\pdfoutput=1
\documentclass[12pt]{article}
\usepackage[letterpaper, left=1.5in, right=1.25in, top=1.25in, bottom=1.25in, includefoot]{geometry}

\usepackage{expl3,xparse}

\usepackage[T1]{fontenc}
\usepackage[type1text]{newtxtext}
\usepackage[bottom]{footmisc}

\usepackage{xcolor}

\usepackage[splitindex]{imakeidx}
\makeindex[intoc=true,options=-s dotfill,columns=2]
\makeindex[intoc=true,name=notation,options=-s dotfill,columns=3,title={Index for Notations}]

\usepackage{titlesec}
\titleformat{\section}
	{\normalfont\Large\bfseries}{\S\,\thesection}{1em}{}

\usepackage[shortlabels]{enumitem}

\usepackage[leqno]{mathtools}
\usepackage{amssymb}
\usepackage{amsthm}

\usepackage[varvw]{newtxmath}
\usepackage[scr=rsfso,cal=euler]{mathalpha}
\providecommand\bm[1]{\boldsymbol{#1}}

\RenewDocumentCommand \subset {} {\subseteq}

\RenewDocumentCommand \supset {} {\supseteq}

\RenewDocumentCommand \le {} { \leqslant }
\RenewDocumentCommand \ge {} { \geqslant }

\ExplSyntaxOn
\cs_set_eq:NN \IfBlankTF \tl_if_blank:nTF
\ExplSyntaxOff

\usepackage{interval}
\usepackage{mleftright}
\DeclarePairedDelimiterX\abs[1]\lvert\rvert
  { \ifblank{#1}{\:\cdot\:}{#1} }
\DeclarePairedDelimiterX\norm[1]\lVert\rVert
  { \ifblank{#1}{\:\cdot\:}{#1} }
\DeclarePairedDelimiterX\variables[1]\lparen\rparen
  { \ifblank{#1}{\:\cdot\:}{#1} }
\DeclarePairedDelimiterX\ceil[1]\lceil\rceil
  { \ifblank{#1}{\:\cdot\:}{#1} }
\DeclarePairedDelimiterX\floor[1]\lfloor\rfloor
  { \ifblank{#1}{\:\cdot\:}{#1} }
\DeclarePairedDelimiterX\bra[1]\langle\vert
  { \ifblank{#1}{\:\cdot\:}{#1} }
\DeclarePairedDelimiterX\ket[1]\vert\rangle
  { \ifblank{#1}{\:\cdot\:}{#1} }
\DeclarePairedDelimiterX\braket[2]\langle\rangle
  {
    \ifblank{#1}{\:\cdot\:}{#1}
    \,\delimsize\vert\,\mathopen{}
    \ifblank{#2}{\:\cdot\:}{#2}
  }
\DeclarePairedDelimiterX\pairing[2]\langle\rangle
  { \ifblank{#1}{\:\cdot\:}{#1}, \ifblank{#2}{\:\cdot\:}{#2} }
\DeclarePairedDelimiterX\inner[2]\lparen\rparen
  { \ifblank{#1}{\:\cdot\:}{#1}, \ifblank{#2}{\:\cdot\:}{#2} }
	
\providecommand\given{}
\newcommand\SetSymbol[1][]
  { \nonscript\:#1\vert\allowbreak\nonscript\:\mathopen{} }
\DeclarePairedDelimiterX\Set[1]\{\}
  { \renewcommand\given{\SetSymbol[\delimsize]}#1 }
\DeclarePairedDelimiterX\GSet[1]\langle\rangle
  { \renewcommand\given{\SetSymbol[\delimsize]}#1 }

\NewDocumentCommand \fun { m e{^_} O{} }
  {%
    \operatorname{#1}%
    \IfValueT{#2}{\sp{#2}}%
    \IfValueT{#3}{\sb{#3}}%
    \IfBlankTF{#4}{}
      {\mleft(#4\mright)}%
  }
\NewDocumentCommand \mfrak { m } { \fun{\mathfrak{#1}} }
\NewDocumentCommand \mcal { m } { \fun{\mathcal{#1}} }
\NewDocumentCommand \mscr { m } { \fun{\mathscr{#1}} }
\NewDocumentCommand \msf { m } { \fun{\mathsf{#1}} }

\NewDocumentCommand \mbb { m } { \fun{\mathbb{#1}} }
\NewDocumentCommand \grp { m } { \msf{#1} }

\NewDocumentCommand \N {} { \mathbb{N} }
\NewDocumentCommand \Z {} { \mathbb{Z} }
\NewDocumentCommand \Q {} { \mathbb{Q} }
\NewDocumentCommand \R {} { \mathbb{R} }
\NewDocumentCommand \C {} { \mathbb{C} }
\NewDocumentCommand \F {} { \mathbb{F} }
\NewDocumentCommand \Aff {} { \mbb{A} }
\NewDocumentCommand \Vect {} { \mbb{V} }
\NewDocumentCommand \Gm {} { \mbb{G}_{\rm m} }
\NewDocumentCommand \Ga {} { \mbb{G}_{\rm a} }

\NewDocumentCommand \Apt {}
  { \fun{\mathscr{A}} }
\NewDocumentCommand \vApt {}
  { \fun{\prescript{v\mkern-7mu\relax}{}{\mathscr{A}}} }
\NewDocumentCommand \Build {}
  { \fun{\mathscr{B}} }
\NewDocumentCommand \vBuild {}
  { \fun{\prescript{v\!}{}{\mathscr{B}}} }
\NewDocumentCommand \vAff {}
  { \fun{\prescript{v\mkern-6mu\relax}{}{\mathbb{A}}} }
\NewDocumentCommand \vPhi {}
  { \fun{\prescript{v}{}{\Phi}} }
\NewDocumentCommand \vSigma {}
  { \fun{\prescript{v}{}{\Sigma}} }
\NewDocumentCommand \vWeyl {}
  { \fun{\prescript{v}{}{\mathit{W}}} }
\NewDocumentCommand \valpha {}
  { \prescript{v\!}{}{\alpha} }
\NewDocumentCommand \vnu {}
  { \prescript{v\!}{}{\nu} }
\NewDocumentCommand \vf {}
  { \prescript{v\mkern-6mu\relax}{}{f} }
\NewDocumentCommand \vC {}
  { \prescript{v\!}{}{C} }
\NewDocumentCommand \vF {}
  { \prescript{v\!}{}{F} }

\NewDocumentCommand \qbinom { m m }
  { \genfrac{[}{]}{0pt}{}{#1}{#2} }
\NewDocumentCommand \vect { m } { \mathbf{#1} }
\NewDocumentCommand \sequence { m O{1} m }
  { \ensuremath{ {#1}_{#2},\cdots,{#1}_{#3} } }

\NewDocumentCommand \odif {} {	\fun{\mathrm{d}}	}%
\NewDocumentCommand \adif {} {	\fun{\Delta}	}%
\NewDocumentCommand \asum {} {	\fun{\Sigma}	}%

\usepackage{tikz}
\usetikzlibrary{cd} 
\usetikzlibrary{arrows.meta}
\usetikzlibrary{backgrounds}
\usetikzlibrary{calc}
\usetikzlibrary{decorations.pathmorphing}
\usetikzlibrary{intersections}
\usetikzlibrary{shapes}
\usetikzlibrary{fit}
\tikzset{%
Venn/.style={
ellipse, draw, color=gray!50,
minimum height=50pt,minimum width=30pt}
}
\usetikzlibrary{braids}

\usepackage{calc}
\usetikzlibrary{shapes.geometric}

\usepackage{dynkin-diagrams}
\pgfkeys{/Dynkin diagram,
  edge length=0.4in,
  fold radius=0.4in
}

\usepackage{adjustbox}

\usepackage{thmtools}

\usepackage[pageanchor]{hyperref}

\usepackage{caption}

\usepackage[capitalize,nameinlink]{cleveref}

\crefname{section}{\S}{\S}
\usepackage{crossreftools}
\crefformat{footnote}{#2\footnotemark[#1]#3}

\hypersetup{%
	colorlinks=true,
	linkcolor=blue,
	citecolor=magenta,
	urlcolor=magenta
}

\numberwithin{equation}{subsection}
\numberwithin{figure}{subsection}
\declaretheorem[numberwithin=section]{theorem}
\counterwithin{theorem}{section} 
\declaretheorem[numberwithin=section]{lemma}
\counterwithin{lemma}{subsection} 
\declaretheorem[sibling=lemma]{proposition}
\counterwithin{proposition}{subsection} 
\declaretheorem[sibling=lemma]{corollary}
\counterwithin{corollary}{subsection} 
\declaretheorem[sibling=lemma,style=definition]{definition}
\counterwithin{definition}{subsection} 
\declaretheorem[sibling=lemma,style=definition]{example}
\counterwithin{example}{subsection} 
\declaretheorem[sibling=lemma,style=definition]{construction}
\counterwithin{construction}{subsection} 

\declaretheorem[sibling=lemma,style=definition,refname={Convention,Conventions}]{convention}
\counterwithin{para}{subsection} 
\declaretheorem[style=remark,numbered=no]{remark}

\let\oldsection\section
\renewcommand{\section}{
  \renewcommand{\theequation}{\thesection.\arabic{equation}}
  \oldsection}
\let\oldsubsection\subsection
\renewcommand{\subsection}{
  \renewcommand{\theequation}{\thesubsection.\arabic{equation}}
  \oldsubsection}

\Crefname{equation}{}{}
\Crefname{para}{}{}

\setlist[enumerate]{itemsep=0mm}

\newlist{thmlist}{enumerate}{1}
\setlist[thmlist]%
  {label=\textup{(\roman{thmlisti})},%
	ref={(\roman{thmlisti})}}%
\Crefname{thmlisti}{Theorem}{Theorems}

\newlist{lemlist}{enumerate}{1}
\setlist[lemlist]%
  {label=\textup{(\roman{lemlisti})},%
	ref={(\roman{lemlisti})}}%
\Crefname{lemlisti}{Lemma}{Lemmas}

\newlist{proplist}{enumerate}{1}
\setlist[proplist]%
  {label=\textup{(\roman{proplisti})},%
	ref={(\roman{proplisti})}}%
\Crefname{proplisti}{Proposition}{Propositions}

\newlist{corolist}{enumerate}{1}
\setlist[corolist]%
  {label=\textup{(\roman{corolisti})},%
	ref={(\roman{corolisti})}}%
\Crefname{corolisti}{Corollary}{Corollaries}

\newlist{paralist}{enumerate}{1}
\setlist[paralist]%
  {label=\textup{(\roman{paralisti})},%
	ref={(\roman{paralisti})}}%
\Crefname{paralisti}{}{}

\newlist{eglist}{enumerate}{1}
\setlist[eglist]%
  {label=\textup{(\alph{eglisti})},%
	ref={(\alph{eglisti})}}%
\Crefname{eglisti}{Example}{Examples}

\makeatletter
\renewcommand{\p@thmlisti}{\perh@ps{\thetheorem}}
\renewcommand{\p@lemlisti}{\perh@ps{\thelemma}}
\renewcommand{\p@proplisti}{\perh@ps{\theproposition}}
\renewcommand{\p@corolisti}{\perh@ps{\thecorollary}}
\renewcommand{\p@eglisti}{\perh@ps{\theexample}}
\renewcommand{\p@paralisti}{\perh@ps{\thepara}}
\protected\def\perh@ps#1#2{\textup{#1#2}}
\newcommand{\itemrefperh@ps}[2]{\textup{#2}}
\newcommand{\itemref}[1]{\begingroup\let\perh@ps\itemrefperh@ps\ref{#1}\endgroup}
\makeatother

\makeatletter
\DeclareRobustCommand{\crefnosort}[1]{%
  \begingroup\@cref@sortfalse\cref{#1}\endgroup
}
\makeatother

\declaretheoremstyle[
	headindent=\parindent
]{itemlike}
\declaretheorem[
	style=itemlike,
	name={},
	numberwithin=subsection
]{step}
\Crefname{step}{}{}

\makeatletter
\renewcommand{\p@step}{%
	\ref@subsection{\arabic{section}}{\arabic{subsection}}%
}
\protected\def\ref@subsection#1#2{%
  \ifnum#1=\value{section}%
		\ifnum#2=\value{subsection}%
		\else
			{\@arabic{#1}.\@arabic{#2}.}%
		\fi
  \else
		{\@arabic{#1}.\@arabic{#2}.}%
  \fi
}
\makeatother

\declaretheorem[
	style=itemlike,
	name={},
	numberwithin=step
]{substep}
\Crefname{substep}{}{}

\allowdisplaybreaks

\crefformat{page}{#2pp.~#1#3}

\usepackage[alphabetic,initials]{amsrefs}

\title{Simplicial volumes in Bruhat-Tits buildings of split classical type}
\author{Xu Gao}

\begin{document}
\maketitle
\begin{abstract}
	In a Bruhat-Tits building of split classical type (that is, of type $A_n$, $B_n$, $C_n$, $D_n$, and any combination of them) over a local field, the \emph{simplicial volume} counts the vertices within the given simplicial distance from a special vertex. 
	
	This paper aims to study the asymptotic growth of the simplicial volume. 
	A formula of the simplicial volume is deduced from the theory of concave functions. Then the dominant term in its asymptotic growth is found using the theory of $q$-exponential polynomials developed in this paper.
\end{abstract}
\tableofcontents

\clearpage
\section{Introduction}\label{sec:Intro}
\index{building}%
\emph{Buildings} are important geometric/combinatorial objects and were first introduced by Jacques Tits (see \cites{Tits74,Bourbaki}) in the 1950s-1960s to study semisimple groups (more generally, reductive groups) over arbitrary fields. %
Later, François Bruhat and Jacques Tits developed a specialized variant to study reductive groups over a non-Archimedean valued field (see \cites{BT-1,BT-2,BT84,BT87} for the original sources). 

\index{path}%
\index{path!length of}%
\index{simplicial distance}%
\index[notation]{d(x,y)@$d(x,y)$}%
A Bruhat-Tits building over a local field is, in particular, a polysimplicial complex. From the viewpoint of incidence geometry, we have a \emph{simplicial distance} on it. 
For any two vertices $x$ and $y$ on a Bruhat-Tits building, a \emph{path} from $x$ to $y$ is a sequence of adjacent vertices $x_{0},x_{1},\cdots,x_{l}$ with $x_{0}=x$ and $x_{l}=y$. 
The number $l$ is called the \emph{length} of the path. 
Then the \emph{simplicial distance} between $x$ and $y$ is the minimum length of a path from $x$ to $y$, and we denote it by $d(x,y)$. %

\index[notation]{Build@$\Build$}%
\index{simplicial ball}%
\index[notation]{B (x,r)@$B(x,r)$}%
\index{simplicial sphere}%
\index[notation]{partial (x,r)@$\partial(x,r)$}%
Let $\Build$ be a Bruhat-Tits building and $x$ a vertex in it. 
The \emph{simplicial ball with center $x$ and radius $r$} is the set of all vertices with simplicial distance at most $r$ from $x$: 
\begin{equation*}
	B(x,r):=
	\Set*{
			y \text{ is a vertex in $\Build$}
			\given 
			d(x,y)\le r
		}.
\end{equation*}
The \emph{simplicial sphere with center $x$ and radius $r$} is the set of all vertices with simplicial distance exactly $r$ from $x$: 
\begin{equation*}
	\partial(x,r):=
	\Set*{
			y \text{ is a vertex in $\Build$}
			\given 
			d(x,y) = r
		}.
\end{equation*}
In a Bruhat-Tits building, a vertex is either special or adjacent to a special one. 
If two vertices $x$ and $y$ are adjacent, then we have 
\begin{equation*}
	B(x,r-1) \subset B(y,r) \subset B(x,r+1).
\end{equation*}
In this sense, we may focus on special vertices.

\index[notation]{o@$o$}%
\index[notation]{B (r)@$B(r)$}%
\index[notation]{partial (r)@$\partial(r)$}%
\index[notation]{simplicial volume@$\fun{SV}[\:\cdot\:]$}%
\index[notation]{simplicial surface area@$\fun{SSA}[\:\cdot\:]$}%
\index{simplicial volume}%
\index{simplicial surface area}%
In the rest of the paper, $o$ will be a fixed special vertex. 
The set $B(o,r)$ will be denoted by $B(r)$ for short, and its cardinality will be denoted by $\fun{SV}[r]$. 
Likewise, the set $\partial(o,r)$ and its cardinality will be denoted by $\partial(r)$ and $\fun{SSA}[r]$ respectively. 
The functions $\fun{SV}[\:\cdot\:]$ and $\fun{SSA}[\:\cdot\:]$ are called 
the \emph{simplicial volume} and the \emph{simplicial surface area} in $\Build$ respectively.

Before moving on, let me explain what does $\fun{SV}[r]$ count in the case where $\Build$ is of split type $A_{n}$. 
Let $V$ be a vector space of dimension $n+1$ over the ground local field $K$. 
Then vertices in $\Build$ can be interpreted as \emph{homothetic} classes of \emph{lattices} in $V$ (see, e.g. \citelist{\cite{BTAG}*{2.22}\cite{BT-1}*{\S 10.2}\cite{BT84}*{1.7}}). 
Fix a lattice $L_{0}$ so that its homothetic class $[L_{0}]$ is taken to be the reference point $o$. 
Then the quantity $\fun{SV}[r]$ counts, up to homotheties, the lattices $L$ between $L_{0}$ and $\varpi^rL_{0}$ (see \cite{Junecue}*{2.1.1}), where $\varpi$ is any uniformizer of $K$.

The purpose of this paper is to analyze the asymptotic growths of the simplicial volume and the simplicial surface area.
Note that 
\begin{equation*}
	\partial(r) = 
	B(r) \setminus B(r-1).
\end{equation*}
Therefore, for sufficiently large $r$, we have 
\begin{equation*}
	C_{1}\cdot\fun{SSA}[r] \le \fun{SV}[r] \le C_{2}\cdot\fun{SSA}[r],
\end{equation*}
where $C_{1}$, $C_{2}$ are positive constants. 
We use the asymptotic notation $\fun{SV}[r]\asymp\fun{SSA}[r]$ to denote this fact.
\index[notation]{asymp@$\asymp$}%

One of the main theorems in this paper is the following.
\begin{theorem}\label{thm:AsymptoticDominanceOfSV}
	\index[notation]{epsilon(n)@$\varepsilon(n)$}%
	\index[notation]{pi(n)@$\pi(n)$}%
	Let $\mathscr{B}$ be an irreducible Bruhat-Tits building of split classical type over a local field $K$ with residue cardinality $q$. 
	Then the simplicial volume $\fun{SV}[\:\cdot\:]$ and the simplicial surface area $\fun{SSA}[\:\cdot\:]$ in it have the following asymptotic dominant relation:
	\begin{equation*}
		\fun{SV}[r] \asymp 
		\fun{SSA}[r] \asymp 
			r^{\varepsilon(n)}q^{\pi(n)r},
	\end{equation*}
	where $\varepsilon(n)$ and $\pi(n)$ are given in the following table.	
	\begin{center}
		\captionsetup{type=table}
		\begin{tabular}{ccc}
			Split type of $\mathscr{B}$ & $\varepsilon(n)$ & $\pi(n)$ \\
			\hline
			$A_{n}$ ($n$ is odd) & $0$ & $(\frac{n+1}{2})^2$ \\ 
			$A_{n}$ ($n$ is even) & $1$ & $\frac{n}{2}(\frac{n}{2}+1)$ \\
			$B_{n}$ ($n=3$) & $0$ & $5$ \\ 
			$B_{n}$ ($n\ge 4$) & $0$ & $\frac{n^2}{2}$ \\
			$C_{n}$ ($n\ge 2$) & $0$ & $\frac{n(n+1)}{2}$ \\
			$D_{n}$ ($n=4$) & $2$ & $6$ \\ 
			$D_{n}$ ($n\ge 5$) & $1$ & $\frac{n(n-1)}{2}$
		\end{tabular}
		\caption{}\label{table:Asymptotics}
	\end{center}
\end{theorem}
This theorem talks about \emph{irreducible} Bruhat-Tits buildings of \emph{split classical types} only. 
But we will see in \cref{subsec:ReduceToIrr} that asymptotic  results for general Bruhat-Tits buildings of split classical types can be deduced from the irreducible ones.
\begin{remark}
	When the Bruhat-Tits building $\mathscr{B}$ is of split type $A_{n}$, this asymptotic dominant relation is given in \cite{Junecue}*{2.1.2}. We refer to \cite{Junecue} for an application of it.
\end{remark}

In order to do asymptotic analysis, we need formulas for the simplicial volume $\fun{SV}[\:\cdot\:]$ and the simplicial surface area $\fun{SSA}[\:\cdot\:]$ in terms of the root system $\Phi$ and the ground local field $K$. This is achieved by the following theorem. 
\begin{theorem}\label{thm:SimplicialVolumeFormula}
	\index{type}%
	\index[notation]{ceil@$\ceil{\:\cdot\:}$}%
	\index[notation]{q@$q$}%
	\index[notation]{Delta@$\Delta$}%
	\index[notation]{I@$I$}%
	\index[notation]{Poincare (Phi,I)@$\mscr{P}_{\Phi;I}$}%
	\index[notation]{C@$\vC$}%
	\index[notation]{B (r,C,I)@$B(r,\vC,I)$}%
	\index[notation]{partial (r,C,I)@$\partial(r,\vC,I)$}%
	Let $\Build$ be a Bruhat-Tits building of split type $\Phi$ over a local field $K$ with residue cardinality $q$. Then the simplicial volume $\fun{SV}[\:\cdot\:]$ and the simplicial surface area $\fun{SSA}[\:\cdot\:]$ in it can be computed by the following formulas:
	\begin{align*}
		\fun{SV}[r] &= 
		\sum_{I\subset\Delta}
		\frac{
			\mscr{P}_{\Phi;I}[q]
		}{
			q^{\fun{deg}[\mscr{P}_{\Phi;I}]}
		}\sum_{x\in B(r,\vC,I)}
		\prod_{a(x)>0}q^{\ceil{a(x)}},\\
		\fun{SSA}[r] &= 
		\sum_{I\subset\Delta}
		\frac{
			\mscr{P}_{\Phi;I}[q]
		}{
			q^{\fun{deg}[\mscr{P}_{\Phi;I}]}
		}\sum_{x\in \partial(r,\vC,I)}
		\prod_{a(x)>0}q^{\ceil{a(x)}},
	\end{align*}
	where 
	\begin{itemize}
		\item $\ceil{\:\cdot\:}$ is the ceiling function, 
		\item $\Delta$ is a basis of the root system $\Phi$, 
		\item $\mscr{P}_{\Phi;I}$ is the Poincar\'{e} polynomial associated to the pair $(\Phi,I)$, 
		\item $\vC$ is a Weyl chamber of $\Phi$, 
		\item and the index sets $B(r,\vC,I)$ (resp. $\partial(r,\vC,I)$) consists of the vertices in $\overline{o+\vC}$ having type $I$ with simplicial distance at most $r$ (resp. exactly $r$) from $o$.
	\end{itemize}
\end{theorem}

In order to apply the formulas, we need to find explicit descriptions of the index sets $B(r,\vC,I)$ and $\partial(r,\vC,I)$. The key step is the following characterization of the simplicial distance.
\begin{theorem}\label{thm:SimplicialDistance}
	Let $\Apt$ be an irreducible affine apartment of split classical type $\Phi$ and $o$ a fixed special vertex in it. Let $\vC$ be a Weyl chamber of $\Phi$ and $a_{0}$ the highest root relative to $\vC$. 
	Then, for a vertex $x$ in $\overline{o+\vC}$, we have:
	\begin{equation*}
		d(x,o)\le r \iff a_{0}(x-o)\le r.
	\end{equation*}
\end{theorem}
\begin{remark}
	While the formulas in \cref{thm:SimplicialVolumeFormula} apply to general Bruhat-Tits building of split type, \cref{thm:SimplicialDistance} only holds when $\Phi$ is of classical type. 
	This is the main reason why this paper focuses on classical type.
\end{remark}

Once the explicit descriptions of the index sets are obtained, we can immediately see that each formula in \cref{thm:SimplicialVolumeFormula} can be expanded into a finite linear combination of multi-summations of the form 
\begin{equation*}
	\sum_{c_{1},\cdots,c_{t}}
		q^{	L(c_{1},\cdots,c_{t})+e(c_{1},\cdots,c_{t})	},
\end{equation*}
where $L(c_{1},\cdots,c_{t})$ is a linear form of the variables $c_{1},\cdots,c_{t}$ and $e(c_{1},\cdots,c_{t})$ is a parity function of $c_{1},\cdots,c_{t}$. 
In order to handle such multi-summations, the notion of \emph{(super) $q$-exponential polynomials} is introduced and studied.

Then we are able to prove \cref{thm:AsymptoticDominanceOfSV} and the following improvement. 
\begin{theorem}\label{thm:AsymptoticGrowthOfSV}
	Notations are as in \cref{thm:AsymptoticDominanceOfSV}.
	\begin{thmlist}
		\item Suppose $\mathscr{B}$ is of split type $A_{n}$, $C_{n}$, $B_{3}$, or $D_{4}$. Then the simplicial volume $\fun{SV}[\:\cdot\:]$ in it has the following asymptotic growth as $r\to\infty$:
		\begin{equation*}
			\fun{SV}[r]\sim
				\tilde{C}(n)\cdot r^{\varepsilon(n)}q^{\pi(n)r},
		\end{equation*}
		where $\tilde{C}(n)$ is a positive number that is a rational function of $q$. 
		Similarly, the simplicial surface area $\fun{SSA}[\:\cdot\:]$ has the following asymptotic growth as $r\to\infty$:
		\begin{equation*}
			\fun{SSA}[r]\sim
				{C}(n)\cdot r^{\varepsilon(n)}q^{\pi(n)r},
		\end{equation*}
		where ${C}(n)$ is a positive number that is a rational function of $q$.  
		\item Suppose $\mathscr{B}$ is of split type $B_{n}$ ($n\ge 4$) or $D_{n}$ ($n\ge 5$). Then the simplicial volume $\fun{SV}[\:\cdot\:]$ in it has the following asymptotic growth as $r\to\infty$:
		\begin{align*}
			\fun{SV}[2r]&\sim
				\tilde{C}_{0}(n)\cdot r^{\varepsilon(n)}q^{2\pi(n)r},\\
			\fun{SV}[2r+1]&\sim
				\tilde{C}_{1}(n)\cdot r^{\varepsilon(n)}q^{2\pi(n)r},
		\end{align*}
		where $\tilde{C}_{0}(n)$ and $\tilde{C}_{1}(n)$ are positive numbers that are rational functions of $q$. 
		Similarly, the simplicial surface area $\fun{SSA}[\:\cdot\:]$ in $\Build$ has the following asymptotic growth as $r\to\infty$:
		\begin{align*}
			\fun{SSA}[2r]&\sim
				{C}_{0}(n)\cdot r^{\varepsilon(n)}q^{2\pi(n)r},\\
			\fun{SSA}[2r+1]&\sim
				{C}_{1}(n)\cdot r^{\varepsilon(n)}q^{2\pi(n)r},
		\end{align*}
		where ${C}_{0}(n)$ and ${C}_{1}(n)$ are positive numbers that are rational functions of $q$. 
	\end{thmlist}
\end{theorem}
Indeed, we will prove stronger results (see \cref{thm:Asymp:An,thm:Asymp:Cn,thm:Asymp:CnSp,thm:Asymp:Bn,thm:Asymp:BnSp}) and give explicit formulas for the involved constants.

\paragraph*{Plan}
This paper is organized as follows. 

In \cref{sec:Preliminaries}, we will review the theory of Bruhat-Tits buildings and fix conventions. 

In \cref{sec:Formula}, we will prove the formulas of the simplicial volume and the simplicity surface area shown in \cref{thm:SimplicialVolumeFormula}.

In \cref{sec:Vertices}, we will deduce explicit descriptions of the index sets $B(r,\vC,I)$ and $\partial(r,\vC,I)$ from well-known results on root systems.

In \cref{sec:Asymptotic}, we will introduce the notion of \emph{(super) $q$-exponential polynomials} and prepare for the second half of this paper. 

In \cref{sec:An,sec:Cn,sec:Bn,sec:Dn}, we will study the asymptotic behaviors of the simplicial volume and simplicity surface area in Bruhat-Tits buildings of split type $A_{n}$, $B_{n}$, $C_{n}$, and $D_{n}$.

\clearpage
\section{Preliminaries and notations}\label{sec:Preliminaries}
This section aims to explain the terminology and notation used in this paper.

\subsection{Buildings and apartments}
Let us begin with general notions of (Euclidean) buildings and apartments. We refer to \cite{Rousseau} and \cite{Bourbaki}*{chap.V, \S 3} for details.

\index[notation]{Aff@$\Aff$}%
\index[notation]{Aff vectorial@$\vAff$}%
\index{reflection}%
\index[notation]{r(H)@$r_{H}$}%
\index[notation]{Weyl@$W$}%
\index[notation]{Weyl vectorial@$\vWeyl$}%
\index{reflection group}
The starting point is a Euclidean affine space $\Aff$. 
Its associated vector space is denoted by $\vAff$. 
A \emph{reflection} on $\Aff$ is an affine isometry whose fixed-point set is a hyperplane. 
Note that the reflections $r_{H}$ on $\Aff$ and the hyperplanes $H$ in $\Aff$ are one-one corresponding to each other. 
Let $W$ be a group of affine isometries generated by reflections and $\vWeyl$ the group of their vectorial parts. 
Then $W$ is called a \emph{reflection group} if $\vWeyl$ is finite. 

\begin{definition}
	\index[notation]{Apt@$\Apt$}%
	\index{apartment}%
	\index{apartment!Euclidean}%
	A \emph{(Euclidean) apartment} $\Apt$ is a Euclidean affine space $\Aff$ equipped with a reflection group $W$ (called its \emph{Weyl group}) on it.
\end{definition}

\index{apartment!irreducible}%
\index{apartment!essential}%
\index{apartment!trivial}%
An apartment $\Apt$ is \emph{irreducible} (resp. \emph{essential}, \emph{trivial}, etc.), and its Weyl group $W$ is said to be so if $\vWeyl$ acts irreducibly (resp. essentially, trivially, etc.) on $\vAff$. 

\index{wall}%
\index{facet}%
\index{cover}%
\index{face}%
\index{polysimplicial complex}%
The hyperplanes corresponding to the reflections in the Weyl group $W$ are called the \emph{walls} of the apartment $\Apt$. 
They decompose the Euclidean space $\Aff$ into cells, called \emph{facets}. We say a facet $F$ \emph{covers} another facet $F'$, or $F'$ is a \emph{face} of $F$, if the closure of $F$ contains $F'$. 
In this way, we can view an apartment $\Apt$ as a \emph{polysimplicial complex}, that is a cellular complex which is a product of simplicial complexes. 

\begin{definition}\label{def:Euclidean_buidings}
	\index{building}%
	\index{building!Euclidean}%
	\index[notation]{Build@$\Build$}%
	\index[notation]{Apts@$\mcal{A}$}%
	\index[notation]{Weyl@$W$}%
	\index{Weyl group}%
	Let $\Apt$ be an apartment. 
	A \emph{(Euclidean) building of type $\Apt$} is a polysimplicial complex $\Build$ equipped with a family $\mcal{A}$ of subcomplexes, such that the following axioms are satisfied. 
  \begin{enumerate}[
			label=\textbf{EB\arabic*.},
			start=0,
			align=left
		]
    \item Every $A\in\mcal{A}$ is isomorphic to $\Apt$.
    \item For any two cells $F$ and $F'$, there is an $A\in\mcal{A}$ containing them.
    \item If $A$ and $A'$ are two members of $\mcal{A}$ containing both $F$ and $F'$, then there is an isomorphism between $A$ and $A'$ fixing $F$ and $F'$ pointwise.
  \end{enumerate} 
	Members of $\mcal{A}$ are called \emph{apartments in $\Build$}. Cells in $\Build$ are called \emph{facets}. 
	The Weyl group $W$ of $\Apt$ is also called the \emph{Weyl group} of $\Build$. 
\end{definition}
\begin{remark}
	An apartment $\Apt$ is a building, where the family $\mcal{A}$ consists of exactly one member, namely $\Apt$ itself.
\end{remark}

\index{building!irreducible}%
\index{building!essential}%
\index{building!trivial}%
\index{vertex}%
\index[notation]{V@$\mcal{V}$}%
A building $\Build$ of type $\Apt$ is \emph{irreducible} (resp. \emph{essential}, \emph{trivial}, etc.) if $\Apt$ is. 
From now on, \emph{all buildings are assumed to be essential}. 
Then the minimal facets have dimension $0$ and are called \emph{vertices}. 
The set of vertices is denoted by $\mcal{V}$. 

\index{translation group}%
\index{chamber}%
\index{alcove}%
The kernel of $W\to\vWeyl$ is called the \emph{translation group}. It is either discrete or dense. 
From now on, \emph{all translation groups are assumed to be discrete}. 
Then the maximal facets are the connected components of the complement of the union of all walls. 
They are called \emph{chambers} if the translation group is trivial and \emph{alcoves} if not.

\index{Weyl group!linear}%
\index{building!vectorial}%
\index{Weyl group!affine}%
\index{building!affine}%
\index{building!spherical}%
If $W\cong\vWeyl$, we say $W$ is \emph{linear} and the building is a \emph{vectorial building}. 
Otherwise, we say $W$ is \emph{affine} and the building is an \emph{affine building}. 
Note that the combinatorial information (the family of apartments and the poset of facets) in a vectorial building can be read from its unit sphere. 
We use the term \emph{spherical building} to refer to any cellular complex (equipped with a family of subcomplexes) encoding the same combinatorial information as a vectorial building. 

\index[notation]{Apt vectorial@$\vApt$}%
\index{building!spherical!at $x$}%
\index[notation]{Build (x)@$\Build_{x}$}%
\index{special vertex}%
Any apartment $\Apt=(\Aff,W)$ admits a vectorial apartment $(\vAff,\vWeyl)$, denoted by $\vApt$. 
However, there is no similar construction for affine buildings since the underlying space of a building $\Build$ is not necessarily an affine space.
Instead, we can consider the subcomplex of facets covering a given point $x$. 
This complex, equipped with the family of subcomplexes inheriting from $\Build$, is not necessarily a Euclidean building but a spherical building. 
This spherical building is called the \emph{spherical building at $x$}, denoted by $\Build_{x}$. 
A point $x\in\Build$ is called \emph{special} if $\Build_{x}$ encodes the same combinatorial information as $\vBuild$. 
This is the case if and only if the stabilizer $W_{x}$ of $x$ in the Weyl group $W$ is isomorphic to $\vWeyl$ through the composition $W_{x}\hookrightarrow W\to\vWeyl$. 
It is clear that a special point must be a vertex.

\subsection{Affine roots}\label{subsec:Affine_roots}
To explain what is \emph{a building of split classical type}, we discuss various notions of \emph{roots} in this subsection. We refer to \cite{BT-1}*{\S 1} and \cite{Bourbaki}*{chap.VI, \S 1-2} for details.

\index{root!affine}%
\index[notation]{alpha@$\alpha$}%
\index[notation]{alpha partial@$\partial\alpha$}%
\index{root system!affine}%
\index[notation]{Sigma@$\Sigma$}%
\index[notation]{r(alpha)@$r_{\alpha}$}%
Let $\Apt=(\Aff,W)$ be an apartment. 
An \emph{affine root} in it is a closed half-space $\alpha$ of $\Aff$ bounded by a wall. This wall is denoted by $\partial\alpha$. 
The set of all affine roots is called the \emph{affine root system of $\Apt$} and is denoted by $\Sigma$. 
It is clear that the affine root system $\Sigma$ determines the apartment $\Apt$: the Weyl group $W$ is generated by the reflections $r_{\alpha}$ with respect to the walls $\partial\alpha$ of affine roots $\alpha\in\Sigma$. 
Then we say $(\Aff,\Sigma)$ is an apartment by an abuse of language.

\index[notation]{alpha (f)@$\alpha_{f}$}%
Let $f$ be an affine function on $\Aff$. 
We use the notation $\alpha_{f}$ to denote the closed half-space \emph{defined} by it: 
\begin{equation*}
	\alpha_{f}:=\Set*{ x\in\Aff \given f(x) \ge 0 }.
\end{equation*}
Up to a positive scale, the affine function $f$ defining a given closed half-space is unique. 
This allows us to use affine functions to talk about affine roots.

\index{root system!vectorial}%
\index[notation]{Sigma vectorial@$\vSigma$}%
\index{root!vectorial}%
\index{vectorial part}%
\index[notation]{alpha vectorial@$\valpha$}%
The affine root system of the vectorial apartment $\vApt$ of $\Apt$ is called its \emph{vectorial root system} and is denoted by $\vSigma$. 
Members of $\vSigma$ are called \emph{vectorial roots}. 
For an affine root $\alpha$ in $\Apt$, there is a unique vectorial root $\valpha\in\vSigma$ associated to it: if $f$ is an affine function defining $\alpha$, then $\valpha$ is defined by the vectorial part $\vf$ of $f$. 
This vectorial root is called the \emph{vectorial part} of $\alpha$. 

\begin{definition}\label{def:root_system}
	\index[notation]{Vect@$\Vect$}%
	\index[notation]{Vect dual@$\Vect^{\ast}$}%
	\index[notation]{Phi@$\Phi$}%
	\index{root system}%
	\index[notation]{r(a)@$r_{a}$}%
	\index{root system!reduced}%
	Let $\Vect$ be a Euclidean vector space and $\Vect^{\ast}$ its dual space identified with $\Vect$ through the inner product. 
  A finite spanning subset $\Phi\subset \Vect^{\ast}\setminus\Set*{0}$ is called a \emph{root system} on $\Vect$ if 
  \begin{enumerate}[
			label=\textbf{RS\arabic*.},
			align=left
		]
    \item for any $a\in\Phi$, the reflection $r_{a}$ with respect to $\fun{Ker}[a]$ leaves $\Phi$ stable;
    \item for any $a,b\in\Phi$, $r_{a}(b)-b\in\Z a$.
  \end{enumerate}
	A root system $\Phi$ is \emph{reduced} if
  \begin{enumerate}[
			label=\textbf{RS\arabic*.},
			resume,
			align=left
		]
    \item for any $a\in\Phi$, $\R a\cap\Phi=\Set*{\pm a}$.
  \end{enumerate}

	\index{root}%
	\index{coroot}%
	\index[notation]{Phi check@$\Phi^{\vee}$}%
	\index{root subsystem}%
	\index{root system!closed subset of}%
  Elements of ${\Phi}$ are called \emph{roots} in $\Phi$. 
  For a root $a\in\Phi$, the vector $a^{\vee}$ orthogonal to $\fun{Ker}[a]$ satisfying $a(a^{\vee})=2$ is called its \emph{coroot}. 
	The coroots form a root system ${\Phi^{\vee}}$ on $\Vect^{\ast}$.  
	A subset $\Psi\subset\Phi$ is called a \emph{root subsystem} if for any $a\in\Psi$, $r_{a}(\Psi)=\Psi$, and is said to be \emph{closed} if for any $a,b\in\Psi$ such that $a+b$ is a root, $a+b\in\Psi$. 
\end{definition}
The following construction shows the relation between a vectorial root system and a root system.
\begin{construction}
	\index{Weyl group}%
	\index[notation]{Weyl vectorial (Phi)@$\vWeyl[\Phi]$}%
	\index[notation]{alpha vectorial@$\valpha_{a}$}%
	Let $(\Phi,\Vect)$ be a root system with its underlying Euclidean vector space. 
	The reflections $(r_{a})_{a\in\Phi}$ generates a linear reflection group $\vWeyl[\Phi]$ on $\Vect$, called the \emph{Weyl group} of this root system. 
	Then $(\Vect,\vWeyl[\Phi])$ is a vectorial apartment whose vectorial root system $\vSigma$ is $\Set*{	\valpha_{a}	\given	a\in\Phi	}$, where $\valpha_{a}:=\Set*{	\vect{v}\in\Vect	\given	a(\vect{v}) \ge 0	}$. 

	Note that: 1, the map $\Phi\to\vSigma\colon a\mapsto\valpha_{a}$ is injective if and only if $\Phi$ is reduced; 2, non-isomorphic root systems may give isomorphic vectorial root systems. 
\end{construction}

\index{echelonnage@\'{e}chelonnage}%
\index[notation]{E@$\mscr{E}$}%
Let $\Apt=(\Aff,\Sigma)$ be an affine apartment and $\Phi$ a root system on $\vAff$. 
An \emph{\'{e}chelonnage} of $\Phi$ by $\Sigma$ is a correspondence $\mscr{E}\subset\Phi\times\Sigma$ such that $(a,\alpha)\in\mscr{E}$ implies $\valpha_{a}=\valpha$ and that $\mscr{E}$ is stable under the obvious action of $W$. 
Note that if such an \'{e}chelonnage exists, then the Weyl group of $\Phi$ is the same as $\vWeyl$. So an \'{e}chelonnage $\mscr{E}$ tells us how the affine root system $\Sigma$ is related to a root system $\Phi$. 

\index{root!affine!expression of}%
\index[notation]{a plus k@$a+k$}%
\index[notation]{Gamma (a)@$\Gamma_{a}$}%
Now, suppose we have an \'{e}chelonnage $\mscr{E}$. Then, for any $(a,\alpha)\in\mscr{E}$, there is a unique affine function $f$ on $\Aff$ defining $\alpha$ and having vectorial part $a$. 
If we fix a reference point $o$ in $\Aff$, this function can be written as $f(x)=a(x-o)+k$ for some $k\in\R$. 
In this sense, we say $f$ (and $\alpha$) has \emph{expression $a+k$}. 
If the reference point $o$ is implied, we simply use $a+k$ to denote this affine function. 
Let $\Gamma_{a}:=\Set*{ k\in\R \given \alpha_{a+k}\in\Sigma }$.
Then an \'{e}chelonnage $\mscr{E}$ of $\Phi$ by $\Sigma$ is the same datum as an assignment $a\in\Phi\mapsto\Gamma_{a}$. Such a datum characterizes the extra information of a Bruhat-Tits building (ref. \cref{subsec:BT-building}) respecting its valuations. 
Hence, we make the following definition.
\begin{definition}	
	\index{building!affine!of type $\mscr{E}$}%
	An \emph{affine building of type $\mscr{E}$} is an affine building equipped with the \'{e}chelonnage $\mscr{E}$. 
\end{definition}

\index[notation]{Sigma (x)@$\Sigma_{x}$}%
\index[notation]{Phi (x)@$\Phi_{x}$}%
\index{root subsystem!at $x$}%
Let $x$ be a point in $\Aff$. 
Let $\Sigma_{x}$ be the set of affine roots $\alpha$ such that $x\in\partial\alpha$. 
Then, through the \'{e}chelonnage $\mscr{E}$, it corresponds to the following subset of $\Phi$: 
\begin{equation*}
	\Phi_{x}	:=	
		\Set*{ a\in\Phi \given (a,\alpha)\in\mscr{E}\text{ for some }\alpha\in\Sigma_{x} }.
\end{equation*}
This subset a closed root subsystem of $\Phi$ and is called the \emph{root subsystem of $\Phi$ at $x$}.  
It is clear that $x$ is special if and only if $\Phi_{x}=\Phi$. 

\begin{construction}\label{eg:apt_split_type}
	\index[notation]{Apt (Phi,Gamma)@$\Apt[\Phi,\Gamma]$}%
	\index[notation]{Apt (Phi)@$\Apt[\Phi]$}%
	\index[notation]{E (Phi)@$\mscr{E}[\Phi]$}%
	\index{building!affine!of split type $\Phi$}%
	Let $(\Phi,\Vect)$ be a reduced root system with its underlying Euclidean vector space and $\Gamma$ a discrete subgroup of $(\R,+)$. 
	Let $\Sigma:=\Set*{ \alpha_{a+k} \given a\in\Phi, k\in\Gamma }$. Then it is an affine root system on $\Vect$, viewed as a Euclidean affine space. 
	We thus obtain an affine apartment $\Apt[\Phi,\Gamma]$. 
	We simply denote it by $\Apt[\Phi]$ if the discrete subgroup $\Gamma$ is implied. 
	Then there is an obvious \'{e}chelonnage: 
	\begin{equation*}
		\mscr{E}[\Phi]	:=
			\Set*{ (a,\alpha_{a+k})\in\Phi\times\Sigma \given k\in\Gamma }.
	\end{equation*} 
	An affine building of type $\mscr{E}[\Phi]$ is also said to be \emph{of split type $\Phi$}. 
\end{construction}

\subsection{Types and colors}
The purpose of is subsection is to explain the usage of the terms \emph{type} and \emph{color} in this paper, their relation to the root system, and relevant conventions.

\index{type function}%
\index{cotype function}%
\index[notation]{tau@$\tau$}%
Let $\Apt=(\Aff,W)$ be an apartment. 
A \emph{type function} (resp. \emph{cotype function}) on $\Apt$ is a strictly order-reversing (resp. order-preserving) map $\tau$ from the complex of facets to a power set $(\mscr{P}[\mfrak{I}],\subset)$ mapping maximal facets to $\emptyset$ (resp. $\mfrak{I}$) and is \emph{$W$-stable}: for any facet $F$ and any $w\in W$, $\tau(F)=\tau(w.F)$. 

\index{panel}%
\index[notation]{PC@$\mcal{P}_{C}$}%
\index[notation]{VC@$\mcal{V}_{C}$}%
The closure $\overline{C}$ of a maximal facet $C$ in $\Apt$ is a fundamental domain of $W$ in $\Aff$ (see e.g. \cite{Bourbaki}*{chap.V, \S 3, no.3, thm.2}). 
Hence, any facet is transformed by $W$ to a unique face of $C$. 
Let $\mcal{P}_{C}$ be the set of \emph{panels}, namely maximal proper faces of $C$. Then a type function is determined by a bijection $\mcal{P}_{C}\to \mfrak{I}$. 
Similarly, let $\mcal{V}_{C}$ be the set of \emph{vertices} covered by $C$. Then a cotype function is determined by a bijection $\mcal{V}_{C}\to \mfrak{I}$. 

\begin{construction}[\cite{Bourbaki}*{chap.VI, \S 1}]\label{eg:Weyl_chamber}
	\index{system of positive roots}%
	\index{root!positive}%
	\index[notation]{Phi positive@$\Phi^{+}$}%
	\index{root!negative}%
	\index[notation]{Phi negative@$\Phi^{-}$}%
	\index{root!simple}%
	\index{root system!basis of}%
	\index[notation]{Delta@$\Delta$}%
	Let $(\Phi,\Vect)$ be a root system with its underlying Euclidean vector space. 
	A closed subset $\Phi^{+}$ of $\Phi$ is called a \emph{system of positive roots} if $\Phi\setminus\Phi^{+}=-\Phi^{+}$. 
	The set $-\Phi^{+}$ is called the system of \emph{negative roots} and is denoted by $\Phi^{-}$. 
	A positive root is called a \emph{simple root} if it cannot be written as the sum of two positive roots. 
	The set $\Delta$ of simple roots forms a \emph{basis} of $\Phi$ in the sense that any root is a $\Z$-linear combination of simple roots with either all non-negative or all non-positive coefficients.

	Let $\Delta$ be a basis of $\Phi$. 
	Then the following set is a vectorial chamber: 
	\index[notation]{C@$\vC$}%
	\begin{equation*}
		\vC	:=
			\Set*{
				\vect{v}\in \Vect
				\given 
				a(\vect{v})>0\text{ for all }a\in\Delta
			}.
	\end{equation*}
	It is called the \emph{Weyl chamber} associated to $\Delta$. 
	\index{Weyl chamber}%
  Conversely, let $\vC$ be a vectorial chamber. 
	Then, for any $\vect{v}\in \vC$, consider the following subset of $\Phi$:
	\begin{equation*}
		\Phi^{+}:=\Set*{a\in\Phi\given a(\vect{v})>0}.
	\end{equation*}
	It forms a system of positive roots and is independent of the choice of $\vect{v}$. Then one can obtain a basis $\Delta$ by taking the simple roots. 
	There is a more geometric description: their null-sets $\Set*{\fun{Ker}[a]}_{a\in\Delta}$ are precisely the walls enclosing $\vC$. 

	\index[notation]{Phi (I)@$\Phi_{I}$}%
	\index[notation]{Psi (I)@$\Psi_{I}$}%
	\index{root system!parabolic subset of}%
	Now, we have a bijection $\mcal{P}_{\vC}\to\Delta$, mapping each panel $P$ to the simple root $a\in\Delta$ whose null-set $\fun{Ker}[a]$ contains $P$. 
	This defines a type function, for which a \emph{type} is a subset $I$ of $\Delta$. 
	Let $\Phi_{I}$ denote the root subsystem of $\Phi$ generated by $I$. 
	Let $\Psi_{I}=\Phi_{I}\cup\Phi^{+}$. 
	Then subset $\Psi_{I}$ has the property that $\Psi_{I}\cup(-\Psi_{I})=\Phi$ and is closed. 
	Such a subset of $\Phi$ is said to be \emph{parabolic}. 
	Then we have a bijection from the parabolic subsets of $\Phi$ containing $\Phi^{+}$ to the types. 
\end{construction}
\begin{convention}\label{con:type}
	\index{type}%
	\index[notation]{t (I)@$t_{I}$}%
	\index[notation]{l i (I)@$\ell_{i}(I)$}%
	\index[notation]{l i@$\ell_{i}$}%
	\index[notation]{l 0@$\ell_{0}$}%
	Given a basis $\Delta=\Set*{a_{1},\cdots,a_{n}}$, a \emph{type} is a subset of $\Delta$, and is identified with a subset of $\Set*{1,\cdots,n}$. 
	For a type $I$ of $\Delta$, 
	we use $t_{I}$ to denote the cardinality of $\Delta\setminus I$ and $\ell_{i}(I)$ ($1\le i\le t_{I}$) the $i$-th index in $\Delta\setminus I$. 
	We use the convention that $\ell_{0}=0$.
	We will omit $I$ if there is no ambiguity. 
\end{convention}
\begin{construction}[\cite{Bourbaki}*{chap.VI, \S 1-2}]\label{eg:Fund_alcove}
	\index[notation]{a 0@$a_{0}$}%
	\index{root!highest}%
	Let $(\Phi,\Vect)$ be a reduced root system with its underlying Euclidean vector space. Suppose $\Phi$ is \emph{irreducible}, namely it cannot be written as the union of two proper subsets such that they are orthogonal to each other. 
	Let $(\vC,\Phi^{+},\Delta)$ be a triple of a Weyl chamber, a system of positive roots, and a basis of $\Phi$ as in \cref{eg:Weyl_chamber}. 
	Then there is a unique root $a_{0}$ having the largest coefficients. This $a_{0}$ is called the \emph{highest root} relative to this triple. 
	
	\index[notation]{a i@$a_{i}$}%
	\index[notation]{alpha i@$\alpha_{i}$}%
	\index[notation]{alpha 0@$\alpha_{0}$}%
	\index[notation]{alpha plus@$\alpha_{+}$}%
	\index{root system!affine!basis of}%
	\index[notation]{Delta tilde@$\widetilde{\Delta}$}%
	\index[notation]{D@$\mscr{D}$}%
	\index[notation]{C@$C$}%
	\index{fundamental alcove}%
	Let $\Apt[\Phi]$ be the affine apartment of split type $\Phi$ and fix a special vertex $o$ in it as the reference point. 
	Suppose $\Delta=\Set*{a_{1},\cdots,a_{n}}$. 
	Let $\alpha_{i}:=\alpha_{a_{i}+0}$ ($1\le i\le n$) and $\alpha_{0}:=\alpha_{-a_{0}+}$, where $\alpha_{-a_{0}+}$ is the intersection of all $\alpha_{-a_{0}+k}$ with $k\in\Gamma$ and $k>0$. 
	Then $\Set*{\alpha_{0},\alpha_{1},\cdots,\alpha_{n}}$ form a \emph{basis} $\widetilde{\Delta}$ of the affine root system $\Sigma$ in the sense that the intersection $\mscr{D}$ of its members is a fundamental domain of $W$ in $\Aff$. 
	The interior $C$ of $\mscr{D}$ is an alcove, called the \emph{fundamental alcove} associated to the basis $\widetilde{\Delta}$. 
	Conversely, let $C$ be an alcove in $\Apt[\Phi]$. 
	Then the affine roots containing $C$ whose boundary avoids exactly one extreme point of $C$ form a basis of $\Sigma$. 

	Now, we have a bijection $\mcal{V}_{C}\to\widetilde{\Delta}$, mapping each vertex $v$ to the affine root $\alpha\in\widetilde{\Delta}$ whose boundary avoids $v$. 
	This defines a cotype function. 
\end{construction}
\begin{convention}\label{con:color}
	\index[notation]{h i@$h_{i}$}%
	\index{color}%
	Given a basis $\Delta=\Set*{a_{1},\cdots,a_{n}}$ of an irreducible $\Phi$, the highest root relative to it is denoted by $a_{0}$ and the coefficients are denoted by $h_{1},\cdots,h_{n}$, namely
	\begin{equation}\label{eq:highest_root}
		a_{0}=h_{1}a_{1}+\cdots+h_{n}a_{n}.
	\end{equation}
	Given a basis $\widetilde{\Delta}=\Set*{\alpha_{0},\alpha_{1},\cdots,\alpha_{n}}$ of $\Sigma$, a \emph{color} of a vertex $v$ is the index $i$ ($0\le i \le n$) such that $v$ is mapped to $\alpha_{i}$ through the cotype function in \cref{eg:Fund_alcove}. 
\end{convention}

\subsection{Reductive groups}
In this subsection, we will review the notion of reductive groups and recall how to exhibit a root system in a reductive group. 
Then we introduce the essential of Tits buildings. 
We refer to \cite{Milne} for reductive groups over a field and to \cite{SGA3} for reductive group schemes over a general base. 

\index[notation]{K@$K$}%
\index[notation]{K algebraic closure@$K^{\rm a}$}%
In what follows, ${K}$ is a field and $K^{\rm a}$ is an algebraic closure of it.

\begin{definition}\label{def:alg_grp}
	\index{group scheme}%
	\index{algebraic group}%
	A \emph{group scheme} means a group object in the category of schemes. An \emph{algebraic group} (defined over $K$) is then a group scheme of finite type over $K$. 
\end{definition}

\index[notation]{G@$\grp{G}$}%
\index[notation]{G (R) base change@$\grp{G}_R$}%
\index[notation]{G (R)@$\grp{G}[R]$}%
\index[notation]{G grp@$G$}%
\index[notation]{G grp (R)@$G_{R}$}%
We will use bold letters like $\grp{G}$ to denote algebraic groups defined over $K$. 
For any $K$-algebra $R$, the group scheme obtained by base change $\grp{G}\otimes_KR$ is denoted by $\grp{G}_R$ and the group of $R$-points is denoted by $\grp{G}[R]$.  
Moreover, $\grp{G}[K]$ is simply denoted by $G$ and $\grp{G}_{R}[R]\cong\grp{G}[R]$ is simply denoted by $G_{R}$. 

\begin{definition}
	\index{algebraic group!linear}%
	A \emph{linear algebraic group} is a closed algebraic subgroup of the \emph{general linear group} $\grp{GL}[V]$ for some finite-dimensional vector space $V$ over $K$.
	Linear algebraic groups are precisely the affine algebraic groups (see e.g. \cite{Milne}*{1.43 and 4.10}). 
\end{definition}

\begin{definition}
	\index{algebraic group!solvable}%
	\index{algebraic group!radical of}%
	\index[notation]{radical (G)@$\mscr{R}[\grp{G}]$}%
	An algebraic group is \emph{solvable} if it admits a subnormal series with Abelian factors.
	Let $\grp{G}$ be a smooth connected linear algebraic group.
	Then there is a largest smooth connected solvable norm subgroup 
	$\mscr{R}[\grp{G}]$ of $\grp{G}$ (see e.g. \cite{Milne}*{6.44}). 
	It is called the \emph{radical} of $\grp{G}$.
\end{definition}

\begin{definition}
	\index{algebraic group!unipotent}%
	\index{algebraic group!unipotent radical of}%
	\index[notation]{radical u (G)@$\mscr{R}_{u}[\grp{G}]$}%
	An algebraic group is \emph{unipotent} if every nonzero linear representation of it has a nonzero fixed vector.
	Let $\grp{G}$ be a smooth connected linear algebraic group.
	Then there is a largest smooth connected unipotent norm subgroup 
	$\mscr{R}_{u}[\grp{G}]$ of $\grp{G}$ (see e.g. \cite{Milne}*{6.46}). 
	It is called the \emph{unipotent radical} of $\grp{G}$.
\end{definition}

\begin{definition}
	\index{reductive group}%
	An algebraic group $\grp{G}$ is \emph{reductive} (resp. \emph{semisimple}) if its \emph{geometric unipotent radical} $\mscr{R}_{u}[\grp{G}_{K^{\rm a}}]$ (resp. \emph{geometric radical} $\mscr{R}[\grp{G}_{K^{\rm a}}]$) is trivial.
\end{definition}

\begin{definition}
	\index{split reductive group}%
	\index[notation]{T@$\grp{T}$}%
	\index[notation]{GT@$(\grp{G},\grp{T})$}%
	\index{reductive group!splittable}%
	A \emph{split reductive group} is a pair $(\grp{G},\grp{T})$ of a reductive group and a split maximal torus in it.
	If such a pair exists, we say $\grp{G}$ is \emph{splittable}.
\end{definition}

Now, we show how to exhibit a root system in a split reductive group.
\begin{definition}\label{def:Weyl_group}
	\index[notation]{g Lie@$\mfrak{g}$}%
	\index{split character group}%
	\index[notation]{Xgrp (T)@$\grp{X}[\grp{T}]$}%
	\index[notation]{g Lie (a)@$\mfrak{g}_{a}$}%
	\index{split reductive group!root of}%
	\index{split reductive group!root system of}%
	\index[notation]{Phi (GT)@$\fun{\Phi}[\grp{G},\grp{T}]$}%
	\index[notation]{a radical ray@$(a)$}%
	\index{root!radical ray of}%
	\index[notation]{Phi vectorial (GT)@$\fun{\vPhi}[\grp{G},\grp{T}]$}%
	Let $\grp{G}$ be a smooth connected linear algebraic group and $\grp{T}$ a split torus in it. 
	Then $\grp{T}$ acts diagonalizably (via the adjoint representation) on the Lie algebra $\mfrak{g}$ of $\grp{G}$. Therefore, we have a decomposition: 
  \begin{equation}\label{eq:root_space_decom}
    \mfrak{g}=\bigoplus_{a\in\grp{X}[\grp{T}]}\mfrak{g}_{a},
  \end{equation}
  where: 
	$\grp{X}[\grp{T}]$ is the \emph{split character group} of $\grp{T}$, namely $\fun{Hom}[\grp{T},\Gm]$, and 
	each $\mfrak{g}_{a}$ is the subspace of $\mfrak{g}$ on which $\grp{T}$ acts through a character $a\in\grp{X}[\grp{T}]$. 
	If $\mfrak{g}_{a}\neq 0$, then $a$ is called a \emph{root}. 
	The set of all roots is denoted by $\fun{\Phi}[\grp{G},\grp{T}]$ and is called the \emph{root system} of the pair $(\grp{G},\grp{T})$.
	For a root $a\in\fun{\Phi}[\grp{G},\grp{T}]$, the set $(a)$ of all positive real multiples of it is called its \emph{radical ray}. 
	The set of all radical rays is denoted by $\fun{\vPhi}[\grp{G},\grp{T}]$. 
	
	\index{split cocharacter group}%
	\index[notation]{Xgrp check (T)@$\grp{X}^{\vee}[\grp{T}]$}%
	\index[notation]{N@$\grp{N}$}%
	\index[notation]{Z@$\grp{Z}$}%
	\index{Weyl group}%
	\index{split reductive group!Weyl group of}%
	\index[notation]{Weyl vectorial (GT)@$\vWeyl[\grp{G},\grp{T}]$}%
	Let $\grp{X}^{\vee}[\grp{T}]$ be the \emph{split cocharacter group} of $\grp{T}$, namely $\fun{Hom}[\Gm,\grp{T}]$. 
	Then the normalizer $\grp{N}=\grp{N}_{\grp{G}}[\grp{T}]$ acts on $\grp{X}^{\vee}[\grp{T}]$ by conjugations. 
	Since the centralizer $\grp{Z}=\grp{Z}_{\grp{G}}[\grp{T}]$ acts trivially on them, we obtain actions of $\grp{N}/\grp{Z}$ on $\grp{X}^{\vee}[\grp{T}]$. 
	The quotient $\grp{N}/\grp{Z}$ is denoted by $\vWeyl[\grp{G},\grp{T}]$ and is called the \emph{Weyl group} of the pair $(\grp{G},\grp{T})$. 
\end{definition}

\index[notation]{Xgrp@$\grp{X}$}%
\index[notation]{Xgrp check@$\grp{X}^{\vee}$}%
\index[notation]{Phi@$\Phi$}%
\index[notation]{Phi vectorial@$\vPhi$}%
\index[notation]{Weyl vectorial@$\vWeyl$}%
From now on, we will omit $\grp{G}$ and $\grp{T}$ from notations such as $\grp{X}[\grp{T}]$, $\fun{\Phi}[\grp{G},\grp{T}]$, $\fun{\vPhi}[\grp{G},\grp{T}]$, $\grp{X}^{\vee}[\grp{T}]$, $\vWeyl[\grp{G},\grp{T}]$, and any others defined later if the pair $(\grp{G},\grp{T})$ is implied.



\begin{definition}
	\index{root subgroup}%
	\index[notation]{U (a)@$\grp{U}_{(a)}$}%
	Let $\grp{G}$ be a smooth connected linear algebraic group and $\grp{T}$ a split torus in it. 
	Then the \emph{root subgroup} $\grp{U}_{(a)}$ associated to a radical ray $(a)$ is the largest connected closed subgroup of $\grp{G}$ such that: 1, it is normalized by $\grp{T}$; 2, any characters appearing in the adjoint representation of $\grp{T}$ on the Lie algebra of $\grp{U}_{(a)}$ belongs to $(a)$. 
	We refer to \cite{Milne}*{16.i} for a construction of such groups.
	
	\index{Levi subgroup}%
	\index[notation]{G (a) Levi@$\grp{G}_{(a)}$}%
	\index{Levi subgroup!derived}%
	\index[notation]{G (a) derived Levi@$\grp{G}^{(a)}$}%
	\index[notation]{beta (a)@$\beta_{(a)}$}%
	Let $(a)\in\vPhi$ such that $-(a)\in\vPhi$. 
	Then the \emph{Levi subgroup} $\grp{G}_{(a)}$ (reps. \emph{derived Levi subgroup} $\grp{G}^{(a)}$) associated to $(a)$ is the subgroup of $\grp{G}$ generated by $\grp{U}_{(a)}$, $\grp{U}_{-(a)}$ and $\grp{T}$ (resp. by $\grp{U}_{(a)}$ and $\grp{U}_{-(a)}$). 
	Then we have open immersions: 
	\begin{align*}
		\grp{U}_{(a)}\times\grp{U}_{-(a)}\longrightarrow\grp{G}^{(a)}
		&&\text{and}&&
		\grp{U}_{-(a)}\times\grp{T}\times\grp{U}_{(a)}\longrightarrow\grp{G}_{(a)}.
	\end{align*}
	Moreover, there is a rational map $\beta_{(a)}\colon\grp{U}_{(a)}\times\grp{U}_{-(a)}\dashrightarrow\grp{U}_{-(a)}\times\grp{T}\times\grp{U}_{(a)}$ compatible with them. We refer to \cite{BT-2}*{1.1.10} for this fact. 
\end{definition}

\index{coroot!associated to $a$}%
\index[notation]{a check@$a^{\vee}$}%
If $(\grp{G},\grp{T})$ be a split reductive group, then $(\grp{G}_{(a)},\grp{T})$ is a split reductive group of semisimple rank one and $\grp{G}^{(a)}$ is the derived group of $\grp{G}_{(a)}$ (see. e.g. \cite{Milne}*{21.c} or \cite{BT-2}*{\S 1.1}). 
Since $\grp{T}\cap\grp{G}^{(a)}$ is one-dimensional, there is a unique cocharacter $a^{\vee}\in\grp{X}^{\vee}$ parameterizing it such that $a\circ a^{\vee}=2$. 
This cocharacter is called the \emph{coroot associated to $a$}. 
Note that the rational map $\beta_{(a)}$ can be expressed using $a^{\vee}$ (see e.g. \cite{BT-2}*{3.2.7.(3)} or \cite{SGA3}*{XX, 2.1.(F)}).

\index{split reductive group!coroot system of}%
\index[notation]{Phi check (GT)@$\fun{\Phi}^{\vee}[\grp{G},\grp{T}]$}%
\index[notation]{Phi check@$\Phi^{\vee}$}%
\index{split reductive group!coroot space of}%
\index[notation]{Vect (GT)@$\Vect[\grp{G},\grp{T}]$}%
\index{coroot space}%
\index[notation]{Vect@$\Vect$}%
\index[notation]{Apt vectorial (GT)@$\vApt[\grp{G},\grp{T}]$}%
\index[notation]{Apt vectorial@$\vApt$}%
The set of coroots is called the \emph{coroot system} of $(\grp{G},\grp{T})$ and is denoted by $\fun{\Phi}^{\vee}[\grp{G},\grp{T}]$.
Then the underlying space of $\Phi$ is the subspace of $\grp{X}^{\vee}\otimes_{\Z}{\R}$ spanned by $\fun{\Phi}^{\vee}$ and equipped with a $\vWeyl$-invariant inner product. 
It is called the \emph{coroot space} of $(\grp{G},\grp{T})$ and is denoted by $\Vect[\grp{G},\grp{T}]$.
With the help of $\beta_{(a)}$, we can see that $\Phi$ is a \emph{reduced} root system on $\Vect$ and that the Weyl group $\vWeyl$ acts on $\Vect$ as a reflection group and is identified with the Weyl group of $\Phi$. 
In this way, we obtain a vectorial apartment $\vApt[\grp{G},\grp{T}]=(\Vect,\vWeyl)$. 

\begin{remark}
	\index[notation]{U (a)@$\grp{U}_{a}$}%
	\index[notation]{G (a) Levi@$\grp{G}_{a}$}%
	\index[notation]{G (a) derived Levi@$\grp{G}^{a}$}%
	\index[notation]{beta (a)@$\beta_{a}$}%
	Since we can identify a radical ray $(a)\in\vPhi$ with a unique root $a\in\Phi$, we denote $\grp{U}_{(a)}$, $\grp{G}_{(a)}$, $\grp{G}^{(a)}$, and $\beta_{(a)}$ by $\grp{U}_{a}$, $\grp{G}_{a}$, $\grp{G}^{a}$, and $\beta_{a}$ respectively.
\end{remark}

\begin{definition}
	\index{parabolic subgroup}%
	\index[notation]{P@$\grp{P}$}%
	\index{Borel subgroup}%
	Let $\grp{G}$ be a smooth connected linear algebraic group. 
	Then a \emph{parabolic subgroup} of it is a smooth subgroup $\grp{P}$ such that $\grp{G}/\grp{P}$ is a complete variety. 
	A subgroup of $\grp{G}$ is \emph{Borel} if it is smooth, connected, solvable, and parabolic. 
\end{definition}
\begin{proposition}[\cite{Milne}*{21.d and 21.i}]\label{prop:parabolic_type}
	Let $(\grp{G},\grp{T})$ be a split reductive group. 
	Then there are natural one-to-one correspondences between the following sets:
  \begin{enumerate}[label=\textup{(\alph*)}]
    \item \index[notation]{B@$\grp{B}$}%
			The set of Borel subgroups $\grp{B}$ of $\grp{G}$ containing $\grp{T}$. 
		\item \index[notation]{C@$\vC$}
			The set of Weyl chambers $\vC$ in the vectorial apartment $\vApt$. 
		\item \index[notation]{Phi positive@$\Phi^{+}$}%
			The set of systems of positive roots $\Phi^{+}$ in the root system $\Phi$. 
		\item \index[notation]{Delta@$\Delta$}%
			The set of bases $\Delta$ of $\Phi$. 
  \end{enumerate}
	The Weyl group $\vWeyl$ acts simply transitively on each set.
	Moreover, after choosing a quadruple $(\grp{B},\vC,\Phi^{+},\Delta)$, we have the following isomorphic posets.
	\begin{enumerate}[label=\textup{(\alph*)}]
    \item \index[notation]{P@$\grp{P}$}%
			The poset of parabolic subgroups $\grp{P}$ of $\grp{G}$ containing $\grp{B}$. 
    \item \index[notation]{F vectorial@$\vF$}%
			The poset of faces $\vF$ of the Weyl chamber $\vC$. 
		\item \index[notation]{Psi@$\Psi$}%
			The poset of parabolic subsets $\Psi$ of $\Phi$ containing $\Phi^{+}$. 
		\item \index[notation]{I@$I$}%
			The poset of subsets $I$ of $\Delta$. 
	\end{enumerate}
\end{proposition}
\begin{convention}\label{con:typeOfP}
	\index{type!parabolic subgroup}%
	\index{type!facet}%
	\index{type!parabolic subset}%
	\index{type}%
	If $(\grp{P},\vF,\Psi,I)$ is a quadruple as above, then we say that each of them \emph{have type $I$}, where $I$ is identified as a subset of $\Set*{1,\cdots,n}$ in \cref{con:type}.
\end{convention}

\begin{theorem}[\citelist{\cite{Rousseau}*{\S 10}\cite{Tits74}*{\S 5}}]
	\index[notation]{Build Tits (G)@$\vBuild[\grp{G}]$}%
  Let $(\grp{G},\grp{T})$ be a split reductive group. Then there is a unique (up to unique isomorphism) $G$-set 
	$\vBuild[\grp{G}]$ 
	containing $\Vect$ and satisfying the following: 
  \begin{thmlist}
    \item $\vBuild[\grp{G}]=\bigcup_{g\in G}g.\Vect$;
    \item $N$ stabilizes $\Vect$ and acts on it through $\vWeyl$;
    \item for every $a\in\Phi$, the fixator of $\valpha_{a}$ is $T\cdot U_{a}$.
  \end{thmlist}
\end{theorem}

\index{Tits building} %
\index[notation]{T (g) conjugate@$\grp{T}^{g}$}%
Then $\vBuild[\grp{G}]$ is a building of type $\vApt[\grp{G},\grp{T}]$, called the \emph{Tits building} of $\grp{G}$. 
Note that the apartment structure on $g.\Vect$ is $\vApt[\grp{G},\grp{T}^{g}]$, where $\grp{T}^{g}$ is the conjugate of $\grp{T}$ by $g$. 
Since split maximal tori of a splittable reductive group $\grp{G}$ are conjugate to each other (see e.g. \cite{Milne}*{17.105}), $\vBuild[\grp{G}]$ is independent of the choice of $\grp{T}$.
Moreover, \cref{prop:parabolic_type} together with the isogeny and existence theorem (see e.g. \cite{Milne}*{23.56}) implies that the Tits building $\vBuild[\grp{G}]$ depends only on the root system $\Phi$ and the ground field $K$.

\subsection{Bruhat-Tits buildings}\label{subsec:BT-building}
This subsection is about the Bruhat-Tits theory of constructing an affine building from a split reductive group over a valued field. 
We refer to \cites{BT-1,BT-2} for details. 

\index[notation]{valuation@$\fun{val}[\:\cdot\:]$}%
\index[notation]{Gamma@$\Gamma$}%
\index[notation]{K circ@$K^{\circ}$}%
\index[notation]{pi@$\varpi$}%
\index[notation]{kappa@$\kappa$}%
In what follows, $K$ is equipped with a discrete valuation $\fun{val}[\:\cdot\:]\colon K\to\R\cup\Set*{\infty}$. 
Its valuation group $\fun{val}[K^{\times}]$ is denoted by $\Gamma$. 
We use $K^{\circ}$ to denote the valuation ring, $\varpi$ a fixed uniformizer, and $\kappa$ the residue field.

\begin{definition}
	\index{root group datum}%
	\index[notation]{TUM@$(T,(U_{a},M_{a})_{a\in\Phi})$}%
	Let $\Phi$ be a root system and $G$ a group. 
	A \emph{root group datum\footnote{It is called a \emph{root datum} in \cite{BT-1}*{6.1.1}.} of type $\Phi$ in $G$} is a system $(T,(U_{a},M_{a})_{a\in\Phi})$, 
	where $T$ is a subgroup of $G$, each $U_{a}$ is a non-trivial subgroup of $G$, and each $M_{a}$ is a right congruence class modulo $T$, satisfying the axioms \textbf{(RD 1)}-\textbf{(RD 6)} in \cite{BT-1}*{6.1.1}. 
	This root group datum is said to be \emph{generating} when $G$ is generated by the subgroups $T$ and $U_{a}$ for $a\in\Phi$.
\end{definition}

\index[notation]{TU@$(T,(U_{a})_{a\in\Phi})$}%
It turns out that, $M_{a}$ is completely determined by $U_{a}$, $U_{-a}$, and $T$. 
From now on, we say that $(T,(U_{a})_{a\in\Phi})$ is a root group datum by an abuse of language.
	
\begin{example}\label{eg:rgd_reductive}
	Let $(\grp{G},\grp{T})$ be a split reductive group. 
	Then there is a generating root group datum $(T,(U_{a})_{a\in\Phi})$ of type $\Phi$ in $G$, where each $U_{a}$ is (the group of $K$-points of) the root subgroup $\grp{U}_{a}$.
\end{example}

\begin{definition}\label{def:valuation}
	\index{root group datum!valuation on}%
	\index[notation]{varphi@$\varphi=(\varphi_{a})_{a\in\Phi}$}%
	\index[notation]{Gamma (a)@$\Gamma_{a}$}%
	\index[notation]{U (a,k)@$U_{a,k}$}%
	A \emph{valuation on the root group datum $(T,(U_{a},M_{a})_{a\in\Phi})$} is a family 
	$\varphi=(\varphi_{a})_{a\in\Phi}$ 
	of functions $\varphi_{a}\colon U_{a}\to\R\cup\Set*{\infty}$ satisfying the axioms \textbf{(V 0)}-\textbf{(V 5)} in \cite{BT-1}*{6.2.1}. 
	For each $a\in\Phi$, let $\Gamma_{a}$ denote the set $\varphi_{a}(U_{a}\setminus\Set*{1})$ and $(U_{a,k})_{k\in\R}$ the filtration on $U_{a}$ induced by $\varphi_{a}$, namely $U_{a,k}:=\varphi_{a}^{-1}([k,\infty])$.
\end{definition}

\index{root group datum!valuation on!special}%
\index{root group datum!valuation on!discrete}%
For simplicity and the purpose of this paper, all root systems $\Phi$ in what follows are assumed to be \emph{reduced}.	We refer to \cites{BT-1,BT-2} for general statements.	
A valuation $\varphi$ is said to be \emph{special} if $0\in\Gamma_{a}$ for all $a\in\Phi$, and \emph{discrete} if each $\Gamma_{a}$ is a discrete subset of $\R$. 
If $\varphi$ is both special and discrete, then each $\Gamma_{a}$ is a discrete subgroup of $(\R,+)$. 

\begin{example}[\cite{BT-1}*{6.2.3}]\label{eg:vrgd_reductive}
	\index[notation]{u a@$u_{a}$}%
	\index{split reductive group!pinning of}%
	Let $(\grp{G},\grp{T})$ be a split reductive group. 
	A \emph{pinning} on it is a basis $\Delta$ of $\Phi$ together with a family of isomorphisms $(u_{a}\colon\Ga\to\grp{U}_{a})_{a\in\Delta}$. 
	Then the family $(u_{a})_{a\in\Delta}$ extends uniquely to a coherent system of isomorphisms $(u_{a})_{a\in\Phi}$, called a 
	\emph{Chevalley system}.
	\index{split reductive group!Chevalley system of}%
	We refer to \cite{BT-2}*{3.2.2} or \cite{SGA3}*{XXIII, 6.2} for more details. 
	Given such a system $(u_{a})_{a\in\Phi}$, the family $(\fun{val}\circ u_{a}^{-1})_{a\in\Phi}$
	is a special and discrete valuation on the root group datum $(T,(U_{a})_{a\in\Phi})$. 
\end{example}

\index{root group datum!valuation on!equipollent}%
\index[notation]{varphi plus v@$\varphi+\vect{v}$}%
Let $(T,(U_{a})_{a\in\Phi})$ be a generating root group datum and $\varphi=(\varphi_{a})_{a\in\Phi}$ a valuation on it.  
For any vector $\vect{v}$ in the underlying space $\Vect$ of $\Phi$, the family $\psi=(\psi_{a})_{a\in\Phi}$ given by $\psi_{a}\colon u\mapsto\varphi_{a}(u)+a(\vect{v})$ is a valuation and is denoted by $\varphi+\vect{v}$. 
The valuations $\varphi$ and $\psi=\varphi+\vect{v}$ are said to be \emph{equipollent}. 

\index[notation]{Aff@$\Aff$}%
\index[notation]{Sigma@$\Sigma$}%
\index[notation]{E@$\mscr{E}$}%
\index[notation]{U (alpha)@$U_{\alpha}$}%
\index[notation]{U (alpha) plus@$U_{\alpha+}$}%
Let $\Aff$ be the equipollent class of $\varphi$. 
Then $\Aff$ is an affine space under $\Vect$. 
We fix $\varphi$ as the reference point. 
Let 
\begin{align*}
	\Sigma	&:=
		\Set*{ \alpha_{a+k} \given a\in\Phi, k\in\Gamma_{a} }
	&\text{and}&&
	\mscr{E}	&:=
		\Set*{ (a,\alpha_{a+k}) \given a\in\Phi, k\in\Gamma_{a} }.
\end{align*}
For $\alpha=\alpha_{a+k}\in\Sigma$, let $U_{\alpha}:=U_{a,k}$ and $U_{\alpha+}:=\bigcup_{h>k}U_{a,h}$. 

\begin{remark}[\cite{BT-2}*{4.2.9}]
	\index{root group datum!valuation on!compatible with $\fun{val}[\:\cdot\:]$}%
	\index[notation]{t in T@$\bm{t}\in T$}%
	\index[notation]{Apt (GT)@$\Apt[\grp{G},\grp{T}]$}%
	Let $(\grp{G},\grp{T})$ be a split reductive group. 
	A valuation $\varphi=(\varphi_{a})_{a\in\Phi}$ on a root group datum $(T,(U_{a})_{a\in\Phi})$ is said to be \emph{compatible with $\fun{val}[\:\cdot\:]$} if for all $u\in U_{a}$ and $\bm{t}\in T$, we have $\varphi_{a}\left( \bm{t}u\bm{t}^{-1} \right) = \varphi_{a}(u) + \fun{val}[a(\bm{t})]$. 
	The valuation in \cref{eg:vrgd_reductive} is such a valuation. 
	The set of compatible valuations forms an equipollent class $\Aff$. 
	Then $\Sigma$ and $\mscr{E}$ are given as in \cref{eg:apt_split_type}. 
	We use $\Apt[\grp{G},\grp{T}]$ to denote the affine apartment of type $\mscr{E}$ obtained in this way. 
\end{remark}

\index[notation]{N @$N$}%
\index[notation]{nu vectorial@$\vnu$}%
\index[notation]{varphi acted by m@$m.\varphi$}%
\index[notation]{nu@$\nu$}%
Let $N$ be the subgroup of $G$ generated by $T$ and $M_{a}$ for all $a\in\Phi$. 
Then $N$ normalizes $T$ and induces an epimorphism $\vnu\colon N\to\vWeyl[\Phi]$ such that 
$\vnu(M_{a})=\Set*{r_{a}}$ and that $\fun{Ker}[\vnu]=T$. 
Next, $N$ acts on the valuations as follows.
Let $\varphi=(\varphi_{a})_{a\in\Phi}$ be a valuation and $m\in N$. 
Then $m.\varphi$ is the valuation given by 
$(m.\varphi)_{a}\colon u\mapsto\varphi_{\vnu(m)^{-1}.a}(m^{-1}um)$.
Note that $m.(\varphi+\vect{v}) = m.\varphi+\vnu(m).\vect{v}$. 
Hence, $N$ stabilizes each equipollent class $\Aff$ and induces a group homomorphism $\nu\colon N\to\fun{Aut}[\Aff]$ whose vectorial part is $\vnu$. 

\index[notation]{H@$H$}%
\index[notation]{Weyl hat@$\widehat{W}$}%
\index[notation]{Weyl@$W$}%
Let $H=\fun{Ker}[\nu]$ and $\widehat{W}=\nu(N)$. 
In general, $\widehat{W}$ is larger than the expected affine Weyl group $W$ making $(\Aff,\Sigma)$ an affine apartment. 
Indeed, $W$ is generated by $\Set*{r_{\alpha}}_{\alpha\in\Sigma}$. 

\begin{remark}[\cite{BT-2}*{4.2.7}]
	\index[notation]{vect t@$\vect{v}_{\bm{t}}$}%
	\index{coroot lattice}%
	\index[notation]{Q (check)@$\mcal{Q}^{\vee}$}%
	Let $(\grp{G},\grp{T})$ be a split reductive group. 
	Then the group $N$ is (the group of $K$-points of) the normalizer $\grp{N}$ of $\grp{T}$, the epimorphism $\vnu\colon N\to\vWeyl[\Phi]$ comes from the quotient $N\to N/T\cong\vWeyl[\Phi]$ as in \cref{def:Weyl_group}, and the translation subgroup $\nu(T)$ of $\widehat{W}$ has the following characterization: for any $\bm{t}\in T$, the translation vector $\vect{v}_{\bm{t}}$ of $\nu(\bm{t})$ is determined by 
	\begin{equation*}
		\chi(\vect{v}_{\bm{t}}) = 
			-\fun{val}[\chi(\bm{t})],
		\qquad\text{for all}\qquad
			\chi\in\Vect^{\ast}\cap\grp{X}.
	\end{equation*}
	On the other hand, the translation subgroup of $W$ is $\mcal{Q}^{\vee}\otimes_{\Z}\Gamma$, where $\mcal{Q}^{\vee}$ is the \emph{coroot lattice}, namely the lattice in $\Vect$ generated by the coroots. 
	In general, they are different. 
	But if $\grp{G}$ is semisimple and simply-connected, then $W$ equals $\widehat{W}$. 
\end{remark}

From now on, assume a special and discrete valuation $\varphi$ is fixed.

\index[notation]{Omega@$\Omega$}%
\index[notation]{U (Omega)@$U_{\Omega}$}%
\index[notation]{Sigma (Omega)@$\Sigma_{\Omega}$}%
\index[notation]{Phi (Omega)@$\Phi_{\Omega}$}%
\index[notation]{N hat (Omega)@$\widehat{N}_{\Omega}$}%
\index[notation]{P hat (Omega)@$\widehat{P}_{\Omega}$}%
Let $\Omega$ be a nonempty subset of $\Aff$. 
The subgroup of $G$ generated by $U_{\alpha}$ for all affine roots such that $\alpha\supset\Omega$ is denoted by $U_{\Omega}$. 
Then the image of $N\cap U_{\Omega}$ under $\nu\colon N\to\widehat{W}$ is generated by the reflections $r_{\alpha}$ for affine roots $\alpha$ such that $\Omega\subset\partial\alpha$. 
Let $\Sigma_{\Omega}$ be the set of such affine roots and $\Phi_{\Omega}$ the corresponding root subsystem of $\Phi$ under $\mscr{E}$.
Then the vectorial part $\vnu(N\cap U_{\Omega})$ equals the linear Weyl group of $\Phi_{\Omega}$.
Finally, let $\widehat{N}_{\Omega}$ denote the fixator of $\Omega$ in $N$ and $\widehat{P}_{\Omega}:=\widehat{N}_{\Omega}\cdot U_{\Omega}$. 

\begin{definition}\label{def:BT-building}
	\index{Bruhat-Tits building}%
	\index[notation]{Build (varphi)@$\Build[\varphi]$}%
	\index[notation]{Build@$\Build$}%
	The \emph{Bruhat-Tits building} of $G$ (equipped with the generating root group datum $(T,(U_{a})_{a\in\Phi})$ and the valuation $\varphi$ on it) is the quotient $G$-set $\Build[\varphi]$ of $G\times\Aff$ under the following equivalent relation:
  \begin{equation}\label{eq:BT-building}
    (g,x)\sim(h,y) \iff \exists n\in N : y=\nu(n).x,\ g^{-1}hn\in\widehat{P}_{x}.
  \end{equation}
	We will simply denote this set by $\Build$ if there is no ambiguity. 
\end{definition}
\begin{theorem}[\cite{BT-1}*{\S 7.4}]
	The Bruhat-Tits building is independent of the choice of $\varphi\in\Aff$. 
	Identifying $\Aff$ with the subset $\Set*{1}\times\Aff$ of $\Build$, we have:
  \begin{thmlist}
    \item $\Build=\bigcup_{g\in G}g.\Aff$;
    \item each $U_{\alpha}$ fixes $\alpha\in\Sigma$ pointwise;
    \item for each nonempty subset $\Omega\subset\Aff$, its fixator is $\widehat{P}_{\Omega}$, which acts transitively on apartments containing $\Omega$;
    \item the stabilizer (resp. fixator) of $\Aff$ is $N$ (resp. $H$).
  \end{thmlist}
\end{theorem}

\index{strongly transitive}%
\index{type-preserving}%
Then we see that $\Build$ is a building of type $(\Aff,W)$. 
Like the Tits building, the action of $G$ on $\Build$ is \emph{strongly transitive}: it acts transitively on the pairs $(A,C)$ where $A$ is an apartment, and $C$ is a chamber in $A$. 
However, unlike the Tits building, the action of $G$ on $\Build$ is not necessarily \emph{type-preserving}: since $\widehat{W}$ is usually larger than the affine Weyl group $W$, there cloud be some $g\in G$ acts incompatible with the type function. 

\begin{construction}\label{eg:BT-building_split}
	\index[notation]{Build (G)@$\Build[\grp{G}]$}%
	Let $(\grp{G},\grp{T})$ be a split reductive group. 
	Then there is a unique Bruhat-Tits building associated to it through \cref{def:BT-building}. 
	Moreover, this building is of type $\Apt[\grp{G},\grp{T}]$ and is independent of $\grp{T}$. 
	We denote it by $\Build[\grp{G}]$. 
	Then, using the isogeny and existence theorem, we can see that it depends only on the root system $\Phi$ and the ground field $K$.
	To obtain a strongly transitive and type-preserving action on the building, we may take $\grp{G}$ to be semisimple and simply-connected.
\end{construction}

\subsection{Concave functions and smooth models}\label{subsec:ConcaveSM}
One important ingredient in Bruhat-Tits theory is the theory of various smooth models associated to concave functions. 
In this subsection, we follow \cites{BT-1,BT-2,Yu15} to review this topic. 

	In what follows, $\Phi$ is a \emph{reduced} root system, $(T,(U_{a})_{a\in\Phi})$ is a generating root group datum  in $G$, and we fix a \emph{special and discrete} valuation $\varphi$ as the reference point of its equipollent class $\Aff$.
	We refer to \cites{BT-1,BT-2} for statements in general settings.	

\index{extended real numbers}%
\index[notation]{R tilde@$\widetilde{\R}$}%
\index[notation]{R plus@$\R+$}%
\index[notation]{k plus@$k+$}%
Let's first introduce the \emph{ordered monoid of extended real numbers} $\widetilde{\R}$. 
Informally, $\widetilde{\R}$ is the union of $\R$, $\R+$, and $\Set*{\infty}$, where $\R+:=\Set*{ k+ \given k\in\R }$ (one can think $k+$ as $k$ plus an infinitesimal) and $\infty$ is the positive infinity. 
Then the commutative addition and the total order on $\R$ are extended to $\widetilde{\R}$ intuitively.	We refer to \cite{BT-1}*{6.4.1} for details. 
\begin{convention}\label{con:filtration}
	\index[notation]{F (k)@$(F_{k})_{k\in\R}$}%
	\index[notation]{F (lambda)@$(F_{\lambda})_{\lambda\in\widetilde{\R}}$}%
	\index[notation]{F (infinity)@$F_{\infty}$}%
	\index{jump}%
	\index[notation]{ceil@$\ceil{\:\cdot\:}$}%
	Whenever we have a filtration $(F_{k})_{k\in\R}$ (for instance, the filtration $(U_{a,k})_{k\in\R}$ of $U_{a}$ in \cref{def:valuation}), we can extend it to $(F_{\lambda})_{\lambda\in\widetilde{\R}}$ by defining
	\begin{align*}
		F_{\lambda} &= \bigcup_{k\in\R,k\ge\lambda} F_{k},
		&
		F_{\infty} &= \bigcap_{k\in\R} F_{k}.
	\end{align*}
	We say $k\in\R$ is a \emph{jump} of the filtration if $F_{k+}\neq F_{k}$. 
	In our most usage, the jumps are precisely the elements of $\Gamma$. 
	For any $\lambda\in\widetilde{\R}$, we use the notation $\ceil{\lambda}$ to denote the smallest $k\in\Gamma$ such that $\lambda\le k$. 
\end{convention}

\begin{definition}
	\index[notation]{Phi tilde@$\widetilde{\Phi}$}%
	\index{concave function}%
	\index{concave function on $\widetilde{\Phi}$}%
	\index{concave function on $\Phi$}%
	Let $\Phi$ be a root system and $\widetilde{\Phi}=\Phi\cup\Set*{0}$. 
	Then a \emph{concave function (on $\widetilde{\Phi}$)} is a function $f\colon\widetilde{\Phi}\to\widetilde{\R}$	such that for any finite family $(a_{i})_{i\in I}$ in $\widetilde{\Phi}$ satisfying $\sum_{i\in I}a_{i}\in\widetilde{\Phi}$, we have
	\begin{equation*}
		\sum_{i\in I}f(a_{i})\ge f(\sum_{i\in I}a_{i}).
	\end{equation*}
	A concave function $f$ is said to be \emph{on $\Phi$} if $f(0)=0$ and $f(\Phi)\subset\R$. 
\end{definition}

\index{good filtration}%
\index[notation]{H (k)@$(H_{k})_{k\ge 0}$}%
In what follows, we fix a \emph{good filtration} $(H_{k})_{k\ge 0}$ on $H$. 
We refer to \cite{BT-1}*{6.4.38} for its definition, under the name ``\emph{extension of the valuation}''. 

\index[notation]{U (f)@$U_{f}$}%
Let $f$ be a concave function on $\widetilde{\Phi}$. 
We use $U_{f}$ to denote the subgroup of $G$ generated by $U_{a,f(a)}$ for all $a\in\Phi$. 
\index[notation]{P (f)@$P_{f}$}%
Let $P_{f}$ denote the subgroup $H_{f(0)}\cdot U_{f}$, then we have the following multiplication map \cite{BT-1}*{6.4.48}: 
\begin{equation}
	\label{eq:Decomposition:Pf}
	\prod_{a\in\Phi^{-}}U_{a,f(a)}\times H_{f(0)}\times \prod_{a\in\Phi^{+}}U_{a,f(a)} \longrightarrow P_{f}.
\end{equation}
It is injective in general and moreover bijective if $f(0)>0$.

\begin{example}[\citelist{\cite{BT-1}*{6.4.2}\cite{BT-2}*{4.6.26}}]\label{eg:parahoric_subgroup}
	Let $\Omega$ be a nonempty subset of $\Aff$.
	\index[notation]{f (Omega)@$f_{\Omega}$}%
  Define $f_{\Omega}\colon\Phi\to\R\cup\Set*{\infty}$ as follows: 
  \begin{equation*}
    f_{\Omega}(a)=\inf\Set*{k\in\R\given\Omega\subset\alpha_{a+k}}.
  \end{equation*}
  Then $f_{\Omega}$ is a concave function on $\Phi$. 
	Then we have $U_{f_{\Omega}}=U_{\Omega}$ and hence $P_{f_{\Omega}}\subset\widehat{P}_{\Omega}$. 
	\index{parahoric subgroup}%
	When $F$ is a facet in $\Aff$, the group $P_{f_{F}}$ is called a \emph{parahoric subgroup}. 
	\index{Iwahori subgroup}%
	It is called an \emph{Iwahori subgroup} if $F$ is further an alcove. 
	Note that these terms are usually restricted to a specific choice of $H_{0}$: the $H^{\circ}$ in \cpageref{def:MP}. 
\end{example}

\index[notation]{f prime@$f'$}%
\index{concave function!optimization of}%
\index{concave function!optimal}%
\index{concave function!root subsystem associated to}%
\index[notation]{Phi (f)@$\Phi_{f}$}%
Let $f$ be a concave function on $\widetilde{\Phi}$. 
Define $f'\colon\Phi\to\R$ as follows:
\begin{equation*}
	f'(a) := \inf\Set*{ k\in\Gamma_{a}\given k\ge f(a) }.
\end{equation*}
It is called the \emph{optimization} of $f$. 
In general, it is not necessarily a concave function (see \cite{BT-2}*{4.5.3}). 
However, under our assumptions, it is a concave function on $\Phi$. 
When $f'=f$, we say $f$ is \emph{optimal}. %
The set of roots $a\in\Phi$ such that $f'(a)+f'(-a)=0$ is a root subsystem and is denoted by $\Phi_{f}$, called the \emph{root subsystem associated to $f$}. 

\begin{remark}
	Note that for any $a\in\Phi$, we have 
	\begin{equation*}
		f_{\Omega}(a)+f_{\Omega}(-a) = 
		-\inf_{x\in\Omega}a(x) -\inf_{x\in\Omega}(-a(x)) = 
		\sup_{x\in\Omega}a(x)-\inf_{x\in\Omega}a(x) \ge0.
	\end{equation*}
	The equality holds if and only if $a(x)$ is a constant on $\Omega$, namely $\Omega$ is contained in a hyperplane having vectorial part $\fun{Ker}[a]$. 
	Therefore, $\Phi_{f_{\Omega}}=\Phi_{\Omega}$. 
\end{remark}

\begin{lemma}[\cite{BT-1}*{6.4.23}]\label{lem:f_ast}	
	\index[notation]{f ast@$f^{\ast}$}%
	\index[notation]{G bar (f)@$\overline{G}_{f}$}%
	\index[notation]{U bar (f,a)@$\overline{U}_{f,a}$}%
	\index[notation]{T bar (f)@$\overline{T}_{f}$}%
	Let $f$ be a concave function on $\Phi$. 
	Define $f^{\ast}\colon\widetilde{\Phi}\to\widetilde{\R}$ as follows:
	\begin{equation*}
		f^{\ast}(a) := 
		\begin{dcases*}
			f(a) & if $f(a)+f(-a)>0$, \\
			f(a)+ & if $f(a)+f(-a)=0$.
		\end{dcases*}
	\end{equation*}
	Then $f^{\ast}$ is a concave function on $\widetilde{\Phi}$. 
	Let $\overline{G}_{f}$ denote the quotient $P_{f}/P_{f^{\ast}}$ and $\overline{U}_{f,a}$ (resp. $\overline{T}_{f}$) the image of $U_{a,f(a)}$ (resp. $H_{f(0)}$) in $\overline{G}_{f}$. 
	Then $(\overline{T}_{f},(\overline{U}_{f,a})_{a\in\Phi_{f}})$ is a generating root group datum of type $\Phi_{f}$ in $\overline{G}_{f}$. 
\end{lemma}
\begin{remark}
	Let $x$ be a point in $\Aff$. 
	Since $P_{f_{x}}$ stabilizes $x$, $P_{f_{x}^{\ast}}$ stabilizes the infinitesimal neighborhood of $x$. 
	Then we may think the root group datum $(\overline{T}_{f_{x}},(\overline{U}_{f_{x},a})_{a\in\Phi_{f_{x}}})$ corresponds to the spherical building at $x$.
\end{remark}

\index{Henselian field}%
\index[notation]{GT@$(\grp{G},\grp{T})$}%
\index[notation]{TU@$(T,(U_{a})_{a\in\Phi})$}%
\index[notation]{varphi@$\varphi$}%
In what follows, $K$ is a \emph{Henselian field} in the sense that $K^{\circ}$ is a Henselian ring. 
We also assume that the residue field $\kappa$ is perfect.
Let $(\grp{G},\grp{T})$ be a split reductive group over $K$ and take the generating root group datum $(T,(U_{a})_{a\in\Phi})$ in $G$ following \cref{eg:rgd_reductive}. 
Then we fix a special and discrete valuation $\varphi$ following \cref{eg:vrgd_reductive}. 

In the rest of this subsection, we should introduce the smooth models associated to concave functions.

\begin{lemma}[\citelist{\cite{Yu15}*{6.2}\cite{BT-2}*{\S 4.3}}]\label{lem:SM_Ua}
	\index[notation]{U sm (a,lambda)@$\mfrak{U}_{a,\lambda}$}%
	For each root subgroup $\grp{U}_{a}$, there are connected smooth models $\mfrak{U}_{a,\lambda}$ ($\lambda\in\widetilde{\R}$) of it satisfying the following: 
	\begin{lemlist}
		\item $\mfrak{U}_{a,\lambda}[K^{\circ}]=U_{a,\lambda}$;
		\item the special fiber $(\mfrak{U}_{a,\lambda})_{\kappa}$ is unipotent for all $\lambda$;
		\item \index[notation]{Gamma (U sm (a,lambda))(m)@$\Gamma(\varpi^m,\mfrak{U}_{a,\lambda})$}%
			the congruence subgroup 
		\begin{equation*}
			\Gamma(\varpi^m,\mfrak{U}_{a,\lambda}):=
			\fun{Ker}[\mfrak{U}_{a,\lambda}[K^{\circ}]\to\mfrak{U}_{a,\lambda}[K^{\circ}/\varpi^m]]
		\end{equation*} 
		equals $U_{a,\lambda+m\fun{val}[\varpi]}$ for every positive integer $m$;
		\item \index[notation]{u Lie (a,lambda)@$\mfrak{u}_{a,\lambda}$}%
			the Lie algebras $\mfrak{u}_{a,\lambda}$ of $\mfrak{U}_{a,\lambda}$ form a filtration on the Lie algebra $\mfrak{u}_{a}$ of $\grp{U}_{a}$. 
	\end{lemlist}
	Moreover, the filtration $(U_{a,\lambda})_{\lambda\in\widetilde{\R}}$ extends to an inductive system $(\mfrak{U}_{a,\lambda})_{\lambda\in\widetilde{\R}}$.
\end{lemma}
\begin{remark}
	\index[notation]{W (K circ)@$\mbb{W_{K^{\circ}}[\:\cdot\:]}$}%
	In our situation, the scheme $\mfrak{U}_{a,\lambda}$ can be taken to be $\mbb{W}_{K^{\circ}}[U_{a,\lambda}]$. 
	Here, the $K^{\circ}$-group scheme $\mbb{W}_{K^{\circ}}[M]$ associated to a $K^{\circ}$-module $M$ is the functor $R\rightsquigarrow M\otimes_{K^{\circ}}R$. We refer to \cite{SGA3}*{I, 4.6} for details of $\mbb{W_{K^{\circ}}[\:\cdot\:]}$. 
\end{remark}

\index[notation]{T sm@$\mfrak{T}$}%
\index{Moy-Prasad filtration}%
Suppose we are given a good filtration $(H_\lambda)_{\lambda\ge 0}$ on $H\subset T$. 
Then there is a smooth model $\mfrak{T}$ of $\grp{T}$ such that $\mfrak{T}[K^{\circ}]=H_{0}$. 
We refer to \cite{BT-2}*{\S 4.4} for details. 
Here we follow \cite{Yu15}, considering the \emph{Moy-Prasad filtration}. 
We start with the subgroup 
\index[notation]{H circ@$H^{\circ}$}%
\begin{equation*}
	\label{def:MP}
	H^{\circ}:=
	\Set*{ \bm{t}\in T \given \fun{val}[\chi(\bm{t})]=0,\text{ for all } \chi\in\grp{X}[\grp{T}] }.
\end{equation*}
Then the Moy-Prasad filtration $(H_\lambda)_{\lambda\ge 0}$ is defined as:
\index[notation]{H (lambda)@$(H_\lambda)_{\lambda\ge 0}$}%
\begin{equation*}
	H_{\lambda}:= 
	\Set*{ \bm{t}\in H^{\circ}\given \fun{val}[\chi(\bm{t})-1]\ge \lambda,\text{ for all } \chi\in\grp{X}[\grp{T}] }.
\end{equation*}
There is also a \emph{Moy-Prasad filtration} $(\mfrak{t}_{\lambda})_{\lambda\ge 0}$ of the Lie algebra $\mfrak{t}$ of $\grp{T}$:
\index[notation]{t Lie (lambda)@$(\mfrak{t}_{\lambda})_{\lambda\ge 0}$}%
\begin{equation*}
	\mfrak{t}_{\lambda}:= 
	\Set*{ t\in \mfrak{t}\given \fun{val}[\odif\chi(t)]\ge \lambda,\text{ for all } \chi\in\grp{X}[\grp{T}] }.
\end{equation*}

\begin{lemma}[\cite{Yu15}*{\S 4}]\label{lem:SM_T}
	\index[notation]{T sm (lambda)@$\mfrak{T}_{\lambda}$}%
	There are connected smooth models $\mfrak{T}_{\lambda}$ ($\lambda\ge0$) of $\grp{T}$ satisfying the following: 
	\begin{lemlist}
		\item $\mfrak{T}_{\lambda}[K^{\circ}]=H_{\lambda}$;
		\item the special fiber $(\mfrak{T}_{\lambda})_{\kappa}$ is unipotent for all $\lambda>0$;
		\item the congruence subgroup 
		\index[notation]{Gamma (T sm (lambda))(m)@$\Gamma(\varpi^m,\mfrak{T}_{\lambda})$}%
		\begin{equation*}
			\Gamma(\varpi^m,\mfrak{T}_{\lambda}):=
			\fun{Ker}[\mfrak{T}_{\lambda}[K^{\circ}]\to\mfrak{T}_{\lambda}[K^{\circ}/\varpi^m]]
		\end{equation*} 
		equals $\grp{T}[K]_{\lambda+m\fun{val}[\varpi]}$ for every positive integer $m$;
		\item the Lie algebra of $\mfrak{T}_{\lambda}$ equals $\mfrak{t}_{\lambda}$.
	\end{lemlist}
	Moreover, the filtration $(H_{\lambda})_{\lambda\ge0}$ extends to an inductive system $(\mfrak{T}_{\lambda})_{\lambda\ge0}$.
\end{lemma}


At this stage, we have inductive systems of group schemes $(\mfrak{T}_{\lambda})_{\lambda\ge0}$ and $(\mfrak{U}_{a,\lambda})_{\lambda\in\widetilde{\R}}$ ($a\in\Phi$). Such a datum is a schematic version of a valuation on a root group datum.

The main theorem of the \emph{schematic Bruhat-Tits theory} is
\begin{theorem}[\citelist{\cite{Yu15}*{8.3}\cite{BT-2}*{\S 4.6}}]\label{thm:SmoothModelPf}
	\index[notation]{G sm (f)@$\mfrak{G}_{f}$}%
	For a concave function $f$ on $\widetilde{\Phi}$, there is a connected smooth model $\mfrak{G}_{f}$ of $\grp{G}$ such that $\mfrak{G}_{f}[K^{\circ}]=P_{f}$. 
	Moreover:
	\begin{thmlist}
		\item The schematic closure of $T$ in $\mfrak{G}_{f}$ is $\mfrak{T}_{f(0)}$.
		\item For each $a\in\Phi$, the schematic closure of $U_{a}$ in $\mfrak{G}_{f}$ is $\mfrak{U}_{a,f(a)}$.
		\item\label{thm:SmoothModelPf:bijective} The multiplication morphism (the products can be taken in any order)
			\begin{equation}\label{eq:SmoothModelPf:decomposition}
				\begin{tikzcd}
					\displaystyle
					\prod_{a\in\Phi_{f}^{-}}
					\mfrak{U}_{a,f(a)}\cdot 
					\mfrak{T}_{f(0)}\cdot 
					\prod_{a\in\Phi_{f}^{+}}
					\mfrak{U}_{a,f(a)} 
					\arrow[r]& 
					\mfrak{G}_{f}
				\end{tikzcd}
			\end{equation}
			is an open immersion. If $f(0)>0$, it induces an isomorphism on special fibers.
	\end{thmlist}	
	Moreover, the assignment $f\rightsquigarrow\mfrak{G}_{f}$ is functorial.
\end{theorem}
\begin{remark}
	Note that \itemref{thm:SmoothModelPf:bijective} implies that \cref{eq:SmoothModelPf:decomposition} gives a bijection on $K^{\circ}$-points (and more generally, $K^{\circ}/I$-points for any ideal $I$) using the Henselian property.
\end{remark}

\index{parahoric group scheme}%
The smooth model $\mfrak{G}_{f_F}$ of $\grp{G}$ associated to a parahoric subgroup $G_{f_F}$ is called a \emph{parahoric group scheme}. 
\index{Chevalley group scheme}%
The parahoric group model $\mfrak{G}_{f_{x}}$ associated to a special vertex $x$ is essentially the \emph{Chevalley group scheme}. 
For the theory of Chevalley group schemes, we refer to \citelist{\cite{BT-2}*{\S 3.2 and 4.6.15}\cite{SGA3}*{XXV}}.

\clearpage
\section{Formula of the simplicial volume}\label{sec:Formula}
\index{local field}%
\index[notation]{q@$q$}%
In this section, we will deduce the formulas for the simplicial volume and simplicity surface area as shown in \cref{thm:SimplicialVolumeFormula}. 
In what follows, $K$ is assumed to be a \emph{local field}.
Namely, $K$ is complete with respect to its valuation $\fun{val}[\:\cdot\:]$, and its residue field $\kappa$ is finite. Let $q$ be the cardinality of $\kappa$. 

\index[notation]{Build@$\Build$}%
\index[notation]{Phi@$\Phi$}%
\index[notation]{G@$\grp{G}$}%
\index[notation]{G grp@$G$}%
The strategy is to employ a \emph{strongly transitive} and \emph{type-preserving} automorphism group. 
Let $\Build$ be a Bruhat-Tits building of split type $\Phi$. 
Then \cref{eg:BT-building_split} tells us that we can realize it as the Bruhat-Tits building of a simply-connected splittable semisimple group $\grp{G}$ having root system $\Phi$ over the ground local field $K$. The group $G$ of $K$-points of $\grp{G}$ is such an automorphism group. 

\index[notation]{o@$o$}%
\index[notation]{T@$\grp{T}$}%
\index[notation]{Apt@$\Apt$}%
\index[notation]{Aff@$\Aff$}%
From now on, we fix a special vertex $o$ in $\Build$ and choose a split maximal torus $\grp{T}$ in $\grp{G}$ such that the apartment $\Apt$ associated to $(\grp{G},\grp{T})$ contains $o$. 
We will follow the notations and conventions in \cref{subsec:BT-building}. 
In particular, $o$ is the reference point of the underlying Euclidean affine space $\Aff$ of $\Apt$. 

\index[notation]{x@$x$}%
For any vertex $x$ in $\Build$, it is clear that a type-preserving automorphism $\phi\in G$ will map a path from $o$ to $x$ to a path from $\phi(o)$ to $\phi(x)$. 
Hence, $G$ preserves the simplicial distance. 
Therefore, we have 
\begin{equation}\label{eq:SimplicialVolumeStrategy}
	\fun{SV}[r]\text{ (resp. $\fun{SSA}[r]$)} =
	\sum_{x}\left[
		P_{o}:
		P_{o,x}
	\right],
\end{equation} 
where
\index[notation]{P (o)@$P_{o}$}%
\index[notation]{P (o,x)@$P_{o,x}$}%
\begin{itemize}
	\item $P_{o}$ is the stabilizer of $o$ in $G$,
	\item $P_{o,x}$ is the stabilizer of $x$ in $P_{o}$, and
	\item \index[notation]{D@$\mscr{D}$}%
		the summation is taking over the intersection of $B(r)$ (resp. $\partial(r)$) with a fundamental domain $\mscr{D}$ of the action of $P_{o}$.
\end{itemize}

The computation will be done as follows. 
In \cref{subsec:ParahoricReduction}, we break the index $\left[P_{o}:P_{o,x}\right]$ into two factors.
In \cref{subsec:PoincarePolynomials}, we will see that the first factor can be given by Poincar\'{e} polynomials.
In \cref{subsec:ConcaveF}, we will compute the second factor using the theory of concave functions. 
In \cref{subsec:FundamentalDomain}, we will describe a fundamental domain of the action of $P_{o}$ and finally prove \cref{thm:SimplicialVolumeFormula}.

\subsection{Parahoric reduction}\label{subsec:ParahoricReduction}
The goal of this subsection is to break the index $\left[P_{o}:P_{o,x}\right]$ into two factors, one of which is a power of $q$. 
For this purpose, we need some facts about concave functions, recalling \cref{subsec:ConcaveSM}.

First note that $P_{o}$ is a \emph{parahoric subgroup} of $G$: it is indeed $P_{f_{o}}$ using the notations in \cref{eg:parahoric_subgroup}. 
Then we have a generating root group datum $(\overline{T}_{o},(\overline{U}_{o,a})_{a\in\Phi})$ of type $\Phi$ in the quotient $P_{o}/P_{f_{o}^{\ast}}$, following \cref{lem:f_ast}.
Moreover, using \cref{thm:SmoothModelPf}, we can see that this datum arises from a split reductive group over $\kappa$. 
\index[notation]{T bar (o)@$\overline{T}_{o}$}%
\index[notation]{U bar (o,a)@$\overline{U}_{o,a}$}%

\begin{lemma}\label{lem:parahoric_reduction}
	\index[notation]{radical u bar (f)@$\grp{\overline{R}}_{f}$}%
	\index[notation]{G sm bar (f)@$\mfrak{\overline{G}}_{f}$}%
	\index[notation]{T sm bar (f)@$\mfrak{\overline{T}}_{f}$}%
	\index[notation]{U sm bar (f,a)@$\mfrak{\overline{U}}_{f,a}$}%
	Let $f$ be a concave function on $\Phi$. 
	Denote the unipotent radical and the reductive quotient of $(\mfrak{G}_{f})_{\kappa}$ by 
	$\grp{\overline{R}}_{f}$ and 
	$\mfrak{\overline{G}}_{f}$ respectively. 
	Let 
	$\mfrak{\overline{T}}_{f}$ (resp. 
	$\mfrak{\overline{U}}_{f,a}$) 
	be the image of $(\mfrak{T}_{0})_{\kappa}$ (resp. $(\mfrak{U}_{a,f(a)})_{\kappa}$) in $\mfrak{\overline{G}}_{f}$. 
	Then $(\mfrak{\overline{G}}_{f},\mfrak{\overline{T}}_{f})$ is a split reductive group with root system $\Phi_{f}$ and root subgroups $(\mfrak{\overline{U}}_{f,a})_{a\in\Phi_{f}}$. 
	Moreover, the generating root group datum $(\overline{T}_{f},(\overline{U}_{f,a})_{a\in\Phi_{f}})$ associated to $(\mfrak{\overline{G}}_{f},\mfrak{\overline{T}}_{f})$ is the same as in \cref{lem:f_ast}. 
\end{lemma}
\begin{proof}
	Applying \cref{thm:SmoothModelPf} to $f$, we see that:
	\begin{lemlist}
		\item\label{lem:parahoric_reduction:1} $(\mfrak{T}_{0})_{\kappa}$ is a split maximal torus in $(\mfrak{G}_{f})_{\kappa}$, the pair $\left((\mfrak{G}_{f})_{\kappa},(\mfrak{T}_{0})_{\kappa}\right)$ has root system $\Phi$, and for any $a\in\Phi$, $(\mfrak{U}_{a,f(a)})_{\kappa}$ is the root subgroup associated to it.
	\end{lemlist}	
	We also refer to \cite{BT-2}*{4.6.4} for a direct proof. 
	
	By \cite{BT-2}*{1.1.11}, the multiplication morphism 
	\begin{equation*}
		\label{lem:parahoric_reduction:eq}\tag{$\ast$}
		\begin{tikzcd}
			\displaystyle
			\prod_{a\in\Phi_{f}^{-}}
			\left(
				\grp{\overline{R}}_{f}\cap(\mfrak{U}_{a,f(a)})_{\kappa}
			\right)\cdot 
			\mscr{R}_{u}[(\mfrak{T}_{0})_{\kappa}]\cdot 
			\prod_{a\in\Phi_{f}^{+}}
			\left(
				\grp{\overline{R}}_{f}\cap(\mfrak{U}_{a,f(a)})_{\kappa}
			\right)
			\arrow[r]& 
			\grp{\overline{R}}_{f}
		\end{tikzcd}
	\end{equation*}
	is an isomorphism. Hence, \itemref{lem:parahoric_reduction:1} implies that first statement of this lemma except that the root system is $\Phi_{f}$.

	Next, applying \cref{thm:SmoothModelPf} to both $f$ and $f^{\ast}$, and using the inductive systems in \cref{lem:SM_T,lem:SM_Ua}, we see that the inclusion $P_{f^{\ast}}\subset P_{f}$ extends to a homomorphism $\mfrak{G}_{f_{\ast}}\to\mfrak{G}_{f}$ so that
	\begin{lemlist}[resume]
		\item\label{lem:parahoric_reduction:2} through the homomorphism $(\mfrak{G}_{f_{\ast}})_{\kappa}\to(\mfrak{G}_{f})_{\kappa}$, the unipotent group $(\mfrak{T}_{0+})_{\kappa}$ is mapped onto the unipotent radical $\mscr{R}_{u}[(\mfrak{T}_{0})_{\kappa}]$ of $(\mfrak{T}_{0})_{\kappa}$, and for any $a\in\Phi$, $(\mfrak{U}_{a,f^{\ast}(a)})_{\kappa}$ is mapped onto the intersection of the unipotent radical $\grp{\overline{R}}_{f}$ and the root subgroup $(\mfrak{U}_{a,f(a)})_{\kappa}$.
	\end{lemlist}	
	We also refer to \cite{BT-2}*{4.6.10} for another proof.

	Now, \itemref{lem:parahoric_reduction:2} tells us that, through the reduction $P_{f^{\ast}}\subset P_{f}\twoheadrightarrow \mfrak{G}_{f}[\kappa]$, $T_{0+}$ (resp. $U_{a,f^{\ast}(a)}$) is mapped to the group of $\kappa$-points of $\mscr{R}_{u}[(\mfrak{T}_{0})_{\kappa}]$ (resp. $\grp{\overline{R}}_{f}\cap(\mfrak{U}_{a,f(a)})_{\kappa}$). Note that, $\grp{\overline{R}}_{f}\cap(\mfrak{U}_{a,f(a)})_{\kappa}$ is the entire $(\mfrak{U}_{a,f(a)})_{\kappa}$ if and only if $f(a)=f^{\ast}(a)$. 
	Then, using the isomorphism \cref{lem:parahoric_reduction:eq}, the second statement of the lemma follows, and we see that the root system of $(\mfrak{\overline{G}}_{f},\mfrak{\overline{T}}_{f})$ is $\Phi_{f}$.
\end{proof}

\index[notation]{P bar (o)@$\overline{P}_{o}$}%
\index[notation]{P bar (o,x)@$\overline{P}_{o,x}$}%
\index[notation]{P sm bar (o,x)@$\mfrak{\overline{P}}_{o,x}$}%
\index[notation]{G sm bar (o)@$\mfrak{\overline{G}}_{o}$}%
Now, back to our situation. 
Let $\overline{P}_{o}$ denote the quotient $P_{o}/P_{f_{o}^{\ast}}$ and  $\overline{P}_{o,x}$ the image of $P_{o,x}$ in it. 
Then we claim that $\overline{P}_{o,x}$ is the group of $\kappa$-points of a parabolic subgroup $\mfrak{\overline{P}}_{o,x}$ of $\mfrak{\overline{G}}_{o}$, the reductive quotient of $(\mfrak{G}_{f_{o}})_{\kappa}$. 
To see this, first note that $P_{o,x}$ is the group $P_{f_{\Set*{o,x}}}$ defined in \cref{eg:parahoric_subgroup}. 
Then we consider the following lemma.

\begin{lemma}
	\index[notation]{Psi (f,g)@$\Psi_{f,g}$}%
	\index[notation]{P sm bar (f,g)@$\mfrak{\overline{P}}_{f,g}$}%
	Let $f,g$ be two concave functions on $\Phi$ with $g\ge f$. Suppose 
	\begin{equation*}
		\Psi_{f,g}:=
		\Set*{ a\in\Phi\given f(a)=g(a) }
	\end{equation*}
	is a parabolic subset of $\Phi$. 
	Then the image of $(\mfrak{G}_{g})_{\kappa}$ in $\mfrak{\overline{G}}_{f}$ is a parabolic subgroup $\mfrak{\overline{P}}_{f,g}$ containing $\mfrak{\overline{T}}_{f}$, corresponding to the parabolic subset $\Psi_{f,g}$. 
\end{lemma}
\begin{proof}
	Since $g\ge f$, we have $P_{g}\subset P_{f}$, which extends to a homomorphism $\mfrak{G}_{g}\to\mfrak{G}_{f}$. 
	By the proof of \cref{lem:parahoric_reduction}, we have the follows.
	First, since $g^{\ast}\ge f^{\ast}$, the image of $\mscr{R}_{u}[(\mfrak{G}_{g})_{\kappa}]$ in $(\mfrak{G}_{f})_{\kappa}$ is contained in $\mscr{R}_{u}[(\mfrak{G}_{f})_{\kappa}]$.
	Then, for each $a\in\Phi$, the image of $(\mfrak{U}_{a,g(a)})_{\kappa}$ in $(\mfrak{G}_{f})_{\kappa}$ is either the entire $(\mfrak{U}_{a,f(a)})_{\kappa}$ if $g(a)=f(a)$ or is contained in $\mscr{R}_{u}[(\mfrak{G}_{f})_{\kappa}]$ if $g(a)\ge f^{\ast}(a)$. 
	Therefore, the image of $(\mfrak{G}_{g})_{\kappa}$ in $\mfrak{\overline{G}}_{f}$ is generated by $\mfrak{\overline{T}}_{f}$ and $\mfrak{\overline{U}}_{f,a}$ for all $a\in\Psi$. 
	Then the statement follows. 
\end{proof}

\index[notation]{Psi (o,x)@$\Psi_{o,x}$}%
\index[notation]{C@$\vC$}%
\index[notation]{Phi positive@$\Phi^{+}$}%
\index[notation]{Phi (o,x)@$\Phi_{o,x}$}%
\index[notation]{I (o,x)@$I_{o,x}$}%
The parabolic subset of $\Phi$ corresponding to $\mfrak{\overline{P}}_{o,x}$ is 
\begin{equation*}
	\Psi_{o,x}	:=
		\Set*{ a\in\Phi \given a(x)\ge 0 }.
\end{equation*}  
Now, we choose a Weyl chamber $\vC$ such that $x\in\overline{o+\vC}$. 
Let $\Phi^{+}$ be the system of positive roots corresponding to $\vC$. 
Then we have $\Psi_{o,x}=\Phi^{+}\cup\Phi_{o,x}$, where $\Phi_{o,x}$ is the root subsystem associated to the concave function $f_{\Set*{o,x}}$:
\begin{equation*}
	\Phi_{o,x} :=
		\Set*{ a\in\Phi \given f_{\Set*{o,x}}(a)\in\Gamma, f_{\Set*{o,x}}(a)+f_{\Set*{o,x}}(-a) = 0 } = 
		\Set*{ a\in\Phi \given a(x)= 0 }.
\end{equation*}
The simple roots in $\Phi^{+}\cap\Phi_{o,x}$ form a \emph{type} $I_{o,x}$ (see \cref{con:type,con:typeOfP}).

\begin{convention}\label{con:typeOfx}
	\index{type!point}%
	Fix a choice of $\vC$. We say a point $x\in\overline{o+\vC}$ \emph{has type $I$} if $I_{o,x}=I$. 
\end{convention}

At this stage, we have 
\begin{equation}\label{eq:IndexDecom}
	\left[P_{o}:P_{o,x}\right] = 
	\left[\overline{P}_{o}:\overline{P}_{o,x}\right]\cdot\left[P_{f_{o}^{\ast}}:P_{f_{o}^{\ast}}\cap P_{o,x}\right],
\end{equation}
where $\overline{P}_{o}$ is (the group of $\kappa$-points of) a splittable reductive group with root system $\Phi$, $\overline{P}_{o,x}$ is (the group of $\kappa$-points of) a parabolic subgroup of the former having type $I_{o,x}$, and 
$P_{f_{o}^{\ast}}$ is a \emph{pro-unipotent group} in the sense that it is a projective limit of groups, each of them is the group of $\kappa$-points of an unipotent group over $\kappa$. 
\index{pro-unipotent group}%

\subsection{Poincar\'{e} polynomials of parabolic subgroups}\label{subsec:PoincarePolynomials}

This subsection treats the first factor $\left[\overline{P}_{o}:\overline{P}_{o,x}\right]$ in \cref{eq:IndexDecom}. That is the index of (the group of $\kappa$-points of) a parabolic subgroup in a splittable reductive group over $\kappa$. 

\index[notation]{GT bar@$(\grp{\overline{G}},\grp{\overline{T}})$}%
\index[notation]{P bar@$\grp{\overline{P}}$}%
\index[notation]{F vectorial@$\vF$}%
\index[notation]{Psi@$\Psi$}%
\index[notation]{I@$I$}%
\index[notation]{Build Tits@$\vBuild$}%
Let $(\grp{\overline{G}},\grp{\overline{T}})$ be a split reductive group over $\kappa$ and $(\grp{\overline{P}},\vF,\Psi,I)$ a quadruple as in \cref{prop:parabolic_type}. 
Then $\grp{\overline{G}}[\kappa]$ acts strongly transitively and type-preserving on the Tits building $\vBuild$ of $\grp{\overline{G}}$, and $\grp{\overline{P}}[\kappa]$ is the stabilizer of $\vF$. 
Hence, the quotient $\grp{\overline{G}}[\kappa]/\grp{\overline{P}}[\kappa]$ counts the facets in $\vBuild$ having type $I$. 
Note that $\grp{\overline{G}}[\kappa]/\grp{\overline{P}}[\kappa]\cong\grp{\overline{G}/\overline{P}}[\kappa]$ according to \emph{Lang's theorem} (see, e.g. \cite{Milne}*{17.98}). 

\index[notation]{Weyl vectorial (I)@$\prescript{v}{}{W_{I}}$}%
\index[notation]{Weyl vectorial (I)@$\vWeyl^{I}$}%
Let $\vWeyl$ be the Weyl group of $(\grp{\overline{G}},\grp{\overline{T}})$. 
Then the \emph{generalized Bruhat decomposition} (see, e.g. \cite{Milne}*{21.h and 21.i}) says that
\begin{equation*}
	\grp{\overline{G}}/\grp{\overline{P}} = \bigsqcup_{\bar{w}\in \vWeyl/\prescript{v}{}{W_{I}}}\grp{C}[\bar{w}] \cong \bigsqcup_{w\in \vWeyl^{I}}\grp{C}[w],
\end{equation*}
where each $\grp{C}[\bar{w}]$ (as well as $\grp{C}[w]$) is an affine space of dimension $\ell(w)$ (the \emph{length} of $w$ in $\vWeyl$), called the \emph{Schubert cell} of $w$, $\prescript{v}{}{W_{I}}$ is the subgroup of $\vWeyl$ generated by reflections with respect to the simple roots in $I$, and $\vWeyl^{I}$ is a system of representatives. 

\begin{definition}
	\index{Poincar\'{e} polynomial}%
	\index[notation]{Poincare (Phi,I)@$\mscr{P}_{\Phi;I}$}%
	\index[notation]{Poincare (Phi)@$\mscr{P}_{\Phi}$}%
	The \emph{Poincar\'{e} polynomial} of the pair $(\Phi,I)$ is the following: 
	\begin{equation}\label{eq:Poincare:def}
		\mscr{P}_{\Phi;I}[z] := 
			\sum_{w\in \vWeyl^{I}} z^{\ell(w)}.
	\end{equation}
	When $I=\emptyset$, it is denoted by $\mscr{P}_{\Phi}$ and called the \emph{Poincar\'{e} polynomial} of $\Phi$. 
\end{definition}

Then we have $\grp{\overline{G}/\overline{P}}[\kappa]=\mscr{P}_{\Phi;I}[q]$.
Note that the image of ${\prod_{a\in\Phi\setminus\Psi}\grp{\overline{U}}_{a}} \to {\grp{\overline{G}}/\grp{\overline{P}}}$	is the \emph{big cell}. 
Hence, $\fun{deg}[\mscr{P}_{\Phi;I}]=\abs*{\Phi\setminus\Psi}$. 

At this stage, we already know that:
\begin{lemma}\label{lem:Poincare:degree}
	The index $\left[\overline{P}_{o}:\overline{P}_{o,x}\right]$ is computed by an integral polynomial $\mscr{P}_{\Phi;I}$ of degree $\abs*{\Phi\setminus\Psi}$.
\end{lemma}
The rest of this subsection aims to deduce $\mscr{P}_{\Phi;I}$ from the information of $(\Phi,I)$.

\begin{lemma}\label{lem:Poincare:Phi}
	$\mscr{P}_{\Phi;I} = \mscr{P}_{\Phi} / \mscr{P}_{\Phi_{I}}$, where $\Phi_{I}$ is the root subsystem of $\Phi$ generated by $I$.
\end{lemma}
\begin{proof}
	Let $\grp{\overline{B}}$ be a Borel subgroup of $\grp{\overline{G}}$ contained in $\grp{\overline{P}}$. Then we have 
	\begin{equation*}\label{lem:Poincare:Phi:ast}\tag{$\ast$}
		\left[\overline{G}:\overline{P}\right] = 
		\left[\overline{G}:\overline{B}\right] \cdot
		\left[\overline{P}:\overline{B}\right]. 
	\end{equation*}
	Let $\grp{\overline{L}}$ be the Levi subgroup of $\grp{\overline{P}}$. Then $(\grp{\overline{L}},\grp{\overline{T}})$ is a split reductive group with root system $\Phi_{I}$. Moreover, $\grp{\overline{B}}\cap\grp{\overline{L}}$ is a Borel subgroup of $\grp{\overline{L}}$. Then we have 
	\begin{equation*}
		\grp{\overline{P}}/\grp{\overline{B}} = 
		\grp{\overline{B}}\grp{\overline{L}}/\grp{\overline{B}} = 
		\grp{\overline{L}}/\grp{\overline{B}}\cap\grp{\overline{L}}.
	\end{equation*}
	Applying this to \cref{lem:Poincare:Phi:ast}, the statement follows.
\end{proof}

\begin{lemma}
	Suppose $\Phi$ can be decomposed into root subsystems $\Phi_{1},\cdots,\Phi_{s}$. Then we have $\mscr{P}_{\Phi}(z) = \mscr{P}_{\Phi_{1}}\cdots\mscr{P}_{\Phi_{s}}$.
\end{lemma}
\begin{proof}
	This is because the decomposition of split reductive groups corresponds to the decomposition of Weyl groups and root systems.
\end{proof}

\index[notation]{Poincare (Xn)@$\mscr{P}_{X_{n}}$}%
Hence, it suffices to know the Poincar\'{e} polynomials of irreducible root systems. 
When $\Phi$ is irreducible of type $X_{n}$, we will denote its Poincar\'{e} polynomial by $\mscr{P}_{X_{n}}$. 
\begin{lemma}\label{lem:PoincareOfXn}
	\index[notation]{d i@$[d_{i}](\:\cdot\:)$}%
	Let $\Phi$ be a reduced root system of rank $n$. Then there are positive integers $d_{1},\cdots,d_{n}$ depending only on the Weyl group $W$ of $\Phi$, such that  
	\begin{equation*}
		\mscr{P}_{\Phi}[z] = 
		\prod_{i=1}^{n}[d_{i}](z),
	\end{equation*}
	where $[d_{i}](z):=1+z+\cdots+z^{d_{i}-1}$. 
\end{lemma}
\begin{proof}
	Let $\grp{G}$ be a complex semisimple group, $\grp{T}$ a maximal torus in it, and $\grp{B}$ a Borel subgroup of $\grp{G}$ containing $\grp{T}$. 
	Let $\Phi$ be the associated root system, $\vWeyl$ the Weyl group, and $\Vect$ be the complexification of the coroot space. 
	Then we have:
	\begin{lemlist}
		\item The complex singular cohomology ring of $\grp{G}/\grp{B}$ vanishes at odd degree and has a basis dual to the Schubert cells (see \cite{BGG73}). 
		\item 
			The \emph{Borel's theorem} (see \cite{Borel}) says that, after dividing degree by two, the complex singular cohomology ring of $\grp{G}/\grp{B}$ is isomorphic to the \emph{coinvariant algebra} $\C[\Vect]_{\vWeyl}$, which is $\C[\Vect]\otimes_{\C[\Vect]^{\vWeyl}}\C$, where $\C[\Vect]$ is the ring of complex polynomial functions on $\Vect$ and $\C[\Vect]^{\vWeyl}$ is the subalgebra of invariant. 
		\item 
		\index[notation]{d i@$d_{i}$}%
			The \emph{Chevalley-Shephard-Todd theorem} (see \cite{Bourbaki}*{chap.VI, \S 3 no.3 thm.3}) says that $\C[\Vect]^{\vWeyl}$ is a polynomial algebra generated by homogeneous polynomials on $\Vect$. 
			Let $d_{1},\cdots,d_{n}$ be the degrees of them. 
	\end{lemlist}
	
	\index{Hilbert-Poincar\'{e} series}%
	Recall that the \emph{Hilbert-Poincar\'{e} series} of a graded commutative $\C$-algebra $S_{\bullet}$ is defined to be $\sum_{d}\fun{dim}_{\C}[S_d]z^d$. 
	Now, considering the Hilbert-Poincar\'{e} series of above graded algebras, the statement follows.
\end{proof}

\index{Weyl group!degrees of}%
The numbers $d_{1},\cdots,d_{n}$ are called the \emph{degrees} of $\vWeyl$ (and of $\Phi$). 
When $\Phi$ is irreducible, they can be found in \cite{Bourbaki}*{chap.VI, \S 4}. 

\index[notation]{Poincare (Xn,I)@$\mscr{P}_{X_{n};I}$}%
For irreducible root systems of type $A_{n}$, $B_{n}$, $C_{n}$, and $D_{n}$, the explicit formulas for their Poincar\'{e} polynomials $\mscr{P}_{X_{n};I}$ with various types $I$ are listed in \cref{sec:Poincares}.

\subsection{Concave functions}\label{subsec:ConcaveF}
This subsection treats the second factor $\left[P_{f_{o}^{\ast}}:P_{f_{o}^{\ast}}\cap P_{o,x}\right]$ in \cref{eq:IndexDecom}. 
It has to be a power of $q$ since $P_{f_{o}^{\ast}}$ is a pro-unipotent group.

\index[notation]{f (o ast x)@$f_{o^\ast x}$}%
First, let $f_{o^\ast x}$ be the following concave function:
\begin{equation*}
	f_{o^\ast x}\colon a\in\widetilde{\Phi} \longmapsto
	\max\Set*{f_{o}^{\ast}(a),f_{\Set*{o,x}}(a)} = \max\Set*{0+,-a(x)}.
\end{equation*}
Then, from the definition of $P_{f}$, we have $P_{f_{o^\ast x}}=P_{f_{o}^{\ast}}\cap P_{o,x}$. Note that both $f_{o^\ast x}$ and $f_{o}^{\ast}$ take the value $0+$ at $0\in\widetilde{\Phi}$.

\begin{lemma}
	Let $f,g$ be two concave functions on $\widetilde{\Phi}$ such that $f(0)=g(0)>0$ and $g\ge f$. Then we have 
	\begin{equation}
		\label{eq:IndexDecomRoots}
		\left[P_{f}:P_{g}\right] = 
		\prod_{a\in\Phi}\abs*{\varphi_{a}^{-1}[f(a),g(a)]},
	\end{equation}
	where $\varphi=(\varphi_{a})_{a\in\Phi}$ is the valuation corresponding to the reference point $o$. 
\end{lemma}
\begin{proof}
	There are two ways to show this. By \cref{lem:SM_T,lem:SM_Ua}, we can extend the decomposition \cref{eq:root_space_decom} to obtain a Lie algebra version of \cref{eq:IndexDecomRoots}. Hence, if the characteristic of $K$ is $0$, the statement follows from the bijective exponential maps of unipotent groups. 
	
	In general case, we can consider the morphism $\mfrak{G}_{g}\to\mfrak{G}_{f}$ obtained by extending the inclusion $P_{g}\subset P_{f}$. Then the multiplicative morphism \cref{eq:SmoothModelPf:decomposition} induces the following commutative diagram for all positive integer $i$.
	\begin{equation*}
		\begin{tikzcd}
			\displaystyle
			{
				\prod_{a\in\Phi^{-}} \mfrak{U}_{a,g(a)}[K^{\circ}/\varpi^i] \cdot
				\mfrak{T}_{g(0)}[K^{\circ}/\varpi^i] \cdot
				\prod_{a\in\Phi^{+}} \mfrak{U}_{a,g(a)}[K^{\circ}/\varpi^i]
			} \ar[d]\arrow[r]	& 
			{
				\mfrak{G}_{g}[K^{\circ}/\varpi^i]
			} \ar[d] \\
			\displaystyle
			{
				\prod_{a\in\Phi^{-}} \mfrak{U}_{a,f(a)}[K^{\circ}/\varpi^i] \cdot
				\mfrak{T}_{f(0)}[K^{\circ}/\varpi^i] \cdot
				\prod_{a\in\Phi^{+}} \mfrak{U}_{a,f(a)}[K^{\circ}/\varpi^i]
			} \arrow[r]	&
			{
				\mfrak{G}_{f}[K^{\circ}/\varpi^i]
			}
		\end{tikzcd}
	\end{equation*}
	By \cref{thm:SmoothModelPf:bijective}, since $f(0)>0$ and $g(0)>0$, the horizontals are isomorphisms. 
	Since $f(0)=g(0)$, at the level of $K^{\circ}/\varpi^i$, we have 
	\begin{align*}
		\MoveEqLeft
		\fun{Coker}[
			\mfrak{G}_{g}[K^{\circ}/\varpi^i]
			\to
			\mfrak{G}_{f}[K^{\circ}/\varpi^i]
		] \\
		& \cong
		\prod_{a\in\Phi}
		\fun{Coker}[
			\mfrak{U}_{a,g(a)}[K^{\circ}/\varpi^i]
			\to
			\mfrak{U}_{a,f(a)}[K^{\circ}/\varpi^i]
		].
	\end{align*}
	By \cref{lem:SM_Ua}, for each $a\in\Phi$, we have 
	\begin{align*}
		\MoveEqLeft
		\fun{Coker}[
			\mfrak{U}_{a,g(a)}[K^{\circ}/\varpi^i]
			\to
			\mfrak{U}_{a,f(a)}[K^{\circ}/\varpi^i]
		] \\
		& \cong
		\fun{Coker}[
			U_{a,g(a)}\otimes_{K^{\circ}}K^{\circ}/\varpi^i
			\to
			U_{a,f(a)}\otimes_{K^{\circ}}K^{\circ}/\varpi^i
		] \\
		& =
		U_{a,f(a)}/U_{a,g(a)} \otimes_{K^{\circ}}K^{\circ}/\varpi^i
		=\varphi_{a}^{-1}[f(a),g(a)] \otimes_{K^{\circ}}K^{\circ}/\varpi^i,
	\end{align*}
	which equals $\varphi_{a}^{-1}[f(a),g(a)]$ if $i\cdot\fun{val}[\varpi]>g(a)-f(a)$ (see \cref{eg:vrgd_reductive}). 
	
	Now, we pass to the limit of the following projective system of homomorphisms.
	\begin{equation*}
		\begin{tikzcd}
			\cdots \arrow[r,two heads]& 
			{\mfrak{G}_{g}[K^{\circ}/\varpi^{i+1}]} 
			\arrow[d]\arrow[r,two heads]& 
			{\mfrak{G}_{g}[K^{\circ}/\varpi^{i}]} 
			\arrow[d]\arrow[r,two heads]& 
			\cdots \arrow[r,two heads]& 
			{\mfrak{G}_{g}[\kappa]} \arrow[d]\\
			\cdots \arrow[r,two heads]& 
			{\mfrak{G}_{f}[K^{\circ}/\varpi^{i+1}]} 
			\arrow[r,two heads]& 
			{\mfrak{G}_{f}[K^{\circ}/\varpi^{i}]} 
			\arrow[r,two heads]& 
			\cdots \arrow[r,two heads]& 
			{\mfrak{G}_{f}[\kappa]}
		\end{tikzcd}
	\end{equation*}
	Then we have 
	\begin{align*}
		P_{f}/P_{g}
		&= \varprojlim_{i}
		\fun{Coker}[
			\mfrak{G}_{g}[K^{\circ}/\varpi^i]
			\to
			\mfrak{G}_{f}[K^{\circ}/\varpi^i]
		]	\\
		&= \varprojlim_{i}
		\prod_{a\in\Phi}\varphi_{a}^{-1}[f(a),g(a)]\otimes_{K^{\circ}}K^{\circ}/\varpi^i	
		= \prod_{a\in\Phi}\varphi_{a}^{-1}[f(a),g(a)].
	\end{align*}
	Then \cref{eq:IndexDecomRoots} follows.
\end{proof}

Applying \cref{eq:IndexDecomRoots} to $f_{o^\ast x}$ and $f_{o}^{\ast}$, we have 
\begin{align*}
	\left[P_{f_{o}^{\ast}}:P_{f_{o}^{\ast}}\cap P_{o,x}\right]
	&=
	\prod_{a\in\Phi}\varphi_{a}^{-1}[0+,\max\Set*{0+,-a(x)}] \\
	&=
	\prod_{a\in\Phi}\varphi_{-a}^{-1}[0+,\max\Set*{0+,a(x)}].
\end{align*}
Then by the definition of $\varphi$ (see \cref{eg:vrgd_reductive}), we have 
\index[notation]{exp(q)@$\fun{exp}_{q}[\:\cdot\:]$}%
\begin{align}
	\label{eq:IndexSecond}
		\left[P_{f_{o}^{\ast}}:P_{f_{o}^{\ast}}\cap P_{o,x}\right]
		&=
		\prod_{a\in\Phi}
			\fun{exp}_{q}[
				\frac{\ceil{\max\Set*{0+,a(x)}}-\ceil{0+}}
					{\fun{val}[\varpi]}
			] \\
		\nonumber &= 
		\prod_{a(x)>0}
			\fun{exp}_{q}[
				\frac{\ceil{a(x)}-\ceil{0+}}
					{\fun{val}[\varpi]}
			],
\end{align}
where $\fun{exp}_{q}[\:\cdot\:]$ is the exponent function with base $q$.

\subsection{Fundamental domain and the proof of \cref{thm:SimplicialVolumeFormula}}\label{subsec:FundamentalDomain}
The following lemma gives us a fundamental domain of $P_{o}$ in $\Build$.
\begin{lemma}\label{lem:FundamentalDomain}
	The convex cone $\overline{o+\vC}$ is a fundamental domain of $P_{o}$.
\end{lemma}
\begin{proof}
	Let $x$ be any point in $\Build$. We need to show that there exists some $g_{x}\in P_{o}$ mapping $x$ into $\overline{o+\vC}$.
	First, let $g.\Aff$ be an apartment containing both $o$ and $x$. More precisely, suppose $o=[g,o+\vect{v}_{0}]$ and $x=[g,o+\vect{v}]$. 
	Then, from the equivalence relation \cref{eq:BT-building}, there is an $n\in N$, such that $o+\vect{v}_{0}=\nu(n).o$ and $gn\in P_{o}$. 
	Let $\vect{v}_{1}\in\Vect$ be the vector $\nu(n)^{-1}.(o+\vect{v})-o$. 
	Since $\overline{\vC}$ is the fundamental domain of $\vWeyl$ in $\Vect$, there is a $w\in\vWeyl$ such that $w.\vect{v}_{1}\in\overline{\vC}$. 
	Now, let $n_{1}$ be a preimage of $w$ under $N_{o}\to W_{o}\cong\vWeyl$. 
	Then $n_{1}n^{-1}g^{-1}\in P_{o}$ and it maps $x$ into $\overline{o+\vC}$.

	On the other hand, if there are two points $x,y\in \overline{o+\vC}$ such that $y=g.x$ for some $g\in P_{o}$. Then, by the \emph{vectorial Bruhat decomposition} \cite{BT-1}*{7.3.4}, we have 
	\begin{equation*}
		g=h_{1}nh_{2},
	\end{equation*}
	where $h_{1},h_{2}\in B_{o,\vC}$ and $n\in N$. Therefore, $n\in N_{o}$, which implies $x=y$ since $\overline{o+\vC}$ is the fundamental domain of $W_{o}$.
\end{proof}

\index[notation]{D (C)@$\mscr{D}[\vC]$}%
We will denote $\overline{o+\vC}$ by $\mscr{D}[\vC]$ to emphasize that it is a fundamental domain. 
Applying \cref{lem:FundamentalDomain} to \cref{eq:SimplicialVolumeStrategy} and using \cref{eq:IndexDecom}, we have 
\begin{align}
	\label{eq:SimplicialVolumeFD}
		\fun{SV}[r] &= 
		\sum_{x\in B(r)\cap\mscr{D}[\vC]}
		\left[\overline{P}_{o}:\overline{P}_{o,x}\right]\cdot\left[P_{f_{o}^{\ast}}:P_{f_{o}^{\ast}}\cap P_{o,x}\right],\\
	\label{eq:SimplicialSurfaceAreaFD}
		\fun{SSA}[r] &= 
		\sum_{x\in \partial(r)\cap\mscr{D}[\vC]}
		\left[\overline{P}_{o}:\overline{P}_{o,x}\right]\cdot\left[P_{f_{o}^{\ast}}:P_{f_{o}^{\ast}}\cap P_{o,x}\right].	
\end{align}
By \cref{lem:Poincare:degree}, the first factor $\left[\overline{P}_{o}:\overline{P}_{o,x}\right]$ is computed by the Poincar\'{e} polynomial $\mscr{P}_{\Phi;{I_{o,x}}}[q]$, which depends only on the \emph{type} of $x$ (see \cref{con:typeOfx}). 
Hence, we can decompose the index sets $B(r)\cap\mscr{D}[\vC]$ and $\partial(r)\cap\mscr{D}[\vC]$ according to the types:
\index[notation]{B (r,C,I)@$B(r,\vC,I)$}%
\index[notation]{partial (r,C,I)@$\partial(r,\vC,I)$}%
\begin{align}
	\label{eq:defIndexSets:Ball}
		B(r,\vC,I) &:=
		\Set*{
			x\in B(r)\cap\mscr{D}[\vC] 
			\given
			\text{ $x$ has type $I$ }
		},\\
	\label{eq:defIndexSets:Sphere}
		\partial(r,\vC,I) &:=
		\Set*{
			x\in \partial(r)\cap\mscr{D}[\vC] 
			\given
			\text{ $x$ has type $I$ }
		}.
\end{align} 
Then \cref{eq:SimplicialVolumeFD,eq:SimplicialSurfaceAreaFD} become the following ones:
\begin{align*}
	\fun{SV}[r] &= 
	\sum_{	I\subset\Delta	}
		\mscr{P}_{\Phi;I}[q]	
		\sum_{	x\in B(r,\vC,I)	}
		\left[P_{f_{o}^{\ast}}:P_{f_{o}^{\ast}}\cap P_{o,x}\right],\\
	\fun{SSA}[r] &= 
	\sum_{	I\subset\Delta	}
		\mscr{P}_{\Phi;I}[q]	
		\sum_{	x\in \partial(r,\vC,I)	}
		\left[P_{f_{o}^{\ast}}:P_{f_{o}^{\ast}}\cap P_{o,x}\right].	
\end{align*}
Applying \cref{eq:IndexSecond} to the above, we have
\begin{align*}
	\fun{SV}[r] &= 
	\sum_{	I\subset\Delta	}
		\mscr{P}_{\Phi;I}[q]	
		\sum_{	x\in B(r,\vC,I)	}
		\prod_{a(x)>0}
			\fun{exp}_{q}[
				\frac{\ceil{a(x)}-\ceil{0+}}
					{\fun{val}[\varpi]}
			],\\
	\fun{SSA}[r] &= 
	\sum_{	I\subset\Delta	}
		\mscr{P}_{\Phi;I}[q]	
		\sum_{	x\in \partial(r,\vC,I)	}
		\prod_{a(x)>0}
			\fun{exp}_{q}[
				\frac{\ceil{a(x)}-\ceil{0+}}
					{\fun{val}[\varpi]}
			].	
\end{align*}
Note that, the ceiling function $\ceil{\:\cdot\:}$ used here follows \cref{con:filtration}. 
If we use the usual ceiling function instead and note that $\fun{deg}[\mscr{P}_{\Phi;{I_{o,x}}}]=\abs*{\Phi\setminus\Psi_{I_{o,x}}}$ equals the number of roots $a\in\Phi$ such that $a(x)>0$, then we obtain the following formulas:
\begin{align}
	\label{eq:SimplicialVolumeFormula}
		\fun{SV}[r] &= 
		\sum_{I\subset\Delta}
		\frac{
			\mscr{P}_{\Phi;I}[q]
		}{
			q^{\fun{deg}[\mscr{P}_{\Phi;I}]}
		}\sum_{x\in B(r,\vC,I)}
		\prod_{a(x)>0}q^{\ceil{a(x)}},\\
	\label{eq:SimplicialSurfaceAreaFormula}
		\fun{SSA}[r] &= 
		\sum_{I\subset\Delta}
		\frac{
			\mscr{P}_{\Phi;I}[q]
		}{
			q^{\fun{deg}[\mscr{P}_{\Phi;I}]}
		}\sum_{x\in \partial(r,\vC,I)}
		\prod_{a(x)>0}q^{\ceil{a(x)}}.
\end{align}
This proves \cref{thm:SimplicialVolumeFormula}.

\begin{remark}
	If the valuation $\fun{val}[\:\cdot\:]$ is normalized, namely $\fun{val}[\varpi]=1$ and hence $\Gamma$ equals the additive group of integers $\Z$, then the two versions of ceiling functions $\ceil{\:\cdot\:}$ coincide and the formulas \cref{eq:SimplicialVolumeFormula,eq:SimplicialSurfaceAreaFormula} can be understood in either way.
\end{remark}
\begin{convention}\label{con:valuation}
	From now on, we assume the valuation $\fun{val}[\:\cdot\:]$ is normalized.
\end{convention}

\subsection{Variants of the simplicial volume}
\index{simplicial volume!variant of}%
\index{simplicial surface area!variant of}%
Let $\tau\colon\mcal{V}\to\mfrak{I}$ be a function factoring through the type function $x\mapsto I_{x,o}$. 
Then we can define the $\tau$-variants of the simplicial volume $\fun{SV}[\:\cdot\:]$ and the simplicial surface area $\fun{SSA}[\:\cdot\:]$ as follows. 
For any $\dagger\in\mfrak{I}$, the quantities $\fun{SV}_{\dagger}[r]$ and $\fun{SSA}_{\dagger}[r]$ count the following sets respectively: 
\index[notation]{B dagger(r)@$B_{\dagger}(r)$}%
\index[notation]{partial dagger(r)@$\partial_{\dagger}(r)$}%
\begin{align*}
	B_{\dagger}(r)	&:=	
		\Set*{	x\in B(r)	\given	\tau(x)=\dagger	},\\
	\partial_{\dagger}(r)	&:=	
		\Set*{	x\in \partial(r)	\given	\tau(x)=\dagger	}.
\end{align*}
Following \cref{eq:defIndexSets:Ball,eq:defIndexSets:Sphere}, we can introduce the following subsets:
\index[notation]{B dagger(r,C,I)@$B_{\dagger}(r,\vC,I)$}%
\index[notation]{partial dagger(r,C,I)@$\partial_{\dagger}(r,\vC,I)$}%
\begin{align*}
		B_{\dagger}(r,\vC,I) &:=
		\Set*{
			x\in B_{\dagger}(r)\cap\mscr{D}[\vC] 
			\given
			\text{ $x$ has type $I$ }
		},\\
		\partial_{\dagger}(r,\vC,I) &:=
		\Set*{
			x\in \partial_{\dagger}(r)\cap\mscr{D}[\vC] 
			\given
			\text{ $x$ has type $I$ }
		}.
\end{align*}
Then the same argument for \cref{thm:SimplicialVolumeFormula} works and gives us the following formulas:
\begin{align}
	\label{eq:SimplicialVolumeFormulaVar}
		\fun{SV}_{\dagger}[r] &= 
		\sum_{I\subset\Delta}
		\frac{
			\mscr{P}_{\Phi;I}[q]
		}{
			q^{\fun{deg}[\mscr{P}_{\Phi;I}]}
		}\sum_{x\in B_{\dagger}(r,\vC,I)}
		\prod_{a(x)>0}q^{\ceil{a(x)}},\\
	\label{eq:SimplicialSurfaceAreaFormulaVar}
		\fun{SSA}_{\dagger}[r] &= 
		\sum_{I\subset\Delta}
		\frac{
			\mscr{P}_{\Phi;I}[q]
		}{
			q^{\fun{deg}[\mscr{P}_{\Phi;I}]}
		}\sum_{x\in \partial_{\dagger}(r,\vC,I)}
		\prod_{a(x)>0}q^{\ceil{a(x)}}.
\end{align}
Such variants may be interesting when we need to focus on certain types of vertices (for instance, when not all vertices are special). 
\index[notation]{dagger@$\dagger$}%
In this paper, we will consider $\dagger=$ ``being special'', although other variants are also worth considering.

\clearpage
\section{Vertices in an apartment}
\label{sec:Vertices}
The purpose of this section is to deduce explicit descriptions of the index sets $B(r,\vC,I)$ and $\partial(r,\vC,I)$ in \cref{thm:SimplicialVolumeFormula}. 
In \cref{subsec:ReduceToIrr}, this problem will be reduced to irreducible cases. 
Then, in \cref{subsec:IndexSet}, a framework will be established, which allow us to give explicit descriptions of the index sets when the Bruhat-Tits building $\Build$ is of split classical type. 
Finally, in \cref{subsec:VerticesInAptAn,subsec:VerticesInAptCn,subsec:VerticesInAptBn,subsec:VerticesInAptDn}, explicit descriptions of the index sets are obtained by a study of the vertices in the affine apartments of split type $A_{n}$, $C_{n}$, $B_{n}$, and $D_{n}$ respectively.

\subsection{Reduce to irreducible ones}\label{subsec:ReduceToIrr}
\begin{proposition}\label{prop:ReduceToIrr}
	Suppose $\Build$ is decomposed into irreducible ones:
	\begin{equation*}
		\Build = 
		\Build_{1}\times\cdots\times\Build_{m}.
	\end{equation*}
	Let $\fun{SSA}_{i}[\:\cdot\:]$ ($1\le i\le m$) be the simplicial surface area in $\Build_{i}$.
	Then, we have 
	\begin{equation*}
		\fun{SSA}[r] = \sum_{r_{1}+\cdots+r_{m}=r}\fun{SSA}_{1}[r_{1}]\cdots\fun{SSA}_{m}[r_{m}].
	\end{equation*}
\end{proposition}
\begin{proof}
	Let $\partial_{i}(r)$ ($1\le i\le m$) be the simplicial sphere in $\Build_{i}$, then we need to show:
	\begin{equation*}
		\partial(r) = 
		\bigsqcup_{r_{1}+\cdots+r_{m}=r}
			\partial^{1}(r_{1})\times\cdots\times\partial^{m}(r_{m}).
	\end{equation*}
	This follows from the following lemma.
\end{proof}
\begin{lemma}
	Suppose $\Build=\Build_{1}\times\Build_{2}$ is a decomposition of Bruhat-Tits buildings. 
	Let $d_{i}(\:\cdot\:,\:\cdot\:)$ ($i=1,2$) be the simplicial distance on $\Build_{i}$ and $\fun{pr}_{i}$ the canonical projection from $\Build$ to $\Build_{i}$. Then we have 
	\begin{equation}
		\label{eq:SimplicialDistanceDecom}
		d(x,y) = d_{1}(\fun{pr}_{1}[x],\fun{pr}_{1}[y])+d_{2}(\fun{pr}_{2}[x],\fun{pr}_{2}[y]).
	\end{equation}
\end{lemma}
\begin{proof}
	First, the left-hand side is no larger than the right-hand side by triangle inequality. To show the equality, we only need to show that if $x,y$ are adjacent, then $\fun{pr}_{i}[x]=\fun{pr}_{i}[y]$ holds either for $i=1$ or $2$. Indeed, $x,y$ are adjacent means that the segment $[x,y]$ contains no vertex inside it and the set 
	\begin{equation*}
		\Psi=\Set*{ a\in\Phi\given [x,y]\subset \partial\alpha\text{ for some }\alpha\in\Sigma }
	\end{equation*} 
	has rank one less than $\Phi$. Then, we must have $\Psi\cap\Phi_{i}=\Phi_{i}$ for either $i=1$ or $2$. Suppose $\Psi\cap\Phi_{1}=\Phi_{1}$. 
	For any $(a,\alpha)\in\mscr{E}$, $\fun{pr}_{1}(\partial\alpha)$ is either a wall in $\Build_{1}$ (if $a\in\Phi_{1}$) or the entire building (if $a\notin\Phi_{1}$). 
	Therefore, $\fun{pr}_{1}([x,y])$ is a vertex and hence $\fun{pr}_{1}[x]=\fun{pr}_{1}[y]$.
\end{proof}

In particular, in order to compute the simplicial surface area in general, it suffices to do that for irreducible ones.
Since $\fun{SSA}[r]=\fun{SV}[r]-\fun{SV}[r-1]$, the same holds for the simplicial volume.

\subsection{Generality on vertices and the index sets}\label{subsec:IndexSet}
\Cref{prop:ReduceToIrr} suggests that, in order to compute the simplicial volume and simplicial surface area in general Bruhat-Tits buildings, we only need to do so in irreducible ones.
Now, suppose $\Build$ is an irreducible Bruhat-Tits building of split type $X_{n}$. Then we can deduce explicit descriptions of the index sets $B(r,\vC,I)$ and $\partial(r,\vC,I)$ as follows.
\begin{step}
	\index[notation]{R (m)@$\R^m$}%
	\index[notation]{Vect@$\Vect$}%
	\index[notation]{e i@$\vect{e}_{i}$}%
	\index[notation]{chi i@$\chi_{i}$}%
	\index[notation]{inner@$\inner{}{}$}%
	\index[notation]{Aff@$\Aff$}%
	We first fix a realization of the affine apartment $\Apt[\Phi]$. For this purpose, we start with the Euclidean space $\R^m$.  Then the underlying Euclidean space $\Vect$ of $\Phi$ is a certain subspace of $\R^m$. 
	We use $(\vect{e}_{1},\cdots,\vect{e}_{m})$ to denote the standard basis of $\R^m$ and $(\chi_{1},\cdots,\chi_{m})$ the dual basis in $(\R^m)^{\ast}$. 
	By an abuse of notation, we do not distinguish $\chi_{i}$ from its restriction to $\Vect$.
	The standard inner product on $\R^m$ is denoted by $\inner{}{}$.
	Then the underlying Euclidean affine space of $\Apt[\Phi]$ can be written as $\Aff=o+\Vect$. 
	We keep the convention that any linear function on $\Vect$ is also viewed as an affine function on $\Aff$ by taking $o$ as the reference point. 

	\index[notation]{Phi@$\Phi$}%
	\index[notation]{Q (check)@$\mcal{Q}^{\vee}$}%
	\index[notation]{Apt (Phi)@$\Apt[\Phi]$}%
	\index[notation]{Weyl vectorial@$\vWeyl$}%
	\index[notation]{Weyl@$W$}%
	The root system $\Phi$ can be written in terms of linear functions on $\Vect$, and we thus obtain a concrete description of the \emph{coroot lattice} $\mcal{Q}^{\vee}$ in $\Vect$. 
	This gives the translation group of the apartment $\Apt[\Phi]$. 
	On the other hand, the action of the Weyl group $\vWeyl$ has a concrete geometric interpretation on $\Vect$. 
	Then the affine Weyl group $W$ is obtained as the semi-product of them.  
\end{step}

\begin{step}\label{step:WeylChamber}
	\index[notation]{C@$\vC$}%
	\index[notation]{Phi positive@$\Phi^{+}$}%
	\index[notation]{Delta@$\Delta$}%
	\index[notation]{a i@$a_{i}$}%
	\index[notation]{a 0@$a_{0}$}%
	\index[notation]{two rho@$2\rho$}%
	Next, we choose a Weyl chamber $\vC$ and describe the associated system of positive roots $\Phi^{+}$, the system of simple roots $\Delta=\Set*{	a_{1},\cdots,a_{n}	}$, and the highest root $a_{0}$. In addition, we express the sum of positive roots $2\rho$ by simple roots. 
\end{step}

\begin{step}
	\index{fundamental coweights}%
	\index[notation]{omega i@$\omega_{i}$}%
	\index{coweights lattice}%
	\index[notation]{P check@$\mcal{P}^{\vee}$}%
	The \emph{fundamental coweights} $\omega_{1},\cdots,\omega_{n}$ relative to $\Delta$ are the vectors in $\Vect$ such that $a_{i}(\omega_{j})=\delta_{ij}$ for all $1\le i,j\le n$. 
	They form a basis of the \emph{coweight lattice} $\mcal{P}^{\vee}$ in $\mathbb{V}$.
	Recall that special vertices in $\Apt[\Phi]$ are points $x\in\Aff$ such that $a(x)\in\Gamma$ for all $a\in\Phi$. 
	Hence, the set of special vertices are precisely $o+\mcal{P}^{\vee}\otimes_{\Z}\Gamma$. 
\end{step}

The above can be found in \cite{Bourbaki}*{chap.VI, \S4, no. 5-9}.

\begin{step}
	\index[notation]{Delta tilde@$\widetilde{\Delta}$}%
	\index[notation]{alpha i@$\alpha_{i}$}%
	\index[notation]{alpha 0@$\alpha_{0}$}%
	\index[notation]{C@$C$}%
	\index[notation]{v i@$v_{i}$}%
	\index[notation]{h i@$h_{i}$}%
	The simple roots $a_{1},\cdots,a_{n}$, together with the highest root $a_{0}$, give rise to a basis $\widetilde{\Delta}=\Set*{	\alpha_{0},\alpha_{1},\cdots,\alpha_{n}	}$ of $\Sigma$ as in \cref{eg:Fund_alcove}. 
	Since we have assumed that $\fun{val}[\:\cdot\:]$ is normalized, we have $\alpha_{0}=\Set*{	x\in\Aff	\given	-a_{0}(x)+1\ge 0	}$. 
	Hence, the fundamental alcove $C$ associated to $\widetilde{\Delta}$ can be written as
	\begin{equation*}
		C := \Set*{
			o+\vect{v} \given
			\vect{v}\in \vC, 
			a_{0}(\vect{v})<1
		}.
	\end{equation*} 
	Let $v_{0}=o,v_{1},\cdots,v_{n}$ be its extreme points, where each $v_{i}$ is opposite to the wall $\partial\alpha_{i}$.  
	Then we have (recall \cref{con:color} for $h_{i}$)
	\begin{equation*}
		a_{j}(v_{i}) = h_{i}^{-1}\delta_{ij},
		\qquad\text{for all}\qquad
		1\le j\le n
	\end{equation*}
	Therefore, $v_{i}=o+h_{i}^{-1}\omega_{i}$. 
\end{step}

\begin{step}
	\index[notation]{V i@$\mcal{V}_{i}$}%
	\index[notation]{V@$\mcal{V}$}%
	\index[notation]{h@$h$}%
	We follow \cref{con:color}. 
	Then a vertex \emph{has color $i$} if it is conjugated to $v_{i}$ by the affine Weyl group $W$.
	Let $\mcal{V}_{i}$ be the sets of vertices in $\Apt[\Phi]$ having color $i$. 
	Recall that $W$ is the semi-product of $W_{o}\cong\vWeyl$ and $\mcal{Q}^{\vee}$. 
	Then we have 
	\begin{equation*}
		\mcal{V}_{i} = W_{o}.v_{i} + \mcal{Q}^{\vee}.
	\end{equation*}
	Note that for any $1\le j\le n$, we have 
	\begin{equation*}
		r_{\alpha_{j}}(v_{i}) = 
			v_{i} + a_{j}(v_{i})a_{j}^{\vee} = 
			v_{i} + h_{i}^{-1}\delta_{ij}a_{j}^{\vee}.
	\end{equation*}
	Since $W_{o}$ is generated by $\Set*{	r_{\alpha_{j}}	\given	1\le j\le n	}$, we see that 
	\begin{equation}
		\label{eq:VerticesBetweenCoroots}
		v_{i} + \mcal{Q}^{\vee} \subset 
			\mcal{V}_{i} \subset 
			v_{i} + h_{i}^{-1}\mcal{Q}^{\vee}.
	\end{equation}
	In particular, if $v_{i}$ is a special vertex, then $\mcal{V}_{i} = v_{i} + \mcal{Q}^{\vee}$. 
	In general, $\mcal{V}_{i}$ is obtained by computing $W_{o}.v_{i}$. 
	The set $\mcal{V}$ of vertices in $\Apt[\Phi]$ is then the disjoint union of $\mcal{V}_{i}$ for $0\le i\le n$. 
	Let $h$ be the maximum of $h_{1},\cdots,h_{n}$. 
	Then we have 
	\begin{equation}
		\label{eq:VerticesBetweenCoweights}
		o+\mcal{P}^{\vee}	\subset	
			\mcal{V}	\subset	
			o+\tfrac{1}{h}\mcal{P}^{\vee}.
	\end{equation}
\end{step}

\begin{step}
	The next step is to characterize the simplicial distance in terms of roots. 
	The goal is to prove \cref{thm:SimplicialDistance}. 
	
	We begin with some technical lemmas.
	\begin{lemma}\label{lem:AdjacentVertices}
		Suppose $\Set*{a_{[i]}\given 1\le i\le n}$ is a linearly independent set of roots. Let $x,y$ be two vertices such that $a_{[i]}(x)=a_{[i]}(y)=k_{i}\in\Z$ for all $1\le i\le n$ expect $i=i_{0}$ and that $|a_{[i_{0}]}(x)-a_{[i_{0}]}(y)|=h^{-1}$. Then $x$ and $y$ are adjacent.
	\end{lemma}
	\begin{proof}
		We need to show the segment $[x,y]$ is an edge. Since it already lies in the intersection of the walls $\partial\alpha_{a_{[i]}-k_{i}}$ ($i\neq i_{0}$), it remains to show that there is no vertex inside this segment. Suppose $x_{t}:=tx+(1-t)y$ ($0<t<1$) is a vertex inside the segment $[x,y]$. By \cref{eq:VerticesBetweenCoweights}, we have $ha_{[i_{0}]}(x_{t})$ and $ha_{[i_{0}]}(y)\in\Z$. But
		\begin{equation*}
			ha_{[i_{0}]}(x_{t}) = 
			ha_{[i_{0}]}(y) + th(a_{[i_{0}]}(x)-a_{[i_{0}]}(y)) = 
			ha_{[i_{0}]}(y) \pm t. 
		\end{equation*}
		Therefore, $t\in\Z$, which is a contradiction.
	\end{proof}
	This lemma will be used later. For now, we need another more specific lemma.
	\begin{lemma}\label{lem:quasiAdjacentVertices}
		Suppose $h_{i_{0}}\le 2$. Let $x,y$ be two vertices such that $a_{i}(x)=a_{i}(y)=k_{i}\in\Z$ for all $1\le i\le n$ except $i=i_{0}$ and that either $a_{i_{0}}(x)\in\Z$ or $a_{i_{0}}(y)\in\Z$. Then there is a path of length $m=\abs*{a_{i_{0}}(x)-a_{i_{0}}(y)}h_{i_{0}}$ between them. 
	\end{lemma}
	\begin{proof}
		We may assume that $a_{i_{0}}(y)\in\Z$ and that $a_{i_{0}}(x)-a_{i_{0}}(y)>0$.
		Consider the sequence $x_{j}=\tfrac{j}{m}x+(1-\tfrac{j}{m})y$ ($0\le j\le m$). Then we have $a_{i}(x_{j})=k_{i}$ for all $1\le i\le n$ except $i=i_{0}$. Moreover, since $a_{i_{0}}(x_{j})=a_{i_{0}}(y)+\tfrac{j}{h_{i_{0}}}$, we have $a_{0}(x_{j})\in\Z$.

		Now, we need to show that the segment $[x_{j-1},x_{j}]$ is an edge for all $1\le j\le m$. 
		Since this segment already lies in the intersection of the walls $\partial\alpha_{a_{i}-k_{i}}$ ($i\neq i_{0}$), it remains to show that there is no vertex inside it. 
		Suppose $x_{t}:=tx_{j-1}+(1-t)x_{j}$ ($0<t<1$) is a vertex inside the segment $[x_{j-1},x_{j}]$. Then there must be another root $a=c_{1}a_{1}+\cdots+c_{n}a_{n}$ linearly independent of $a_{i}$ ($i\neq i_{0}$) such that $a(x_{t})\in\Z$. Then $c_{i_{0}}a_{i_{0}}(x_{t})$ has to be a nonzero integer. Now, we have
		\begin{align*}
			c_{i_{0}}a_{i_{0}}\left(x_{t}\right) &= 
			c_{i_{0}}a_{i_{0}}\left(tx_{j-1}+(1-t)x_{j}\right)	\\		&= 
			c_{i_{0}}a_{i_{0}}\left(
				y + \frac{t(j-1)+(1-t)j}{m}(x-y)
			\right)	\\
			&=	
			c_{i_{0}}a_{i_{0}}(y) + \frac{c_{i_{0}}}{h_{i_{0}}}(j-t). 
		\end{align*}
		Therefore, we have $\tfrac{c_{i_{0}}}{h_{i_{0}}}(j-t)\in\Z$.
		By the basic property of the highest root, we must have $1\le c_{i_{0}}\le h_{i_{0}}$. 
		If $c_{i_{0}}=h_{i_{0}}$, then $t\in\Z$, a contradiction. 
		If $c_{i_{0}}<h_{i_{0}}$, since we have assumed that $h_{i_{0}}\le 2$, we must have $c_{i_{0}}=1$. Then $\tfrac{1}{h_{i_{0}}}(j-t)\in\Z$ implies that $t\in\Z$, a contradiction.
	\end{proof}		

	Now, we can prove the following weaker result:
	\begin{lemma}\label{lem:SimplicialDistanceInHighestRootSpecialCase}
		Suppose $h\le 2$. Then, for any \emph{special} vertex $x\in\mscr{D}[\vC]$, we have
		\begin{equation*}
			d(x,o)\le r \iff a_{0}(x)\le r.
		\end{equation*}
	\end{lemma}
	\begin{proof}
		First, since any edge intersects with a wall by a vertex, we must have 
		\begin{equation*}
			d(o,x)\ge a_{0}(x).
		\end{equation*}
		This shows the only if part.
		
		Next, we consider the sequence $x_{i}=x_{i-1}+a_{i}(x)\omega_{i}$ ($1\le i\le n$) with $x_{0}=o$. Then all $x_{i}$ are special vertices and for the successive vertices $x_{i-1}$ and $x_{i}$ have the property that $a_{j}(x_{i-1})=a_{j}(x_{i})$ for all $1\le j\le n$ except $j=i$. 
		Hence, by \cref{lem:quasiAdjacentVertices}, there is a path of length $h_{i}a_{i}(x)$ between them. In this way, we obtain a path from $o$ to $x$ of length 
		\begin{equation*}
			h_{1}a_{1}(x)+\cdots+h_{n}a_{n}(x) = a_{0}(x).
		\end{equation*}
		This proves the lemma.
	\end{proof}
	The argument for the only if part works for \cref{thm:SimplicialDistance}. As for the if part, we can construct a suitable path from $x$ to a special vertex $x_{0}\in\mscr{D}[\vC]$, verifying that
	\begin{equation}\label{eq:SpecialX0inLemma}
		d(x_{0},x)\le\ceil{a_{0}(x)}-a_{0}(x_{0}).
	\end{equation}
	Indeed, if such a path exists, by \cref{lem:SimplicialDistanceInHighestRootSpecialCase}, we have 
	\begin{equation*}
		d(o,x)\le 
		d(o,x_{0}) + d(x_{0},x) \le 
		a_{0}(x_{0})+\ceil{a_{0}(x)}-a_{0}(x_{0}) = 
		\ceil{a_{0}(x)}\le r.
	\end{equation*}
	This proves \cref{thm:SimplicialDistance}.
\end{step}

\begin{step}
	Now, we can describe the index set $B(r,\vC,I)$ and $\partial(r,\vC,I)$ as follows. 
	First, by \cref{thm:SimplicialDistance}, we see that 
	\begin{align*}
		B(r)\cap\mscr{D}[\vC] &= 
		\mscr{D}[\vC]\cap\mcal{V}\cap\alpha_{-a_{0}+r},\\
		\partial(r)\cap\mscr{D}[\vC] &= 
		\mscr{D}[\vC]\cap\mcal{V}\cap(\alpha_{-a_{0}+r}\setminus\alpha_{-a_{0}+r-1}).
	\end{align*}
	For a type $I$, let $\vC_{I}$ denote the face of $\vC$ having type $I$:
	\index[notation]{C (I)@$\vC_{I}$}%
	\begin{equation*}
		\vC_{I}:=\Set*{
				\vect{v}\in \Vect
				\given 
				\forall a\in I, a(\vect{v})=0;
				\forall a\in \Delta\setminus I, a(\vect{v})>0
			}.
	\end{equation*} 
	Recall that for any point $x$, its type is $I_{o,x} = \Set*{	a\in\Delta \given a(x)= 0	}$. 
	Hence, a point $x\in\mscr{D}[\vC]$ has type $I$ if and only if $x\in o+\vC_{I}$.
	Then we have 
	\begin{align}
		\label{eq:IndexSetsUsingRoots:Ball}
			B(r,\vC,I) &= 
			(o+\vC_{I})\cap\mcal{V}\cap\alpha_{-a_{0}+r},\\
		\label{eq:IndexSetsUsingRoots:Sphere}
			\partial(r,\vC,I) &= 
			(o+\vC_{I})\cap\mcal{V}\cap(\alpha_{-a_{0}+r}\setminus\alpha_{-a_{0}+r-1}).
	\end{align}	
	The $\dagger$-variants are obtained similarly.

	\index[notation]{dagger@$\dagger$}%
	By \cref{con:type}, a point $x$ has type $I$ if and only if it is of the form 
	\begin{equation*}
		x=o+c_{1}\omega_{\ell_{1}}+\cdots+c_{t}\omega_{\ell_{t}}.
	\end{equation*} 
	Then the condition $x\in o+\vC_{I}$ can be interpreted as ``$c_{1},\cdots,c_{t}>0$''.
	Next, $x$ is a \emph{special} vertex if and only if $c_{1},\cdots,c_{t}\in\Z$.
	Let $\dagger=$ ``being special''.	
	Then, by \cref{eq:IndexSetsUsingRoots:Ball,eq:IndexSetsUsingRoots:Sphere}, we have the following explicit descriptions:
	\begin{align}
		\label{eq:expressIndexSets_speical:Ball}
			B_{\dagger}(r,\vC,I) &=
				\Set*{
					x=o+c_{1}\omega_{\ell_{1}}+\cdots+c_{t}\omega_{\ell_{t}}
					\given
					\begin{array}{c}
						c_{1},\cdots,c_{t}\in\Z_{>0},\\
						h_{\ell_{1}}c_{1}+\cdots+h_{\ell_{t}}c_{t} \le r
					\end{array}
				},\\
		\label{eq:expressIndexSets_speical:Sphere}
			\partial_{\dagger}(r,\vC,I) &=
				\Set*{
					x=o+c_{1}\omega_{\ell_{1}}+\cdots+c_{t}\omega_{\ell_{t}}
					\given
					\begin{array}{c}
						c_{1},\cdots,c_{t}\in\Z_{>0},\\
						h_{\ell_{1}}c_{1}+\cdots+h_{\ell_{t}}c_{t} = r
					\end{array}
				},
	\end{align}
	\index[notation]{Z > 0@$\Z_{>0}$}%
	where $\Z_{>0}$ denotes the set of positive integers.
\end{step}

\subsection{Vertices in the apartment \texorpdfstring{$\Apt[A_{n}]$}{An}}\label{subsec:VerticesInAptAn}
\begin{step}
	The underlying Euclidean vector space $\Vect$ is the following:
	\begin{equation*}
		\mathbb{V} := 
		\Set*{
			\vect{v} \in\R^{n+1}
			\given 
			\chi_{1}(\vect{v})+\cdots+\chi_{n+1}(\vect{v})=0
		}.
	\end{equation*}
	Its dual space is $\mathbb{V}^{\ast}=(\R\chi_{1}\oplus\cdots\oplus\R\chi_{n+1})/\R(\chi_{1}+\cdots+\chi_{n+1})$.
	After identifying each $\chi_{i}$ with its restriction to $\mathbb{V}$, the root system can be written as follows:
	\begin{equation*}
		\Phi := 
		\Set*{
			\chi_{i}-\chi_{j}
			\given
			1\le i,j\le n+1
		}.
	\end{equation*}
	Then the coroot lattice $\mcal{Q}^{\vee}$ is the restriction of the standard lattice $\Z^{n+1}$ in $\R^{n+1}$ to $\Vect$, and the Weyl group $\vWeyl$ acts on $\mathbb{V}$ as permutations of coordinates.
\end{step}

\begin{step}\label{step:WeylChamberAn}
	We can choose the following Weyl chamber $\vC$:
	\begin{equation*}
		\vC := 
		\Set*{
			\vect{v}\in\mathbb{V}
			\given
			\chi_{i}(\vect{v})>\chi_{j}(\vect{v})
			\text{ for all } 
			1 \le i < j \le n+1
		}.
	\end{equation*}
	Then the system of positive roots $\Phi^{+}$ associated to $\vC$ is the following:
	\begin{equation*}
		\Phi^{+} := 
		\Set*{
			\chi_{i}-\chi_{j}
			\given
			1\le i< j\le n+1
		}.
	\end{equation*}
	Among them, the simple roots are the following:
	\index[notation]{a i@$a_{i}$}%
	\begin{flalign*}
		&& 
			a_{i} &:= 
				\chi_{i}-\chi_{i+1}.
		&
			\mathllap{(1\le i\le n)}
	\end{flalign*}
	Using the basis $\Delta=\Set*{a_{1},\cdots,a_{n}}$, the positive roots can be written as follows:
	\begin{flalign}\label{eq:Roots_An}
		&& 
			\chi_{i}-\chi_{j} &= 
				a_{i}+\cdots+a_{j-1}.
		&
			\mathllap{(1\le i< j\le n+1)}
	\end{flalign}
	Among them, the highest root $a_{0}$ relative to $\Delta$ is 
	\index[notation]{a 0@$a_{0}$}%
	\begin{equation}\label{eq:HighestRoot_An}
		a_{0} :=
			\chi_{1}-\chi_{n+1} = a_{1} + \cdots + a_{n}.
	\end{equation}
	Moreover, the sum of positive roots is 
	\index[notation]{two rho@$2\rho$}%
	\begin{equation}\label{eq:2rho_An}
		2\rho = 
			\sum_{i=1}^{n}i(n+1-i)a_{i}.
	\end{equation} 
\end{step}

\begin{step}
	\index[notation]{omega i@$\omega_{i}$}%
	The fundamental coweights relative to $\Delta$ are the following:
	\begin{flalign*}
		&&
		\omega_{i} &:= 
			(\vect{e}_{1}+\cdots+\vect{e}_{i}) - \tfrac{i}{n+1}(\vect{e}_{1}+\cdots+\vect{e}_{n+1}).
		&
			\mathllap{(1\le i\le n)}
	\end{flalign*}
	Then the special vertices are $o+\Z\omega_{1}+\cdots+\Z\omega_{n}$.
\end{step}

\begin{step}
	Associated to $\Delta$, the fundamental alcove $C$ can be written as follows:
	\begin{equation*}
		C :=
		\Set*{
			x\in\Aff \given
				\chi_{1}(x)>\cdots>\chi_{n+1}(x),
				\chi_{1}(x)-\chi_{n+1}(x)<1
		}.
	\end{equation*}
	The extreme points of $C$ other than $v_{0}=o$ are the following:
	\index[notation]{v i@$v_{i}$}%
	\begin{flalign*}
		&&
		v_{i} &:= o+\omega_{i} = 
			o+(\vect{e}_{1}+\cdots+\vect{e}_{i}) - \tfrac{i}{n+1}(\vect{e}_{1}+\cdots+\vect{e}_{n+1}).
		&
			\mathllap{(1\le i\le n)}
	\end{flalign*} 
	Note that all of them are special vertices.
\end{step}

\begin{step}\label{step:VerticesAn}
	Since each $v_{i}$ is a special vertex, we have 
	\begin{equation*}
		\mcal{V}_{i} = v_{i}+\mcal{Q}^{\vee} = 
		\Set*{
			x\in\Aff\given
			\chi_{j}(x)+\tfrac{i}{n+1}\in\Z
			\text{ for all }1\le j\le n+1
		}.
	\end{equation*}
	In particular, all vertices are special and hence $\mcal{V}=o+\mcal{P}^{\vee}$.
\end{step}

\begin{step}
	Since all vertices are special, the $A_{n}$ case of \cref{thm:SimplicialDistance} follows from \cref{lem:SimplicialDistanceInHighestRootSpecialCase}.
\end{step}

\begin{step}\label{step:IndexBsrciAn}
	Let $I$ be a type and follow \cref{con:type}. 
	Since all vertices are special, by \cref{eq:expressIndexSets_speical:Ball,eq:expressIndexSets_speical:Sphere}, we have the following explicit descriptions:
	\begin{align}
			B(r,\vC,I) &=
				\Set*{
					x=o+c_{1}\omega_{\ell_{1}}+\cdots+c_{t}\omega_{\ell_{t}}
					\given
					\begin{array}{c}
						c_{1},\cdots,c_{t}\in\Z_{>0},\\
						c_{1}+\cdots+c_{t} \le r
					\end{array}
				},\\
		\label{eq:expressIndexSets_An}
			\partial(r,\vC,I) &=
				\Set*{
					x=o+c_{1}\omega_{\ell_{1}}+\cdots+c_{t}\omega_{\ell_{t}}
					\given
					\begin{array}{c}
						c_{1},\cdots,c_{t}\in\Z_{>0},\\
						c_{1}+\cdots+c_{t} = r
					\end{array}
				}.
	\end{align}
\end{step}

\subsection{Vertices in the apartment \texorpdfstring{$\Apt[C_{n}]$ ($n\ge 2$)}{Cn}}\label{subsec:VerticesInAptCn}
\begin{step}\label{step:WeylGroupCn}
	The underlying Euclidean vector space $\Vect$ is the entire $\R^n$, and its dual space $\mathbb{V}^{\ast}$ is thus spanned by the coordinate functions $\chi_{1},\cdots,\chi_{n}$. 
	Then the root system can be written as follows: 
	\begin{equation*}
		\Phi := 
		\Set*{
			\pm\chi_{i}\pm\chi_{j}\given 
			1\le i<j\le n
		}\cup
		\Set*{
			\pm 2\chi_{i}\given
			1\le i\le n
		}.
	\end{equation*}
	Then its coroot lattice $\mcal{Q}^{\vee}$ is precisely the standard lattice $\Z^n$ in $\R^n$, and the Weyl group $\vWeyl$ acts on $\mathbb{V}$ as permutations and sign changes of coordinates.
\end{step}

\begin{step}\label{step:WeylChamberCn}
	We can choose the following Weyl chamber $\vC$:
	\begin{equation*}
		\vC := 
		\Set*{
			\vect{v}\in\mathbb{V}
			\given
			\chi_{i}(\vect{v})>\chi_{j}(\vect{v})>0
			\text{ for all } 
			1 \le i < j \le n
		}.
	\end{equation*}
	Then the system of positive roots $\Phi^{+}$ associated to $\vC$ is the following:
	\begin{equation*}
		\Phi^{+} := 
			\Set*{
				\chi_{i}\pm\chi_{j}\given 
				1\le i<j\le n
			}\cup
			\Set*{
				2\chi_{i}\given
				1\le i\le n
			}.
	\end{equation*}
	Among them, the simple roots are the following:
	\index[notation]{a i@$a_{i}$}%
	\begin{flalign*}
		&&
		a_{i} &:= 
			\chi_{i}-\chi_{i+1},
			&
			\mathllap{(1\le i\le n-1)}\\
		&&
		a_{n} &:= 
			2\chi_{n}.
	\end{flalign*}
	Using the basis $\Delta=\Set*{a_{1},\cdots,a_{n}}$, the positive roots can be written as follows:
	\begin{flalign}
		&&\nonumber
			\chi_{i}-\chi_{j}	&=
				a_{i}+\cdots+a_{j-1},	
				&	\mathllap{(1\le i<j\le n)}\\
		&&\label{eq:Roots_Cn}
			\chi_{i}+\chi_{j}	&=
				a_{i}+\cdots+a_{j-1}+2a_{j}+\cdots+2a_{n-1}+a_{n}	,
				&	\mathllap{(1\le i<j\le n)}\\
		&&\nonumber
			2\chi_{i}	&=
				2a_{i}+\cdots+2a_{n-1}+a_{n}.	
				&	\mathllap{(1\le i\le n)}
	\end{flalign}
	Among them, the highest root $a_{0}$ relative to $\Delta$ is 
	\index[notation]{a 0@$a_{0}$}%
	\begin{equation}\label{eq:HighestRoot_Cn}
		a_{0} :=
			2\chi_{1}	=	2a_{1} + \cdots +2a_{n-1} + a_{n}.
	\end{equation}
	Moreover, the sum of positive roots is 
	\index[notation]{two rho@$2\rho$}%
	\begin{equation}\label{eq:2rho_Cn}
		2\rho = 
			\sum_{i=1}^{n-1}i(2n+1-i)a_{i} + \binom{n+1}{2}a_{n}.
	\end{equation}
\end{step}

\begin{step}
	The fundamental coweights relative to $\Delta$ are the following:
	\index[notation]{omega i@$\omega_{i}$}%
	\begin{flalign*}
		&&
			\omega_{i} &:= 
				\vect{e}_{1} + \cdots + \vect{e}_{i},
		& \mathllap{( 1\le i\le n-1 )}\\
		&&
			\omega_{n} &:= 
				\tfrac{1}{2}(\vect{e}_{1} + \cdots + \vect{e}_{n}).
	\end{flalign*}
	Hence, the coweight lattice $\mcal{P}^{\vee}$ is $\Z^n+\Z\tfrac{1}{2}(\vect{e}_{1} + \cdots + \vect{e}_{n})$.
\end{step}

\begin{step}
	Associated to $\Delta$, the fundamental alcove $C$ can be written as follows:
	\begin{equation*}
		C :=
		\Set*{
			x\in\Aff \given
			\tfrac{1}{2}>\chi_{1}(x)>\cdots>\chi_{n}(x)>0
		}.
	\end{equation*}
	The extreme points of $C$ other than $v_{0}=o$ are the following:
	\index[notation]{v i@$v_{i}$}%
	\begin{flalign*}
		&&
			v_{i} &:= 
				o + \tfrac{1}{2}\omega_{i} =
				o + \tfrac{1}{2}(\vect{e}_{1} + \cdots + \vect{e}_{i}),
				& \mathllap{(1\le i\le n-1)}\\
		&&
			v_{n} &:= 
				o + \omega_{n} =
				o + \tfrac{1}{2}(\vect{e}_{1} + \cdots + \vect{e}_{n}).
	\end{flalign*}
	Note that $v_{n}$ is a special vertex, while $v_{i}$ ($1\le i\le n-1$) are not special.
\end{step}

\begin{step}\label{step:VerticesCn}
	For each $i$, by \cref{step:WeylGroupCn}, $W_{o}.v_{i}$ consists of the points $x\in\Aff$ whose coordinates are either $0$ or $\pm\tfrac{1}{2}$ and exactly $i$ of them are nonzero. Then we have 
	\begin{equation*}
		\mcal{V}_{i} = \Set*{
			x\in\Aff\given
			\chi_{1}(x),\cdots,\chi_{n}(x)\in\tfrac{1}{2}\Z
			\text{ and exactly $i$ of them are non-integers }
		}.
	\end{equation*}
	Hence, we have $\mcal{V} = o+\tfrac{1}{2}\Z^n$. 
	In particular, we have $a_{i}(x)\in\tfrac{1}{2}\Z$ for all $1\le i\le n-1$ and $a_{n}(x)\in\Z$. 
	Conversely, if $a_{i}(x)\in\tfrac{1}{2}\Z$ for all $1\le i\le n-1$ and $a_{n}(x)\in\Z$, then we can see that $x-o\in\tfrac{1}{2}\Z^n$. 
	\index[notation]{omega i prime@$\omega_{i}'$}%
	Let $\omega_{i}'$ denote $h_{i}^{-1}\omega_{i}$. 
	Then we have
	\begin{equation}\label{eq:VerticesCn}
		\mcal{V} = o+\Z\omega_{1}'\oplus\cdots\oplus\Z\omega_{n}'.
	\end{equation}
\end{step}

\begin{step}
	Not every vertex is special. We thus need the following notion:
	\begin{definition}\label{def:jump}
		\index{jump}%
		\index[notation]{J (x)@$J_{x}$}%
		Let $x\in\Aff$ be a point. 
		Then an index $j\in\Set*{1,\cdots,n}$ is called a \emph{jump} if $a_{j}(x)\notin\Z$. 
		The set of jumps of $x$ is denoted by $J_{x}$.
	\end{definition}
	Let $x$ be a vertex in $\mscr{D}[\vC]$ with jumps $j_{1},\cdots,j_{s}$, ordered from smallest to largest.
	Note that we must have $j_{s}<n$. 
	Let $x_{i}=x-\tfrac{1}{2}(\omega_{j_{1}}+\cdots+\omega_{j_{i}})$ for $1\le i\le s$. 
	Then the following lemma tells us that $x_{i}$ and $x_{i+1}$ are adjacent vertices. 
	\begin{lemma}
		Let $x\in\Aff$ be a vertex and $j_{1}$ its smallest jump. Then $y=x-\tfrac{1}{2}\omega_{j_{1}}$ is a vertex in $\mscr{D}[\vC]$ adjacent to $x$.
	\end{lemma}
	\begin{proof}
		First note that, by \cref{step:VerticesCn}, we have $a_{j}(x)\in\tfrac{1}{2}\Z\setminus\Z$ for all $j\in J_{x}$. Hence, $J_{y}=J_{x}\setminus\Set*{j_{1}}$. 
		We define the roots $a_{[i]}$ ($1\le i\le n$) as follows. 
		If $i\in J_{y}$, let
		\begin{equation*}
			a_{[i]} = 2a_{i}+\cdots+2a_{n-1}+a_{n}.
		\end{equation*}
		Otherwise, let $a_{[i]}=a_{i}$. 
		Then $a_{[1]},\cdots,a_{[n]}$ are linearly independent positive roots, and $a_{[1]}(y),\cdots,a_{[n]}(y)$ are non-negative integers. 
		Hence, $y$ is a vertex in $\mscr{D}[\vC]$.
		Since $x-y=\tfrac{1}{2}\omega_{j_{1}}$, we have $a_{[i]}(x)=a_{[i]}(y)$ for all $i$ except $i=j_{1}$ and $a_{[j_{1}]}(x)-a_{[j_{1}]}(y)=\tfrac{1}{2}$. 
		Then \cref{lem:AdjacentVertices} applies to the roots $a_{[i]}$ ($1\le i\le n$) and the vertices $x$ and $y$.
	\end{proof}

	Then the sequence $(x,x_{1},\cdots,x_{s})$ forms a path from $x$ to $x_{s}$ of length $s$ in $\mscr{D}[\vC]$. Since $x_{s}$ has no jumps, it is a special vertex. Moreover, we have 
	\begin{equation*}
		a_{0}(x)-a_{0}(x_{s}) = 
			\tfrac{1}{2}a_{0}(\omega_{j_{1}}+\cdots+\omega_{j_{s}}) = s.
	\end{equation*}
	Then this $x_{s}$ is the expected $x_{0}$ verifying \cref{eq:SpecialX0inLemma}. 
	Thus, \cref{thm:SimplicialDistance} is proved.
\end{step}

\begin{step}\label{step:IndexBsrciCn}
	Let $I$ be a type and follow \cref{con:type}. 
	By introducing $\omega_{i}'=h_{i}^{-1}\omega_{i}$, we can write a point $x$ having type $I$ as follows:
	\begin{equation*}
		x = o+c_{1}\omega_{\ell_{1}}'+\cdots+c_{t}\omega_{\ell_{t}}'.
	\end{equation*} 
	By \cref{eq:VerticesCn}, such an $x$ is a vertex if and only if $c_{1},\cdots,c_{t}\in\Z$. Therefore, by \cref{eq:IndexSetsUsingRoots:Ball,eq:IndexSetsUsingRoots:Sphere}, we have the following explicit description:
	\begin{align}
			B(r,\vC,I) &=
				\Set*{
					x=o+c_{1}\omega_{\ell_{1}}'+\cdots+c_{t}\omega_{\ell_{t}}'
					\given
					\begin{array}{c}
						c_{1},\cdots,c_{t}\in\Z_{>0},\\
						c_{1}+\cdots+c_{t} \le r
					\end{array}
				},\\
		\label{eq:expressIndexSets_Cn}
			\partial(r,\vC,I) &=
				\Set*{
					x=o+c_{1}\omega_{\ell_{1}}'+\cdots+c_{t}\omega_{\ell_{t}}'
					\given
					\begin{array}{c}
						c_{1},\cdots,c_{t}\in\Z_{>0},\\
						c_{1}+\cdots+c_{t} = r
					\end{array}
				}.
	\end{align}
\end{step}

\subsection{Vertices in the apartment \texorpdfstring{$\Apt[B_{n}]$ ($n\ge 3$)}{Bn}}\label{subsec:VerticesInAptBn}
\begin{step}\label{step:WeylGroupBn}
	The underlying Euclidean vector space $\Vect$ is the entire $\R^n$, and its dual space $\mathbb{V}^{\ast}$ is thus spanned by the coordinate functions $\chi_{1},\cdots,\chi_{n}$. 
	Then the root system can be written as follows: 
	\begin{equation*}
		\Phi := 
		\Set*{
			\pm\chi_{i}\pm\chi_{j}\given 
			1\le i<j\le n
		}\cup
		\Set*{
			\pm \chi_{i}\given
			1\le i\le n
		}.
	\end{equation*}
	Then its coroot lattice is the following sublattice of the standard lattice $\Z^n$ in $\R^n$: 
	\begin{equation*}
		\mcal{Q}^{\vee} := 
		\Set*{
			\vect{v}\in\Z^n\given
			\inner{\vect{v}}{\vect{v}}\in 2\Z
		}.
	\end{equation*}
	The Weyl group $\vWeyl$ acts on $\mathbb{V}$ as permutations and sign changes of coordinates. 
\end{step}

\begin{step}\label{step:WeylChamberBn}
	We can choose the following Weyl chamber $\vC$:
	\begin{equation*}
		\vC := 
		\Set*{
			\vect{v}\in\mathbb{V}
			\given
			\chi_{i}(\vect{v})>\chi_{j}(\vect{v})>0
			\text{ for all } 
			1 \le i < j \le n
		}.
	\end{equation*}
	Then the system of positive roots $\Phi^{+}$ associated to $\vC$ is the following:
	\begin{equation*}
		\Phi^{+} := 
			\Set*{
				\chi_{i}\pm\chi_{j}\given 
				1\le i<j\le n
			}\cup
			\Set*{
				\chi_{i}\given
				1\le i\le n
			}.
	\end{equation*}
	Among them, the simple roots are the following:
	\index[notation]{a i@$a_{i}$}%
	\begin{flalign*}
		&&
			a_{i} &:= 
				\chi_{i}-\chi_{i+1},
				&
				\mathllap{(1\le i\le n-1)}\\
		&&
			a_{n} &:=\chi_{n}.
	\end{flalign*}
	Using the basis $\Delta=\Set*{a_{1},\cdots,a_{n}}$, the positive roots can be written as follows:
	\begin{flalign}
		&&\nonumber
			\chi_{i}-\chi_{j}	&=
				a_{i}+\cdots+a_{j-1},	
				&	\mathllap{(1\le i<j\le n)}\\
		&&\label{eq:Roots_Bn}
			\chi_{i}+\chi_{j}	&=
				a_{i}+\cdots+a_{j-1}+2a_{j}+\cdots+2a_{n},	
				&	\mathllap{(1\le i<j\le n)}\\
		&&\nonumber
			\chi_{i}	&=
				a_{i}+\cdots+a_{n}.	
				&	\mathllap{(1\le i\le n)}
	\end{flalign}
	Among them, the highest root $a_{0}$ relative to $\Delta$ is
	\index[notation]{a 0@$a_{0}$} %
	\begin{equation}\label{eq:HighestRoot_Bn}
		a_{0} :=
			\chi_{1}+\chi_{2}	=	a_{1} + 2a_{2} + \cdots + 2a_{n}.
	\end{equation}
	Moreover, the sum of positive roots is 
	\index[notation]{two rho@$2\rho$}%
	\begin{equation}\label{eq:2rho_Bn}
		2\rho =
			\sum_{i=1}^{n}i(2n-i)a_{i}.
	\end{equation}
\end{step}

\begin{step}\label{step:CoweightsBn}
	The fundamental coweights relative to $\Delta$ are the following:
	\index[notation]{omega i@$\omega_{i}$}%
	\begin{flalign*}
		&&
			\omega_{i} &= 
			\vect{e}_{1} + \cdots + \vect{e}_{i}.
			& \mathllap{(1\le i\le n)}
	\end{flalign*}
	Hence, the coweight lattice $\mcal{P}^{\vee}$ is precisely the standard lattice $\Z^n$.
\end{step}

\begin{step}
	Associated to $\Delta$, the fundamental alcove $C$ can be written as follows:
	\begin{equation*}
		C =
		\Set*{
			x\in\Aff \given
			\begin{array}{c}
				\chi_{1}(x)>\cdots>\chi_{n}(x)>0,\\ 
				\chi_{1}(x)+\chi_{2}(x)<1
			\end{array}
		}.
	\end{equation*}
	The extreme points of $C$ other than $v_{0}=o$ are the following:
	\index[notation]{v i@$v_{i}$}%
	\begin{flalign*}
		&&
			v_{1} &=
			o + \omega_{1} = 
			o + \vect{e}_{1},\\
		&&
			v_{i} &= 
			o + \tfrac{1}{2}\omega_{i} =
			o + \tfrac{1}{2}(\vect{e}_{1} + \cdots + \vect{e}_{i}).
			& \mathllap{(2\le i\le n)}
	\end{flalign*}
	Note that $v_{1}$ is a special vertex, while $v_{i}$ ($2\le i\le n$) are not special.
\end{step}

\begin{step}\label{step:VerticesBn}
	First, apply the affine Weyl group $W$ to $v_{0}$, we have 
	\begin{equation*}
		\mcal{V}_{0} = \Set*{
			x\in\Aff\given
			\chi_{1}(x),\cdots,\chi_{n}(x)\in\Z, 
			\chi_{1}(x)+\cdots+\chi_{n}(x)
			\text{ is even }
		}.
	\end{equation*}
	Applying $W_{o}$ to $v_{1}$, we see that $W_{o}.v_{1}$ consists of the points $x\in\Aff$ having one coordinate being $1$ or $-1$ and all others are $0$. Then we have:
	\begin{equation*}
		\mcal{V}_{1} = \Set*{
			x\in\Aff\given
			\chi_{1}(x),\cdots,\chi_{n}(x)\in\Z,
			\chi_{1}(x)+\cdots+\chi_{n}(x)
			\text{ is odd }
		}
	\end{equation*}
	For each $i>1$, $W_{o}.v_{i}$ consists of the points $x\in\Aff$ whose coordinates are either $0$ or $\pm\tfrac{1}{2}$ and exactly $i$ of them are nonzero. Then we have:
	\begin{equation*}
		\mcal{V}_{i} = \Set*{
			x\in\Aff\given
			\chi_{1}(x),\cdots,\chi_{n}(x)\in\tfrac{1}{2}\Z
			\text{ and exactly $i$ of them are non-integers }
		}.
	\end{equation*}
	Hence, the vertices are all the point $x\in\Aff$ such that $\chi_{j}(x)\in\tfrac{1}{2}\Z$ for all $j$ and the number of non-integer coordinates is not $1$.	In particular, we have 
	\begin{equation*}
		\mcal{V}\subset 
			o + \tfrac{1}{2}\Z^n =
			o + \tfrac{1}{2}\mcal{P}^{\vee}.
	\end{equation*}
	However, the equality doesn't hold. For instance, the point $o+\tfrac{1}{2}\omega_{1}$ is clearly not a vertex. Another example is $o+\tfrac{1}{2}(\omega_{i-1}+\omega_{i})$ where all $\chi_{j}(x)$ are integers except $j=i$.

	To better describe the vertices, we need the notion introduced in \cref{def:jump}.
	Then the complement of $\mcal{V}$ in $o + \tfrac{1}{2}\mcal{P}^{\vee}$ can be described as follows.
	\begin{lemma}\label{lem:VerticesBn}
		A point $x\in o + \tfrac{1}{2}\mcal{P}^{\vee}$ belongs to the complement if and only if either $J_{x}=\Set*{j_{1},j_{2}}$ and $j_{2}-j_{1}=1$, or $J_{x}=\Set*{1}$.
	\end{lemma}
	\index[notation]{Xi@$\Xi$}%
	We will use $\Xi$ to denote the set of points $x\in\Aff$ having the property in the lemma.
	\begin{proof}
		First, points in $\Xi$ cannot be vertices.
		If $J_{x}=\Set*{1}$, then $x$ is not a vertex since all $\chi_{j}(x)$ are integers except $j=1$. 
		If $J_{x}=\Set*{j_{1},j_{2}}$ and $j_{2}-j_{1}=1$, then $x$ is not a vertex since all $\chi_{j}(x)$ are integers except $j=j_{2}$.
		
		Conversely, suppose $x\in o + \tfrac{1}{2}\mcal{P}^{\vee}$ and $x\notin\Xi$, then there are four cases:
		\begin{enumerate}
			\item $\abs*{J_{x}}\ge 3$. Then at least two coordinates of $x$ are non-integers. 
			\item $J_{x}=\Set*{j_{1},j_{2}}$ and $j_{2}-j_{1}>1$. 
			Then $\chi_{j_{1}+1}(x),\cdots,\chi_{j_{2}}(x)$ are non-integers. 
			\item $J_{x}=\Set*{j_{1}}$ and $j_{1}>1$. 
			Then $\chi_{1}(x),\cdots,\chi_{j_{1}}(x)$ are non-integers. 
			\item $x$ has no jumps. Then it is a special vertex. 
		\end{enumerate} 
		In any of above cases, $x$ is a vertex by our characterization.
	\end{proof}
	
	\index[notation]{dagger@$\dagger$}%
	We illustrate the structure of $\mcal{V}$ by the following diagram:
	
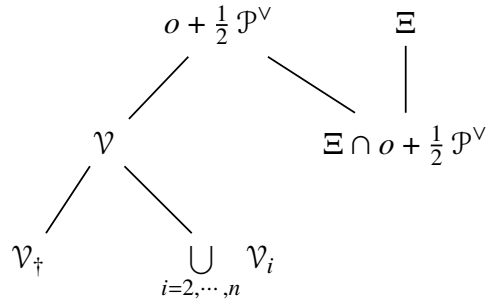
\begin{equation*}
		\begin{tikzcd}[column sep=tiny]
			&& {	o+\tfrac{1}{2}\mcal{P}^{\vee}	}
				\ar[dl,no head]\ar[dr,no head]	
			&	{	\Xi	}
				\ar[d,no head] \\
			&	{	\mcal{V}	}
				\ar[dl,no head]\ar[dr,no head]	
			&&	{	\Xi\cap o+\tfrac{1}{2}\mcal{P}^{\vee}	} \\
			{	\mcal{V}_{\dagger}	} 
			&&	{	\bigcup\limits_{i=2,\cdots,n}\mcal{V}_{i}	}	&
		\end{tikzcd}
	\end{equation*}
	where $\dagger$ denotes ``being special''.
\end{step}

\begin{step}
	Let $x\in\mscr{D}[\vC]$ be a vertex having jumps $j_{1},\cdots,j_{s}$, ordered from smallest to largest. To construct a path between $x$ and a special vertex $x_{0}$ in $\mscr{D}[\vC]$ verifying \cref{eq:SpecialX0inLemma}, we need the following lemmas.
	\begin{lemma}\label{lem:ConstructX0forBn:1}
		Suppose either $j_{s}-j_{s-1}>1$ or $s>3$. 
		Then $y=x-\tfrac{1}{2}\omega_{j_{1}}$ is a vertex in $\mscr{D}[\vC]$ adjacent to $x$.
	\end{lemma}
	\begin{proof}
		Since $\mcal{V}\subset o+\tfrac{1}{2}\mcal{P}^{\vee}$, we have $J_{y}=J_{x}\setminus\Set*{j_{1}}$. 
		We define the roots $a_{[j]}$ ($1\le j\le n$) as follows. First, let 
		\begin{equation*}
			a_{[j_{s}]} = 
			\begin{dcases*}
				a_{j_{s}-1}+2a_{j_{s}}+\cdots+2a_{n} & 
				if $j_{s}-j_{s-1}>1$,\\
				a_{j_{s-2}}+\cdots+a_{j_{s}-1}+2a_{j_{s}}+\cdots+2a_{n} & otherwise.
			\end{dcases*}
		\end{equation*}
		For $2\le i\le s-1$, let 
		\begin{equation*}
			a_{[j_{i}]} = a_{j_{i}}+\cdots+a_{j_{i+1}}. 
		\end{equation*}
		Finally, if $j\notin J_{y}$, let $a_{[j]}=a_{j}$. 
		Then $a_{[1]},\cdots,a_{[n]}$ are linearly independent positive roots, and $a_{[1]}(y),\cdots,a_{[n]}(y)$ are non-negative integers. 
		Hence, $y$ is a vertex in $\mscr{D}[\vC]$.
		Since $x-y=\tfrac{1}{2}\omega_{j_{1}}$, we have $a_{[j]}(x)=a_{[j]}(y)$ for all $j$ except $j=j_{1}$ and $a_{[j_{1}]}(x)-a_{[j_{1}]}(y)=\tfrac{1}{2}$. 
		Then \cref{lem:AdjacentVertices} applies to the roots $a_{[j]}$ ($1\le j\le n$) and the vertices $x$ and $y$.
	\end{proof}
	Let $y=x-\tfrac{1}{2}\omega_{j_{1}}$ be as in \cref{lem:ConstructX0forBn:1}. Then we must have $a_{0}(y)\in\Z$ and 
	\begin{align*}
		\ceil{a_{0}(x)}-a_{0}(y) &=
			\begin{dcases*}
				a_{0}(x)-a_{0}(y) & if $j_{1}>1$,\\
				a_{0}(x)+\tfrac{1}{2}-a_{0}(y) & if $j_{1}=1$;
			\end{dcases*} \\
		a_{0}(x)-a_{0}(y) &= 
			\tfrac{1}{2}a_{0}(\omega_{j_{1}}) = 
			\begin{dcases*}
				1 & if $j_{1}>1$,\\
				\tfrac{1}{2} & if $j_{1}=1$.
			\end{dcases*}
	\end{align*}
	Hence, $\ceil{a_{0}(x)}-a_{0}(y) = 1$.
	By repeating using \cref{lem:ConstructX0forBn:1}, we can reduce our problem to the case where $s=3$, or further $s=1$ if we start with $j_{s}-j_{s-1}>1$. 

	Now, we may assume either $s\le 3$ with $j_{s}-j_{s-1}=1$ or $s=1$.
	\begin{lemma}\label{lem:ConstructX0forBn:2}
		Suppose $s=3$ and $j_{1}>1$. Then $x_{0}=x-\tfrac{1}{2}(\omega_{j_{1}}-\omega_{j_{2}}+\omega_{j_{3}})$ is a special vertex in $\mscr{D}[\vC]$ adjacent to $x$ and verifying \cref{eq:SpecialX0inLemma}.
	\end{lemma}
	\begin{proof}
		First note that $a_{1}(x_{0}),\cdots,a_{n}(x_{0})$ are non-negative integers. Hence, $x_{0}$ is a special vertex in $\mscr{D}[\vC]$. 
		Since $j_{1}>1$, we have $a_{0}(x)\in\Z$ and 
		\begin{equation*}
			a_{0}(x)-a_{0}(x_{0}) = \tfrac{1}{2}a_{0}(\omega_{j_{1}}-\omega_{j_{2}}+\omega_{j_{3}}) = 1.
		\end{equation*}
		Hence, it remains to show that $x_{0}$ is adjacent to $x$.

		To do this, we define the roots $a_{[j]}$ ($1\le j\le n$) as follows. 
		If $j\neq j_{2},j_{3}$, let $a_{[j]}=a_{j}$. Otherwise, let
		\begin{align*}
				a_{[j_{2}]} &= 
				a_{j_{1}}+\cdots+a_{j_{2}},\\
				a_{[j_{3}]} &= 
				a_{j_{2}}+\cdots+a_{j_{3}}.
		\end{align*}
		Then $a_{[1]},\cdots,a_{[n]}$ are linearly independent. 
		Since $x-x_{0}=\tfrac{1}{2}(\omega_{j_{1}}-\omega_{j_{2}}+\omega_{j_{3}})$, we have $a_{[j]}(x)=a_{[j]}(x_{0})$ for all $j$ except $j=j_{1}$ and $a_{[j_{1}]}(x)-a_{[j_{1}]}(x_{0})=\tfrac{1}{2}$. 		
		Then \cref{lem:AdjacentVertices} applies to the roots $a_{[j]}$ ($1\le j\le n$) and the vertices $x$ and $x_{0}$.
	\end{proof}	

	\begin{lemma}\label{lem:ConstructX0forBn:3}
		Suppose $s=3$, $j_{3}-j_{2}=1$ and $j_{1}=1$, then $y=x-\tfrac{1}{2}(-\omega_{j_{1}}+\omega_{j_{2}})$ is a vertex in $\mscr{D}[\vC]$ adjacent to $x$.
	\end{lemma}
	\begin{proof}
		We define the roots $a_{[j]}$ ($1\le j\le n$) as follows. 
		If $j\neq j_{2},j_{3}$, let $a_{[j]}=a_{j}$. Otherwise, let
		\begin{align*}
				a_{[j_{2}]} &= 
				a_{j_{1}}+\cdots+a_{j_{2}},\\
				a_{[j_{3}]} &= 
				a_{j_{1}}+\cdots+a_{j_{2}}+2a_{j_{3}}+\cdots+2a_{n}.
		\end{align*}
		Then $a_{[1]},\cdots,a_{[n]}$ are linearly independent positive roots, and $a_{[1]}(y),\cdots,a_{[n]}(y)$ are non-negative integers. 
		Hence, $y$ is a vertex in $\mscr{D}[\vC]$.
		Since $x-y=\tfrac{1}{2}(-\omega_{j_{1}}+\omega_{j_{2}})$, we have $a_{[j]}(x)=a_{[j]}(y)$ for all $j$ except $j=j_{1}$ and $a_{[j_{1}]}(x)-a_{[j_{1}]}(y)=-\tfrac{1}{2}$. 
		Then \cref{lem:AdjacentVertices} applies to the roots $a_{[j]}$ ($1\le j\le n$) and the vertices $x$ and $y$.
	\end{proof}
	Let $y=x-\tfrac{1}{2}(-\omega_{j_{1}}+\omega_{j_{2}})$ be as in \cref{lem:ConstructX0forBn:3}. Then the only jump of $y$ is $j_{3}>1$. 
	Hence, $a_{0}(y)\in\Z$ and $\ceil{a_{0}(x)}-a_{0}(y)=1$. 
	Therefore, \cref{lem:ConstructX0forBn:3} reduces our problem to the case where $s=1$. 

	Note that, by \cref{lem:VerticesBn}, $s=2$ and $j_{s}-j_{s-1}=1$ contradict to each other. Therefore, we may assume $s=1$ now. By \cref{lem:VerticesBn} again, we must have $j_{1}>1$.
	Let $x_{0}=x-\tfrac{1}{2}\omega_{j_{1}}$. 
	Then it is a special vertex in $\mscr{D}[\vC]$ since $a_{1}(x_{0}),\cdots,a_{n}(x_{0})$ are non-negative integers. 
	Applying \cref{lem:AdjacentVertices} to the simple roots $a_{1},\cdots,a_{n}$ and the vertices $x$ and $x_{0}$, we see that they are adjacent. Moreover, we have $a_{0}(x)-a_{0}(x_{0})=1$ verifying \cref{eq:SpecialX0inLemma}. 
	This finishes the proof of \cref{thm:SimplicialDistance}.
\end{step}

\begin{step}\label{step:IndexBsrciBn}
	\index[notation]{X (I)@$\mcal{X}[I]$}%
	\index[notation]{X (I,r)@$\mcal{X}[I,r]$}%
	Let $I$ be a type and follow \cref{con:type}. 
	For any set $\mcal{X}$ of points, we introduce the following subsets:
	\begin{equation}\label{eq:IndexSetX}
		\begin{aligned}
			\mcal{X}[I] &:= 
			(o+\vC_{I})\cap\mcal{X},\\
			\mcal{X}[I,r] &:=
			(o+\vC_{I})\cap\mcal{X}\cap(\alpha_{-a_{0}+r}\setminus\alpha_{-a_{0}+r-1}).
		\end{aligned}
	\end{equation}
	Then we have $\mcal{V}[I,r]=\partial(r,\vC,I)$ by \cref{eq:IndexSetsUsingRoots:Ball,eq:IndexSetsUsingRoots:Sphere}.

	\index[notation]{omega i prime@$\omega_{i}'$}%
	By introducing $\omega_{i}'=h_{i}^{-1}\omega_{i}$, we can write a point $x$ having type $I$ as follows:
	\begin{equation*}
		x = o+c_{1}\omega_{\ell_{1}}'+\cdots+c_{t}\omega_{\ell_{t}}'.
	\end{equation*} 
	Then $a_{0}(x)\le r$ if and only if $c_{1}+\cdots+c_{t}\le r$.

	\index[notation]{X 0@$\mcal{X}^{0}$}%
	\index[notation]{X 1@$\mcal{X}^{1}$}%
	Consider the set $o+\tfrac{1}{2}\mcal{P}^{\vee}$ and recall that $h_{1}=1$ while $h_{2}=\cdots=h_{n}=2$. 
	Let $\mcal{X}^{0}$ be the set $o+\Z\omega_{1}'+\cdots+\Z\omega_{n}'$ and $\mcal{X}^{1}=\mcal{X}^{0}-\tfrac{1}{2}\omega_{1}$. Then we have 
	\begin{equation*}
		o+\tfrac{1}{2}\mcal{P}^{\vee} = 
			\mcal{X}^{0}\cup\mcal{X}^{1}.
	\end{equation*}
	This gives a superset of the $\mcal{V}$.

	\index[notation]{Xi J@$\Xi_{J}$}%
	\index[notation]{X J@$\mcal{X}_{J}$}%
	By \cref{lem:VerticesBn}, the complement of $\mcal{V}$ in $o+\tfrac{1}{2}\mcal{P}^{\vee}$ is the restriction of $\Xi$. Inspired by this and \cref{def:jump}, we can consider the following sets for each $J\subset\Set*{1,\cdots,n}$:
	\begin{align*}
		\Xi_{J} &:=
			\Set*{	x\in\Aff	\given	J_{x}=J	}
		&\text{and}&&
		\mcal{X}_{J} &:=
			\Xi_{J}\cap o+\tfrac{1}{2}\mcal{P}^{\vee}.
	\end{align*}
	Then we have 
	\begin{equation*}
		o+\tfrac{1}{2}\mcal{P}^{\vee} \setminus \mcal{V} = 
		\mcal{X}_{\Set*{1}}\cup\mcal{X}_{\Set*{1,2}}\cup\cdots\cup\mcal{X}_{\Set*{n-1,n}}.
	\end{equation*}
	Note that, for any $J$, we have 
	\begin{equation}\label{eq:XJfromX0Bn}
		\mcal{X}_{J} = \mcal{X}_{\emptyset}-\sum_{j\in J}\tfrac{1}{2}\omega_{j}.
	\end{equation}
	Moreover, it is clear that $\mcal{X}_{\emptyset}$ is precisely $o+\mcal{P}^{\vee}$, the set of special vertices.


	Next, we consider $\mcal{X}[I]$ for above sets.
	First, if $\ell_{1}>1$, then $\mcal{V}[I]\subset\mcal{X}^{0}[I]$. Otherwise, $\mcal{V}[I]\cap\mcal{X}^{1}[I]\neq\emptyset$. 
	For any $J$, it is clear that $\mcal{X}_{J}[I]\neq\emptyset$ if and only if $I\cap J=\emptyset$. If this is the case, we have the following refinement of \cref{eq:XJfromX0Bn}:
	\begin{equation}\label{eq:XIJfromXI0Bn}
		\mcal{X}_{J}[I] = \mcal{X}_{\emptyset}[I]-\sum_{j\in J}\tfrac{1}{2}\omega_{j}.
	\end{equation}

	Finally, we consider $\mcal{X}[I,r]$. First, it is clear that 
	\begin{equation}
		\label{eq:IndexBsrciBn:pu0}
		\mcal{X}^{0}[I,r] =
		\Set*{
			x = o+c_{1}\omega_{\ell_{1}}'+\cdots+c_{t}\omega_{\ell_{t}}'
			\given
			\begin{array}{c}
				c_{1},\cdots,c_{t}>0,\\
				c_{1},\cdots,c_{t}\in\Z_{>0},\\
				c_{1}+\cdots+c_{t} = r
			\end{array}
		}.
	\end{equation}
	Also note that 
	\begin{equation}
		\label{eq:IndexBsrciBn:pu1}
		\mcal{X}^{1}[I,r] = 
		\mcal{X}^{0}[I,r]-\tfrac{1}{2}\omega_{1}.
	\end{equation}
	Then we need to work out $\mcal{X}_{J}[I,r]$. For $\mcal{X}_{\emptyset}[I,r]$, an explicit description is given in \cref{eq:expressIndexSets_speical:Ball,eq:expressIndexSets_speical:Sphere}. For general $J$, we have the following refinement of \cref{eq:XIJfromXI0Bn}:
	\begin{lemma}\label{lem:XIJfromXI0Bn}
		\index[notation]{deltal (J)@$\delta(J)$}%
		Suppose $I\cap J=\emptyset$. 
		Then we have 
		\begin{equation*}
			\mcal{X}_{J}[I,r] = 
				\mcal{X}_{\emptyset}[I,r+\abs*{J}-\delta(J)]
				-\sum_{j\in J}\tfrac{1}{2}\omega_{j},
		\end{equation*}
		where $\delta(J)$ is defined as follows:
		\begin{equation*}
			\delta(J) = 
			\begin{dcases*}
				1 & if $1\in J$,\\
				0 & otherwise.
			\end{dcases*}
		\end{equation*}
	\end{lemma}
	\begin{proof}
		By \cref{eq:XIJfromXI0Bn}, it suffices to show that for any $x\in\mcal{X}_{\emptyset}[I,r]$,
		\begin{equation*}\label{eq:tag:lem:XIJfromXI0Bn}\tag{$\ast$}
			\ceil*{a_{0}\left(x-\sum_{j\in J}\tfrac{1}{2}\omega_{j}\right)} = 
				a_{0}(x)-\abs*{J}+\delta(J).
		\end{equation*}
		Note that $a_{0}(x)\in\Z$ and that
		\begin{equation*}
			a_{0}\left(\sum_{j\in J}\tfrac{1}{2}\omega_{j}\right) = 
			\begin{dcases*}
				\abs*{J}-\tfrac{1}{2} & if $1\in J$,\\
				\abs*{J} & otherwise.
			\end{dcases*}
		\end{equation*}
		Then \cref{eq:tag:lem:XIJfromXI0Bn} follows.
	\end{proof}
	
	We illustrate above discussions by the following diagrams:
	\begin{center}
		\captionsetup{type=figure}
		\begin{equation*}
			\begin{tikzcd}[column sep=0pt]
				& {	\mcal{X}^{0}[I,r]	}
					\ar[dl,no head]\ar[dr,no head]	
				&	{	\bigcup\limits_{J=\Set*{2,3},\cdots,\Set*{n-1,n}}\Xi_{J}	}
					\ar[d,no head] \\
				\mathclap{	\mcal{V}[I,r]	}\ar[d,no head]	
				&&	{	\bigcup\limits_{J=\Set*{2,3},\cdots,\Set*{n-1,n}}\mcal{X}_{J}[I,r]	} \\
				\mathclap{	\mcal{V}_{\dagger}[I,r]=\mcal{X}_{\emptyset}[I,r]	}
				&&
			\end{tikzcd}
		\end{equation*}
		\caption{Vertices of type $I$ in $\Apt[B_{n}]$ ($\ell_{1}>1$)}\label{figure:VerticesOfIBnl1}
	\end{center}
	\begin{center}
		\captionsetup{type=figure}
		\begin{equation*}
			\begin{tikzcd}[column sep=0pt]
				& {	\mcal{X}^{0}[I,r]\cup\mcal{X}^{1}[I,r]	}
					\ar[dl,no head]\ar[dr,no head]	
				&	{	\bigcup\limits_{J=\Set*{1},\Set*{1,2},\cdots,\Set*{n-1,n}}\Xi_{J}	}
					\ar[d,no head] \\
				\mathclap{	\mcal{V}[I,r]	}\ar[d,no head]	
				&&	{	\bigcup\limits_{J=\Set*{1},\Set*{1,2},\cdots,\Set*{n-1,n}}\mcal{X}_{J}[I,r]	} \\
				\mathclap{	\mcal{V}_{\dagger}[I,r]=\mcal{X}_{\emptyset}[I,r]	}
				&&
			\end{tikzcd}
		\end{equation*}
		\caption{Vertices of type $I$ in $\Apt[B_{n}]$ ($\ell_{1}=1$)}\label{figure:VerticesOfIBne1}
	\end{center}
\end{step}

\subsection{Vertices in the apartment \texorpdfstring{$\Apt[D_{n}]$ ($n\ge 4$)}{Dn}}\label{subsec:VerticesInAptDn}
\begin{step}
	The underlying Euclidean vector space $\Vect$ is the entire $\R^n$, and its dual space $\mathbb{V}^{\ast}$ is thus spanned by the coordinate functions $\chi_{1},\cdots,\chi_{n}$. 
	Then the root system can be written as follows: 
	\begin{equation*}
		\Phi = 
		\Set*{
			\pm\chi_{i}\pm\chi_{j}\given 
			1\le i<j\le n
		}.
	\end{equation*}
	Then its coroot lattice is the following sublattice of the standard lattice $\Z^n$ in $\R^n$: 
	\begin{equation*}
		\mcal{Q}^{\vee} := 
		\Set*{
			\vect{v}\in\Z^n\given
			\inner{\vect{v}}{\vect{v}}\in 2\Z
		}.
	\end{equation*}
	The Weyl group $\vWeyl$ acts on $\mathbb{V}$ as permutations and even number of sign changes of coordinates. 
\end{step}

\begin{step}\label{step:WeylChamberDn}
	We can choose the following Weyl chamber $\vC$:
	\begin{equation*}
		\vC = 
		\Set*{
			\vect{v}\in\mathbb{V}
			\given
			\chi_{i}(\vect{v})>\abs*{\chi_{j}(\vect{v})}
			\text{ for all } 
			1 \le i < j \le n
		}.
	\end{equation*}
	Then the system of positive roots $\Phi^{+}$ associated to $\vC$ is the following:
	\begin{equation*}
		\Phi^{+} := 
			\Set*{
				\chi_{i}\pm\chi_{j}\given 
				1\le i<j\le n
			}.
	\end{equation*}
	Among them, the simple roots are the following:
	\index[notation]{a i@$a_{i}$}%
	\begin{flalign*}
		&&
			a_{i} &:= 
				\chi_{i}-\chi_{i+1}, 
				&	\mathllap{(1\le i\le n-1)}\\
		&&
			a_{n}	&:=
				\chi_{n-1}+\chi_{n}.
	\end{flalign*}
	Using the basis $\Delta=\Set*{a_{1},\cdots,a_{n}}$, the positive roots can be written as follows:
	\begin{flalign}
		&&\nonumber
			\chi_{i}-\chi_{j}	&=
				a_{i}+\cdots+a_{j-1},	
				&	\mathllap{(1\le i<j\le n)}\\
		&&\nonumber
			\chi_{i}+\chi_{n}	&=
				a_{i}+\cdots+a_{n-2}+a_{n},	
				&	\mathllap{(1\le i\le n-1)}\\
		&&\label{eq:Roots_Dn}
			\chi_{i}+\chi_{n-1}	&=
				a_{i}+\cdots+a_{n-2}+a_{n-1}+a_{n},	
				&	\mathllap{(1\le i\le n-2)}\\
		&&\nonumber
			\chi_{i}+\chi_{j}	&=
				a_{i}+\cdots+a_{j-1}
				&	\mathllap{(1\le i<j\le n-2)}\\
		&&\nonumber 
				&\quad +2a_{j}+\cdots+2a_{n-2} + a_{n-1}+a_{n}.
	\end{flalign}
	Among them, the highest root $a_{0}$ relative to $\Delta$ is 
	\index[notation]{a 0@$a_{0}$} %
	\begin{equation}\label{eq:HighestRoot_Dn}
		a_{0} :=
			\chi_{1}+\chi_{2}	=	
				a_{1} + 2a_{2} + \cdots + 2a_{n-2} + a_{n-1} + a_{n}.
	\end{equation}
	Moreover, the sum of positive roots is 
	\index[notation]{two rho@$2\rho$}%
	\begin{equation}\label{eq:2rho_Dn}
		2\rho =
			\sum_{i=1}^{n-2}i(2n-1-i)a_{i}+\binom{n}{2}(a_{n-1}+a_{n}).
	\end{equation}
\end{step}

\begin{step}\label{step:CoweightsDn}
	The fundamental coweights relative to $\Delta$ are the following:
	\index[notation]{omega i@$\omega_{i}$}%
	\begin{flalign*}
		&&
			\omega_{i} &= 
				\vect{e}_{1} + \cdots + \vect{e}_{i},
				& \mathllap{(1\le i\le n-2)}\\
		&&
			\omega_{n-1} &= 
				\tfrac{1}{2}(\vect{e}_{1} + \cdots + \vect{e}_{n-1}-\vect{e}_{n}),\\
		&&
			\omega_{n} &= 
				\tfrac{1}{2}(\vect{e}_{1} + \cdots + \vect{e}_{n}).
	\end{flalign*}
	Hence, the coweight lattice $\mcal{P}^{\vee}$ is $\Z^n+\Z\tfrac{1}{2}(\vect{e}_{1} + \cdots + \vect{e}_{n})$.
\end{step}

\begin{step}
	Associated to $\Delta$, the fundamental alcove $C$ can be written as follows:
	\begin{equation*}
		C =
		\Set*{
			x\in\Aff \given
			\begin{array}{c}
				\chi_{1}(x)>\cdots>\chi_{n-1}(x)>\abs*{\chi_{n}(x)},\\ 
				\chi_{1}(x)+\chi_{2}(x)<1
			\end{array}
		}.
	\end{equation*}
	The extreme points of $C$ other than $v_{0}=o$ are the following:
	\index[notation]{v i@$v_{i}$}%
	\begin{flalign*}
		&&
			v_{1} &= 
				o + \omega_{1} =
				o + \vect{e}_{1},\\
		&&
			v_{i} &=
				o + \tfrac{1}{2}\omega_{i} =
				o + \tfrac{1}{2}(\vect{e}_{1} + \cdots + \vect{e}_{i}),
				& \mathllap{(2\le i\le n-2)}\\
		&&
			v_{n-1} &=
				o + \omega_{n-1} = 
				o + \tfrac{1}{2}(\vect{e}_{1} + \cdots + \vect{e}_{n-1}-\vect{e}_{n}),\\
		&&
			v_{n} &=
				o + \omega_{n} = 
				o + \tfrac{1}{2}(\vect{e}_{1} + \cdots + \vect{e}_{n}).
	\end{flalign*}
	Note that $v_{1},v_{n-1},v_{n}$ are special vertices, while $v_{i}$ ($2\le i\le n-2$) are not special.
\end{step}

\begin{step}\label{step:VerticesDn}
	First, apply the affine Weyl group $W$ to $v_{0}$, we have 
	\begin{equation*}
		\mcal{V}_{0} = \Set*{
				x\in\Aff\given
				\chi_{1}(x),\cdots,\chi_{n}(x)\in\Z, 
				\chi_{1}(x)+\cdots+\chi_{n}(x)
				\text{ is even }
		}.
	\end{equation*}
	Applying $W_{o}$ to $v_{1}$, we see that $W_{o}.v_{1}$ consists of the points $x\in\Aff$ having one coordinate being $1$ or $-1$ and all others are $0$. Then we have:
	\begin{equation*}
		\mcal{V}_{1} = \Set*{
			x\in\Aff\given
			\chi_{1}(x),\cdots,\chi_{n}(x)\in\Z,
			\chi_{1}(x)+\cdots+\chi_{n}(x)
			\text{ is odd }
		}
	\end{equation*}
	For each $1<i<n-1$, $W_{o}.v_{i}$ consists of the points $x\in\Aff$ whose coordinates are either $0$ or $\pm\tfrac{1}{2}$ and exactly $i$ of them are nonzero. Then we have:
	\begin{equation*}
		\mcal{V}_{i} = \Set*{
			x\in\Aff\given
			\chi_{1}(x),\cdots,\chi_{n}(x)\in\tfrac{1}{2}\Z
			\text{ and exactly $i$ of them are non-integers }
		}.
	\end{equation*}
	Then $W_{o}.v_{n-1}$ consists of the points $x\in\Aff$ whose coordinates are $\pm\tfrac{1}{2}$ and odd numbers of them are negative. Then we have:
	\begin{equation*}
		\mcal{V}_{n-1} = \Set*{
			x\in\Aff\given
			\chi_{1}(x),\cdots,\chi_{n}(x)\in\tfrac{1}{2}\Z\setminus\Z,
			\chi_{1}(x)+\cdots+\chi_{n}(x)-\tfrac{n}{2}
			\text{ is odd }
		}
	\end{equation*}
	Then $W_{o}.v_{n}$ consists of the points $x\in\Aff$ whose coordinates are $\pm\tfrac{1}{2}$ and even numbers of them are negative. Then we have:
	\begin{equation*}
		\mcal{V}_{n} = \Set*{
			x\in\Aff\given
			\chi_{1}(x),\cdots,\chi_{n}(x)\in\tfrac{1}{2}\Z\setminus\Z,
			\chi_{1}(x)+\cdots+\chi_{n}(x)-\tfrac{n}{2}
			\text{ is even }
		}
	\end{equation*}
	Hence, the vertices are all the point $x\in\Aff$ such that $\chi_{j}(x)\in\tfrac{1}{2}\Z$ for all $j$ and the number of non-integer coordinates is neither $1$ nor $n-1$.	In particular, we have 
	\begin{equation}\label{eq:verticesInZn}
		\mcal{V}\subsetneq
			o + \tfrac{1}{2}\Z^n\subsetneq
			o + \tfrac{1}{2}\mcal{P}^{\vee}.
	\end{equation}

	\index[notation]{omega i prime@$\omega_{i}'$}%
	\index[notation]{X 00@$\mcal{X}^{00}$}%
	\index[notation]{X 01@$\mcal{X}^{10}$}%
	\index[notation]{X 10@$\mcal{X}^{01}$}%
	\index[notation]{X 11@$\mcal{X}^{11}$}%
	\index[notation]{X 0@$\mcal{X}^{(0)}$}%
	\index[notation]{X 1@$\mcal{X}^{(1)}$}%
	Let $\omega_{i}'=h_{i}^{-1}\omega_{i}$. Consider the following sets: 
	\begin{equation}\label{eq:SetsX:Dn}
		\begin{aligned}
			\mcal{X}^{00}	&:=
				o+\Z\omega_{1}'+\cdots+\Z\omega_{n}', &
			\mcal{X}^{01}	&:=
				\mcal{X}^{00}-\tfrac{1}{2}(\omega_{n-1}+\omega_{n}),\\
			\mcal{X}^{10}	&:=
				\mcal{X}^{00}-\tfrac{1}{2}\omega_{1}, &
			\mcal{X}^{11}	&:=
				\mcal{X}^{00}-\tfrac{1}{2}(\omega_{1}+\omega_{n-1}+\omega_{n}),\\
			\mcal{X}^{(0)}	&:=
				\mcal{X}^{00}\cup\mcal{X}^{10}, &
			\mcal{X}^{(1)}	&:=
				\mcal{X}^{01}\cup\mcal{X}^{11}.
		\end{aligned}
	\end{equation}

	\begin{lemma}\label{lem:VerticesDnX}
		We have $\mcal{X}^{(0)}\cup\mcal{X}^{(1)} = o+\tfrac{1}{2}\Z^n$.
	\end{lemma}
	\begin{proof}
		It is clear that $\mcal{X}^{(0)}\cup\mcal{X}^{(1)} \subset o + \tfrac{1}{2}\Z^n$.
		Conversely, we have 
		\begin{align*}
			\tfrac{1}{2}\vect{e}_{1} &= 
				\tfrac{1}{2}\omega_{1},\\
			\tfrac{1}{2}\vect{e}_{2} &= 
				\omega_{2}'-\tfrac{1}{2}\omega_{1},\\
			&\shortvdotswithin{=}
			\tfrac{1}{2}\vect{e}_{i} &=	
				\omega_{i}'-\omega_{i-1}',\\
			&\shortvdotswithin{=}
			\tfrac{1}{2}\vect{e}_{n-1} &=	
				\tfrac{1}{2}(\omega_{n-1}+\omega_{n})-\omega_{n-2}',\\
			\tfrac{1}{2}\vect{e}_{n} &=	
				\omega_{n}-\tfrac{1}{2}(\omega_{n-1}+\omega_{n}).
		\end{align*}
		Then the statement follows.
	\end{proof}

	To better describe the vertices, we need the notion introduced in \cref{def:jump}.
	Then the complement of $\mcal{V}$ in $\mcal{X}^{(0)}\cup\mcal{X}^{(1)}$ can be described as follows.
	\begin{lemma}\label{lem:VerticesDn0}
		A point $x\in \mcal{X}^{(0)}$ is not a vertex if and only if either $J_{x}=\Set*{j_{1},j_{2}}$ and $j_{2}-j_{1}=1$, or $J_{x}=\Set*{1}$.
	\end{lemma}
	\index[notation]{Xi 0@$\Xi^{(0)}$}%
	We will use $\Xi^{(0)}$ to denote the set of points $x\in\Aff$ having the property in the lemma.
	\begin{proof}
		First, points in $\Xi^{(0)}$ cannot be vertices.
		If $J_{x}=\Set*{1}$, then $x$ is not a vertex since either all $\chi_{j}(x)$ are integers except $j=1$ (when $a_{n-1}(x)+a_{n}(x)$ is even), or all $\chi_{j}(x)$ are non-integers except $j=1$ (when $a_{n-1}(x)+a_{n}(x)$ is odd). 
		Next, suppose $J_{x}=\Set*{j_{1},j_{2}}$ and $j_{2}-j_{1}=1$. 
		If $j_{2}<n-1$, then $x$ is not a vertex since either all $\chi_{j}(x)$ are integers except $j=j_{2}$ (when $a_{n-1}(x)+a_{n}(x)$ is even), or all $\chi_{j}(x)$ are non-integers except $j=j_{2}$ (when $a_{n-1}(x)+a_{n}(x)$ is odd). 
		If $j_{2}=n-1$, then $x$ is not a vertex since $\chi_{n}(x)\notin\tfrac{1}{2}\Z$.
		If $j_{2}=n$, then $J_{x}=\Set*{n-1,n}$ and we leave this situation in \cref{lem:VerticesDn1}.		
		
		Conversely, suppose $x\in \mcal{X}^{(0)}$ and $x\notin\Xi^{(0)}$. 
		Then there are four cases:
		\begin{enumerate}
			\item $\abs*{J_{x}}\ge 3$. Then among the coordinates of $x$, at least two of them are integers and two of them are non-integers. 
			\item $J_{x}=\Set*{j_{1},j_{2}}$ and $1<j_{2}-j_{1}<n-2$. 
			Then $\chi_{j_{1}+1}(x),\cdots,\chi_{j_{2}}(x)$ are either all the integer coordinates of $x$, or all the non-integer coordinates of $x$.
			\item $J_{x}=\Set*{j_{1}}$ and $1<j_{1}<n-1$. 
			Then $\chi_{1}(x),\cdots,\chi_{j_{1}}(x)$ are either all the integer coordinates of $x$, or all the non-integer coordinates of $x$.
			\item $x$ has no jumps. Then it is a special vertex. 
		\end{enumerate} 
		In any of above cases, $x$ is a vertex by our characterization.
	\end{proof}

	\begin{lemma}\label{lem:VerticesDn1}
		A point $x\in \mcal{X}^{(1)}$ is not a vertex if and only if $J_{x}\subset\Set*{n-2,n-1,n}$.
	\end{lemma}
	Note that $x\in \mcal{X}^{(1)}$ implies that $\Set*{n-1,n}\subset J_{x}$. 
	\index[notation]{Xi 1@$\Xi^{(1)}$}%
	We will use $\Xi^{(1)}$ to denote the set of points $x\in\Aff$ having the property that $\Set*{n-1,n}\subset J_{x}\subset\Set*{n-2,n-1,n}$.
	\begin{proof}
		First, points in $\Xi^{(1)}$ cannot be vertices.
		Indeed, if $x\in \Xi^{(1)}$, then $a_{j}(x)\in\Z$ for all $j<n-2$.
		Hence, $\chi_{1}(x),\cdots,\chi_{n-2}(x)$ are either all integers or all non-integers. 
		Hence, by \cref{eq:verticesInZn}, for $x$ to be a vertex, we must have that $2\chi_{n-1}(x)$ and $2\chi_{n}(x)$ are integers in the same parity. 
		But $\Set*{n-1,n}\subset J_{x}$ implies that they are not.

		Conversely, suppose $x\in \mcal{X}^{(1)}$ and $x\notin\Xi^{(1)}$. 
		Then $\Set*{n-1,n}\subset J_{x}$ implies that exactly one of $\chi_{n-1}(x)$ and $\chi_{n}(x)$ is an integer. 
		If $j<n-2$ is an index in $J_{x}$, then exactly one of $\chi_{j}(x)$ and $\chi_{j+1}(x)$ is an integer. 
		Hence, among the coordinates of $x$, at least two of them are integers and two of them are non-integers. 
	\end{proof}
	
	\index[notation]{dagger@$\dagger$}%
	We illustrate the structure of $\mcal{V}$ by the following diagrams:
	
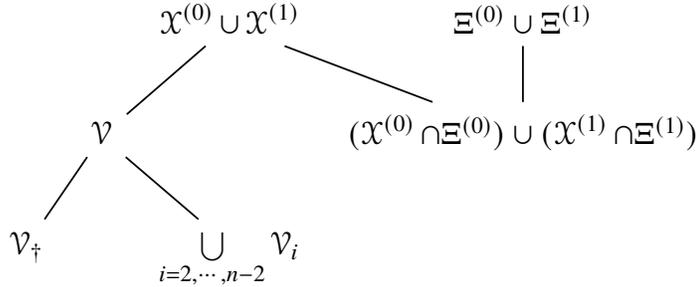
\begin{equation*}
		\begin{tikzcd}[column sep=tiny]
			&& {	\mcal{X}^{(0)}\cup\mcal{X}^{(1)}	}
				\ar[dl,no head]\ar[dr,no head]	
			&	{	\Xi^{(0)}\cup\Xi^{(1)}	}
				\ar[d,no head] \\
			&	{	\mcal{V}	}
				\ar[dl,no head]\ar[dr,no head]	
			&&	{	(\mcal{X}^{(0)}\cap\Xi^{(0)})\cup(\mcal{X}^{(1)}\cap\Xi^{(1)})	} \\
			{	\mcal{V}_{\dagger}	} 
			&&	{	\bigcup\limits_{i=2,\cdots,n-2}\mcal{V}_{i}	}	&
		\end{tikzcd}
	\end{equation*}
	where $\dagger$ denotes ``being special''.
\end{step}

\begin{step}
	Let $x\in\mscr{D}[\vC]$ be a vertex. 
	We divide into two cases: $x\in\mcal{X}^{(0)}$ or $x\in\mcal{X}^{(1)}$. 

	First, let us assume $x\in\mcal{X}^{(0)}$ and suppose $x$ has jumps $j_{1},\cdots,j_{s}$, ordered from smallest to largest. 
	To construct a path between $x$ and a special vertex $x_{0}$ in $\mscr{D}[\vC]$ verifying \cref{eq:SpecialX0inLemma}, we need the following lemmas.
	\begin{lemma}\label{lem:ConstructX0forDn:01}
		Suppose either $j_{s}-j_{s-1}>1$ or $s>3$. Then $y=x-\tfrac{1}{2}\omega_{j_{1}}$ is a vertex in $\mscr{D}[\vC]\cap\mcal{X}^{(0)}$ adjacent to $x$.
	\end{lemma}
	\begin{proof}
		It is clear that $y\in\mcal{X}^{(0)}$. 
		Then the proof is similar to \cref{lem:ConstructX0forBn:1} except that the root $a_{[j_{s}]}$ is defined as follows:
		\begin{equation*}
			a_{[j_{s}]} = 
			\begin{dcases*}
				a_{j_{s}-1}+2a_{j_{s}}+\cdots+2a_{n-2}+a_{n-1}+a_{n} & if $j_{s}-j_{s-1}>1$,\\
				a_{j_{s-2}}+\cdots+a_{j_{s}-1}+2a_{j_{s}}+\cdots+2a_{n-2}+a_{n-1}+a_{n} & otherwise,\\
			\end{dcases*}
		\end{equation*}
		Hence, we omit the proof here.
	\end{proof}

	Note that $a_{0}(y)\in\Z$ and $\ceil{a_{0}(x)}-a_{0}(x) = 1$. 
	Hence, by repeating using \cref{lem:ConstructX0forDn:01}, we can reduce our problem to the case where $s=3$, or further $s=1$ if we start with $j_{s}-j_{s-1}>1$. 

	Now, we may assume either $s\le 3$ with $j_{s}-j_{s-1}=1$ or $s=1$.
	\begin{lemma}\label{lem:ConstructX0forDn:02}
		Suppose $s=3$ and $j_{1}>1$. 
		Then $x_{0}=x-\tfrac{1}{2}(\omega_{j_{1}}-\omega_{j_{2}}+\omega_{j_{3}})$ is a special vertex in $\mscr{D}[\vC]$ adjacent to $x$ and verifying \cref{eq:SpecialX0inLemma}.
	\end{lemma}
	\begin{proof}
		The proof is the same as \cref{lem:ConstructX0forBn:2}.
	\end{proof}	

	\begin{lemma}\label{lem:ConstructX0forDn:03}
		Suppose $s=3$, $j_{3}-j_{2}=1$ and $j_{1}=1$, then $y=x-\tfrac{1}{2}(-\omega_{j_{1}}+\omega_{j_{2}})$ is a vertex in $\mscr{D}[\vC]\cap\mcal{X}^{(0)}$ adjacent to $x$.
	\end{lemma}
	\begin{proof}
		It is clear that $y\in\mcal{X}^{(0)}$. 
		Then the proof is similar to \cref{lem:ConstructX0forBn:3} except that the root $a_{[j_{s}]}$ is defined as follows:
		\begin{equation*}
			a_{[j_{s}]} = 
				a_{j_{1}}+\cdots+a_{j_{2}}+2a_{j_{3}}+\cdots+2a_{n-2}+a_{n-1}+a_{n}.
		\end{equation*}
		Hence, we omit the proof here.
	\end{proof}

	Note that the only jump of $y$ is $j_{3}>1$. 
	Hence, $a_{0}(y)\in\Z$ and $\ceil{a_{0}(x)}-a_{0}(y)=1$. 
	Therefore, \cref{lem:ConstructX0forDn:03} reduces our problem to the case where $s=1$. 

	Note that, by \cref{lem:VerticesDn0}, $s=2$ and $j_{s}-j_{s-1}=1$ contradict to each other. 
	Therefore, we may assume $s=1$ now. 
	By \cref{lem:VerticesDn0} again, we must have $j_{1}>1$.
	Let $x_{0}=x-\tfrac{1}{2}\omega_{j_{1}}$. 
	Then it is in a special vertex in $\mscr{D}[\vC]$ since $a_{1}(x_{0}),\cdots,a_{n}(x_{0})$ are non-negative integers. 
	Applying \cref{lem:AdjacentVertices} to the simple roots $a_{1},\cdots,a_{n}$ and the vertices $x$ and $x_{0}$, we see that they are adjacent. Moreover, we have $a_{0}(x)-a_{0}(x_{0})=1$ verifying \cref{eq:SpecialX0inLemma}. 
	This finishes the proof of \cref{thm:SimplicialDistance} when $x\in\mcal{X}^{(0)}$.

	Next, let us assume $x\in\mcal{X}^{(1)}$. Then we must have $\Set*{n-1,n}\subset J_{x}$. 
	Suppose $x$ has jumps $j_{1},\cdots,j_{s},n-1,n$, ordered from smallest to largest. 
	To construct a path between $x$ and a special vertex $x_{0}$ in $\mscr{D}[\vC]$ verifying \cref{eq:SpecialX0inLemma}, we need the following lemmas.

	\begin{lemma}\label{lem:ConstructX0forDn:11}
		Suppose either $j_{s}-j_{s-1}>1$ or $s>2$. Then $y=x-\tfrac{1}{2}\omega_{j_{1}}$ is a vertex in $\mscr{D}[\vC]\cap\mcal{X}^{(1)}$ adjacent to $x$.
	\end{lemma}
	\begin{proof}
		It is clear that $y\in\mcal{X}^{(1)}$. 
		When $j_{s}-j_{s-1}>1$ or $s>3$, the proof is similar to \cref{lem:ConstructX0forDn:01} except that we have to define $a_{[n-1]}$ and $a_{[n]}$ as follows:
		\begin{align*}
			a_{[n-1]} &= 
				a_{j_{s}}+\cdots+a_{n-2}+a_{n-1},\\
			a_{[n]} &= 
				a_{j_{s}}+\cdots+a_{n-2}+a_{n}.
		\end{align*}
		When $s=3$ and $j_{3}-j_{2}=1$, the proof still works if we define $a_{[j_{s}]}$ as follows:
		\begin{equation*}
			a_{[j_{3}]} = a_{j_{2}}+\cdots+a_{n}. \qedhere
		\end{equation*}
	\end{proof}

	Note that $a_{0}(y)\in\Z$ and $\ceil{a_{0}(x)}-a_{0}(x) = 1$. 
	Hence, by repeating using \cref{lem:ConstructX0forDn:11}, we can reduce our problem to the case where $s=2$, or further $s=1$ if we start with $j_{s}-j_{s-1}>1$. 

	Now, we may assume either $s\le 2$ with $j_{s}-j_{s-1}=1$ or $s=1$.
	\begin{lemma}\label{lem:ConstructX0forDn:12}
		Suppose $s=2$ and $j_{1}>1$. 
		Then $x_{0}=x-\tfrac{1}{2}(\omega_{j_{1}}-\omega_{j_{2}}+\omega_{n-1}+\omega_{n})$ is a special vertex in $\mscr{D}[\vC]$ adjacent to $x$ and verifying \cref{eq:SpecialX0inLemma}.
	\end{lemma}
	\begin{proof}
		First note that $a_{0}(x)\in\Z$ and 
		\begin{equation*}
			a_{0}(x)-a_{0}(x_{0}) = \tfrac{1}{2}a_{0}(\omega_{j_{1}}-\omega_{j_{2}}+\omega_{n-1}+\omega_{n}) = 1.
		\end{equation*}
		Then the proof is similar to \cref{lem:ConstructX0forBn:2} except that there is no $j_3$ and that we need to define $a_{[n-1]}$ and $a_{[n]}$ as follows:
		\begin{align*}
				a_{[n-1]} &=
					a_{j_{2}}+\cdots+a_{n-2}+a_{n-1},\\
				a_{[n]} &=
					a_{j_{2}}+\cdots+a_{n-2}+a_{n}.
		\end{align*}
		Then the statement follows.
	\end{proof}	


	\begin{lemma}\label{lem:ConstructX0forDn:13}
		Suppose $s=2$ and $j_{1}=1$, then $x_{0}=x-\tfrac{1}{2}(-\omega_{j_{1}}+\omega_{j_{2}}-\omega_{n-1}+\omega_{n})$ is a special vertex in $\mscr{D}[\vC]$ adjacent to $x$ and verifying \cref{eq:SpecialX0inLemma}.
	\end{lemma}
	\begin{proof}
		First note that $a_{0}(x)\in\tfrac{1}{2}+\Z$ and 
		\begin{equation*}
			a_{0}(x)-a_{0}(x_{0}) = \tfrac{1}{2}a_{0}(-\omega_{j_{1}}+\omega_{j_{2}}-\omega_{n-1}+\omega_{n}) = \tfrac{1}{2}.
		\end{equation*}
		Then the proof is similar to \cref{lem:ConstructX0forDn:12} except that $a_{[n]}$ is defined as follows:
		\begin{equation*}
				a_{[n]} =
					a_{1}+\cdots+a_{n}.
		\end{equation*}
		Then the statement follows.
	\end{proof}


	\begin{lemma}\label{lem:ConstructX0forDn:14}
		Suppose $s=1$. 
		Then $x_{0}=x-\tfrac{1}{2}(\omega_{j_{1}}-\omega_{n-1}+\omega_{n})$ is a special vertex in $\mscr{D}[\vC]$ adjacent to $x$ and verifying \cref{eq:SpecialX0inLemma}.
	\end{lemma}
	\begin{proof}
		First note that 
		\begin{equation*}
			a_{0}(x)-a_{0}(x_{0}) = \tfrac{1}{2}a_{0}(\omega_{j_{1}}-\omega_{n-1}+\omega_{n}) = \tfrac{1}{2}a_{0}(\omega_{j_{1}}).
		\end{equation*}
		Hence, $\ceil{a_{0}(x)}-a_{0}(x_{0})=1$. 
		Then the proof is similar to \cref{lem:ConstructX0forDn:12} or \cref{lem:ConstructX0forDn:13} except that there is no $j_2$ and that the root $a_{[n]}$ is defined as follows:
		\begin{equation*}
				a_{[n]} =
					a_{n-2}+a_{n-1}+a_{n}.
		\end{equation*}
		Here we need $j_1<n-2$, which is guaranteed by \cref{lem:VerticesDn1}. 
	\end{proof}
	This finishes the proof of \cref{thm:SimplicialDistance}.
\end{step}

\begin{step}\label{step:IndexBsrciDn}
	Let $I$ be a type and follow \cref{con:type}. 
	Using the notations introduced in \cref{eq:IndexSetX}, we have $\mcal{V}[I,r]=\partial(r,\vC,I)$. 
	We also have introduced the sets $\mcal{X}^{00}$, $\mcal{X}^{10}$, $\mcal{X}^{01}$, $\mcal{X}^{11}$, $\mcal{X}^{(0)}$, and $\mcal{X}^{(1)}$ in \cref{eq:SetsX:Dn}. 
	Inspired \cref{lem:VerticesDnX,lem:VerticesDn0,lem:VerticesDn1} and \cref{def:jump}, we consider the following sets for each $J\subset\Set*{1,\cdots,n}$:
	\index[notation]{Xi J@$\Xi_{J}$}%
	\index[notation]{X J@$\mcal{X}_{J}$}%
	\begin{align*}
		\Xi_{J} &:=
			\Set*{	x\in\Aff	\given	J_{x}=J	}
		&\text{and}&&
		\mcal{X}_{J} &:=
			\Xi_{J}\cap o+\tfrac{1}{2}\Z^n.
	\end{align*}
	Note that $o+\tfrac{1}{2}\Z^n = \mcal{X}^{(0)}\cup\mcal{X}^{(1)}$. 
	Then we have 
	\begin{align*}
		\MoveEqLeft
		o+\tfrac{1}{2}\Z^n \setminus \mcal{V} \\
		&= 
		\mcal{X}_{\Set*{1}}\cup\mcal{X}_{\Set*{1,2}}\cup\cdots\cup\mcal{X}_{\Set*{n-3,n-2}}\cup\mcal{X}_{\Set*{n-1,n}}\cup\mcal{X}_{\Set*{n-2,n-1,n}}.
	\end{align*}
	Note that, for any $J$, we have 
	\begin{equation}\label{eq:XJfromX0Dn}
		\mcal{X}_{J} = \mcal{X}_{\emptyset}-\sum_{j\in J}\tfrac{1}{2}\omega_{j}.
	\end{equation}
	Moreover, it is clear that $\mcal{X}_{\emptyset}$ is precisely $o+\mcal{P}^{\vee}$, the set of special vertices.

	Next, we consider $\mcal{X}[I]$ for above sets. 
	First, if $\Set*{n-1,n}\cap I\neq\emptyset$ then $\mcal{V}[I]\subset\mcal{X}^{(0)}[I]$. 
	Otherwise, $\mcal{V}[I]\cap\mcal{X}^{(1)}[I]\neq\emptyset$. 
	In each case, we have $\mcal{V}[I]\cap(\mcal{X}^{10}[I]\cup\mcal{X}^{11}[I])=\emptyset$ if and only if $\ell_{1}>1$. 
	For any $J$, it is clear that $\mcal{X}_{J}[I]\neq\emptyset$ if and only if $I\cap J=\emptyset$. If this is the case, we have the following refinement of \cref{eq:XJfromX0Dn}:
	\begin{equation}\label{eq:XIJfromXI0Dn}
		\mcal{X}_{J}[I] = \mcal{X}_{\emptyset}[I]-\sum_{j\in J}\tfrac{1}{2}\omega_{j}.
	\end{equation}

	Finally, we consider $\mcal{X}[I,r]$. First, it is clear that 
	\begin{equation}
		\label{eq:IndexBsrciDn:pu00}
		\mcal{X}^{00}[I,r] =
		\Set*{
			x = o+c_{1}\omega_{\ell_{1}}'+\cdots+c_{t}\omega_{\ell_{t}}'
			\given
			\begin{array}{c}
				c_{1},\cdots,c_{t}>0,\\
				c_{1},\cdots,c_{t}\in\Z_{>0},\\
				c_{1}+\cdots+c_{t} = r
			\end{array}
		}.
	\end{equation}
	Then the followings follow from \cref{eq:SetsX:Dn}:
	\begin{align}
		\label{eq:IndexBsrciDn:pu01}
			\mcal{X}^{01}[I,r] &= 
				\mcal{X}^{00}[I,r+1]-\tfrac{1}{2}(\omega_{n-1}+\omega_{n}),\\
		\label{eq:IndexBsrciDn:pu10}
			\mcal{X}^{10}[I,r] &= 
				\mcal{X}^{00}[I,r]-\tfrac{1}{2}\omega_{1},\\
		\label{eq:IndexBsrciDn:pu11}
			\mcal{X}^{11}[I,r] &= 
				\mcal{X}^{00}[I,r+1]-\tfrac{1}{2}(\omega_{1}+\omega_{n-1}+\omega_{n}).
	\end{align}
	Then we need to work out $\mcal{X}_{J}[I,r]$. For $\mcal{X}_{\emptyset}[I,r]$, an explicit description is given in \cref{eq:expressIndexSets_speical:Ball,eq:expressIndexSets_speical:Sphere}. 
	For general $J$, we have the following refinement of \cref{eq:XIJfromXI0Dn}:
	\begin{lemma}\label{lem:XIJfromXI0Dn}
		\index[notation]{deltal (J)@$\delta(J)$}%
		Suppose $I\cap J=\emptyset$. 
		Then we have 
		\begin{equation*}
			\mcal{X}_{J}[I,r] = 
				\mcal{X}_{\emptyset}[I,r+\abs*{J}-\delta(J)]-\sum_{j\in J}\tfrac{1}{2}\omega_{j},
		\end{equation*}
		where $\delta(J)$ is defined as follows:
		\begin{equation*}
			\delta(J) = 
			\begin{dcases*}
				2 & if $\Set*{1,n-1,n}\subset J$,\\
				0 & if $\Set*{1,n-1,n}\cap J=\emptyset$,\\
				1 & otherwise.
			\end{dcases*}
		\end{equation*}
	\end{lemma}
	\begin{proof}
		By \cref{eq:XIJfromXI0Bn}, it suffices to show that for any $x\in\mcal{X}_{\emptyset}[I,r]$,
		\begin{equation*}\label{eq:tag:lem:XIJfromXI0Dn}\tag{$\ast$}
			\ceil*{a_{0}\left(x-\sum_{j\in J}\tfrac{1}{2}\omega_{j}\right)} = 
				a_{0}(x)-\abs*{J}+\delta(J).
		\end{equation*}
		Note that $a_{0}(x)\in\Z$ and that
		\begin{equation*}
			a_{0}\left(\sum_{j\in J}\tfrac{1}{2}\omega_{j}\right) = 
			\begin{dcases*}
				\abs*{J}-\tfrac{3}{2} & if $\Set*{1,n-1,n}\subset J$,\\
				\abs*{J}-1 & if $1\notin J$ but $\Set*{n-1,n}\subset J$,\\
				\abs*{J}-\tfrac{1}{2} & if $1\in J$ but $\Set*{n-1,n}\cap J=\emptyset$,\\
				\abs*{J} & if $\Set*{1,n-1,n}\cap J=\emptyset$.
			\end{dcases*}
		\end{equation*}
		Then \cref{eq:tag:lem:XIJfromXI0Dn} follows.
	\end{proof}
	
	We illustrate above discussions by the following diagrams:
	\begin{center}
		\captionsetup{type=figure}
		\begin{equation*}
			\begin{tikzcd}[column sep=0pt]
				& {	\mcal{X}^{00}[I,r]	}
					\ar[dl,no head]\ar[dr,no head]	
				&	{	\bigcup\limits_{J=\Set*{2,3},\cdots,\Set*{n-3,n-2}}\Xi_{J}	}
					\ar[d,no head] \\
				\mathclap{	\mcal{V}[I,r]	}\ar[d,no head]	
				&&	{	\bigcup\limits_{J=\Set*{2,3},\cdots,\Set*{n-3,n-2}}\mcal{X}_{J}[I,r]	} \\
				\mathclap{	\mcal{V}_{\dagger}[I,r]=\mcal{X}_{\emptyset}[I,r]	}
				&&
			\end{tikzcd}
		\end{equation*}\vspace*{-3mm} 
		\caption{Vertices of type $I$ in $\Apt[D_{n}]$ ($1\in I$ and $\Set*{n-1,n}\cap I\neq\emptyset$)}\label{figure:VerticesOfIDnl1ne}
	\end{center}

	\begin{center}
		\captionsetup{type=figure}
		\begin{equation*}
			\begin{tikzcd}[column sep=0pt]
				& {	\mcal{X}^{00}[I,r]\cup\mcal{X}^{01}[I,r]	}
					\ar[dl,no head]\ar[dr,no head]	
				&	{	\bigcup\limits_{
							\crampedsubstack{
								J=\Set*{2,3},\cdots,\Set*{n-3,n-2},\\
								\Set*{n-1,n},\Set*{n-2,n-1,n}
								}}\Xi_{J}	}
					\ar[d,no head] \\
				\mathclap{	\mcal{V}[I,r]	}\ar[d,no head]	
				&&	{	\bigcup\limits_{
								\crampedsubstack{
									J=\Set*{2,3},\cdots,\Set*{n-3,n-2},\\
									\Set*{n-1,n},\Set*{n-2,n-1,n}
									}}\mcal{X}_{J}[I,r]	} \\
				\mathclap{	\mcal{V}_{\dagger}[I,r]=\mcal{X}_{\emptyset}[I,r]	}
				&&
			\end{tikzcd}
		\end{equation*}\vspace*{-3mm} 
		\caption{Vertices of type $I$ in $\Apt[D_{n}]$ ($1\in I$ and $\Set*{n-1,n}\cap I=\emptyset$)}\label{figure:VerticesOfIDnl1e}
	\end{center}

	\begin{center}
		\captionsetup{type=figure}
		\begin{equation*}
			\begin{tikzcd}[column sep=0pt]
				& {	\mcal{X}^{00}[I,r]\cup\mcal{X}^{10}[I,r]	}
					\ar[dl,no head]\ar[dr,no head]	
				&	{	\bigcup\limits_{J=\Set*{1},\Set*{1,2},\cdots,\Set*{n-3,n-2}}\Xi_{J}	}
					\ar[d,no head] \\
				\mathclap{	\mcal{V}[I,r]	}\ar[d,no head]	
				&&	{	\bigcup\limits_{J=\Set*{1},\Set*{1,2},\cdots,\Set*{n-3,n-2}}\mcal{X}_{J}[I,r]	} \\
				\mathclap{	\mcal{V}_{\dagger}[I,r]=\mcal{X}_{\emptyset}[I,r]	}
				&&
			\end{tikzcd}
		\end{equation*}\vspace*{-3mm} 
		\caption{Vertices of type $I$ in $\Apt[B_{n}]$ ($1\notin I$ and $\Set*{n-1,n}\cap I\neq\emptyset$)}\label{figure:VerticesOfIDne1ne}
	\end{center}

	\begin{center}
		\captionsetup{type=figure}
		\begin{equation*}
			\begin{tikzcd}[column sep=0pt]
				& {	\mcal{X}^{(0)}[I,r]\cup\mcal{X}^{(1)}[I,r]	}
					\ar[dl,no head]\ar[dr,no head]	
				&	{	\bigcup\limits_{
							\crampedsubstack{
								J=\Set*{1},\Set*{1,2},\Set*{2,3},\cdots,\Set*{n-3,n-2},\\
								\Set*{n-1,n},\Set*{n-2,n-1,n}
								}}\Xi_{J}	}
					\ar[d,no head] \\
				\mathclap{	\mcal{V}[I,r]	}\ar[d,no head]	
				&&	{	\bigcup\limits_{
					\crampedsubstack{
						J=\Set*{1},\Set*{1,2},\Set*{2,3},\cdots,\Set*{n-3,n-2},\\
						\Set*{n-1,n},\Set*{n-2,n-1,n}
						}}\mcal{X}_{J}[I,r]	} \\
				\mathclap{	\mcal{V}_{\dagger}[I,r]=\mcal{X}_{\emptyset}[I,r]	}
				&&
			\end{tikzcd}
		\end{equation*}\vspace*{-3mm} 
		\caption{Vertices of type $I$ in $\Apt[B_{n}]$ ($1\notin I$ and $\Set*{n-1,n}\cap I=\emptyset$)}\label{figure:VerticesOfIDne1e}
	\end{center}
\end{step}

\clearpage
\section{Asymptotic analysis}\label{sec:Asymptotic}
This section aims to provide tools to analyze the asymptotic behavior of the simplicial volume and the simplicial surface area. 

We have already seen that $\fun{SV}[r]\asymp\fun{SSA}[r]$ in \cref{sec:Intro}, the introduction. 
Hence, in order to prove \cref{thm:AsymptoticDominanceOfSV}, it suffices to prove the simplicial surface area part. 
Likewise, the simplicial volume part of \cref{thm:AsymptoticGrowthOfSV} can be deduced from the simplicial surface area part, either by \cref{lem:asymptotic_equality} or direct computation. 
Therefore, it suffices to consider the asymptotic analysis of the simplicial surface area $\fun{SSA}[\:\cdot\:]$ only.

In the formulas \cref{eq:SimplicialVolumeFormula,eq:SimplicialSurfaceAreaFormula}, there are only finitely many $I\subset\Delta$ and each $\mscr{P}_{\Phi;I}[q]$ is an integral polynomial. 
Hence, the asymptotic study of $\fun{SSA}[\:\cdot\:]$ can be reduced to summations of the following form:
\begin{equation}\label{eq:SX}
	\fun{S}_{\mcal{X}[I]}[r] := 
		\sum_{x\in\mcal{X}[I,r]}\prod_{a\in\Phi^{+}}q^{\ceil{a(x)}},
\end{equation}
where $\mcal{X}$ is a set of points, and the notations $\mcal{X}[I]$ and $\mcal{X}[I,r]$ follow \cref{eq:IndexSetX}.

The growth of $\fun{S}_{\mcal{X}[I]}[r]$ varies for different types $I$. 
For the purpose of asymptotic analysis, only the dominant ones are relevant. 
To better analyze their growth, we introduce the following auxiliary functions:
\begin{equation}\label{eq:SXasymp}
	\fun{S}^{\asymp}_{\mcal{X}[I]}[r] := 
		\sum_{x\in\mcal{X}[I,r]}q^{2\rho(x)},
\end{equation}
where $2\rho$ is the sum of positive roots. 
Note that
\begin{equation*}
	2\rho(x)=\sum_{a\in\Phi^{+}}a(x)	\le	
	\sum_{a\in\Phi^{+}}\ceil{a(x)}	\le	
	\sum_{a\in\Phi^{+}}(a(x)+1)=
	2\rho(x)+\fun{deg}[\mscr{P}_{\Phi;I}].
\end{equation*}
Hence, we have $\fun{S}_{\mcal{X}[I]}[r]\asymp\fun{S}^{\asymp}_{\mcal{X}[I]}[r]$. 
But the later one is easier to study.

This section is structured as follows. 
In \cref{subsec:DiscreteCalculus}, we will introduce \emph{$q$-numbers} and \emph{$q$-functions} and discuss the \emph{discrete calculus} on them. 
We will then only focus on the $q$-functions defined by \emph{(super) $q$-exponential polynomials}. 
To study them algebraically, we will review gradings and filtrations in \cref{subsec:gradedAlg}. 
Then in \cref{subsec:EP,subsec:SEP}, we will introduce \emph{(super) $q$-exponential polynomials}  and study the asymptotic properties of the $q$-functions defined by them.
Finally, with those notions in hand, we will study the asymptotic growth of multi-summations in \cref{subsec:AGMS,subsec:AGMS2}.

\subsection{Discrete calculus of \texorpdfstring{$q$}{q}-functions}\label{subsec:DiscreteCalculus}
It is often more convenient to treat $q$ as a formal variable when we apply algebraic operations to $\fun{S}_{\mcal{X}[I]}[r]$ and $\fun{S}^{\asymp}_{\mcal{X}[I]}[r]$. 
But to carry out the asymptotic analysis, we need to view $q$ as a real number. 
Inspired by this, we have the following definition.
\begin{definition}\label{def:qFunction}
	\index{q-number@$q$-number}%
	\index{q-function@$q$-function}%
	\index{q-number@$q$-number!level of}%
	\index{q-function@$q$-function!level of}%
	\index[notation]{Q (q,-)@$\mbb{Q}[q;-]$}%
	Let $q$ be a formal variable and $h$ a positive integer. 
	Then a \emph{$q$-number (of level $h$)} is a rational function of $q^{1/h}$ over $\Q$ having no poles on the half real line $\R_{>1}:=\Set*{	r\in\R	\given r>1	}$. 
	Let $\mbb{Q}[q;h]$ denote the ring of $q$-numbers of level $h$. 
	Then a \emph{$q$-function (of level $h$)} is a function defined for sufficiently large integers and valued in $\mbb{Q}[q;h]$. 
\end{definition}
\begin{example}
	Let $h$ be a positive integer larger than $1$. 
	Then $(q^{1/h}-1)^{-1}$ is a $q$-number of level $h$, while $(q-h)^{-1}$ is not a $q$-number. 
\end{example}
\begin{remark}
	\index[notation]{R (>1)@$\R_{>1}$}%
	A rational function of $q^{1/h}$ is in particular an algebraic function of $q$ and hence we can talk about its poles. 
	On the half real line $\R_{>1}$, the function $q^{1/h}$ has a unique real-valued branch. 
	This allows us to treat $q$-numbers as real-valued continuous functions on $\R_{>1}$. 
\end{remark}

\index{q-number@$q$-number!pointwise topology of}%
\index{q-function@$q$-function!limit of}%
\index[notation]{f (q)@$f_{q}$}%
Each $\mbb{Q}[q;h]$ is a principal ideal domain. 
When the level $h$ varies, they form an inductive system. 
Let $\mbb{Q}[q;-]$ denote the inductive limit. 
We will view it as the ring of all $q$-numbers. 
On this principal ideal domain, we will consider the \emph{pointwise topology} inheriting from the algebra $\mscr{C}[\R_{>1}]$ of real-valued continuous functions on $\R_{>1}$. 
In particular, if $f$ is a $q$-function, then the limit of $f(z)$ as $z\to\infty$ is defined pointwise:
\begin{equation*}
	\lim_{z\to\infty}f(z) = 
		\left(\lim_{z\to\infty}f(z)(q)\right)_{q>1}.
\end{equation*}
Then we can view each $q$-function $f$ as a family of discrete functions $(f_{q})_{q>1}$ indexed by the half real line $\R_{>1}$, where $f_{q}(z):=f(z)(q)$.

\begin{definition}\label{def:AsymptoticGrowth}
	\index{q-function@$q$-function!asymptotic equality of}%
	\index{asymptotically equal}%
	\index[notation]{sim@$\sim$}%
	Let $f$ and $g$ be two $q$-functions. 
	We say that they are \emph{asymptotically equal} and that $f$ \emph{has asymptotic growth} $g$, denoted by $f(z)\sim g(z)$, if 
	\begin{equation*}
		\lim_{z\to\infty}\frac{f(z)}{g(z)} = 1. 
	\end{equation*}
\end{definition}

We also need asymptotic dominant relations of $q$-functions. 
Like the topology, these notions are defined \emph{pointwise}.
\begin{definition}
	\index{q-number@$q$-number!positive}%
	\index{q-number@$q$-number!non-negative}%
	\index{q-function@$q$-function!eventually positive}%
	\index{q-function@$q$-function!eventually non-negative}%
	A $q$-number $C$ is said to be \emph{positive} (resp. \emph{non-negative}) if for all $q>1$, $C(q)$ is a positive (resp. non-negative) real number. 
	A $q$-function $f$ is said to be \emph{eventually positive} (resp. \emph{eventually non-negative}) if  for all $q>1$, $f_{q}$ is an eventually positive (resp. eventually non-negative) function, namely: $f_{q}(z)>0$ (resp. $f_{q}(z)\ge 0$) for sufficiently large $z$.
\end{definition}
\begin{definition}
	\index{q-function@$q$-function!dominant}%
	\index{asymptotic dominant relation}%
	\index{dominate}%
	\index[notation]{greater@$\gg$}%
	\index[notation]{asymp@$\asymp$}%
	Let $f$ and $g$ be two $q$-functions. 
	We say that $f$ \emph{dominates} $g$, denoted by $f(z)\gg g(z)$, if there exists a positive $q$-number $C$ such that $\abs*{f} - C\cdot\abs*{g}$ is an eventually non-negative $q$-function. 
	We will denote $f(z)\asymp g(z)$ if both $f(z)\gg g(z)$ and $g(z)\gg f(z)$.
\end{definition}

Then we can consider the \emph{discrete calculus} on $q$-numbers. 
\begin{definition}\label{def:difference}
	\index{q-function@$q$-function!difference of}%
	\index[notation]{dif (f)@$\adif{f}$}%
	Let $f$ be a $q$-function. 
	Its \emph{difference} $\adif{f}$ is the following $q$-function:
	\begin{equation*}
		\adif{f}(z) := f(z+1) - f(z).
	\end{equation*}
\end{definition}

The \emph{difference operator} $\adif$ is $\mbb{Q}[q;-]$-linear and satisfies the \emph{Leibniz rule}:
\begin{equation}
	\label{eq:LeibnizRule}
	\adif[fg] = f\cdot\adif{g} + g\cdot\adif{f} + \adif{f} \cdot\adif{g}.
\end{equation}

\begin{definition}
	\index{q-function@$q$-function!eventually strictly increasing}%
	A $q$-function $f$ is said to be \emph{eventually strictly increasing} if for all $q>1$, $f_{q}$ is an eventually strictly increasing function.
\end{definition}
Clearly, $f$ is eventually strictly increasing if and only if $\adif{f}$ is eventually positive.
\begin{definition}
	\index{q-function@$q$-function!unbounded}%
	A $q$-function $f$ is said to be \emph{unbounded} if for all $q$-number $C$, the $q$-function $\abs*{f}-C$ is eventually positive.
\end{definition}
\begin{lemma}\label{lem:asymptotic_equality}
	Let $f$ and $g$ be two eventually strictly increasing unbounded $q$-functions. 
	Then we have $f(z)\sim g(z)$ if and only if $\adif{f}(z) \sim \adif{g}(z)$.
\end{lemma}
\begin{proof}
	Apply \emph{Stolz-Ces\`{a}ro theorem} (see e.g. \cite{choudary2014real}*{theorem 2.7.2}) to $f_{q}$ and $g_{q}$ for all $q>1$. 
	Then the statement follows.
\end{proof}

We also need the discrete version of integrals. 
\begin{definition}
	\index{q-function@$q$-function!anti-difference of}%
	\index[notation]{sum (f)@$\asum{f}$}%
	\index[notation]{sum (f) anchor (a)@$\asum_{a}{f}$}%
	Let $f$ be a $q$-function. 
	Then an \emph{anti-difference} of $f$ is a $q$-function $g$ such that $\adif{g}=f$. 
	Since $\fun{Ker}[\adif]$ consists of constant $q$-functions, we see that the anti-difference is not unique but unique up to a constant $q$-function. 
	By an abuse of notation, we will use $\asum{f}$ to denote an anti-difference of $f$. 
	Let $a$ be an integer in the domain of $f$. 
	Then the \emph{anti-difference of $f$ with anchor $a$}, denoted by $\asum_{a}{f}$, is defined as follows:
	\begin{equation*}
		\asum_{a}{f}(z) := (\asum{f})(z) - (\asum{f})(a).
	\end{equation*}
	Note that $\asum_{a}{f}$ is well-defined although $\asum{f}$ is not.
\end{definition}

Note that, if $a,b$ are two integers, then we have the following summation formula:
\begin{equation}\label{eq:FTC}
	\sum_{z=a}^{b-1}f(z) = 
		(\asum{f})(b) - (\asum{f})(a) = 
		(\asum_{a}{f})(b).
\end{equation}

We will consider the following notions of $q$-numbers and $q$-functions.
\begin{definition}
	\index{q-number@$q$-number!primary}%
	\index{q-function@$q$-function!primary}%
	A $q$-number is said to be \emph{primary} if it is of level one. 
	Then a $q$-function is said to be \emph{primary} if its values are primary $q$-numbers. 
\end{definition}


\subsection{Weakly graded algebras}\label{subsec:gradedAlg}
Before moving on, let's review gradings and then the filtrations induced by them. 
In the study of (super) $q$-exponential polynomials, it is the filtration induced by a grading, rather than the grading itself, will play an essential role. 

Throughout this subsection, $R$ is a commutative ring and $\Gamma$ is an additive monoid. 
In the applications later, $\Gamma$ will be $\N$, $\Q$, $\F_2$, or products of them. 
\begin{definition}\label{def:grading}
	\index{grading}%
	\index{graded module}%
	\index{grading!homogeneous component}%
	\index{grading!homogeneous element}%
	\index[notation]{M (g)@$M_{g}$}%
	A \emph{$\Gamma$-grading} on an $R$-module $M$ is a decomposition into a direct sum 
	\begin{equation*}
		M = \bigoplus_{g\in\Gamma}M_{g},
	\end{equation*}
	where each $M_{g}$ is an $R$-submodule, called the \emph{homogeneous component of grade $g$}. 
	Elements of $M_{g}$ are said to be \emph{homogeneous of grade $g$}. 
	A general element $m$ of $M$ is decomposed into homogeneous elements $m_{g}$ ($g\in\Gamma$), each $m_{g}$ is called its \emph{homogeneous component of grade $g$}. 
	An $R$-module equipped with a \emph{$\Gamma$-grading} is called a \emph{$\Gamma$-graded module over $R$}. 
\end{definition}

\begin{definition}
	\index{operator on graded module!respect the grading}%
	\index{operator on graded module!shift the grading!homogeneously}%
	Let $M$ be a $\Gamma$-graded module over $R$ and $h\in\Gamma$. 
	Then an operator $T$ on $M$ is said to \emph{respect the grading} if $T(M_{g})\subset M_{g}$ for all $g\in\Gamma$ and \emph{shift the grading homogeneously by $h$} if $T(M_{g})\subset M_{g+h}$ for all $g\in\Gamma$. 
\end{definition}

\begin{lemma}\label{lem:grading_shifting_homogeneously}
	Suppose $\Gamma$ is a group and $h\in\Gamma$. 
	Let $M$ be a $\Gamma$-graded projective module over $R$ and $T$ a surjective $R$-linear operator $T$ on $M$ shifting the grading homogeneously by $h$. 
	Then there is a section of $T$ shifts the grading homogeneously by $-h$. 
\end{lemma}
\begin{proof}
	The conditions on $T$ imply that its restriction to each $M_{g}$ is surjective onto $M_{g+h}$. 
	Since $M_{g+h}$ is projective, $T|_{M_{g}}$ admits a section $S_{g}\colon M_{g+h}\to M_{g}$. 
	Then the desired section of $T$ is the direct sum of the sections $S_{g}$. 
\end{proof}

\begin{remark}
	\index[notation]{Grothendieck (Gamma)@$\mcal{G}[\Gamma]$}%
	The \emph{Grothendieck group} $\mcal{G}[\Gamma]$ of $\Gamma$ is the universal Abelian group under $\Gamma$. 
	If $M$ is a $\Gamma$-graded module, then we will treat it as a $\mcal{G}[\Gamma]$-graded module by defining $M_{g}=\Set*{0}$ if $g\in\mcal{G}[\Gamma]\setminus\Gamma$. 
	With this convention, \cref{lem:grading_shifting_homogeneously} holds even without assuming that $\Gamma$ is a group.
\end{remark}
\begin{corollary}\label{lem:grading_respecting}
	Let $M$ be a $\Gamma$-graded projective module over $R$ and $T$ a surjective $R$-linear operator $T$ on $M$ respecting the grading. 
	Then there is a section of $T$ respects the grading. 
\end{corollary}

\begin{definition}
	\index{graded algebra}%
	A \emph{$\Gamma$-graded algebra over $R$} is an $R$-algebra $A$ equipped with a $\Gamma$-grading such that 
	\begin{equation*}
		A_{g}A_{h} \subset A_{g+h},  
		\qquad\text{for all}\qquad
		g,h\in\Gamma.
	\end{equation*}
	Note that $A_{0}$ is a subalgebra and each $A_{g}$ ($g\in\Gamma$) is an $A_{0}$-bimodule. 
	The subalgebra $A_{0}$ is called the \emph{subalgebra of grade $0$}. 
	
	\index{graded algebra!graded module of}%
	If $A$ is a $\Gamma$-graded algebra over $R$, then a \emph{$\Gamma$-graded (left) $A$-module} is a $\Gamma$-graded module $M$ over $R$ equipped with a (left) $A$-action such that 
	\begin{equation*}
		A_{g}M_{h} \subset M_{g+h}, 
		\qquad\text{for all}\qquad
		g,h\in\Gamma.
	\end{equation*}
\end{definition}



\begin{definition}\label{def:homogeneous_basis}
	\index{graded algebra!homogeneous basis of}%
	\index{graded algebra!coefficient attached to $e_{g}$}%
	Let $A$ be an $R$-algebra $A$ equipped with a $\Gamma$-grading. 
	Suppose that $A_{0}$ is a subalgebra of $A$ and that each homogeneous component $A_{g}$ ($g\in\Gamma$) is a free $A_{0}$-module of rank one. 
	Then the grading can be written as follows:
	\begin{equation*}
		A = \bigoplus_{g\in\Gamma}A_{0}e_{g},
	\end{equation*}
	where each $e_{g}\in A_{g}$ is a generator of the $A_{0}$-module $A_{g}$ and $e_{0}=1\in A_{0}$. 
	The family $(e_{g})_{g\in\Gamma}$ is called a \emph{homogeneous basis of $A$ over $A_{0}$}. 	
	Note that a homogeneous basis of $A$ determines the  $\Gamma$-grading on it. 
	Let $a\in A$. 
	Then its homogeneous component of grade $g$ is of the form $a_{g}=c_{g}e_{g}$ with $c_{g}\in A_{0}$. 
	The element $c_{g}$ is called the \emph{coefficient} of $a_{g}$ and is said to be a \emph{coefficient of $a$} and \emph{attached to $e_{g}$ in $a$}. 
\end{definition}
\begin{remark}
	The assumption in \cref{def:homogeneous_basis} does not require $A$ to be a $\Gamma$-graded algebra. 
\end{remark}

In the rest of this subsection, $\Gamma$ is a totally ordered additive monoid (for instance, $\Gamma$ is a submonoid of $(\R,+)$). 

\begin{definition}\label{def:filtration}
	\index{filtration}%
	\index[notation]{M (le g)@$M_{\le g}$}%
	\index[notation]{M (approx g)@$M_{\approx g}$}%
	A \emph{$\Gamma$-filtration} on an $R$-module $M$ is a family of $R$-submodules $(M_{\le g})_{g\in\Gamma}$ of $M$ such that $M_{\le g} \subset M_{\le h}$ whenever $g\le h$ and that $\bigcup_{g\in\Gamma}M_{\le g}=M$. 
	We will use $M_{\approx g}$ to denote the set of elements in $M_{\le g}$ but not in any $M_{\le h}$ with $h<g$. 

	\index{grading!filtration induced by}%
	A $\Gamma$-grading on an $R$-module $M$ induces a $\Gamma$-filtration as follows: 
	\begin{flalign*}
		&& M_{\le g} &:= \bigoplus_{h\le g}M_{h}. & \mathllap{(g\in\Gamma)}
	\end{flalign*}
	Then we say an element $m\in M$ is \emph{of grade $g$} if $m\in M_{\approx g}$. 
	If an element is of grade $g$, then its homogeneous component of grade $g$ will be called its \emph{leading term}.
\end{definition}

Note that a nonzero homogeneous element of grade $g$ is of grade $g$. 
For a general nonzero element $m\in M$, its grade is the largest $g$ such that the homogeneous component of grade $g$ of $m$ is nonzero. 

\begin{convention}
	If $M$ is $\N$-graded, we will say that $0\in M$ has grade $-1$. 
	If $M$ is $\Q$-graded, we will say that $0\in M$ has grade $-\infty$. 
\end{convention}

\begin{definition}
	\index{operator on graded module!respect the filtration strictly}%
	\index{operator on graded module!shift the grading}%
	Let $M$ be a $\Gamma$-graded module over $R$ and $h\in\Gamma$. 
	Then an operator $T$ on $M$ is said to \emph{respect the filtration strictly} if $T(M_{\approx g})\subset M_{\approx g}$ for all $g\in\Gamma$ and \emph{shift the grading by $h$} if $T(M_{\approx g})\subset M_{\approx g+h}$ for all $g\in\Gamma$. 
\end{definition}

\begin{lemma}\label{lem:grading_shifting}
	Suppose $\Gamma$ is a totally ordered group and $h\in\Gamma$. 
	Let $M$ be a $\Gamma$-graded module over $R$ and $T$ a surjective $R$-linear operator $T$ on $M$ shifting the grading by $h$. 
	Then any section of $T$ shifts the grading by $-h$. 
\end{lemma}
\begin{proof}
	Since $M_{\approx g}\cap M_{\approx g'}=\emptyset$ whenever $g\neq g'$, the conditions on $T$ imply that the preimage of $M_{\approx g+h}$ under $T$ is precisely $M_{\approx g}$. 
	Hence, if $S$ is a section of $T$, we have $S(M_{\approx g+h})\subset M_{\approx g}$ as expected. 
\end{proof}

\begin{remark}
	The total order on $\Gamma$ can be canonically extended to its Grothendieck group $\mcal{G}[\Gamma]$. 
	Then \cref{lem:grading_shifting} holds without assuming that $\Gamma$ is a group.
\end{remark}

\begin{corollary}\label{lem:filtration_respecting}
	Let $M$ be a $\Gamma$-graded module over $R$ and $T$ a surjective $R$-linear operator $T$ on $M$ respecting the filtration strictly. 
	Then any section of $T$ respects the filtration strictly. 
\end{corollary}

\begin{definition}
	\index{weakly graded algebra}%
	\index{weakly graded module}%
	A \emph{weakly $\Gamma$-graded algebra over $R$} is an $R$-algebra $A$ equipped with a $\Gamma$-grading such that 
	\begin{equation*}
		A_{\approx g}A_{\approx h} \subset A_{\approx g+h}
		\qquad\text{for all}\qquad
		g,h\in\Gamma.
	\end{equation*}
	
	If $A$ is a weakly $\Gamma$-graded algebra over $R$, then 
	a \emph{weakly $\Gamma$-graded (left) $A$-module} is a $\Gamma$-graded module $M$ over $R$ equipped with a (left) $A$-action such that 
	\begin{equation*}
		A_{\approx g}M_{\approx h} \subset M_{\approx g+h}
		\qquad\text{for all}\qquad
		g,h\in\Gamma. 
	\end{equation*}
\end{definition}

\begin{remark}
	If $A$ is a weakly $\Gamma$-graded algebra over $R$. 
	Then a free (left) module over $A$ is naturally a weakly $\Gamma$-graded (left) $A$-module. 
\end{remark}

\begin{example}\label{eg:polynomials}
	\index{degree of a polynomial}
	The ring of polynomials $R[z]$ over $R$ is naturally a $\N$-graded algebra, where the monomials $(z^{n})_{n\in\N}$ forms a homogeneous basis of it. 
	However, if the characteristic of $R$ is $0$, then the induced $\N$-filtration on $R[z]$, namely the \emph{degree} filtration, can also be induced from the following alternative grading: 
	\begin{equation*}
		R[z] = \bigoplus_{n\in\N}R\binom{z}{n}, 
	\end{equation*}
	where 
	\index[notation]{z choose n@$\binom{z}{n}$}%
	\begin{equation*} 
		\binom{z}{n}	:=	\frac{1}{n!}z(z-1)\cdots(z-n+1). 
	\end{equation*}
	In this paper, this grading will be called the \emph{degree}. 
	Note that this convention is different from the usual one. 
	In particular, this grading does not make $R[z]$ into a $\N$-graded algebra over $R$, only a weakly $\N$-graded algebra. 
\end{example}

\subsection{\texorpdfstring{$q$}{q}-exponential polynomials}\label{subsec:EP}
In this and the next section, we will introduce \emph{(super) $q$-exponential polynomials} and study their interaction with anti-difference operators. 
\begin{definition}
	\index{q-polynomial@$q$-polynomial}%
	\index[notation]{Qq [z]@$\mbb{Q}[q;-][z]$}%
	A \emph{$q$-polynomial} is a polynomial with $q$-number coefficients. 
	Following the usual notation, we will denote the ring of $q$-polynomials by $\mbb{Q}[q;-][z]$.
\end{definition}
\begin{definition}\label{def:EP}
	\index{q-exponential@$q$-exponential}%
	\index{q-exponential polynomial@$q$-exponential polynomial}%
	\index[notation]{Qq [z]qQz@$\mbb{Q}[q;-][z]q^{\Q z}$}%
	A \emph{$q$-exponential polynomial} is a finite formal sum 
	\begin{equation}
		\label{eq:ExponentialPolynomial}
		f(z) = \sum_{\nu}f_{\nu}(z)q^{\nu z},
	\end{equation}
	where $\nu\in\Q$ and each $f_{\nu}(z)$ is a $q$-polynomial. 
	The ring of $q$-exponential polynomials will be denoted by $\mbb{Q}[q;-][z]q^{\Q z}$. 
\end{definition}

\begin{definition}\label{def:gradeP}
	\index{q-polynomial@$q$-polynomial!degree of}%
	\index{q-polynomial@$q$-polynomial!leading coefficient of}%
	Following \cref{eg:polynomials}, we will consider the following grading on $\mbb{Q}[q;-][z]$ and call it the \emph{degree}: 
	\begin{equation*}
		\mbb{Q}[q;-][z] = \bigoplus_{n\in\N}\mbb{Q}[q;-]\binom{z}{n}.
	\end{equation*}
	That is to say, an element of grade $n$ in the sense of \cref{def:filtration} will be said to be of \emph{degree} $n$. 
	However, note that this grading only makes $\mbb{Q}[q;-][z]$ a weakly $\N$-graded algebra over $\mbb{Q}[q;-]$. 
	Let $f$ be a $q$-polynomial. 
	We will use $\fun{deg}[f]$ to denote its \emph{degree}.
	The \emph{leading coefficient} $\fun{lead}[f]$ of $f$ is defined to be the coefficient of its leading term, namely the $q$-number attached to $\binom{z}{\fun{deg}[f]}$ in $f$. 
\end{definition}


\begin{definition}\label{def:gradeEP}
	\index{q-exponential polynomial@$q$-exponential polynomial!order of}%
	\index{q-exponential polynomial@$q$-exponential polynomial!degree of}%
	\index{q-exponential polynomial@$q$-exponential polynomial!leading coefficient of}%
	\index[notation]{ord (f)@$\fun{ord}[f]$}%
	\index[notation]{deg (f)@$\fun{deg}[f]$}%
	\index[notation]{lead (f)@$\fun{lead}[f]$}%
	The following grading on $\mbb{Q}[q;-][z]q^{\Q z}$ will be called the \emph{order}:
	\begin{equation*}
		\mbb{Q}[q;-][z]q^{\Q z} = 
			\bigoplus_{\nu\in\Q}\mbb{Q}[q;-][z]q^{\nu z}.
	\end{equation*}
	That is to say, an element of grade $\nu$ in the sense of \cref{def:filtration} will be said to be of \emph{order} $\nu$. 
	Note that this grading makes $\mbb{Q}[q;-][z]q^{\Q z}$ a $\Q$-graded algebra over $\mbb{Q}[q;-][z]$. 
	We will use $\fun{ord}[f]$ to denote the \emph{order} of a $q$-exponential polynomial $f$. 

	Each homogeneous component $\mbb{Q}[q;-][z]q^{\nu z}$ is 
	a free module of rank one over the weakly $\N$-graded algebra $\mbb{Q}[q;-][z]$ and thus naturally a weakly $\N$-graded module. 
	The \emph{degree} $\fun{deg}[f]$ and the \emph{leading coefficient} $\fun{lead}[f]$ of a $q$-exponential polynomial $f$ are defined to be the degree and the leading coefficient of its leading term in the order grading. 
\end{definition}

\begin{example}
	Let $f$ be a $q$-exponential polynomial as in \cref{eq:ExponentialPolynomial}. 
	Then its order is the largest $\nu\in\Q$ such that $f_{\nu}\neq 0$, its degree and leading coefficient is the degree and the leading coefficient of the $q$-polynomial $f_{\fun{ord}[f]}$. 
\end{example}

\begin{remark}
	Every $q$-polynomial will be viewed as a $q$-exponential polynomial which is homogeneous of order $0$. 
\end{remark}

It is clear that a $q$-exponential polynomial $f$ defines a $q$-function.  
We will use the same notion to denote this $q$-function. 
Then it has the following asymptotic growth:
\begin{equation}
	\label{eq:AsymptoticOfEP}
		f(z)\sim\fun{lead}[f]\binom{z}{\fun{deg}[f]}q^{\fun{ord}[f]z}.
\end{equation}
We thus introduce the following convention.
\begin{convention}
	Let $\fun{S}$ be a $q$-function and $f$ a $q$-exponential polynomial. 
	If $\fun{S}$ can be defined by a $q$-exponential polynomial whose leading term is the same as $f(z)$, then we will say that $\fun{S}[z]$ \emph{has asymptotic growth} $f(z)$ and write  
	\begin{equation*}
		\fun{S}[z] \sim f(z)
	\end{equation*}
	by an abuse of language. 	
	Note that this implies that $\fun{S}$ is asymptotically equal to the $q$-function defined by $f$. 
\end{convention}

Now, we turn to the discrete calculus.
\begin{definition}\label{def:differenceOfEP}
	\index{q-polynomial@$q$-polynomial!difference of}%
	\index{q-exponential polynomial@$q$-exponential polynomial!difference of}%
	\index[notation]{dif (f)@$\adif{f}$}%
	The \emph{difference operator} $\adif$ on $q$-polynomials is the $\mbb{Q}[q;-]$-linear operator vanishing on constant $q$-polynomials and satisfying the following:
	\begin{flalign}\label{eq:differenceOfP}
		&&
		\adif{\binom{z}{n}}	&=	
			\binom{z}{n-1}. 
			& \mathllap{(n\ge 1)}
	\end{flalign}
	This operator extends to $q$-exponential polynomials as follows: 
	\begin{flalign}\label{eq:differenceOfEP}
		&&
		\adif[\binom{z}{n}q^{\nu z}] &:= 
			\left(
				(q^{\nu}-1)\binom{z}{n} + q^{\nu}\binom{z}{n-1}
			\right)q^{\nu z}. 
			& \mathllap{(n\ge 1,\nu\neq 0)}
	\end{flalign}
\end{definition}
It is straightforward to verify that the difference operator $\adif$ satisfies the \emph{Leibniz rule} \cref{eq:LeibnizRule}. 
For a $q$-exponential polynomial $f$, the $q$-function defined by $\adif{f}$ is precisely the \emph{difference} of the $q$-function defined by $f$.

The following lemma follows from the definition.
\begin{lemma}\label{lem:adifEP}
	The linear operator $\adif$ respects the order grading on $\mbb{Q}[q;-][z]q^{\Q z}$. 
	On each homogeneous component $\mbb{Q}[q;-][z]q^{\nu z}$ ($\nu\neq 0$), the operator $\adif$ respects the degree filtration strictly. 
	On the subalgebra $\mbb{Q}[q;-][z]$, the operator $\adif$ shifts the degree homogeneously by $-1$. 
\end{lemma}



Next, we will introduce the anti-difference operators. 
\begin{lemma}\label{lem:asumEP}
	\index[notation]{sum (f)@$\asum{f}$}%
	The linear operators $\adif$ admits a section $\asum$ such that
	\begin{lemlist}
		\item \label{item:asumEP:order} it respects the order;
		\item \label{item:asumEP:degree} on each homogeneous component of order $\nu\neq 0$, it respects the degree filtration strictly;
		\item \label{item:asumEP:degree0} on the subalgebra of order $0$, it shifts the degree homogeneously by $1$. 
	\end{lemlist}
	Moreover, if $f$ is a $q$-exponential polynomial, then we have the following formula:
	\begin{equation}\label{eq:asumEP}
		\fun{lead}[\asum{f}] = 
			\begin{dcases*}
				\left(q^{\fun{ord}[f]}-1\right)^{-1}\fun{lead}[f] & 
					if $\fun{ord}[f]\neq 0$,\\
				\fun{lead}[f] & 
					if $\fun{ord}[f] = 0$.
			\end{dcases*}
	\end{equation}
\end{lemma}
\begin{proof}
	First note that the restriction of the linear operator $\adif$ to each homogeneous component $\mbb{Q}[q;-][z]q^{\nu z}$ ($\nu\neq 0$) is bijective, while its restriction to the subalgebra $\mbb{Q}[q;-][z]$ is surjective. 
	In particular, the linear operator $\adif$ itself is surjective. 
	Since $\mbb{Q}[q;-][z]q^{\Q z}$ is a free module over $\mbb{Q}[q;-]$, 
	the statements on orders and degrees follows from \cref{lem:adifEP} by applying \cref{lem:grading_respecting,lem:filtration_respecting,lem:grading_shifting_homogeneously} to the operator $\adif$. 
	As for the leading coefficients, $\fun{ord}[f] = 0$ case follows from \cref{eq:differenceOfP}. 
	If $\fun{ord}[f] > 0$, by \cref{eq:differenceOfEP}, we have
	\begin{equation*}
		\fun{lead}[\adif{f}] = 
			\left(q^{\fun{ord}[f]}-1\right)\fun{lead}[f].
	\end{equation*}
	Replacing $f$ by $\asum{f}$, \cref{eq:asumEP} follows. 
\end{proof}

\begin{definition}\label{def:antidifferenceOfEP}
	\index{q-polynomial@$q$-polynomial!anti-difference of}%
	\index{q-exponential polynomial@$q$-exponential polynomial!anti-difference of}%
	\index[notation]{sum (f) anchor (a)@$\asum_{a}{f}$}%
	The linear operator $\asum$ in \cref{lem:asumEP} is called the \emph{free anti-difference operator}. 
	Let $a$ be an integer. 
	The \emph{anti-difference operator with anchor $a$}, denoted by $\asum_{a}$, is the linear operator $\asum-\fun{ev}_{a}\circ\asum$, where $\fun{ev}_{a}$ evaluates a $q$-exponential polynomial $f(z)$ at $z=a$. 
\end{definition}
For a $q$-exponential polynomial $f$, the $q$-function defined by  $\asum_{a}{f}$ is the \emph{anti-difference with anchor $a$} of the $q$-function defined by $f$. 

We end this subsection with discussions of primarity. 
\begin{definition}
	Let $f$ be a $q$-exponential polynomial as in \cref{eq:ExponentialPolynomial}. 
	Then $f$ is said to be \emph{primary} if its coefficients are primary $q$-numbers and $f_{\nu} = 0$ for all $\nu\notin\Z$. 
\end{definition}
Clearly, primary $q$-exponential polynomials define primary $q$-functions. 
\begin{lemma}\label{lem:asumEP:comment}
	Let $f$ be a $q$-exponential polynomial. 
	If $\fun{ord}[f] \ge 0$, then the $q$-functions defined by $\asum{f}$ for all $a\in\Z$ are asymptotically equal to the $q$-function defined by $\asum{f}$. 
	If $f$ is a primary, then so are $\adif{f}$, $\asum{f}$, and $\asum_{a}{f}$ ($a\in\Z$).
\end{lemma}
\begin{proof}
	We have $\asum_{a}{f}-\asum{f}\in\mbb{Q}[q;-]$. 
	Note that elements of $\mbb{Q}[q;-]$ have order $0$ and degree $0$, while $\asum{f}$ has nonzero order or degree by \cref{lem:asumEP}.
	Hence, we have $\asum_{a}{f}\sim\asum{f}$. 
	The last statement is evident. 
\end{proof}

\subsection{Super \texorpdfstring{$q$}{q}-exponential polynomials}\label{subsec:SEP}
\begin{definition}\label{def:parityfunction}
	\index{parity function}%
	\index{parity q-function@parity $q$-function}%
	A \emph{parity function} is a function which is defined on integers and factors through the projection $\Z\to\mathbb{F}_{2}$. 
	A parity function valued in $q$-numbers is called a \emph{parity $q$-function}.
\end{definition}


\begin{definition}
	\index{super q-polynomial@super $q$-polynomial}%
	\index{super q-exponential@super $q$-exponential}%
	\index{super q-exponential polynomial@super $q$-exponential polynomial}%
	A \emph{super $q$-polynomial} is a polynomial with coefficients in parity $q$-functions. 
	A \emph{super $q$-exponential polynomial} is a finite formal sum 
	\begin{equation}
		\label{eq:superExponentialPolynomial}
		f(z) = \sum_{\nu}f_{\nu}(z)q^{\nu z},
	\end{equation}
	where $\nu\in\Q$ and each $f_{\nu}(z)$ is a super $q$-polynomial. 
\end{definition}

To better understand the structure of the ring of super $q$-exponential polynomials, we recall the following notions.
\begin{definition}\label{def:super}
	\index{superalgebra}%
	\index{superalgebra!even part of}%
	\index{superalgebra!odd part of}%
	A \emph{superalgebra} over a commutative ring $R$ is a $\F_{2}$-graded algebra $A$ over $R$. 
	The $\F_2$-grading $A = A_{0}\oplus A_{1}$ is called the \emph{parity}. 
	In particular, the subalgebra $A_{0}$ is called the \emph{even part} and the $A_{0}$-module $A_{1}$ is called the \emph{odd part}. 
	For any element $a\in A$, its homogeneous component of parity $0$ (resp. $1$) is called its \emph{even part} (resp. \emph{odd part}). 
	If $A$ is a superalgebra, then an \emph{$A$-supermodule} is a $\F_{2}$-graded $A$-module. 
\end{definition}

\begin{convention}
	\index[notation]{$(-1)^{z}$}%
	We will use $(-1)^{z}$ to denote the parity function mapping even numbers to $1$ and odd numbers to $-1$.
\end{convention}
Then the following lemma is easy to verify.
\begin{lemma}\label{lem:parity}
	Let $e(z)$ be a parity function. 
	Then we have
	\begin{equation*}
		e(z) =	\frac{1}{2}(e(0)+e(1)) + \frac{1}{2}(e(0)-e(1))(-1)^{z}.
	\end{equation*}
\end{lemma}

By this lemma, we have the following. 
\begin{corollary}
	\index[notation]{Qq [-1]@$\mbb{Q}[q;-][(-1)^{z}]$}%
	The ring of parity $q$-functions with formal variable $z$ is precisely the $\mbb{Q}[q;-]$-algebra generated by $(-1)^{z}$ and is a superalgebra decomposed into even and odd parts as follows: 
	\begin{equation*}
		\mbb{Q}[q;-][(-1)^{z}] = 
			\mbb{Q}[q;-] \oplus \mbb{Q}[q;-](-1)^{z}. 
	\end{equation*}
\end{corollary}

\index[notation]{Qq [-1,z]@$\mbb{Q}[q;-][(-1)^{z},z]$}%
Then the ring of super $q$-polynomials can be denoted by $\mbb{Q}[q;-][(-1)^{z},z]$. 

\begin{corollary}\label{cor:superP}
	The ring of super $q$-polynomials is a superalgebra decomposed into even and odd parts as follows: 
	\begin{equation*}
		\mbb{Q}[q;-][(-1)^{z},z] = 
			\mbb{Q}[q;-][z] \oplus \mbb{Q}[q;-][z](-1)^{z}. 
	\end{equation*}
\end{corollary}

\begin{definition}
	\index{super q-polynomial@super $q$-polynomial!degree of}%
	\index{super q-polynomial@super $q$-polynomial!even degree of}%
	\index{super q-polynomial@super $q$-polynomial!odd degree of}%
	\index{super q-polynomial@super $q$-polynomial!leading coefficient of}%
	\index{super q-polynomial@super $q$-polynomial!even leading coefficient of}%
	\index{super q-polynomial@super $q$-polynomial!odd leading coefficient of}%
	The \emph{degree} grading on $\mbb{Q}[q;-][(-1)^{z},z]$ is defined similarly to that on $\mbb{Q}[q;-][z]$ in \cref{def:gradeP} and makes it a weakly $\N$-graded algebra over $\mbb{Q}[q;-][(-1)^{z}]$. 
	Since this weakly graded algebra is also a superalgebra, the degree grading induces a grading on its even part and one on its odd parts. 
	They are called the \emph{even degree} and the \emph{odd degree}. 
	Note that the even degree is precisely the degree grading on $\mbb{Q}[q;-][z]$ defined in \cref{def:gradeP}, and the odd degree is precisely the natural grading on a free module of rank one over $\mbb{Q}[q;-][z]$. 

	\index[notation]{deg (f)@$\fun{deg}[f]$}%
	\index[notation]{deg (f) 0@$\fun{deg}_{0}[f]$}%
	\index[notation]{deg (f) 1@$\fun{deg}_{1}[f]$}%
	\index[notation]{lead (f)@$\fun{lead}[f]$}%
	\index[notation]{lead (f) 0@$\fun{lead}_{0}[f]$}%
	\index[notation]{lead (f) 1@$\fun{lead}_{1}[f]$}%
	Let $f$ be a super $q$-polynomial. 
	We will use $\fun{deg}[f]$ to denote its \emph{degree}. 
	Its \emph{even degree} $\fun{deg}_{0}[f]$ (resp. \emph{odd degree} $\fun{deg}_{1}[f]$) is the even degree (resp. odd degree) of its even part (resp. odd part). 
	The \emph{leading coefficient} $\fun{lead}[f]$ of $f$ is defined to be the coefficient of its leading term, namely the parity $q$-function attached to $\binom{z}{\fun{deg}[f]}$ in $f$. 
	Its \emph{even leading coefficient} $\fun{lead}_{0}[f]$ (resp. \emph{odd leading coefficient} $\fun{lead}_{1}[f]$) is the leading coefficient of its even part (resp. odd part). 
\end{definition}

\begin{example}
	Let $f$ be a super $q$-polynomial as follows:
	\begin{equation*}
		f(z) = f_{0}(z) + f_{1}(z)(-1)^{z},
	\end{equation*}
	where $f_{0}$ and $f_{1}$ are $q$-polynomials. 
	Then the even degree of $f$ is $\fun{deg}[f_{0}]$, the odd degree of $f$ is $f$ is $\fun{deg}[f_{1}]$, and the degree of $f$ is the larger one of them. 
	If $\fun{deg}[f_{0}]>\fun{deg}[f_{1}]$, then the leading coefficient of $f$ is precisely its even leading coefficient, which is $\fun{lead}[f_{0}]$. 
	If $\fun{deg}[f_{0}]<\fun{deg}[f_{1}]$, then the leading coefficient of $f$ is precisely its odd leading coefficient multiplied by $(-1)^{z}$, which is $\fun{lead}[f_{1}](-1)^{z}$. 
	If $\fun{deg}[f_{0}]=\fun{deg}[f_{1}]$, then the leading coefficient of $f$ is the parity $q$-function $\fun{lead}[f_{0}]+\fun{lead}[f_{1}](-1)^{z}$.
\end{example}

\index[notation]{Qq [-1,z]qQz@$\mbb{Q}[q;-][(-1)^{z},z]q^{\Q z}$}%
Similarly to \cref{def:EP}, the ring of super $q$-exponential polynomials will be denoted by  $\mbb{Q}[q;-][(-1)^{z},z]q^{\Q z}$.


\begin{definition}
	\index{super q-exponential polynomial@super $q$-exponential polynomial!order of}%
	\index{super q-exponential polynomial@super $q$-exponential polynomial!degree of}%
	\index{super q-exponential polynomial@super $q$-exponential polynomial!even degree of}%
	\index{super q-exponential polynomial@super $q$-exponential polynomial!odd degree of}%
	\index{super q-exponential polynomial@super $q$-exponential polynomial!leading coefficient of}%
	\index{super q-exponential polynomial@super $q$-exponential polynomial!even leading coefficient of}%
	\index{super q-exponential polynomial@super $q$-exponential polynomial!odd leading coefficient of}%
	\index[notation]{ord (f)@$\fun{ord}[f]$}%
	\index[notation]{deg (f)@$\fun{deg}[f]$}%
	\index[notation]{deg (f) 0@$\fun{deg}_{0}[f]$}%
	\index[notation]{deg (f) 1@$\fun{deg}_{1}[f]$}%
	\index[notation]{lead (f)@$\fun{lead}[f]$}%
	\index[notation]{lead (f) 0@$\fun{lead}_{0}[f]$}%
	\index[notation]{lead (f) 1@$\fun{lead}_{1}[f]$}%
	The \emph{order} grading on $\mbb{Q}[q;-][(-1)^{z},z]q^{\Q z}$ is defined similarly to that on $\mbb{Q}[q;-][z]q^{\Q z}$ in \cref{def:gradeEP} and makes it a $\Q$-graded algebra over the superalgebra $\mbb{Q}[q;-][(-1)^{z},z]$. 	
	We will use $\fun{ord}[f]$ to denote the \emph{order} of a super $q$-exponential polynomial $f$. 

	Each homogeneous component $\mbb{Q}[q;-][(-1)^{z},z]q^{\nu z}$ is a free supermodule of rank one over the weakly $\N$-graded superalgebra $\mbb{Q}[q;-][(-1)^{z},z]$ and thus naturally a weakly $\N$-graded supermodule. 
	The \emph{degree} $\fun{deg}[f]$, the \emph{even degree} $\fun{deg}_{0}[f]$, the \emph{odd degree} $\fun{deg}_{1}[f]$, the \emph{leading coefficient} $\fun{lead}[f]$, the \emph{even leading coefficient} $\fun{lead}_{0}[f]$, and the \emph{odd leading coefficient} $\fun{lead}_{1}[f]$ of a super $q$-exponential polynomial $f$ are defined to be the degree, the even degree, the odd degree, the leading coefficient, the even leading coefficient, and the odd leading coefficient of its leading term in the order grading. 
\end{definition}

\begin{example}
	Let $f$ be a $q$-exponential polynomial as in \cref{eq:superExponentialPolynomial}. 
	Then the order of $f$ is the largest $\nu\in\Q$ such that $f_{\nu}\neq 0$ and its leading term is the product of the super $q$-polynomial $f_{\fun{ord}[f]}$ and $q^{\fun{ord}[f] z}$. 
\end{example}

\begin{remark}
	Every super $q$-polynomial will be viewed as a super $q$-exponential polynomial which is homogeneous of order $0$. 
\end{remark}

It is clear that a super $q$-exponential polynomial $f$ defines a $q$-function. 
We will use the same notion to denote this $q$-function. 
Then we have the following asymptotic equalities: 
\begin{equation}\label{eq:AsymptoticOfSEP}
	\begin{aligned}
		f(z) &\sim 
			\left(
				\fun{lead}_{0}[f] \binom{z}{\fun{deg}_{0}[f]} + 
				\fun{lead}_{1}[f] \binom{z}{\fun{deg}_{1}[f]} (-1)^{z}
			\right) q^{\fun{ord}[f]z} \\ 
		&\sim 
			\fun{lead}[f] \binom{z}{\fun{deg}[f]} q^{\fun{ord}[f]z}.
	\end{aligned}
\end{equation}
Note that the leading coefficient $\fun{lead}[f]$ is a parity $q$-function rather than a $q$-number. 
In particular, the asymptotic behaviors of $f(z)$ along even integers and odd integers are different if $\fun{deg}_{0}[f] = \fun{deg}_{1}[f]$. 
\begin{convention}
	Let $\fun{S}$ be a $q$-function and $f$ a super $q$-exponential polynomial. 
	If $\fun{S}$ can be defined by a super $q$-exponential polynomial whose leading term is the same as $f(z)$, then we will say that $\fun{S}$ \emph{has asymptotic growth} $f$ and write  
	\begin{equation*}
		\fun{S}[z] \sim f(z)
	\end{equation*}
	by an abuse of language. 	
	Note that this implies that $\fun{S}$ is asymptotically equal to the $q$-function defined by $f$. 
\end{convention}

Now, we turn to the discrete calculus. 
\begin{definition}
	\index{super q-polynomial@super $q$-polynomial!difference of}%
	\index{super q-exponential polynomial@super $q$-exponential polynomial!difference of}%
	\index[notation]{dif (f)@$\adif{f}$}%
	The \emph{difference operator} $\adif$ on super $q$-exponential polynomials is the extension of the difference operator defined in \cref{def:differenceOfEP} satisfying the Leibniz rule \cref{eq:LeibnizRule} and acts on parity $q$-functions as in \cref{def:difference}.
\end{definition}

\begin{remark}
	By \cref{lem:parity}, the action of $\adif$ on parity $q$-functions is determined by its action on $(-1)^{z}$. 
	Note that $\adif{(-1)^{z}}=-2(-1)^{z}$. 
	Hence, the action of $\adif$ on the superalgebra of super $q$-exponential polynomials respects the parity.
\end{remark}

For a super $q$-exponential polynomial $f$, the $q$-function defined by $\adif{f}$ is precisely the \emph{difference} of the $q$-function defined by $f$.

The following lemma follows from the definition.
\begin{lemma}\label{lem:adifSEP}
	The linear operator $\adif$ respects the parity and the order grading on $\mbb{Q}[q;-][(-1)^{z},z]q^{\Q z}$. 
	On each homogeneous component $\mbb{Q}[q;-][(-1)^{z},z]q^{\nu z}$ ($\nu\neq 0$), the operator $\adif$ respects the degree filtration, the even degree filtration on its even part, and the odd degree filtration on its odd parts strictly. 
	On the subalgebra $\mbb{Q}[q;-][(-1)^{z},z]$, the operator $\adif$ shifts the even degree homogeneously by $-1$ and respects the odd degree filtration strictly.
\end{lemma}



Next, we will introduce the anti-difference operators. 
\begin{lemma}\label{lem:asumSEP}
	\index[notation]{sum (f)@$\asum{f}$}%
	The linear operators $\adif$ admits a section $\asum$ such that
	\begin{lemlist}
		\item \label{item:asumSEP:order} it respects the parity and the order;
		\item \label{item:asumSEP:degree} on each homogeneous component of order $\nu\neq 0$, it respects the degree filtration, the even degree filtration on its even part, and odd degree filtration on its odd part strictly; 
		\item \label{item:asumSEP:degree0} on the subalgebra of order $0$, it shifts the even degree homogeneously by $1$ and respects the odd degree filtration strictly.
	\end{lemlist}
	Moreover, if $f$ is a $q$-exponential polynomial, then we have  the following formulas: 
	\begin{equation}\label{eq:asumSEP}
		\begin{aligned}
			\fun{lead}_{0}[\asum{f}] &= 
				\begin{dcases*}
					\left(q^{\fun{ord}[f]}-1\right)^{-1}\fun{lead}_{0}[f] 
						& if $\fun{ord}[f]\neq 0$,\\
					\fun{lead}_{0}[f] 
						& if $\fun{ord}[f] = 0$,
				\end{dcases*}
				\\
			\fun{lead}_{1}[\asum{f}] &= 
				-\left(q^{\fun{ord}[f]}+1\right)^{-1}\fun{lead}_{1}[f].
		\end{aligned}
	\end{equation}
\end{lemma}
\begin{proof}
	First note that the restriction of the linear operator $\adif$ to each homogeneous component $\mbb{Q}[q;-][(-1)^{z},z]q^{\nu z}$ ($\nu\neq 0$) is bijective, while its restriction to the subalgebra $\mbb{Q}[q;-][(-1)^{z},z]$ is surjective. 
	In particular, the linear operator $\adif$ itself is surjective. Since $\mbb{Q}[q;-][(-1)^{z},z]q^{\Q z}$ is a free module over $\mbb{Q}[q;-]$, the statements on orders and degrees follows from \cref{lem:adifSEP} by applying \cref{lem:grading_respecting,lem:filtration_respecting,lem:grading_shifting_homogeneously} to the operator $\adif$. 
	The statements on even leading coefficients follows from \cref{lem:asumEP}. 
	As for the odd leading coefficients, 
	first note that for each $\nu\in\Q$, we have 
	\begin{equation*}
		\adif[(-1)^{z}\binom{z}{n}q^{\nu z}] = 
			- \left(q^{\nu}+1\right)(-1)^{z}\binom{z}{n}q^{\nu z} 
			- q^{\nu}(-1)^{z}\binom{z}{n-1}q^{\nu z}.
	\end{equation*}
	Therefore, we have  
	\begin{equation*}
		\fun{lead}_{1}[\adif{f}] = 
			-\left(q^{\fun{ord}[f]}+1\right)\fun{lead}_{1}[f].
	\end{equation*}
	Replacing $f$ by $\asum{f}$, \cref{eq:asumEP} follows. 
\end{proof}

\begin{definition}
	\index{super q-polynomial@super $q$-polynomial!anti-difference of}%
	\index{super q-exponential polynomial@super $q$-exponential polynomial!anti-difference of}%
	\index[notation]{sum (f) anchor (a)@$\asum_{a}{f}$}%
	The linear operator $\asum$ in \cref{lem:asumSEP} is called the \emph{free anti-difference operator}. 
	Let $a$ be an integer. 
	The \emph{anti-difference operator with anchor $a$}, denoted by $\asum_{a}$, is the linear operator $\asum-\fun{ev}_{a}\circ\asum$, where $\fun{ev}_{a}$ evaluates a super $q$-exponential polynomial $f(z)$ at $z=a$. 
\end{definition}
For a super $q$-exponential polynomial $f$, the $q$-function defined by $\asum_{a}{f}$ is the \emph{anti-difference with anchor $a$} of the $q$-function defined by $f$.

We end this subsection with discussions of primarity. 
\begin{definition}
	Let $f$ be a super $q$-exponential polynomial as in \cref{eq:superExponentialPolynomial}. 
	Then $f$ is said to be \emph{primary} if its coefficients are primary $q$-numbers and $f_{\nu} = 0$ for all $\nu\notin\Z$. 
\end{definition}
Clearly, primary super $q$-exponential polynomials define primary $q$-functions. 
\begin{lemma}\label{lem:asumSEP:comment}
	Let $f$ be a $q$-exponential polynomial. 
	If $\fun{ord}[f] > 0$ or $\fun{ord}[f] = 0$ with either $\fun{deg}_{0}[f] \ge 0$ or $\fun{deg}_{1}[f] > 0$, then the $q$-functions defined by $\asum{f}$ for all $a\in\Z$ are asymptotically equal to the $q$-function defined by $\asum{f}$. 
	If $f$ is primary, then so are $\adif{f}$, $\asum{f}$, and $\asum_{a}{f}$ ($a\in\Z$).
\end{lemma}
\begin{proof}
	We have $\asum_{a}{f}-\asum{f}\in\mbb{Q}[q;-]$. 
	Elements of $\mbb{Q}[q;-]$ have order $0$ and even degree $0$, while the assumption on $f$ implies that $\asum{f}$ has nonzero order or degree by \cref{lem:asumSEP}. 
	Hence, we have $\asum_{a}{f}\sim\asum{f}$. 
	The last statement is evident. 
\end{proof}

\subsection{Asymptotic growth of multi-summations}\label{subsec:AGMS}
To analyze the growth of the $q$-functions $\fun{S}_{\mcal{X}[I]}[r]$ and $\fun{S}^{\asymp}_{\mcal{X}[I]}[r]$, we need to write them as $q$-exponential polynomials. 
This can be done by considering multi-summations of homogeneous (super) $q$-exponential polynomials. 
In this subsection, we give some general results. 

The strategy is: we will inductively construct a sequence of (super) $q$-exponential polynomials where the final one defines the desired $q$-function, and then we will compare the orders, the degrees, and the leading coefficients of them. 

To better describe the results, let's introduce some conventions. 
\begin{convention}
	\index[notation]{index@$\mfrak{i}$}%
	\index[notation]{X index@$X^{\mfrak{i}}$}%
	\index[notation]{one@$\vect{1}$}%
	\index[notation]{mu dot c@$\bm{\mu}\cdot\vect{c}$}%
	Let $\mfrak{i}$ be an index set. 
	The set of functions from $\mfrak{i}$ to another set $X$ will be denoted by $X^{\mfrak{i}}$. 
	Such a function $\vect{c}$ will be identified with a sequence $(c_i)_{i\in\mfrak{i}}$ indexed by $\mfrak{i}$, where $c_i=\vect{c}(i)$. 
	The constant sequence mapping all $i\in\mfrak{i}$ to $1$ will be denoted by $\vect{1}$. 
	If $\bm{\mu}$ and $\vect{c}$ are two sequences of real numbers indexed by $\mfrak{i}$, then $\bm{\mu}\cdot\vect{c}$ denotes their dot product, namely $\sum_{i\in\mfrak{i}}\mu_ic_i$. 
\end{convention}

\begin{lemma}\label{lem:MultiSum}
	Let $\fun{S}$ be the $q$-function defined by the following multi-summation
	\begin{equation*}
		\fun{S}[z] = 
		\sum_{
			\vect{c}\in\Z_{>0}^{\mfrak{i}}\colon
			\vect{1}\cdot\vect{c} = z
		}q^{\bm{\mu}\cdot\vect{c}},
	\end{equation*}
	where $\bm{\mu}$ is a sequence of non-negative rational numbers. 
	Define the following notations: 
	\begin{itemize}
		\index[notation]{mu max@$\mu_{\max}$}%
		\index[notation]{index max@$\mfrak{i}_{\max}$}%
		\item $\mu_{\max}$ is the maximum of $\bm{\mu}$;
		\item $\mfrak{i}_{\max}$ is the set of indices $i\in\mfrak{i}$ such that $\mu_i = \mu_{\max}$.
	\end{itemize}
	Then $\fun{S}$ can be defined by a $q$-exponential polynomial so that 
	\begin{equation*}
		\fun{S}[z] \sim 
		\prod_{	i\notin \mfrak{i}_{\max}	}
			\left(	q^{\mu_{\max}-\mu_{i}}-1	\right)^{-1}
		\cdot\binom{z}{\abs{\mfrak{i}_{\max}}-1} q^{\mu_{\max} z}.
	\end{equation*}
	Moreover, if $\bm{\mu}$ takes integral values, then the $q$-exponential polynomial is primary. 
\end{lemma}
\begin{remark}
	Note that the $q$-function $\fun{S}$ is eventually positive since the leading coefficient of the $q$-exponential polynomial defining it is positive.
\end{remark}
\begin{proof}
	First note that the condition on the sequence $\vect{c}$ of variables is stable under reindexing. 
	Hence, we may assume $\mfrak{i}=\Set*{1,\cdots,t}$ and $\mfrak{i}_{\max} = \Set*{1,\cdots,i_0}$ by reindexing the sequence $\bm{\mu}$ if necessary. 
	We change the variables from $\vect{c}$ to $\vect{b}$ as follows:
	\begin{flalign*}
		&&	b_i &:= c_1 + \cdots + c_i. & \mathllap{(1\le i\le t)}
	\end{flalign*} 
	Then we can write $\fun{S}[z]$ as follows:
	\begin{equation*}
		\fun{S}[z] = 
			q^{\nu_t z}\sum_{b_{t-1}=t-1}^{z-1}q^{\nu_{t-1} b_{t-1}}\cdots\sum_{b_1=1}^{b_2-1}q^{\nu_1 b_1},
	\end{equation*}
	where $\nu_i = \mu_i - \mu_{i+1}$ for $1 \le i \le t-1$ and $\nu_t = \mu_t$. 

	To analyze the growth of $\fun{S}[z]$, we define $f_1,\cdots,f_t$ inductively as follows:
	\begin{flalign*}
		&&
		f_1(z) &= q^{\nu_1 z},\\
		&&
		f_i(z) &= q^{\nu_i z}\asum_{i-1}{f_{i-1}}(z).
		& 
		\mathllap{(1 < i \le t)}
	\end{flalign*}
	Then each $f_i$ is a $q$-exponential polynomial, and we can analyze them by induction. 
	In particular, $f_t$ defines the $q$-function $\fun{S}$ by \cref{eq:FTC}. 
	Moreover, if $\mu_1,\cdots,\mu_t$ are integers, then every $f_i$ is primary by \cref{lem:asumEP:comment}.

	For $1< i \le i_0$, repeatedly applying \cref{item:asumEP:order}, we have
	\begin{equation*}
		\fun{ord}[f_{i}] = 
			\nu_{i} + \fun{ord}[f_{i-1}] = 
			\nu_{i} =
			\begin{dcases*}
				0 & if $i<i_0$,\\
				\mu_{i_0} - \mu_{i_0+1} & if $i=i_0$.
			\end{dcases*}
	\end{equation*}
	By \cref{item:asumEP:degree0}, we have the following recurrence relations: 
	\begin{align*}
		\fun{deg}[f_{i}]& = 
			\fun{deg}[f_{i-1}] + 1,
		&
		\fun{lead}[f_{i}] &= 
			\fun{lead}[f_{i-1}]. 
	\end{align*}
	In particular, we have $\fun{ord}[f_{i_0}] = \mu_{i_0} - \mu_{i_0+1}$, $\fun{deg}[f_{i_0}] = i_0-1$, and $\fun{lead}[f_{i_0}] = 1$. 

	For $i_0 < i \le t$, repeatedly applying \cref{item:asumEP:order}, we have
	\begin{equation*}
		\fun{ord}[f_{i}] = 
			\nu_{i} + \fun{ord}[f_{i-1}] = 
			\nu_{i} + \mu_{i_0} -\mu_{i} = 
			\begin{dcases*}
				\mu_{i_0} - \mu_{i+1} & if $i<t$,\\
				\mu_{i_0} & if $i=t$.
			\end{dcases*} 
	\end{equation*}
	In particular, they are positive. 
	Then by \cref{item:asumEP:degree}, we have the following recurrence relations: 
	\begin{align*}
		\fun{deg}[f_{i}] &= 
			\fun{deg}[f_{i-1}],
		&
		\fun{lead}[f_{i}] &= 
			\left(q^{\mu_{i_0}-\mu_{i}}-1\right)^{-1}\fun{lead}[f_{i-1}].
	\end{align*}
	In particular, we have $\fun{ord}[f_{t}]=\mu_{i_0}$, $\fun{deg}[f_{t}]=i_0-1$, and $\fun{lead}[f_{t}]$ equals the product of $\left(q^{\mu_{i_0}-\mu_{i}}-1\right)^{-1}$ for $i_0+1\le i\le t$. 
	
	Therefore, we have the following:
	\[
		\fun{S}[z] \sim 
			\prod_{i=i_0+1}^{t}
				\left(q^{\mu_{i_0}-\mu_{i}}-1\right)^{-1}
			\cdot\binom{z}{i_0-1} q^{\mu_{i_0} z}.
	\]
	This proves the lemma.
\end{proof}

In the rest of this subsection, we will consider multi-summations involving parity functions. 
We first extend \cref{def:parityfunction} to the following definition.
\begin{definition}
	\index{multivariable parity function}%
	A \emph{multivariable parity function} (indexed by $\mfrak{i}$) is a function defined on $\Z^{\mfrak{i}}$ factoring through the projection $\Z^{\mfrak{i}}\to\F_2^{\mfrak{i}}$. 
\end{definition}

\begin{convention}\label{con:F2asZ}
	\index[notation]{z bar@$\overline{z}$}%
	By an abuse of notation, we will use the same notation to denote a multivariable parity function indexed by $\mfrak{i}$ and a function on defined on $\F_2^{\mfrak{i}}$. 
	In other words, we will treat any sequence in $\F_2$ as a sequence in $\Z$ by viewing $0\in\F_2$ and $1\in\F_2$ as their standard representatives $0\in\Z$ and $1\in\Z$. 	
\end{convention}


\begin{lemma}\label{lem:MultiSumSEP}
	Let $\fun{S}$ be the $q$-function defined by the following multi-summation
	\begin{equation*}
		\fun{S}[z] = 
		\sum_{
			\vect{c}\in\Z_{>0}^{\mfrak{i}}\colon
			\vect{1}\cdot\vect{c} = z
		}q^{\bm{\mu}\cdot\vect{c}+e(\vect{c})},
	\end{equation*}
	where $\bm{\mu}$ is a sequence of non-negative rational numbers and $e$ is a multivariable parity function. 
	Define the following notations: 
	\begin{itemize}
		\index[notation]{mu max@$\mu_{\max}$}%
		\index[notation]{index max@$\mfrak{i}_{\max}$}%
		\item $\mu_{\max}$ is the maximum of $\bm{\mu}$;
		\item $\mfrak{i}_{\max}$ is the set of indices $i\in\mfrak{i}$ such that $\mu_i = \mu_{\max}$. 
	\end{itemize} 
	Then $\fun{S}$ can be defined by a super $q$-exponential polynomial so that
	\begin{equation*}
		\fun{S}[z] \sim 
			\left(	C_{\bm{\mu},e,0}+C_{\bm{\mu},e,1}(-1)^{z}	\right)
			\cdot\binom{z}{\abs{\mfrak{i}_{\max}}-1} q^{\mu_{\max} z},
	\end{equation*}
	where the constants $C_{\bm{\mu},e,0}$ and $C_{\bm{\mu},e,1}$ are defined as follows:
	\begin{align*}
		C_{\bm{\mu},e,0} &:= 
			C_{\bm{\mu}}\cdot 
			\sum_{\vect{s}\in\F_2^{\mfrak{i}}}
				q^{e(\vect{s})+(\mu_{\max}-\bm{\mu})\cdot\vect{s}},
		&
		C_{\bm{\mu},e,1} &:= 
			C_{\bm{\mu}}\cdot 
			\sum_{\vect{s}\in\F_2^{\mfrak{i}}}
				(-1)^{\vect{1}\cdot\vect{s}}
				q^{e(\vect{s})+(\mu_{\max}-\bm{\mu})\cdot\vect{s}},
	\end{align*}
	\index[notation]{mu max minus mu@$\mu_{\max}-\bm{\mu}$}%
	where $\mu_{\max}-\bm{\mu}$ denotes the sequence $(\mu_{\max}-\mu_i)_{i\in\mfrak{i}}$ and $C_{\bm{\mu}}$ the following the constant:
	\begin{equation*}
		C_{\bm{\mu}} :=
			\frac{1}{2^{\abs{\mfrak{i}_{\max}}}}
			\prod_{i\notin \mfrak{i}_{\max}}
				\left(	q^{2(\mu_{\max}-\mu_{i})}-1	\right)^{-1}.
	\end{equation*} 
	Moreover, if $\bm{\mu}$ and $e$ take integral values, then the super $q$-exponential polynomial is primary.
\end{lemma}
\begin{remark}
	Note that the even leading coefficient $C_{\bm{\mu},e,0}$ is positive and the odd leading coefficient $C_{\bm{\mu},e,1}$ satisfies $\abs*{C_{\bm{\mu},e,1}}<C_{\bm{\mu},e,0}$. 
	Hence, the $q$-function $\fun{S}$ is eventually positive. 
	Note that $C_{\bm{\mu},e,1}$ could be $0$, in which case the asymptotic growth of $\fun{S}[z]$ along even integers and odd integers coincide. 
\end{remark}

To prove \cref{lem:MultiSumSEP}, we begin with some special cases. 
\begin{lemma}\label{lem:MultiSumPP}
	Let $\fun{S}$ be the $q$-function defined by the following multi-summation
	\begin{equation*}
		\fun{S}[z] = 
		\sum_{
			\vect{c}\in\Z_{>0}^{\mfrak{i}}\colon
			\vect{1}\cdot\vect{c} = z
		}(-1)^{\vect{s}\cdot\vect{c}},
	\end{equation*}
	where $\vect{s}$ is a sequence of integers. 
	Define the following notations: 
	\begin{itemize}
		\index[notation]{index 0up@$\mfrak{i}^{0}$}%
		\index[notation]{index 1up@$\mfrak{i}^{1}$}%
		\item $\mfrak{i}^{0}$ (resp. $\mfrak{i}^{1}$) is the set of indices $i\in\mfrak{i}$ such that $s_i$ is even (resp. odd). 
	\end{itemize}
	Then $\fun{S}$ can be defined by a primary super $q$-polynomial so that 
	\begin{equation*}
		\fun{S}[z] \sim 
		\left(-\tfrac{1}{2}\right)^{\abs*{\mfrak{i}^{1}}}\binom{z}{\abs*{\mfrak{i}^{0}}-1} + 
		\left(-\tfrac{1}{2}\right)^{\abs*{\mfrak{i}^{0}}}\binom{z}{\abs*{\mfrak{i}^{1}}-1}(-1)^{z}.
	\end{equation*}
\end{lemma}
\begin{proof}
	First, if either $\vect{s}$ contains no even numbers or no odd numbers, then the statement follows from \cref{lem:MultiSum}. 
	We may assume that the sequence $\vect{s}$ contains at least one even number and one odd number. 
	Since the condition on the sequence $\vect{c}$ of variables is stable under reindexing, we may assume $\mfrak{i} = \Set*{1,\cdots,t}$ and $\mfrak{i}^{\circ} = \Set*{2,\cdots,i_0+1}$ by reindexing the sequence $\vect{s}$ if necessary. 
	We change the variables from $\vect{c}$ to $\vect{b}$ as follows:
	\begin{flalign*}
		&&	b_i &:= c_1+\cdots+c_i. & \mathllap{(1\le i\le t)}
	\end{flalign*} 
	Then we can write $\fun{S}[z]$ as follows:
	\begin{equation*}
		\fun{S}[z] = 
			(-1)^{r_t z}
			\sum_{b_{t-1}=t-1}^{z-1}
				(-1)^{r_{t-1} b_{t-1}}\cdots
				\sum_{b_1=1}^{b_2-1}
					(-1)^{r_1 b_1},
	\end{equation*}
	where $r_i = s_i - s_{i+1}$ for $1\le i\le t-1$ and $r_t = s_t$. 
	Then our assumption implies that $r_i$ is even whenever $i\notin\Set*{1,i_0+1,t}$.

	To analyze the growth of $\fun{S}[z]$, we can define $f_1,\cdots,f_t$ inductively as follows: 
	\begin{flalign*}
		&&
		f_1(z) &= (-1)^{r_1 z},\\
		&&
		f_i(z) &= (-1)^{r_i z}\asum_{i-1}{f_{i-1}}(z).
		&
		\mathllap{(1 < i \le t)}
	\end{flalign*}
	Note that $f_1$ fails the condition of \cref{lem:asumSEP:comment}. 
	But we can compute $\asum_{1}f_1$ directly: 
	\begin{equation*}
		\asum_{1}f_1(z) = -\tfrac{1}{2}-\tfrac{1}{2}(-1)^{z}. 
	\end{equation*}
	Then each $f_i$ is a primary super $q$-polynomial by \cref{lem:asumSEP:comment}. 
	Moreover, $f_t$ defines the $q$-function $\fun{S}$ by \cref{eq:FTC}. 
	
	For $i\notin\Set*{1,i_0+1,t}$, since $r_i$ is even, by \cref{item:asumSEP:degree0}, we have the following recurrence relations: 
	\begin{align*}
		\fun{deg}_{0}[f_i] &= 
			\fun{deg}_{0}[f_{i-1}] + 1, &
		\fun{deg}_{1}[f_i] &= 
			\fun{deg}_{1}[f_{i-1}], \\
		\fun{lead}_{0}[f_i] &= 
			\fun{lead}_{0}[f_{i-1}], &
		\fun{lead}_{1}[f_i] &= 
			-\tfrac{1}{2}\fun{lead}_{1}[f_{i-1}].
	\end{align*}
	On the other hand, when $r_{i}$ is odd, we have 
	\begin{align*}
		\fun{deg}_{0}[f_i] &= 
			\fun{deg}_{1}[f_{i-1}], &
		\fun{deg}_{1}[f_i] &= 
			\fun{deg}_{0}[f_{i-1}] + 1, \\
		\fun{lead}_{0}[f_i] &= 
			-\tfrac{1}{2}\fun{lead}_{1}[f_{i-1}], &
		\fun{lead}_{1}[f_i] &= 
			\fun{lead}_{0}[f_{i-1}].
	\end{align*}

	If $i_0 = t-1$, then $r_t=r_{i_0+1}=s_t$ is also even, and we have 
	\begin{align*}
		\fun{deg}_{0}[f_{t}] &= 
			\fun{deg}_{0}[f_{t-1}] + 1 = \cdots 
			= \fun{deg}_{0}[f_{2}] + t - 2 = t - 2, \\
		\fun{deg}_{1}[f_{t}] &= 
			\fun{deg}_{1}[f_{t-1}] = \cdots 
			= \fun{deg}_{1}[f_{2}] = 0, \\
		\fun{lead}_{0}[f_{t}] &= 
			\fun{lead}_{0}[f_{t-1}] = \cdots
			= \fun{lead}_{0}[f_{2}] = -\tfrac{1}{2}, \\
		\fun{lead}_{1}[f_{t}] &= 
			-\tfrac{1}{2}\fun{lead}_{1}[f_{t-1}] = \cdots
			= \left(-\tfrac{1}{2}\right)^{t-2}\fun{lead}_{1}[f_{2}] = \left(-\tfrac{1}{2}\right)^{t-1}. 
	\end{align*}
	Otherwise, both $r_t$ and $r_{i_0+1}$ are odd, and we have
	\begin{align*}
		\fun{deg}_{0}[f_{t}] &= 
			\fun{deg}_{1}[f_{t-1}] = \cdots = 
			\fun{deg}_{1}[f_{i_0+1}] \\ &= 
			\fun{deg}_{0}[f_{i_0}] + 1 = \cdots = 
			\fun{deg}_{0}[f_{2}] + i_0 - 2 + 1 = i_0 - 1, \\
		\fun{deg}_{1}[f_{t}] &= 
			\fun{deg}_{0}[f_{t-1}] + 1 = \cdots = 
			\fun{deg}_{0}[f_{i_0+1}] + t - i_0 - 2 + 1 \\ &=
			\fun{deg}_{1}[f_{i_0}] + t - i_0 - 1 = \cdots = 
			\fun{deg}_{1}[f_{2}] + t - i_0 - 1 = t - i_0 - 1, \\
		\fun{lead}_{0}[f_{t}] &= 
			-\tfrac{1}{2}\fun{lead}_{1}[f_{t-1}] = \cdots = 
			-\tfrac{1}{2}\left(-\tfrac{1}{2}\right)^{t - i_0 - 2}\fun{lead}_{1}[f_{i_0+1}] \\ &=
			\left(-\tfrac{1}{2}\right)^{t - i_0 - 1}\fun{lead}_{0}[f_{i_0}] = \cdots = 
			\left(-\tfrac{1}{2}\right)^{t - i_0 - 1}\fun{lead}_{0}[f_{2}] = \left(-\tfrac{1}{2}\right)^{t - i_0}, \\
		\fun{lead}_{1}[f_{t}] &= 
			\fun{lead}_{0}[f_{t-1}] = \cdots = 
			\fun{lead}_{0}[f_{i_0+1}] \\ &=
			-\tfrac{1}{2}\fun{lead}_{1}[f_{i_0}] = \cdots = 
			-\tfrac{1}{2}\left(-\tfrac{1}{2}\right)^{i_0-2}\fun{lead}_{1}[f_{2}] = \left(-\tfrac{1}{2}\right)^{i_0}.
	\end{align*}
	Then the lemma follows. 
\end{proof}

Next, we consider the following situation.
\begin{lemma}\label{lem:MultiSumEPP}
	Let $\fun{S}$ be the $q$-function defined by the following multi-summation
	\begin{equation*}
		\fun{S}[z] = 
		\sum_{
			\vect{c}\in\Z_{>0}^{\mfrak{i}}\colon
			\vect{1}\cdot\vect{c} = z
		}(-1)^{\vect{s}\cdot\vect{c}}q^{\bm{\mu}\cdot\vect{c}},
	\end{equation*}
	where $\vect{s}$ is a sequence of integers and $\bm{\mu}$ is a sequence of non-negative rational numbers. 
	Define the following notations: 
	\begin{itemize}
		\index[notation]{mu max@$\mu_{\max}$}%
		\index[notation]{index max@$\mfrak{i}_{\max}$}%
		\index[notation]{index max 0up@$\mfrak{i}_{\max}^{0}$}%
		\index[notation]{index max 1up@$\mfrak{i}_{\max}^{1}$}%
		\index[notation]{index 0up@$\mfrak{i}^{0}$}%
		\index[notation]{index 1up@$\mfrak{i}^{1}$}%
		\item $\mu_{\max}$ is the maximum of $\bm{\mu}$;
		\item $\mfrak{i}_{\max}$ is the set of indices $i\in\mfrak{i}$ such that $\mu_i = \mu_{\max}$;
		\item $\mfrak{i}_{\max}^{0}$ (resp. $\mfrak{i}_{\max}^{1}$) is the set of indices $i\in\mfrak{i}_{\max}$ such that $s_i$ is even (resp. odd);
		\item $\mfrak{i}^{0}$ (resp. $\mfrak{i}^{1}$) is the set of indices $i\in\mfrak{i}\setminus\mfrak{i}_{\max}$ such that $s_i$ is even (resp. odd). 
	\end{itemize}
	Then $\fun{S}$ can be defined by a super $q$-exponential polynomial so that 
	\begin{equation*}
		\fun{S}[z] \sim 
			\left(
				f_{\vect{s},\bm{\mu},0}(z) + 
				f_{\vect{s},\bm{\mu},1}(z)(-1)^{z}
			\right)q^{\mu_{\max} z},
	\end{equation*}
	where $f_{\vect{s},\bm{\mu},\square}(z)$ ($\square=0,1$) is the following $q$-polynomial: 
	\begin{equation*}
		f_{\vect{s},\bm{\mu},\square}(z) := 
		\prod_{i\notin\mfrak{i}_{\max}}
			\left(	(-1)^{\square+s_i}q^{\mu_{\max}-\mu_{i}}-1	\right)^{-1}
		\cdot 
		\left(-\tfrac{1}{2}\right)^{\abs*{\mfrak{i}_{\max}^{1-\square}}}
		\binom{z}{\abs*{\mfrak{i}_{\max}^{\square}}-1}. 
	\end{equation*}
	Moreover, if $\bm{\mu}$ takes integral values, then the super $q$-exponential polynomial is primary.
\end{lemma}
\begin{remark}
	Note that the leading coefficient may be negative. 
	However, if $\vect{s}$ contains no even numbers (resp. odd numbers), then the even (resp. odd) leading coefficient is zero and the odd (resp. even) leading coefficient is positive. 
\end{remark}
\begin{proof}
	First note that the condition on the sequence $\vect{c}$ of variables is stable under reindexing.
	Hence, we may assume $\mfrak{i} = \Set*{1,\cdots,t}$, $\mfrak{i}_{\max} = \Set*{1,\cdots,i_0}$, and $\mfrak{i}^{0} = \Set*{i_0+1,\cdots,i_1}$ by reindexing the sequences $\bm{\mu}$ and $\vect{s}$ if necessary. 
	We change the variables from $\vect{c}$ to $\vect{b}$ as follows:
	\begin{flalign*}
		&&	b_i &:= c_1+\cdots+c_i. & \mathllap{(i_0\le i\le t)}
	\end{flalign*} 
	Then we can write $\fun{S}[z]$ as follows:
	\begin{equation*}
		\fun{S}[z] = 
			(-1)^{r_t z}q^{\nu_t z}
			\sum_{b_{t-1}=t-1}^{z-1}
				(-1)^{r_{t-1} b_{t-1}}q^{\nu_{t-1} b_{t-1}}\cdots
				\sum_{b_{i_0}=i_0}^{b_{i_0+1}-1}
					(-1)^{-s_{i_0+1} b_{i_0}}q^{\nu_{i_0} b_{i_0}}
					\fun{S}_{\vect{s}|_{\mfrak{i}_{\max}}}[b_{i_0}],
	\end{equation*}
	where 
	\begin{itemize}
		\item $r_i = s_i - s_{i+1}$ for $i_0+1 \le i\le t-1$ and $r_t = s_t$; 
		\item $\nu_i = \mu_i - \mu_{i+1}$ for $i_0\le i\le t-1$ and $\nu_t = \mu_t$; 
		\item $\vect{s}|_{\mfrak{i}_{\max}}$ is the subsequence of $\vect{s}$ indexed by $\mfrak{i}_{\max}$ and the $q$-function $\fun{S}_{\vect{s}|_{\mfrak{i}_{\max}}}$ is defined as follows: 
		\begin{equation*}
			\fun{S}_{\vect{s}|_{\mfrak{i}_{\max}}}[z] := 
				\sum_{
					\vect{c}\in\Z_{>0}^{\mfrak{i}_{\max}}\colon
					\vect{1}\cdot\vect{c} = z
				}(-1)^{\vect{s}|_{\mfrak{i}_{\max}}\cdot\vect{c}}.
		\end{equation*}
	\end{itemize}
	By \cref{lem:MultiSumPP}, the $q$-function $\fun{S}_{\vect{s}|_{\mfrak{i}_{\max}}}$ can be defined by a primary super $q$-polynomial $f$, for which we have 
	\begin{align*}
		\fun{deg}_{0}[f] &= 
			\abs*{\mfrak{i}_{\max}^{0}}-1, &
		\fun{deg}_{1}[f] &= 
			\abs*{\mfrak{i}_{\max}^{1}}-1, \\
		\fun{lead}_{0}[f] &= 
			\left(-\tfrac{1}{2}\right)^{\abs*{\mfrak{i}_{\max}^{1}}}, &
		\fun{lead}_{1}[f] &= 
			\left(-\tfrac{1}{2}\right)^{\abs*{\mfrak{i}_{\max}^{0}}}.
	\end{align*}

	To analyze the growth of $\fun{S}[z]$, we can define $f_{i_0},\cdots,f_t$ inductively as follows: 
	\begin{flalign*}
		&&
		f_{i_0}(z) &= 
			(-1)^{-s_{i_0+1} z}q^{\nu_{i_0} z}f(z),\\
		&&
		f_i(z) &= 
			(-1)^{r_i z}q^{\nu_i z}\asum_{i-1}{f_{i-1}}(z).
		&
		\mathllap{(i_0 < i \le t)}
	\end{flalign*}
	Then each $f_i$ is a super $q$-exponential polynomial, and we can analyze them by induction. 
	In particular, $f_t$ defines the $q$-function $\fun{S}$ by \cref{eq:FTC}. 
	Moreover, if $\mu_1,\cdots,\mu_t$ are integers, then every $f_i$ is primary by \cref{lem:asumSEP:comment}. 
	
	For each $i_0 < i \le t$, repeatedly applying \cref{item:asumSEP:order}, we have
	\begin{equation*}
		\fun{ord}[f_{i}] = 
			\nu_{i} + \fun{ord}[f_{i-1}] = 
			\nu_{i} + \mu_{i_0} -\mu_{i} = 
			\begin{dcases*}
				\mu_{i_0} - \mu_{i+1} & if $i<t$,\\
				\mu_{i_0} & if $i=t$.
			\end{dcases*} 
	\end{equation*}
	When $i\notin\Set*{i_1,t}$, since $r_i$ is even, by \cref{item:asumSEP:degree}, we have the following recurrence relations: 
	\begin{align*}
		\fun{deg}_{0}[f_i] &= 
			\fun{deg}_{0}[f_{i-1}], &
		\fun{deg}_{1}[f_i] &= 
			\fun{deg}_{1}[f_{i-1}], \\
		\fun{lead}_{0}[f_i] &= 
			\left(q^{\mu_{i_0}-\mu_{i}}-1\right)^{-1}\fun{lead}_{0}[f_{i-1}], &
		\fun{lead}_{1}[f_i] &= 
			\left(-q^{\mu_{i_0}-\mu_{i}}-1\right)^{-1}\fun{lead}_{1}[f_{i-1}].
	\end{align*}
	On the other hand, when $r_{i}$ is odd, we have 
	\begin{align*}
		\fun{deg}_{0}[f_i] &= 
			\fun{deg}_{1}[f_{i-1}], &
		\fun{deg}_{1}[f_i] &= 
			\fun{deg}_{0}[f_{i-1}], \\
		\fun{lead}_{0}[f_i] &= 
			\left(-q^{\mu_{i_0}-\mu_{i}}-1\right)^{-1}\fun{lead}_{1}[f_{i-1}], &
		\fun{lead}_{1}[f_i] &= 
			\left(q^{\mu_{i_0}-\mu_{i}}-1\right)^{-1}\fun{lead}_{0}[f_{i-1}].
	\end{align*}

	If $i_1=t$, namely $\mfrak{i}^{1}=\emptyset$, then both $s_{i_0+1}$ and $r_t=s_t$ are even, and we have 
	\begin{align*}
		\fun{deg}_{0}[f_{t}] &= 
			\fun{deg}_{0}[f_{t-1}] = \cdots = 
			\fun{deg}_{0}[f_{i_0}] = 
			\abs*{\mfrak{i}_{\max}^{0}}-1 = t-1, \\
		\fun{deg}_{1}[f_{t}] &= 
			\fun{deg}_{1}[f_{t-1}] = \cdots = 
			\fun{deg}_{1}[f_{i_0}] = 
			\abs*{\mfrak{i}_{\max}^{1}}-1 = -1, \\
		\fun{lead}_{0}[f_{t}] &= 
			\left(q^{\mu_{i_0}-\mu_{t}}-1\right)^{-1}
			\cdot\fun{lead}_{0}[f_{t-1}] = \cdots \\ &= 
			\prod_{i=i_0+1}^{t}
				\left(q^{\mu_{i_0}-\mu_{i}}-1\right)^{-1}
			\cdot\fun{lead}_{0}[f_{i_0}] =
			\prod_{i=i_0+1}^{t}
				\left(q^{\mu_{i_0}-\mu_{i}}-1\right)^{-1}
			\cdot
				\left(-\tfrac{1}{2}\right)^{\abs*{\mfrak{i}_{\max}^{1}}}.
	\end{align*}

	If $i_1=i_0$, namely $\mfrak{i}^{0}=\emptyset$, then both $s_{i_0+1}$ and $r_t=s_t$ are odd, and we have 
	\begin{align*}
		\fun{deg}_{0}[f_{t}] &= 
			\fun{deg}_{1}[f_{t-1}] = \cdots = 
			\fun{deg}_{1}[f_{i_0}] = 
			\abs*{\mfrak{i}_{\max}^{0}}-1 = -1, \\
		\fun{deg}_{1}[f_{t}] &= 
			\fun{deg}_{0}[f_{t-1}] = \cdots = 
			\fun{deg}_{0}[f_{i_0}] = 
			\abs*{\mfrak{i}_{\max}^{1}}-1 = t-1, \\
		\fun{lead}_{1}[f_{t}] &= 
			\left(q^{\mu_{i_0}-\mu_{t}}-1\right)^{-1}
			\cdot\fun{lead}_{0}[f_{t-1}] = \cdots \\ &= 
			\prod_{i=i_0+1}^{t}
				\left(q^{\mu_{i_0}-\mu_{i}}-1\right)^{-1}
			\cdot\fun{lead}_{0}[f_{i_0}] =
			\prod_{i=i_0+1}^{t}
				\left(q^{\mu_{i_0}-\mu_{i}}-1\right)^{-1}
			\cdot
				\left(-\tfrac{1}{2}\right)^{\abs*{\mfrak{i}_{\max}^{0}}}.
	\end{align*}

	If $i_0 < i_1 < t$, then $s_{i_0+1}$ is even while $r_{i_1}$ and $r_t=s_t$ are odd. We thus have 
	\begin{align*}
		\fun{deg}_{0}[f_{t}] &= 
			\fun{deg}_{1}[f_{t-1}] = \cdots = 
			\fun{deg}_{1}[f_{i_1}] \\ &= 
			\fun{deg}_{0}[f_{i_1-1}] = \cdots = 
			\fun{deg}_{0}[f_{i_0}] = 
			\abs*{\mfrak{i}_{\max}^{0}}-1 \ge 0, \\
		\fun{deg}_{1}[f_{t}] &= 
			\fun{deg}_{0}[f_{t-1}] = \cdots = 
			\fun{deg}_{0}[f_{i_1}] \\ &= 
			\fun{deg}_{0}[f_{i_1-1}] = \cdots = 
			\fun{deg}_{0}[f_{i_0}] = 
			\abs*{\mfrak{i}_{\max}^{1}}-1 \ge 0,\\
		\fun{lead}_{0}[f_{t}] &= 
			\left(-q^{\mu_{i_0}-\mu_{t}}-1\right)^{-1}
			\cdot\fun{lead}_{1}[f_{t-1}] = \cdots = 
			\prod_{i=i_1+1}^{t}
				\left(-q^{\mu_{i_0}-\mu_{i}}-1\right)^{-1}
			\cdot\fun{lead}_{1}[f_{i_1}] \\	&= 
			\prod_{i=i_1+1}^{t}
				\left(-q^{\mu_{i_0}-\mu_{i}}-1\right)^{-1}
			\cdot
				\left(q^{\mu_{i_0}-\mu_{i_1}}-1\right)^{-1}
			\cdot\fun{lead}_{0}[f_{i_1-1}] = \cdots \\ &= 
			\prod_{i=i_1+1}^{t}
				\left(-q^{\mu_{i_0}-\mu_{i}}-1\right)^{-1}
			\cdot
			\prod_{i=i_0+1}^{i_1}
				\left(q^{\mu_{i_0}-\mu_{i}}-1\right)^{-1}
			\cdot\fun{lead}_{0}[f_{i_0}] \\ &= 
			\prod_{i=i_1+1}^{t}
				\left(-q^{\mu_{i_0}-\mu_{i}}-1\right)^{-1}
			\cdot
			\prod_{i=i_0+1}^{i_1}
				\left(q^{\mu_{i_0}-\mu_{i}}-1\right)^{-1}
			\cdot
				\left(-\tfrac{1}{2}\right)^{\abs*{\mfrak{i}_{\max}^{1}}}, \\
		\fun{lead}_{1}[f_{t}] &= 
			\left(q^{\mu_{i_0}-\mu_{t}}-1\right)^{-1}
			\cdot\fun{lead}_{0}[f_{t-1}] = \cdots = 
			\prod_{i=i_1+1}^{t}
				\left(q^{\mu_{i_0}-\mu_{i}}-1\right)^{-1}
			\cdot\fun{lead}_{0}[f_{i_1}] \\ &= 
			\prod_{i=i_1+1}^{t}
				\left(q^{\mu_{i_0}-\mu_{i}}-1\right)^{-1}
			\cdot
				\left(-q^{\mu_{i_0}-\mu_{i_1}}-1\right)^{-1}
			\cdot\fun{lead}_{1}[f_{i_1-1}] = \cdots \\ &=
			\prod_{i=i_1+1}^{t}
				\left(q^{\mu_{i_0}-\mu_{i}}-1\right)^{-1}
			\cdot
			\prod_{i=i_0+1}^{i_1}
				\left(-q^{\mu_{i_0}-\mu_{i}}-1\right)^{-1}
			\cdot\fun{lead}_{1}[f_{i_0}] \\ &=
			\prod_{i=i_1+1}^{t}
				\left(q^{\mu_{i_0}-\mu_{i}}-1\right)^{-1}
			\cdot
			\prod_{i=i_0+1}^{i_1}
				\left(-q^{\mu_{i_0}-\mu_{i}}-1\right)^{-1}
			\cdot
				\left(-\tfrac{1}{2}\right)^{\abs*{\mfrak{i}_{\max}^{0}}}.
	\end{align*}
	Then the lemma follows.
\end{proof}

To deduce \cref{lem:MultiSumSEP} from \cref{lem:MultiSumEPP}, we need the following notions. 

\begin{definition}
	\index{multivariable parity function!Fourier transform of}
	\index[notation]{e hat@$\widehat{e}$}%
	Let $e$ be a multivariable parity function indexed by $\mfrak{i}$. 
	Then its \emph{Fourier transform} $\widehat{e}$ is the following:
	\begin{equation*}
		\widehat{e}(-)	:=	
			\sum_{\vect{s}\in\F_2^{\mfrak{i}}}
			e(\vect{s})(-1)^{\vect{s}\cdot-}.
	\end{equation*}
\end{definition}

\begin{convention}\label{con:partitionOfIndex}
	\index[notation]{index 1@$\mfrak{i}_{1}$}%
	\index[notation]{index 2@$\mfrak{i}_{2}$}%
	\index[notation]{cequences, union of them@$\vect{c}_{1}\sqcup\vect{c}_{2}$}%
	It is often the case that the index set $\mfrak{i}$ admits a partition $\mfrak{i} = \mfrak{i}_{1} \sqcup \mfrak{i}_{2}$. 
	For any sequence $\vect{c}$ indexed by $\mfrak{i}$, we will use $\vect{c}_{1}$ and $\vect{c}_{2}$ to denote the subsequences of $\vect{c}$ indexed by $\mfrak{i}_{1}$ and $\mfrak{i}_{2}$ respectively. 
	Conversely, if $\vect{c}_{1}$ and $\vect{c}_{2}$ are two sequences indexed by $\mfrak{i}_{1}$ and $\mfrak{i}_{2}$ respectively, then we will use $\vect{c}_{1}\sqcup\vect{c}_{2}$ to denote the sequence indexed by $\mfrak{i}$ obtained from them.
\end{convention}

We have the following multivariable version of \cref{lem:parity}. 
\begin{lemma}\label{lem:multi_parity}
	Let $e$ be a multivariable parity function indexed by $\mfrak{i}$. 
	Then we have
	\begin{equation*}
		e(\vect{c})	=	
			\frac{1}{2^{\abs*{\mfrak{i}}}}
			\sum_{\vect{s}\in\F_2^{\mfrak{i}}}
			\widehat{e}(\vect{s})(-1)^{\vect{s}\cdot\vect{c}}.
	\end{equation*}
	Moreover, if $\mfrak{i}$ admits a partition $\mfrak{i} = \mfrak{i}_{1} \sqcup \mfrak{i}_{2}$, then we have
	\begin{equation*}
		\sum_{\vect{s}_{1}\in\F_2^{\mfrak{i}_{1}}}
			e(\vect{s}_{1}\sqcup\vect{c}_{2})(-1)^{\vect{s}_{1}\cdot\vect{c}_{1}} 
			=	
			\frac{1}{2^{\abs*{\mfrak{i}_{2}}}}
			\sum_{\vect{s}_{2}\in\F_2^{\mfrak{i}_{2}}}
				\widehat{e}(\vect{c}_{1}\sqcup\vect{s}_{2})(-1)^{\vect{s}_{2}\cdot\vect{c}_{2}}. 
	\end{equation*}
\end{lemma}
\begin{proof}
	This follows from the general theory of Fourier transforms on finite Abelian groups (see e.g. \cite{Lang}*{chap. XVIII, \S 5 and \S 6}). 
	To verify the lemma directly, note that for any index set $\mfrak{i}$ and any sequence $\vect{c}$ indexed by $\mfrak{i}$, we have 
	\begin{equation*}
		\sum_{\vect{s}\in\F_2^{\mfrak{i}}}
			(-1)^{\vect{s}\cdot\vect{c}} = 
			\begin{dcases*}
				2^{\abs*{\mfrak{i}}} & if $\vect{\overline{c}}=0$,\\
				0 & otherwise.
			\end{dcases*}
	\end{equation*}
	Then the statement follows by straightforward computations.
\end{proof}

We are now able to prove \cref{lem:MultiSumSEP}.
\begin{proof}[Proof of \cref{lem:MultiSumSEP}]
	By \cref{lem:multi_parity}, we can write $\fun{S}$ as follows:
	\begin{equation*}
		\fun{S}[z] = 
			\frac{1}{2^{\abs*{\mfrak{i}}}}
			\sum_{\vect{s}\in\F_2^{\mfrak{i}}}
				\widehat{q^{e(-)}}(\vect{s})\fun{S}_{\vect{s}}[z],
	\end{equation*}
	where each $\fun{S}_{\vect{s}}$ is defined as follows: 
	\begin{equation*}
		\fun{S}_{\vect{s}}[z] = 
			\sum_{
				\vect{c}\in\Z_{>0}^{\mfrak{i}}\colon
				\vect{1}\cdot\vect{c} = z
			}(-1)^{\vect{s}\cdot\vect{c}}q^{\bm{\mu}\cdot\vect{c}}.
	\end{equation*}

	By \cref{lem:MultiSumEPP}, each $\fun{S}_{\vect{s}}$ can be defined by a super $q$-exponential polynomial of order $\mu_{\max}$, even degree $\abs*{\mfrak{i}_{\max}^{0}}-1$, and odd degree $\abs*{\mfrak{i}_{\max}^{1}}-1$ (see there for the notations). Moreover, if $\bm{\mu}$ takes integral values, then these super $q$-exponential polynomials are primary. 

	Note that the degree of $\fun{S}_{\vect{s}}$ achieves its maximum $\abs*{\mfrak{i}_{\max}}-1$ if and only if $\vect{s}$ contains no odd numbers or no even numbers. 
	Furthermore, in that case, both the even and odd leading coefficients of $\fun{S}_{\vect{s}}$ are non-negative (indeed, one is zero and another is positive).  

	Therefore, $\fun{S}$ can be defined by a super $q$-exponential polynomial of order $\mu_{\max}$ and degree $\abs*{\mfrak{i}_{\max}}-1$, and if $\bm{\mu}$ and $e$ take integral values, then this super $q$-exponential polynomial is primary. 
	Moreover, we have 
	\begin{equation*}\label{eq:lem:MultiSumSEP:ast}\tag{$\ast$}
		\fun{S}[z] \sim 
			\frac{1}{2^{\abs*{\mfrak{i}}}}
			\fun{E}[z]\binom{z}{\abs*{\mfrak{i}_{\max}}-1}q^{\mu_{\max} z},
	\end{equation*}
	where the parity $q$-function $\fun{E}$ is given as follows:
	\begin{align*}
		\fun{E}[z] &:= 
			\sum_{\vect{s}\in\F_2^{\mfrak{i}\setminus\mfrak{i}_{\max}}}
				\widehat{q^{e(-)}}(\vect{0}\sqcup\vect{s})
					\prod_{i\notin\mfrak{i}_{\max}}
					\left(	
						(-1)^{s_i}q^{\mu_{\max}-\mu_{i}}-1	
					\right)^{-1} \\
			&\qquad +
			\sum_{\vect{s}\in\F_2^{\mfrak{i}\setminus\mfrak{i}_{\max}}}
				\widehat{q^{e(-)}}(\vect{1}\sqcup\vect{s})
					\prod_{i\notin\mfrak{i}_{\max}}
					\left(	
						(-1)^{1+s_i}q^{\mu_{\max}-\mu_{i}}-1	
					\right)^{-1}
					\cdot (-1)^{z}.
	\end{align*}

	To deduce the formula in \cref{lem:MultiSumSEP} from above one, note that
	\begin{align*}
		\MoveEqLeft
		\prod_{i\notin\mfrak{i}_{\max}}
			\left(	
				q^{2(\mu_{\max}-\mu_{i})}-1	
			\right)
		\prod_{i\notin\mfrak{i}_{\max}}
			\left(	
				(-1)^{s_i}q^{\mu_{\max}-\mu_{i}}-1	
			\right)^{-1} \\
		&= 
		\prod_{i\notin\mfrak{i}_{\max}}
			\left(	
				(-1)^{s_i}q^{\mu_{\max}-\mu_{i}}+1	
			\right)
		= 
		\sum_{\vect{s}'\in\F_2^{\mfrak{i}\setminus\mfrak{i}_{\max}}}
			(-1)^{\vect{s}\cdot\vect{s}'}
			q^{(\mu_{\max}-\bm{\mu})\cdot\vect{s}'},
	\end{align*}
	and similarly 
	\begin{equation*}
		\prod_{i\notin\mfrak{i}_{\max}}
				\left(	
					q^{2(\mu_{\max}-\mu_{i})}-1	
				\right)
			\prod_{i\notin\mfrak{i}_{\max}}
				\left(	
					(-1)^{1+s_i}q^{\mu_{\max}-\mu_{i}}-1	
				\right)^{-1} = 
		\sum_{\vect{s}'\in\F_2^{\mfrak{i}\setminus\mfrak{i}_{\max}}}
			(-1)^{(\vect{1}+\vect{s})\cdot\vect{s}'}
			q^{(\mu_{\max}-\bm{\mu})\cdot\vect{s}'}.
	\end{equation*}
	Therefore, we have 
	\begin{align*}
		\MoveEqLeft
		\prod_{i\notin\mfrak{i}_{\max}}
		\left(	
			q^{2(\mu_{\max}-\mu_{i})}-1	
		\right) 
		\cdot
		\fun{E}[z]\\
		&= 
		\sum_{\vect{s}\in\F_2^{\mfrak{i}\setminus\mfrak{i}_{\max}}}
			\widehat{q^{e(-)}}(\vect{0}\sqcup\vect{s})
			\sum_{\vect{s}'\in\F_2^{\mfrak{i}\setminus\mfrak{i}_{\max}}}
				(-1)^{\vect{s}\cdot\vect{s}'}
				q^{(\mu_{\max}-\bm{\mu})\cdot\vect{s}'} \\
		&\qquad +
		\sum_{\vect{s}\in\F_2^{\mfrak{i}\setminus\mfrak{i}_{\max}}}
			\widehat{q^{e(-)}}(\vect{1}\sqcup\vect{s})
			\sum_{\vect{s}'\in\F_2^{\mfrak{i}\setminus\mfrak{i}_{\max}}}
				(-1)^{(\vect{1}+\vect{s})\cdot\vect{s}'}
				q^{(\mu_{\max}-\bm{\mu})\cdot\vect{s}'}
				\cdot (-1)^{z} \\
		&= 
		\sum_{\vect{s}'\in\F_2^{\mfrak{i}\setminus\mfrak{i}_{\max}}}
		\left(
			\sum_{\vect{s}\in\F_2^{\mfrak{i}\setminus\mfrak{i}_{\max}}}
			\widehat{q^{e(-)}}(\vect{0}\sqcup\vect{s})
				(-1)^{\vect{s}\cdot\vect{s}'}
		\right)
				q^{(\mu_{\max}-\bm{\mu})\cdot\vect{s}'} \\
		&\qquad +
		\sum_{\vect{s}'\in\F_2^{\mfrak{i}\setminus\mfrak{i}_{\max}}}
		\left(
			\sum_{\vect{s}\in\F_2^{\mfrak{i}\setminus\mfrak{i}_{\max}}}
			\widehat{q^{e(-)}}(\vect{1}\sqcup\vect{s})
				(-1)^{\vect{s}\cdot\vect{s}'}
		\right)
				q^{(\mu_{\max}-\bm{\mu})\cdot\vect{s}'}
				\cdot (-1)^{z+\vect{1}\cdot\vect{s}'}.
	\end{align*}
	By \cref{lem:multi_parity}, we have 
	\begin{align*}
		\MoveEqLeft
		\frac{1}{2^{\abs*{\mfrak{i}\setminus\mfrak{i}_{\max}}}}
		\cdot
		\prod_{i\notin\mfrak{i}_{\max}}
		\left(	
			q^{2(\mu_{\max}-\mu_{i})}-1	
		\right) 
		\cdot
		\fun{E}[z]\\
		&= 
		\sum_{\vect{s}'\in\F_2^{\mfrak{i}\setminus\mfrak{i}_{\max}}}
		\left(
			\sum_{\vect{s}\in\F_2^{\mfrak{i}_{\max}}}
				q^{e(\vect{s}\sqcup\vect{s}')}
				(-1)^{\vect{s}\cdot\vect{0}}
		\right)
				q^{(\mu_{\max}-\bm{\mu})\cdot\vect{s}'} \\
		&\qquad +
		\sum_{\vect{s}'\in\F_2^{\mfrak{i}\setminus\mfrak{i}_{\max}}}
		\left(
			\sum_{\vect{s}\in\F_2^{\mfrak{i}\setminus\mfrak{i}_{\max}}}
				q^{e(\vect{s}\sqcup\vect{s}')}
				(-1)^{\vect{s}\cdot\vect{1}}
		\right)
				q^{(\mu_{\max}-\bm{\mu})\cdot\vect{s}'}
				\cdot (-1)^{z+\vect{1}\cdot\vect{s}'} \\
		&= 
		\sum_{\vect{s}\in\F_2^{\mfrak{i}}}
			q^{e(\vect{s})+(\mu_{\max}-\bm{\mu})\cdot\vect{s}}
			\left(
				1 + 
				(-1)^{z+\vect{1}\cdot\vect{s}}
			\right).
	\end{align*}
	Apply this to \cref{eq:lem:MultiSumSEP:ast}, then \cref{lem:MultiSumSEP} follows. 
\end{proof}

\subsection{Asymptotic growth of non-balanced multi-summations}\label{subsec:AGMS2}
This subsection aims to apply the results in \cref{subsec:AGMS} to get asymptotic growth of multi-summations which are \emph{non-balanced} in the sense that the summation condition of variables is no longer $\vect{1}\cdot\vect{c}=z$. 
In our applications, the coefficients in the summation condition can only be either $1$ or $2$, see \cref{step:IndexBsrciAn,step:IndexBsrciCn,step:IndexBsrciBn,step:IndexBsrciDn}. 
Hence, we will assume that the index set $\mfrak{i}$ admits a partition $\mfrak{i} = \mfrak{i}_{1} \sqcup \mfrak{i}_{2}$ and then follow \cref{con:partitionOfIndex}.  

\begin{lemma}\label{lem:MultiSum2}
	Let $\fun{S}$ be the $q$-function defined by the following multi-summation
	\begin{equation*}
		\fun{S}[z] = 
		\sum_{
			\vect{c}\in\Z_{>0}^{\mfrak{i}}\colon
				\vect{1}\cdot\vect{c}_{1} + 
				2(\vect{1}\cdot\vect{c}_{2}) = z
		}q^{\bm{\mu}\cdot\vect{c}},
	\end{equation*}
	where $\bm{\mu}$ is a sequence of non-negative rational numbers.
	For $\square=1,2$, 
	define the following notations: 
	\begin{itemize}
		\index[notation]{mu square max@$\mu_{\square\max}$}%
		\index[notation]{index square max@$\mfrak{i}_{\square\max}$}%
		\item $\mu_{\square\max}$ is the maximum of $\bm{\mu}_{\square}$;
		\item $\mfrak{i}_{\square\max}$ is the set of indices $i\in\mfrak{i}_{\square}$ such that $\mu_i = \mu_{\square\max}$. 
	\end{itemize} 
	Then $\fun{S}$ can be defined by a super $q$-exponential polynomial. 
	\begin{lemlist}
		\item \label{item:lem:MultiSum2:l} 
			If $2\mu_{1\max} > \mu_{2\max}$, then we have 
			\begin{equation*}
				\fun{S}[z] \sim 
					C_{\bm{\mu}}\cdot 
					\sum_{\vect{s}\in\F_2^{\mfrak{i}_{1}\setminus\mfrak{i}_{1\max}}}
						q^{(\mu_{1\max}-\bm{\mu}|_{\mfrak{i}_{1}\setminus\mfrak{i}_{1\max}})\cdot\vect{s}}
					\cdot
					\binom{z}{\abs{\mfrak{i}_{1\max}}-1} q^{\mu_{1\max} z},
			\end{equation*}
			where $\mu_{1\max}-\bm{\mu}|_{\mfrak{i}_{1}\setminus\mfrak{i}_{1\max}}$ denotes the sequence $(\mu_{1\max}-\mu_i)_{i\in\mfrak{i}_{1}\setminus\mfrak{i}_{1\max}}$ and the constant $C_{\bm{\mu}}$ is defined as follows:
			\begin{equation*}
				C_{\bm{\mu}} := 
					\prod_{	i\in \mfrak{i}_{1}\setminus\mfrak{i}_{1\max}	}
						\left(	q^{2\mu_{1\max}-2\mu_{i}}-1	\right)^{-1}
					\prod_{	i\in \mfrak{i}_{2}	}
						\left(	q^{2\mu_{1\max}-\mu_{i}}-1	\right)^{-1}.
			\end{equation*}
		\item \label{item:lem:MultiSum2:s}  
			If $2\mu_{1\max} < \mu_{2\max}$, then we have 
			\begin{equation*}
				\fun{S}[z] \sim 
					\left(	C_{\bm{\mu},0}+C_{\bm{\mu},1}(-1)^{z}	\right)
					\cdot
					\binom{z}{\abs{\mfrak{i}_{2\max}}-1} q^{\tfrac{1}{2}\mu_{2\max} z},
			\end{equation*}
			where the constants $C_{\bm{\mu},0}$ and $C_{\bm{\mu},1}$ are defined as follows:
			\begin{align*}
				C_{\bm{\mu},0} &= 
					C_{\bm{\mu}}\cdot 
					\sum_{\vect{s}\in\F_2^{\mfrak{i}_{1}}}
						q^{(\tfrac{1}{2}\mu_{2\max}-\bm{\mu}_{1})\cdot\vect{s}},
				&
				C_{\bm{\mu},1} &= 
					C_{\bm{\mu}}\cdot 
					\sum_{\vect{s}\in\F_2^{\mfrak{i}_{1}}}
						(-1)^{\vect{1}\cdot\vect{s}}
						q^{(\tfrac{1}{2}\mu_{2\max}-\bm{\mu}_{1})\cdot\vect{s}},
			\end{align*}
			where $\tfrac{1}{2}\mu_{2\max}-\bm{\mu}_{1}$ denotes the sequence $(\tfrac{1}{2}\mu_{2\max}-\mu_i)_{i\in\mfrak{i}_{1}}$ and the constant $C_{\bm{\mu}}$ is defined as follows:
			\begin{equation*}
				C_{\bm{\mu}} := 
					\frac{1}{2^{\abs{\mfrak{i}_{2\max}}}}
					\prod_{	i\in \mfrak{i}_{1}	}
						\left(	q^{\mu_{2\max}-2\mu_{i}}-1	\right)^{-1}
					\prod_{	i\in \mfrak{i}_{2}\setminus\mfrak{i}_{2\max}	}
						\left(	q^{\mu_{2\max}-\mu_{i}}-1	\right)^{-1}.
			\end{equation*}
		\item \label{item:lem:MultiSum2:e}  
			If $2\mu_{1\max} = \mu_{2\max}$, then we have 
			\begin{equation*}
				\fun{S}[z] \sim 
					C_{\bm{\mu}}\cdot 
					\sum_{\vect{s}\in\F_2^{\mfrak{i}_{1}\setminus\mfrak{i}_{1\max}}}
						q^{(\mu_{1\max}-\bm{\mu}|_{\mfrak{i}_{1}\setminus\mfrak{i}_{1\max}})\cdot\vect{s}}
					\cdot
					\binom{z}{\abs{\mfrak{i}_{1\max}}+\abs{\mfrak{i}_{2\max}}-1} q^{\mu_{1\max} z},
			\end{equation*}
			where $\mu_{1\max}-\bm{\mu}|_{\mfrak{i}_{1}\setminus\mfrak{i}_{1\max}}$ denotes the sequence $(\mu_{1\max}-\mu_i)_{i\in\mfrak{i}_{1}\setminus\mfrak{i}_{1\max}}$ and the constant $C_{\bm{\mu}}$ is defined as follows:
			\begin{equation*}
				C_{\bm{\mu}} := 
					\frac{1}{2^{\abs{\mfrak{i}_{2\max}}}}
					\prod_{	i\in \mfrak{i}_{1}\setminus\mfrak{i}_{1\max}	}
						\left(	q^{2\mu_{1\max}-2\mu_{i}}-1	\right)^{-1}
					\prod_{	i\in \mfrak{i}_{2}\setminus\mfrak{i}_{2\max}	}
						\left(	q^{2\mu_{1\max}-\mu_{i}}-1	\right)^{-1}.
			\end{equation*}
	\end{lemlist}
	Moreover, if both $\bm{\mu}_{1}$ and $\tfrac{1}{2}\bm{\mu}_{2}$ take integral values, then the super $q$-exponential polynomial is primary. 
\end{lemma}
\begin{remark}
	Note that the even leading coefficient $C_{\bm{\mu},0}$ is positive and the odd leading coefficient $C_{\bm{\mu},1}$ satisfies $\abs*{C_{\bm{\mu},1}}<C_{\bm{\mu},0}$. 
	Hence, the $q$-function $\fun{S}$ is eventually positive. 
	Note that $C_{\bm{\mu},1}$ could be $0$, in which case the asymptotic growth of $\fun{S}[z]$ along even integers and odd integers coincide. 
\end{remark}


In the proof of above lemma and its many applications, a fundamental trick is to extend the domain of a $q$-function to include non-integers. 
If this $q$-function is defined by a (super) $q$-exponential polynomial, then it is clear how to do this: simply evaluate this (super) $q$-exponential polynomial. 
On the other hand, when the $q$-function is given by a (multi-)summation, it is natural to define its value at non-integer points being zero. 
However, keep these two conventions may cause confusions especially when an asymptotic equality connecting a (multi-)summation and a (super) $q$-exponential polynomial is provided. 
Hence, we will abandon the second convention and use the following one instead.
\begin{convention}\label{con:S(1/2)}
	\index[notation]{S one half z@$\fun{S}[\tfrac{1}{2}z]$}%
	Suppose $\fun{S}$ is a $q$-function such that it can be defined by a (super) $q$-exponential polynomial $f$. 
	When we write $\fun{S}[\tfrac{1}{2}n]$, where $n\in\Z$, we actually mean the evaluation of $f(z)$ at $z=\tfrac{1}{2}n$. 
	Note that this may cause $\fun{S}$ having nonzero value at a half-integer even though $\fun{S}$ may be given by a (multi-)summation. 

	Then by \cref{lem:parity}, the $q$-function that gives $\fun{S}[\tfrac{1}{2}z]$ when $z$ is even and $0$ when $z$ is odd is the following one:
	\begin{equation*}
		\tfrac{1}{2}(1+(-1)^{z})\fun{S}[\tfrac{1}{2}z].
	\end{equation*}
\end{convention}

\begin{proof}[Proof of \cref{lem:MultiSum2}]
	By introducing new variables $\vect{s}\in\F_2^{\mfrak{i}_{}}$ and replacing $\vect{c}_{1}$ by $2\vect{c}_{1}-\vect{s}$, we can write the $q$-function $\fun{S}$ as follows: 
	\begin{equation*}
		\fun{S}[z] = 
		\sum_{
			\crampedsubstack{
				\vect{c}\in\Z_{>0}^{\mfrak{i}},
				\vect{s}\in\F_2^{\mfrak{i}_{1}}\\
				2(\vect{1}\cdot\vect{c}) = z + \vect{1}\cdot\vect{s}
			}
		}q^{	-\bm{\mu}_{1}\cdot\vect{s}	
			+	(2\bm{\mu}_{1}\sqcup\bm{\mu}_{2})\cdot\vect{c}	}.
	\end{equation*}
	Consider the following $q$-function:
	\begin{equation*}
		\fun{S}^{\prime}[z] := 
		\sum_{
			\vect{c}\in\Z_{>0}^{\mfrak{i}}\colon
				\vect{1}\cdot\vect{c} = z
		}q^{	\bm{\mu'}\cdot\vect{c}	},
	\end{equation*}
	where $\bm{\mu'}$ is the sequence $2\bm{\mu}_{1}\sqcup\bm{\mu}_{2}$. 
	Then we have 
	\begin{equation*}
		\fun{S}[z] = 
		\sum_{
			\vect{s}\in\F_2^{\mfrak{i}_{1}}\colon 
			z+\vect{1}\cdot\vect{s}\in 2\Z
		}q^{	-\bm{\mu}_{1}\cdot\vect{s}	}
		\fun{S}^{\prime}[\tfrac{1}{2}(z+\vect{1}\cdot\vect{s})].
	\end{equation*}
	Note that the summation only takes over those sequence $\vect{s}\in\F_2^{\mfrak{i}_{1}}$ satisfying $z+\vect{1}\cdot\vect{s}\in 2\Z$. 
	Hence, following \cref{con:S(1/2)}, we have 
	\begin{equation*}\label{eq:SintoBanlace}\tag{$\ast$}
		\fun{S}[z] = 
		\sum_{
			\vect{s}\in\F_2^{\mfrak{i}_{1}}
		}
		\tfrac{1}{2}\left(1+(-1)^{z+\vect{1}\cdot\vect{s}}\right)
		q^{	-\bm{\mu}_{1}\cdot\vect{s}	}
		\fun{S}^{\prime}[\tfrac{1}{2}(z+\vect{1}\cdot\vect{s})].
	\end{equation*}

	By \cref{lem:MultiSum}, the $q$-function $\fun{S}^{\prime}$ can be defined by a $q$-exponential polynomial $f$ which is primary when $\bm{\mu'}$ only contains integers. 
	Moreover, we have 
	\begin{equation*}
		\fun{S}^{\prime}[z] \sim 
		\prod_{	i\notin \mfrak{i}^{\prime}_{\max}	}
			\left(	q^{\mu'_{\max}-\mu'_{i}}-1	\right)^{-1}
		\cdot\binom{z}{\abs{\mfrak{i}^{\prime}_{\max}}-1} q^{\mu'_{\max} z},
	\end{equation*}
	where $\mu'_{\max}$ is the maximum of the sequence $\bm{\mu'}$ and $\mfrak{i}^{\prime}_{\max}$ is the set of indices $i\in\mfrak{i}$ such that $\mu'_{i}$ achieves this maximum. 
	Applying this to \cref{eq:SintoBanlace} and noticing that 
	\begin{equation*}
		\binom{\tfrac{1}{2}(z+\vect{1}\cdot\vect{s})}{\abs{\mfrak{i}^{\prime}_{\max}}-1}
		q^{	\mu'_{\max} \tfrac{1}{2}(z+\vect{1}\cdot\vect{s})	} 
		\sim 
		\left(\tfrac{1}{2}\right)^{\abs{\mfrak{i}^{\prime}_{\max}}-1}
		q^{	\tfrac{1}{2}\mu'_{\max}\vect{1}\cdot\vect{s} }
		\binom{z}{\abs{\mfrak{i}^{\prime}_{\max}}-1}q^{	\tfrac{1}{2}\mu'_{\max} z	},
	\end{equation*}
	we obtain the following asymptotic equality:
	\begin{align*}
		\fun{S}[z] &\sim
			\tfrac{1}{2^{\abs{\mfrak{i}^{\prime}_{\max}}}}
			\prod_{	i\notin \mfrak{i}^{\prime}_{\max}	}
				\left(	q^{\mu'_{\max}-\mu'_{i}}-1	\right)^{-1}\\
		&\qquad
		\cdot
		\sum_{	\vect{s}\in\F_2^{\mfrak{i}_{1}}	}
			\left(1+(-1)^{z+\vect{1}\cdot\vect{s}}\right)
			q^{	(\tfrac{1}{2}\mu'_{\max} - \bm{\mu}_{1})\cdot\vect{s}	}
			\cdot\binom{z}{\abs{\mfrak{i}^{\prime}_{\max}}-1} 
			q^{	\tfrac{1}{2}\mu'_{\max} z	}.
	\end{align*}
	
	In the sequence $\bm{\mu'}$, the maximum $\mu'_{\max}$ is $\max\Set*{	2\mu_{1\max},	\mu_{2\max}	}$, and we have 
	\begin{equation*}
		\mfrak{i}^{\prime}_{\max} = 
		\begin{dcases*}
			\mfrak{i}_{1\max} 
				& if $2\mu_{1\max} > \mu_{2\max}$,\\
			\mfrak{i}_{2\max} 
				& if $2\mu_{1\max} < \mu_{2\max}$,\\
			\mfrak{i}_{1\max}\sqcup\mfrak{i}_{2\max} 
				& if $2\mu_{1\max} = \mu_{2\max}$.
		\end{dcases*}
	\end{equation*}
	In the first and the third case, the sequence $\tfrac{1}{2}\mu'_{\max} - \bm{\mu}_{1}$ contains a zero. 
	Hence, we have 
	\begin{equation*}
		\sum_{	\vect{s}\in\F_2^{\mfrak{i}_{1}}	}
			(-1)^{z+\vect{1}\cdot\vect{s}}
			q^{	(\tfrac{1}{2}\mu'_{\max} - \bm{\mu}_{1})\cdot\vect{s}	}
			= 0.
	\end{equation*}
	Then the asymptotic relations in the lemma follows. 
	Note that the proof of \cref{lem:MultiSum} also shows that $f(\tfrac{1}{2}z)$ is a primary $q$-exponential polynomial if $\bm{\mu'}$ only contains even integers. 
	Then the last statement follows. 
\end{proof}

The following lemma will not be used in this paper. 
It can be deduced from \cref{lem:MultiSumSEP} similarly to \cref{lem:MultiSum2}
\begin{lemma}\label{lem:MultiSumSEP2}
	Let $\fun{S}$ be the $q$-function defined by the following multi-summation
	\begin{equation*}
		\fun{S}[z] = 
		\sum_{
			\vect{c}\in\Z_{>0}^{\mfrak{i}}\colon
				\vect{1}\cdot\vect{c}_{1} + 
				2(\vect{1}\cdot\vect{c}_{2}) = z
		}q^{\bm{\mu}\cdot\vect{c}+e(\vect{c})},
	\end{equation*}
	where $\bm{\mu}$ is a sequence of non-negative rational numbers and $e$ is a multivariable parity function.
	For $\square=1,2$, 
	define the following notations: 
	\begin{itemize}
		\index[notation]{mu square max@$\mu_{\square\max}$}%
		\index[notation]{index square max@$\mfrak{i}_{\square\max}$}%
		\item $\mu_{\square\max}$ is the maximum of $\bm{\mu}_{\square}$;
		\item $\mfrak{i}_{\square\max}$ is the set of indices $i\in\mfrak{i}_{\square}$ such that $\mu_i = \mu_{\square\max}$. 
	\end{itemize} 
	Then $\fun{S}$ can be defined by a super $q$-exponential polynomial. 
	\begin{lemlist}
		\item \label{item:lem:MultiSumSEP2:l} 
			If $2\mu_{1\max} > \mu_{2\max}$, then we have 
			\begin{equation*}
				\fun{S}[z] \sim 
					\left(	C_{\bm{\mu},e,0}+C_{\bm{\mu},e,1}(-1)^{z}	\right)
					\cdot
					\binom{z}{\abs{\mfrak{i}_{1\max}}-1} q^{\mu_{1\max} z},
			\end{equation*}
			where the constants $C_{\bm{\mu},e,0}$ and $C_{\bm{\mu},e,1}$ are defined as follows:
			\begin{align*}
				C_{\bm{\mu},e,0} &= 
					C_{\bm{\mu}}\cdot 
					\sum_{
						\vect{s}_{1}\in\F_2^{\mfrak{i}_{1}\setminus\mfrak{i}_{1\max}}
					}
						q^{	2(\mu_{1\max}-\bm{\mu}|_{\mfrak{i}_{1}\setminus\mfrak{i}_{1\max}})\cdot\vect{s}_{1}	} \\
				&\qquad\qquad\cdot
					\sum_{
						\vect{s}_{0}\in\F_2^{\mfrak{i}_{1}},
						\vect{s}_{2}\in\F_2^{\mfrak{i}_{2}}
					}
						q^{	e(\vect{s}_{0}\sqcup\vect{s}_{2})	+	(\mu_{1\max}-\bm{\mu}_{1})\cdot\vect{s}_{0} + (2\mu_{1\max}-\bm{\mu}_{2})\cdot\vect{s}_{2}	},
				\\
				C_{\bm{\mu},e,1} &= 
					C_{\bm{\mu}}\cdot 
					\sum_{
						\vect{s}_{1}\in\F_2^{\mfrak{i}_{1}\setminus\mfrak{i}_{1\max}}
					}
						q^{	2(\mu_{1\max}-\bm{\mu}|_{\mfrak{i}_{1}\setminus\mfrak{i}_{1\max}})\cdot\vect{s}_{1}	} \\
				&\qquad\qquad\cdot
					\sum_{
						\vect{s}_{0}\in\F_2^{\mfrak{i}_{1}},
						\vect{s}_{2}\in\F_2^{\mfrak{i}_{2}}
					}
						(-1)^{\vect{1}\cdot\vect{s}_{0}}
						q^{	e(\vect{s}_{0}\sqcup\vect{s}_{2})	+	(\mu_{1\max}-\bm{\mu}_{1})\cdot\vect{s}_{0} + (2\mu_{1\max}-\bm{\mu}_{2})\cdot\vect{s}_{2}	},
			\end{align*}
			where $\mu_{1\max}-\bm{\mu}|_{\mfrak{i}_{1}\setminus\mfrak{i}_{1\max}}$ denotes the sequence $(\mu_{1\max}-\mu_i)_{i\in\mfrak{i}_{1}\setminus\mfrak{i}_{1\max}}$, $\mu_{1\max}-\bm{\mu}_{1}$ the sequence $(\mu_{1\max}-\mu_i)_{i\in\mfrak{i}_{1}}$, $2\mu_{1\max}-\bm{\mu}_{2}$ the sequence $(2\mu_{1\max}-\mu_i)_{i\in\mfrak{i}_{2}}$, and $C_{\bm{\mu}}$ the following constant:
			\begin{equation*}
				C_{\bm{\mu}} := 
					\frac{1}{2^{\abs{\mfrak{i}_{1\max}}}}
					\prod_{	i\in \mfrak{i}_{1}\setminus\mfrak{i}_{1\max}	}
						\left(	q^{4\mu_{1\max}-4\mu_{i}}-1	\right)^{-1}
					\prod_{	i\in \mfrak{i}_{2}	}
						\left(	q^{4\mu_{1\max}-2\mu_{i}}-1	\right)^{-1}.
			\end{equation*}
		\item \label{item:lem:MultiSumSEP2:s}  
			If $2\mu_{1\max} < \mu_{2\max}$, then we have 
			\begin{equation*}
				\fun{S}[z] \sim 
					\left(	C_{\bm{\mu},e,0}+C_{\bm{\mu},e,1}(-1)^{z}	\right)
					\cdot
					\binom{z}{\abs{\mfrak{i}_{2\max}}-1} q^{\tfrac{1}{2}\mu_{2\max} z},
			\end{equation*}
			where the constants $C_{\bm{\mu},e,0}$ and $C_{\bm{\mu},e,1}$ are defined as follows:
			\begin{align*}
				C_{\bm{\mu},e,0} &= 
					C_{\bm{\mu}}\cdot 
					\sum_{
						\vect{s}_{0},\vect{s}_{1}\in\F_2^{\mfrak{i}_{1}},
						\vect{s}_{2}\in\F_2^{\mfrak{i}_{2}}
					}
						\left(	
							1 + 
							(-1)^{
								\vect{1}\cdot\vect{s}_{0}+\vect{1}\cdot\vect{s}_{1}+\vect{1}\cdot\vect{s}_{2}
								}	
						\right)\Big(	\\
				&\qquad\qquad
						q^{	e(\vect{s}_{0}\sqcup\vect{s}_{2})	+	(\tfrac{1}{2}\mu_{2\max}-\bm{\mu}_{1})\cdot\vect{s}_{0} + 2(\tfrac{1}{2}\mu_{2\max}-\bm{\mu}_{1})\cdot\vect{s}_{1} + (\mu_{2\max}-\bm{\mu}_{2})\cdot\vect{s}_{2}	}\Big),
				\\
				C_{\bm{\mu},e,1} &= 
					C_{\bm{\mu}}\cdot 
					\sum_{
						\vect{s}_{0},\vect{s}_{1}\in\F_2^{\mfrak{i}_{1}},
						\vect{s}_{2}\in\F_2^{\mfrak{i}_{2}}
					}
						\left(	
							(-1)^{
								\vect{1}\cdot\vect{s}_{0}
								} + 
							(-1)^{
								\vect{1}\cdot\vect{s}_{1}+\vect{1}\cdot\vect{s}_{2}
								}	
						\right)\Big(	\\
				&\qquad\qquad
						q^{	e(\vect{s}_{0}\sqcup\vect{s}_{2})	+	(\tfrac{1}{2}\mu_{2\max}-\bm{\mu}_{1})\cdot\vect{s}_{0} + 2(\tfrac{1}{2}\mu_{2\max}-\bm{\mu}_{1})\cdot\vect{s}_{1} + (\mu_{2\max}-\bm{\mu}_{2})\cdot\vect{s}_{2}	}\Big),
			\end{align*}
			where $\tfrac{1}{2}\mu_{2\max}-\bm{\mu}_{1}$ denotes the sequence $(\tfrac{1}{2}\mu_{2\max}-\mu_i)_{i\in\mfrak{i}_{1}}$, $\mu_{2\max}-\bm{\mu}_{2}$ the sequence $(\mu_{2\max}-\mu_i)_{i\in\mfrak{i}_{2}}$, and $C_{\bm{\mu}}$ the following the constant: 
			\begin{equation*}
				C_{\bm{\mu}} := 
					\frac{1}{2^{2\abs{\mfrak{i}_{2\max}}}}
					\prod_{	i\in \mfrak{i}_{1}	}
						\left(	q^{2\mu_{2\max}-4\mu_{i}}-1	\right)^{-1}
					\prod_{	i\in \mfrak{i}_{2}\setminus\mfrak{i}_{2\max}	}
						\left(	q^{2\mu_{2\max}-2\mu_{i}}-1	\right)^{-1}.
			\end{equation*}
		\item \label{item:lem:MultiSumSEP2:e}  
			If $2\mu_{1\max} = \mu_{2\max}$, then we have 
			\begin{equation*}
				\fun{S}[z] \sim 
					\left(	C_{\bm{\mu},e,0}+C_{\bm{\mu},e,1}(-1)^{z}	\right)
					\cdot
					\binom{z}{\abs{\mfrak{i}_{1\max}}+\abs{\mfrak{i}_{2\max}}-1} q^{\mu_{1\max} z},
			\end{equation*}
			where the constants $C_{\bm{\mu},e,0}$ and $C_{\bm{\mu},e,1}$ are defined as follows:
			\begin{align*}
				C_{\bm{\mu},e,0} &= 
					C_{\bm{\mu}}\cdot 
					\sum_{
						\vect{s}_{1}\in\F_2^{\mfrak{i}_{1}\setminus\mfrak{i}_{1\max}}
					}
						q^{	2(\mu_{1\max}-\bm{\mu}|_{\mfrak{i}_{1}\setminus\mfrak{i}_{1\max}})\cdot\vect{s}_{1}	} \\
				&\qquad\qquad\cdot
					\sum_{
						\vect{s}_{0}\in\F_2^{\mfrak{i}_{1}},
						\vect{s}_{2}\in\F_2^{\mfrak{i}_{2}}
					}
						q^{	e(\vect{s}_{0}\sqcup\vect{s}_{2})	+	(\mu_{1\max}-\bm{\mu}_{1})\cdot\vect{s}_{0} + (\mu_{2\max}-\bm{\mu}_{2})\cdot\vect{s}_{2}	},
				\\
				C_{\bm{\mu},e,1} &= 
					C_{\bm{\mu}}\cdot 
					\sum_{
						\vect{s}_{1}\in\F_2^{\mfrak{i}_{1}\setminus\mfrak{i}_{1\max}}
					}
						q^{	2(\mu_{1\max}-\bm{\mu}|_{\mfrak{i}_{1}\setminus\mfrak{i}_{1\max}})\cdot\vect{s}_{1}	} \\
				&\qquad\qquad\cdot
					\sum_{
						\vect{s}_{0}\in\F_2^{\mfrak{i}_{1}},
						\vect{s}_{2}\in\F_2^{\mfrak{i}_{2}}
					}
						(-1)^{\vect{1}\cdot\vect{s}_{0}}
						q^{	e(\vect{s}_{0}\sqcup\vect{s}_{2})	+	(\mu_{1\max}-\bm{\mu}_{1})\cdot\vect{s}_{0} + (2\mu_{1\max}-\bm{\mu}_{2})\cdot\vect{s}_{2}	}, 
			\end{align*}
			where $\mu_{1\max}-\bm{\mu}|_{\mfrak{i}_{1}\setminus\mfrak{i}_{1\max}}$ denotes the sequence $(\mu_{1\max}-\mu_i)_{i\in\mfrak{i}_{1}\setminus\mfrak{i}_{1\max}}$, $\mu_{1\max}-\bm{\mu}_{1}$ the sequence $(\mu_{1\max}-\mu_i)_{i\in\mfrak{i}_{1}}$, $\mu_{2\max}-\bm{\mu}_{2}$ the sequence $(\mu_{2\max}-\mu_i)_{i\in\mfrak{i}_{2}}$, and $C_{\bm{\mu}}$ the following constant:
			\begin{equation*}
				C_{\bm{\mu}} := 
					\frac{1}{2^{\abs{\mfrak{i}_{1\max}}+2\abs{\mfrak{i}_{2\max}}}}
					\prod_{	i\in \mfrak{i}_{1}\setminus\mfrak{i}_{1\max}	}
						\left(	q^{4\mu_{1\max}-4\mu_{i}}-1	\right)^{-1}
					\prod_{	i\in \mfrak{i}_{2}\setminus\mfrak{i}_{2\max}	}
						\left(	q^{2\mu_{2\max}-2\mu_{i}}-1	\right)^{-1}.
			\end{equation*}
	\end{lemlist}
	Moreover, if both $\bm{\mu}_{1}$ and $\tfrac{1}{2}\bm{\mu}_{2}$ take integral values, then the super $q$-exponential polynomial is primary.  
\end{lemma}
\begin{remark}
	Note that the even leading coefficient $C_{\bm{\mu},e,0}$ is positive and the odd leading coefficient $C_{\bm{\mu},e,1}$ satisfies $\abs*{C_{\bm{\mu},e,1}}<C_{\bm{\mu},e,0}$. 
	Hence, the $q$-function $\fun{S}$ is eventually positive. 
	Note that $C_{\bm{\mu},e,1}$ could be $0$, in which case the asymptotic growth of $\fun{S}[z]$ along even integers and odd integers coincide. 
\end{remark}

\clearpage
\section{Simplicial volume in buildings of \texorpdfstring{$A_{n}$}{An} type}\label{sec:An}
In this section, we will prove the $A_{n}$ part of \cref{thm:AsymptoticDominanceOfSV,thm:AsymptoticGrowthOfSV}. 
More precisely, we will prove the following stronger theorem. 
\begin{theorem}\label{thm:Asymp:An}
	Let $\mathscr{B}$ be a Bruhat-Tits building of split classical type $A_n$ over a local field $K$ with residue cardinality $q$. 
	Then the simplicial volume $\fun{SV}[\:\cdot\:]$ and the simplicial surface area $\fun{SSA}[\:\cdot\:]$ in it can be defined by primary $q$-exponential polynomials whose leading terms are of the form: 
	\begin{align*}
		\fun{SV}[r] &\sim 
			\tilde{C}(n) \cdot r^{\varepsilon(n)}q^{	\pi(n) r	},&
		\fun{SSA}[r] &\sim 
			C(n) \cdot r^{\varepsilon(n)}q^{	\pi(n) r	},
	\end{align*}
	where $\varepsilon(n)=0$ and $\pi(n)=(\tfrac{n+1}{2})^{2}$ if $n$ is odd, while $\varepsilon(n)=1$ and $\pi(n)=\frac{n}{2}(\frac{n}{2}+1)$ if $n$ is even. 
\end{theorem}
Moreover, we will obtain explicit formulas for the constants $\tilde{C}(n)$ and $C(n)$. 

By the discussion at the beginning of \cref{sec:Asymptotic}, 
this can be done as follows. 
First, we will compute the asymptotic growth of $\fun{S}_{\mcal{V}[I]}[r]$ for each type $I\subset\Delta$. 
This allows us to find the dominant ones. 
On the other hand, by \cref{eq:SimplicialSurfaceAreaFormula}, we have 
\begin{equation}\label{eq:SSA=SumSI:An}
	\fun{SSA}[r] = 
		\sum_{I\subset\Delta}
		\frac{
			\mscr{P}_{A_n;I}[q]
		}{
			q^{\fun{deg}[\mscr{P}_{A_n;I}]}
		}\fun{S}_{\mcal{V}[I]}[r].
\end{equation}
Then we can obtain the asymptotic growths of $\fun{SSA}[r]$ and  $\fun{SV}[r]$.

\subsection{Asymptotic growth of \texorpdfstring{$\fun{S}_{\mcal{V}[I]}[r]$}{SVIr}}
Now, let $I$ be a type and follow \cref{con:type}. 
We are going to compute the asymptotic growth of $\fun{S}_{\mcal{V}[I]}[r]$. 
Since all vertices are special, we have 
\begin{equation*}	
	\fun{S}_{\mcal{V}[I]}[r] = 
		\fun{S}^{\asymp}_{\mcal{V}[I]}[r] = 
		\sum_{x\in\mcal{V}[I,r]}q^{2\rho(x)}.
\end{equation*}
Then by \cref{eq:2rho_An,eq:expressIndexSets_An}, we have 
\begin{equation*}
	\fun{S}_{\mcal{V}[I]}[r] = 
			\sum_{\crampedsubstack{
				c_i\in\Z_{>0}\\
				c_1+\cdots+c_t = r
			}}q^{	\sum\limits_{i=1}^{t}\ell_{i}(n+1-\ell_{i})c_i	}.
\end{equation*}
Now, we apply \cref{lem:MultiSum} to above summation, where the index set $\mfrak{i}$ is $\Set*{1,\cdots,t}$ and the sequence $\bm{\mu}$ is 
\begin{flalign*}
	&&	\mu_i&=\ell_{i}(n+1-\ell_{i}). & \mathllap{(1 \le i \le t)}
\end{flalign*}
Since all members of $\bm{\mu}$ are integers, $\fun{S}_{\mcal{V}[I]}$ can be defined by a primary $q$-exponential polynomial. 
Note that, by \cref{eq:SSA=SumSI:An}, this already implies that $\fun{SV}[\:\cdot\:]$ and  $\fun{SSA}[\:\cdot\:]$ can be defined by primary $q$-exponential polynomials.

The knowledge of quadratic function shows that either $\mfrak{i}_{\max}$ is a singleton $\Set*{i_0}$ or it consists of two consecutive indices $\Set*{i_0,i_0+1}$. 

If $\mfrak{i}_{\max}=\Set*{i_0}$. 
Then $\mu_{\max} = \ell_{i_0}(n+1-\ell_{i_0})$, 
$\mu_{\max} - \mu_i = (\ell_{i_0}-\ell_{i})(n+1-\ell_{i_0}-\ell_{i})$, 
and we have 
\begin{equation}
	\label{eq:AsymptoticOfSI:An:1}
	\fun{S}_{\mcal{V}[I]}[r] \sim 
	\prod_{i\neq i_0}
		\left(
			q^{(\ell_{i_0}-\ell_{i})(n+1-\ell_{i_0}-\ell_{i})}-1
		\right)^{-1}
	\cdot q^{\ell_{i_0}(n+1-\ell_{i_0}) r}.
\end{equation}
In particular, $\fun{S}_{\mcal{V}[I]}$ has order $\ell_{i_0}(n+1-\ell_{i_0})$ and degree $0$.

If $\mfrak{i}_{\max}=\Set*{i_0,i_0+1}$. 
Then $\mu_{\max} = \ell_{i_0}(n+1-\ell_{i_0})$, 
$\mu_{\max} - \mu_i = (\ell_{i_0}-\ell_{i})(n+1-\ell_{i_0}-\ell_{i})$, 
and we have 
\begin{equation}
	\label{eq:AsymptoticOfSI:An:2}
	\fun{S}_{\mcal{V}[I]}[r] \sim 
	\prod_{i\neq i_0,i_0+1}
		\left(
			q^{(\ell_{i_0}-\ell_{i})(n+1-\ell_{i_0}-\ell_{i})}-1
		\right)^{-1}
	\cdot rq^{\ell_{i_0}(n+1-\ell_{i_0}) r}.
\end{equation}
In particular, $\fun{S}_{\mcal{V}[I]}$ has order $\ell_{i_0}(n+1-\ell_{i_0})$ and degree $1$.

\subsection{Dominant types}\label{subsec:dominant:An}
\index{dominant type}%
Now, we are able to figure out for which type $I$, the $q$-function $\fun{S}_{\mcal{V}[I]}$ is dominant among its siblings. We will say that such a type is \emph{dominant}.

When $n$ is odd, we have 
\begin{equation*}
	\mu_{\max}=\ell_{i_0}(n+1-\ell_{i_0})
	\le (\frac{n+1}{2})^2.
\end{equation*}
The equality achieves exactly when $\ell_{i_0}=\frac{n+1}{2}$. 
Therefore, $\fun{S}_{\mcal{V}[I]}$ is dominant exactly when $\tfrac{n+1}{2}\notin{I}$. 
In this case, we have $\mfrak{i}_{\max}=\Set*{i_0}$ and $\ell_{i_0}=\frac{n+1}{2}$. 

When $n$ is even, we have 
\begin{equation*}
	\mu_{\max}=\ell_{i_0}(n+1-\ell_{i_0})
	\le \frac{n}{2}\left(\frac{n}{2}+1\right).
\end{equation*}
The equality achieves exactly when $\ell_{i_0}=\frac{n}{2}$ or $\frac{n}{2}+1$. 
Therefore, $\fun{S}_{\mcal{V}[I]}$ is dominant only if $\Set*{\tfrac{n}{2},\tfrac{n}{2}+1} \not\subset {I}$. 
There are three cases: 
if $\Set*{\tfrac{n}{2},\tfrac{n}{2}+1}\cap{I} = \Set*{\tfrac{n}{2}+1}$, we have $\mfrak{i}_{\max} = \Set*{i_0}$ and $\ell_{i_0}=\frac{n}{2}$; 
if $\Set*{\tfrac{n}{2},\tfrac{n}{2}+1}\cap{I} = \Set*{\tfrac{n}{2}}$, we have $\mfrak{i}_{\max} = \Set*{i_0}$ and $\ell_{i_0}=\frac{n}{2}+1$; 
if $\Set*{\tfrac{n}{2},\tfrac{n}{2}+1}\cap{I} = \emptyset$, we have $\mfrak{i}_{\max} = \Set*{i_0,i_0+1}$, $\ell_{i_0}=\frac{n}{2}$, and $\ell_{i_0+1}=\frac{n}{2}+1$. 
Among them, the last one gives the dominant type since $\fun{S}_{\mcal{V}[I]}$ has degree $1$ in that case while it has degree $0$ in the first two cases.

\subsection{Asymptotic growths of \texorpdfstring{$\fun{SSA}[r]$}{SSA} and \texorpdfstring{$\fun{SV}[r]$}{SV}}
We are now going to obtain the asymptotic growth of $\fun{SSA}[r]$.
By \cref{eq:SSA=SumSI:An}, we have 
\begin{equation*}
	\fun{SSA}[r]	\sim
		\sum_{I\text{ is dominant}}
		\frac{
			\mscr{P}_{A_n;I}[q]
		}{
			q^{\fun{deg}[\mscr{P}_{A_n;I}]}
		}\fun{S}_{\mcal{V}[I]}[r].
\end{equation*}

When $n$ is odd, by \cref{subsec:dominant:An} and \cref{eq:AsymptoticOfSI:An:1}, we see that $\fun{SSA}[\:\cdot\:]$ can be defined by a primary $q$-exponential polynomial so that
\begin{equation}\label{eq:Asymp:An:odd:SSA}
	\fun{SSA}[r] \sim
		C(n) \cdot q^{	(\frac{n+1}{2})^2 r	},
\end{equation}
where the constant $C(n)$ is defined as follows:
\begin{equation}\label{eq:Asymp:An:odd:C}
	\index[notation]{C (n)@$C(n)$}%
	C(n) := 
		\sum_{\crampedsubstack{
			I\subset\Delta \\ 
			\tfrac{n+1}{2}\notin{I}
		}}
		\frac{
			\mscr{P}_{A_n;I}[q]
		}{
			q^{\fun{deg}[\mscr{P}_{A_n;I}]}
		}
		\prod_{\crampedsubstack{
			1 \le i \le t_{I} \\
			\ell_{i}(I) \neq \tfrac{n+1}{2}
		}}
			\left(
				q^{	\left(\ell_{i}(I)-\tfrac{n+1}{2}\right)^2}-1
			\right)^{-1}.
\end{equation}
As a consequence, we see that $\fun{SV}[\:\cdot\:]$ can be defined by a primary $q$-exponential polynomial so that
\begin{equation}\label{eq:Asymp:An:odd:SV}
	\fun{SV}[r] = 
	\sum_{z=0}^{r}\fun{SSA}[z] \sim 
		\frac{
			q^{(\frac{n+1}{2})^2}
		}{
			q^{(\frac{n+1}{2})^2} - 1
		}C(n) \cdot q^{	(\frac{n+1}{2})^2 r	},
\end{equation}

When $n$ is even, by \cref{subsec:dominant:An} and \cref{eq:AsymptoticOfSI:An:2}, we see that $\fun{SSA}[\:\cdot\:]$ can be defined by a primary $q$-exponential polynomial so that
\begin{equation}\label{eq:Asymp:An:even:SSA}
	\fun{SSA}[r] \sim
		C(n) \cdot rq^{	\frac{n}{2}(\frac{n}{2}+1) r	},
\end{equation}
where the constant $C(n)$ is defined as follows:
\begin{equation}\label{eq:Asymp:An:even:C}
	\index[notation]{C (n)@$C(n)$}%
	C(n) := 
		\sum_{\crampedsubstack{
			I\subset\Delta \\ 
			\tfrac{n}{2},\tfrac{n}{2}+1\notin{I}
		}}
		\frac{
			\mscr{P}_{A_n;I}[q]
		}{
			q^{\fun{deg}[\mscr{P}_{A_n;I}]}
		}
		\prod_{\crampedsubstack{
			1 \le i \le t_{I} \\
			\ell_{i}(I) \neq \tfrac{n}{2},\tfrac{n}{2}+1
		}}
			\left(
				q^{	(\ell_{i}(I)-\frac{n}{2})(\ell_{i}(I)-\frac{n}{2}-1)	}-1
			\right)^{-1}.
\end{equation}
As a consequence, we see that $\fun{SV}[\:\cdot\:]$ can be defined by a primary $q$-exponential polynomial so that
\begin{equation}\label{eq:Asymp:An:even:SV}
	\fun{SV}[r] = 
	\sum_{z=0}^{r}\fun{SSA}[z] \sim 
		\frac{
			q^{\frac{n}{2}(\frac{n}{2}+1)}
		}{
			q^{\frac{n}{2}(\frac{n}{2}+1)} - 1
		}C(n) \cdot q^{	\frac{n}{2}(\frac{n}{2}+1) r	},
\end{equation}
Since all vertices are special, we have $\fun{SSA}_{\dagger}[r]=\fun{SSA}[r]$ and $\fun{SV}_{\dagger}[r]=\fun{SV}[r]$, where $\dagger$ denotes ``being special''.

By \cref{eq:Asymp:An:odd:SSA,eq:Asymp:An:odd:C,eq:Asymp:An:odd:SV,eq:Asymp:An:even:SSA,eq:Asymp:An:even:C,eq:Asymp:An:even:SV}, 
we have proved \cref{thm:Asymp:An}. 
Moreover, by \cref{eq:PoincareAnI}, we have the following explicit formulas for the first factor of $C(n)$:
\begin{align*}
	\mscr{P}_{A_n;I}[q]
	&= 
		\qbinom{n+1}{
			\ell_{1}(I),\ell_{2}(I)-\ell_{1}(I),\cdots,\ell_{t}(I)-\ell_{t-1}(I),n+1-\ell_{t}(I)
			}(q), \\
	q^{\fun{deg}[\mscr{P}_{A_n;I}]}
	&=
		\frac{
			q^{\binom{n+1}{2}}
		}{
			q^{\binom{\ell_{1}(I)}{2}}
			q^{\binom{\ell_{2}(I)-\ell_{1}(I)}{2}}\cdots
			q^{\binom{\ell_{t}(I)-\ell_{t-1}(I)}{2}}
			q^{\binom{n+1-\ell_{t}(I)}{2}}
		}.
\end{align*}
See \cref{eq:quantum_multinomial} for the definition of the symbol $\qbinom{\:\cdot\:}{\:\cdot\:,\cdots,\:\cdot\:}$.

\clearpage
\section{Simplicial volume in buildings of \texorpdfstring{$C_{n}$}{Cn} type}\label{sec:Cn}
In this section, we will prove the $C_{n}$ part of \cref{thm:AsymptoticDominanceOfSV,thm:AsymptoticGrowthOfSV}. 
More precisely, we will prove the following stronger theorem. 
\begin{theorem}\label{thm:Asymp:Cn}
	Let $\mathscr{B}$ be a Bruhat-Tits building of split classical type $C_n$ over a local field $K$ with residue cardinality $q$. 
	Then the simplicial volume $\fun{SV}[\:\cdot\:]$ and the simplicial surface area $\fun{SSA}[\:\cdot\:]$ in it can be defined by primary super $q$-exponential polynomials whose leading terms are of the form: 
	\begin{align*}
		\fun{SV}[r] &\sim 
			\tilde{C}(n) \cdot q^{	\tfrac{n(n+1)}{2} r	},&
		\fun{SSA}[r] &\sim 
			C(n) \cdot q^{	\tfrac{n(n+1)}{2} r	},
	\end{align*}
	where $\tilde{C}(n)$ and $C(n)$ are primary $q$-numbers, not just parity $q$-functions.
\end{theorem}
Moreover, we will obtain explicit formulas for the constants $\tilde{C}(n)$ and $C(n)$. 

By the discussion at the beginning of \cref{sec:Asymptotic}, 
this can be done as follows. 
First, we estimate the asymptotic growth of $\fun{S}_{\mcal{V}[I]}[r]$ for each type $I\subset\Delta$ using the auxiliary function $\fun{S}^{\asymp}_{\mcal{V}[I]}$ in \cref{subsec:Cn:AsymptoticSasympVIr}. 
This allows us to find the dominant ones in \cref{subsec:dominant:Cn}. 
Then we can compute the leading coefficient of $\fun{S}_{\mcal{V}[I]}[r]$ for dominant ones in \cref{subsec:Cn:AsymptoticSVIr}. 
Finally, by \cref{eq:SimplicialSurfaceAreaFormula}, we have 
\begin{equation}\label{eq:SSA=SumSI:Cn}
	\fun{SSA}[r] = 
		\sum_{I\subset\Delta}
		\frac{
			\mscr{P}_{C_n;I}[q]
		}{
			q^{\fun{deg}[\mscr{P}_{C_n;I}]}
		}\fun{S}_{\mcal{V}[I]}[r]
		\sim 
		\sum_{I\text{ is dominant}}
		\frac{
			\mscr{P}_{C_n;I}[q]
		}{
			q^{\fun{deg}[\mscr{P}_{C_n;I}]}
		}\fun{S}_{\mcal{V}[I]}[r].
\end{equation}
Then we can obtain the asymptotic growths of $\fun{SSA}[r]$ and $\fun{SV}[r]$. 

Along the discussion, we will also consider the asymptotic growths of $\fun{SSA}_{\dagger}[r]$ and $\fun{SV}_{\dagger}[r]$, where $\dagger$ denotes ``being special''. Namely, we will also prove the following theorem (in \cref{subsec:Cn:AsymptoticSVdaggerIr,subsec:Cn:AsymptoticSSAdagger}).
\begin{theorem}\label{thm:Asymp:CnSp}
	Let $\mathscr{B}$ be a Bruhat-Tits building of split classical type $C_n$ over a local field $K$ with residue cardinality $q$. 
	Then the special simplicial volume $\fun{SV}_{\dagger}[\:\cdot\:]$ and the special simplicial surface area $\fun{SSA}_{\dagger}[\:\cdot\:]$ in it can be defined by primary super $q$-exponential polynomials whose leading terms are of the form: 
	\begin{align*}
		\fun{SV}_{\dagger}[r] &\sim 
			\tilde{C}_{\dagger}(n) \cdot q^{	\tfrac{n(n+1)}{2} r	},&
		\fun{SSA}_{\dagger}[r] &\sim 
			C_{\dagger}(n) \cdot q^{	\tfrac{n(n+1)}{2} r	},
	\end{align*}
	where $\tilde{C}_{\dagger}(n)$ and $C_{\dagger}(n)$ are primary $q$-numbers, not just parity $q$-functions.
\end{theorem}
Moreover, we will obtain explicit formulas for the constants $\tilde{C}_{\dagger}(n)$ and $C_{\dagger}(n)$.

\subsection{Asymptotic growth of \texorpdfstring{$\fun{S}^{\asymp}_{\mcal{V}[I]}[r]$}{SasympVIr}}\label{subsec:Cn:AsymptoticSasympVIr}
Now, let $I$ be a type and follow \cref{con:type}. 
We are going to estimate the asymptotic growth of $\fun{S}_{\mcal{V}[I]}[r]$ up to the leading coefficient. 

By \cref{eq:2rho_Cn,eq:expressIndexSets_Cn}, we have 
\begin{equation*}
	\fun{S}^{\asymp}_{\mcal{V}[I]}[r] = 
		\sum_{\crampedsubstack{
			c_i\in\Z_{>0}\\
			c_1+\cdots+c_t = r
		}}q^{	\sum\limits_{i=1}^{t}\tfrac{1}{2}\ell_{i}(2n+1-\ell_{i})c_i	}.
\end{equation*}
Now, we apply \cref{lem:MultiSum} to above summation, where the index set $\mfrak{i}$ is $\Set*{1,\cdots,t}$ and the sequence $\bm{\mu}$ is 
\begin{flalign*}
	&&	\mu_i&=\tfrac{1}{2}\ell_{i}(2n+1-\ell_{i}). & \mathllap{(1 \le i \le t)}
\end{flalign*}
Since all members of $\bm{\mu}$ are integers, $\fun{S}^{\asymp}_{\mcal{V}[I]}$ can be defined by a primary $q$-exponential polynomial. 
The knowledge of quadratic function shows that $\mfrak{i}_{\max}=\Set*{t}$ with $\mu_{\max}=\tfrac{1}{2}\ell_{t}(2n+1-\ell_{t})$. 
Then we have 
\begin{equation*}
	\fun{S}^{\asymp}_{\mcal{V}[I]}[r] \sim  
		\prod_{	i=1	}^{t-1}
			\left(	q^{\tfrac{1}{2}(\ell_{t}-\ell_{i})(2n+1-\ell_{t}-\ell_{i})} - 1	\right)^{-1}
		\cdot q^{	\tfrac{1}{2}\ell_{t}(2n+1-\ell_{t}) r	}.
\end{equation*}
Since $\fun{S}_{\mcal{V}[I]}[r]\asymp\fun{S}^{\asymp}_{\mcal{V}[I]}[r]$, we see that $\fun{S}_{\mcal{V}[I]}$ has order $\tfrac{1}{2}\ell_{t}(2n+1-\ell_{t})$ and degree $0$.

\subsection{Asymptotic growth of \texorpdfstring{$\fun{S}_{\mcal{V}_{\dagger}[I]}[r]$}{SVdaggerIr}}\label{subsec:Cn:AsymptoticSVdaggerIr}
Next, we are going to compute the asymptotic growth of $\fun{S}_{\mcal{V}_{\dagger}[I]}[r]$. 

If $\ell_{t}<n$, then by \cref{eq:expressIndexSets_speical:Sphere,eq:HighestRoot_Cn,eq:2rho_Cn}, we have
\begin{equation*}
	\fun{S}_{\mcal{V}_{\dagger}[I]}[r] = 
		\sum_{\crampedsubstack{
			c_i\in\Z_{>0}\\
			2c_1+\cdots+2c_t = r
		}}q^{	\sum\limits_{i=1}^{t}\ell_{i}(2n+1-\ell_{i})c_i	}.
\end{equation*}
Now, we apply \cref{lem:MultiSum2} to this summation, where the index set $\mfrak{i}$ is $\Set*{1,\cdots,t}$, the partition $\mfrak{i}=\mfrak{i}_{1}\sqcup\mfrak{i}_{2}$ is $\Set*{1,\cdots,t} = \emptyset\sqcup\Set*{1,\cdots,t}$, and the sequence $\bm{\mu}$ is 
\begin{flalign*}
	&&	\mu_i&=\ell_{i}(2n+1-\ell_{i}). & \mathllap{(1 \le i \le t)}
\end{flalign*}
Since all members of $\bm{\mu}$ are even integers, $\fun{S}_{\mcal{V}_{\dagger}[I]}$ can be defined by a primary super $q$-exponential polynomial. 
The knowledge of quadratic function shows that $\mfrak{i}_{\max}=\mfrak{i}_{2\max}=\Set*{t}$ with $\mu_{\max}=\mu_{2\max}=\ell_{t}(2n+1-\ell_{t})$. 
Then by \cref{item:lem:MultiSum2:s}, we have 
\begin{equation*}
	\fun{S}_{\mcal{V}_{\dagger}[I]}[r] \sim  
		\prod_{	i=1	}^{t-1}
			\left(	q^{(\ell_{t}-\ell_{i})(2n+1-\ell_{t}-\ell_{i})} - 1	\right)^{-1}
		\cdot	\tfrac{1}{2}\left(1+(-1)^{r}\right)
		\cdot q^{	\tfrac{1}{2}\ell_{t}(2n+1-\ell_{t}) r	}.
\end{equation*}
In particular, it has order $\tfrac{1}{2}\ell_{t}(2n+1-\ell_{t})$ and degree $0$. 

If $\ell_{t}=n$, then by \cref{eq:expressIndexSets_speical:Sphere,eq:HighestRoot_Cn,eq:2rho_Cn}, we have
\begin{equation*}
	\fun{S}_{\mcal{V}_{\dagger}[I]}[r] = 
		\sum_{\crampedsubstack{
			c_i\in\Z_{>0}\\
			2c_1+\cdots+2c_{t-1}+c_t = r
		}}q^{	\sum\limits_{i=1}^{t-1}\ell_{i}(2n+1-\ell_{i})c_i + \tfrac{n(n+1)}{2}c_t	}.
\end{equation*}
Now, we apply \cref{lem:MultiSum2} to this summation, where the index set $\mfrak{i}$ is $\Set*{1,\cdots,t}$, the partition $\mfrak{i}=\mfrak{i}_{1}\sqcup\mfrak{i}_{2}$ is $\Set*{1,\cdots,t} = \Set*{t}\sqcup\Set*{1,\cdots,t-1}$, and the sequence $\bm{\mu}$ is 
\begin{flalign*}
	&&
		\mu_i &= \ell_{i}(2n+1-\ell_{i}),
		& \mathllap{(1 \le i < t)} \\
	&&
		\mu_t &= \tfrac{n(n+1)}{2}.
\end{flalign*}
Since $\mu_{t}$ is an integer and all members of $\bm{\mu}_{2}$ are even integers, $\fun{S}_{\mcal{V}_{\dagger}[I]}$ can be defined by a primary super $q$-exponential polynomial. 
The knowledge of quadratic function shows that $\mfrak{i}_{2\max}=\Set*{t-1}$, $\mu_{2\max}=\ell_{t-1}(2n+1-\ell_{t-1})$, and $2\mu_{1\max}>\mu_{2\max}$.  
Then by \cref{item:lem:MultiSum2:l}, we have 
\begin{equation}\label{eq:AsymptoticOfSdaggerI:Cn}
	\fun{S}_{\mcal{V}_{\dagger}[I]}[r] \sim  
		\prod_{	i=1	}^{t-1}
			\left(	q^{(n-\ell_{i})(n+1-\ell_{i})} - 1	\right)^{-1}
		\cdot q^{	\tfrac{n(n+1)}{2} r	}.
\end{equation}
In particular, it has order $\tfrac{n(n+1)}{2}$ and degree $0$.

\subsection{Dominant types}\label{subsec:dominant:Cn}
Now, we are able to figure out which type is dominant. 

We begin with $\fun{S}_{\mcal{V}[I]}$, since $\ell_t \le n$, we have 
\begin{equation*}
	\mu_{\max} = 
		\tfrac{1}{2}\ell_{t}(2n+1-\ell_{t}) \le \tfrac{n(n+1)}{2}.
\end{equation*}
The equality achieves exactly when $\ell_{t}=n$. 
Therefore, $\fun{S}_{\mcal{V}[I]}$ is dominant exactly when $n\notin{I}$. In this case, we have $\mfrak{i}_{\max}=\Set*{t}$ and $\ell_{t}=n$. 

Next, we consider $\fun{S}_{\mcal{V}_{\dagger}[I]}$. 
Then similar argument shows that it is dominant exactly when $n\notin{I}$.

\subsection{Asymptotic growths of \texorpdfstring{$\fun{SSA}_{\dagger}[r]$}{SSAdagger} and \texorpdfstring{$\fun{SV}_{\dagger}[r]$}{SVdagger}}\label{subsec:Cn:AsymptoticSSAdagger}
We are now able to obtain the asymptotic growths of $\fun{SSA}_{\dagger}[r]$ and $\fun{SV}_{\dagger}[r]$. 

By \cref{eq:SSA=SumSI:Cn}, \cref{eq:AsymptoticOfSdaggerI:Cn}, and \cref{subsec:dominant:Cn}, we see that $\fun{SSA}_{\dagger}[\:\cdot\:]$ can be defined by a primary super $q$-exponential polynomial so that
\begin{equation}\label{eq:Asymp:Cn:SSAdagger}
	\fun{SSA}_{\dagger}[r] \sim
		C_{\dagger}(n) \cdot q^{	\tfrac{n(n+1)}{2} r},
\end{equation}
where the constant $C_{\dagger}(n)$ is defined as follows:
\index[notation]{C dagger (n)@$C_{\dagger}(n)$}%
\begin{equation}\label{eq:Asymp:Cn:Cdagger}
	C_{\dagger}(n)	:= 
	\sum_{n\notin I}
		\frac{
			\mscr{P}_{C_n;I}[q]
		}{
			q^{\fun{deg}[\mscr{P}_{C_n;I}]}
		}
		\prod_{i=1}^{t_{I}-1}
			\left(	q^{	(n-\ell_{i}(I))(n+1-\ell_{i}(I))	}-1	\right)^{-1}.
\end{equation}
As a consequence, we see that $\fun{SV}_{\dagger}[\:\cdot\:]$ can be defined by a primary super $q$-exponential polynomial so that
\begin{equation}\label{eq:Asymp:Cn:SVdagger}
	\fun{SV}_{\dagger}[r] = 
	\sum_{z=0}^{r}\fun{SSA}_{\dagger}[z] \sim 
		\frac{
			q^{\tfrac{n(n+1)}{2}}
		}{
			q^{\tfrac{n(n+1)}{2}} - 1
		}C_{\dagger}(n) \cdot q^{	\tfrac{n(n+1)}{2} r}.
\end{equation}

\subsection{Asymptotic growth of dominant \texorpdfstring{$\fun{S}_{\mcal{V}[I]}[r]$}{SVIr}}\label{subsec:Cn:AsymptoticSVIr}
Let $I$ be a type and follow \cref{con:type}. 
We are going to compute the asymptotic growth of $\fun{S}_{\mcal{V}[I]}[r]$ when $I$ is dominant. 

First, we need to write $\fun{S}_{\mcal{V}[I]}[r]$ into a multi-summation. 
To do this, we pick an arbitrary $x\in\mcal{V}[I]$ and investigate the difference between $2\rho(x)$ and the sum of $\ceil{a(x)}$ for $a(x)>0$. 
To better describe this sum, we introduce the following conventions. 
\begin{convention}\label{con:ell_inverse}
	\index[notation]{l t+1@$\ell_{t+1}=n$}%
	For any $j\in\Set*{1,\cdots,n}$, we will use $\ell^{-1}(j)$ to denote the index $i\in\Set*{1,\cdots,t+1}$ such that $\ell_{i-1}< j \le \ell_{i}$, where $\ell_{t+1}=n$. 
\end{convention}
\begin{convention}\label{con:standard_parity}
	\index{parity function!standard}%
	\index[notation]{z bar@$\vect{\overline{z}}$}%
	We will use $\overline{z}$ to denote \emph{standard parity function} mapping even numbers to $0$ and odd numbers to $1$. 	
	Note that $\ceil*{\tfrac{1}{2}z} = \tfrac{1}{2}(z+\overline{z})$.
\end{convention}
\begin{convention}\label{con:SummationSeq}
	The summation $c_i+\cdots+c_j$ is read to be $0$ when $i>j$.
\end{convention}
Now, suppose $x = o+c_1\omega_{\ell_{1}}'+\cdots+c_t\omega_{\ell_{t}}'\in\mcal{V}[I]$. By \cref{eq:Roots_Cn}, we have
\begin{flalign*}
	&&
	(\chi_{j}-\chi_{j'})(x) &= 
		\tfrac{1}{2}
		\left(
			c_{\ell^{-1}(j)}+\cdots+c_{\ell^{-1}(j')-1}
		\right),
		&{(1 \le j < j' \le n)}\\
	&&
	(\chi_{j}+\chi_{j'})(x) &= 
		\tfrac{1}{2}
		\left(
			c_{\ell^{-1}(j)}+\cdots+c_{\ell^{-1}(j')-1}
		\right) + 
		c_{\ell^{-1}(j')}+\cdots+c_{t},
		&{(1 \le j < j' \le n)}\\
	&&
	2\chi_{j}(x) &= 
		c_{\ell^{-1}(j)}+\cdots+c_{t}.
		&{(1 \le j \le n)}
\end{flalign*}
Therefore, we have
\begin{align*}
	\sum_{a\in\Phi^{+}}\ceil{a(x)} &= 
	\sum_{1 \le j < j' \le n}
	\left(
		\ceil{(\chi_{j}-\chi_{j'})(x)} + 
		\ceil{(\chi_{j}+\chi_{j'})(x)}
	\right) + 
	\sum_{j=1}^{n}\ceil{2\chi_{j}(x)} \\
	&= 2\rho(x) + 
	\sum_{1 \le j < j' \le n}
		\overline{	c_{\ell^{-1}(j)}+\cdots+c_{\ell^{-1}(j')-1}	}.
\end{align*}

From above analysis, we can define the parity function $e_{I}$ as follows:
\begin{equation}
	\label{eq:DefineEpsilon:Cn}
	e_{I}(c_1,\cdots,c_{t}) :=
	\sum_{1 \le i < i' \le t+1}
		(\ell_{i}-\ell_{i-1})(\ell_{i'}-\ell_{i'-1})
		\overline{c_i+\cdots+c_{i'-1}}.
\end{equation}
Then we have 
\begin{equation*}
	\sum_{a\in\Phi^{+}}\ceil{a(x)} =
		2\rho(x) + e_{I}(c_1,\cdots,c_{t}).
\end{equation*}

Now, we apply \cref{lem:MultiSumSEP} to the following summation.
\begin{equation*}
	\fun{S}_{\mcal{V}[I]}[r] = 
		\sum_{\crampedsubstack{
			c_i\in\Z_{>0}\\
			c_1+\cdots+c_t = r
		}}q^{	\sum\limits_{i=1}^{t}\tfrac{1}{2}\ell_{i}(2n+1-\ell_{i})c_i	+ e_{I}(c_1,\cdots,c_{t})}.
\end{equation*}
Note that the index set $\mfrak{i}$ is $\Set*{1,\cdots,t}$ and the sequence $\bm{\mu}$ is 
\begin{flalign*}
	&&	\mu_i&=\tfrac{1}{2}\ell_{i}(2n+1-\ell_{i}). & \mathllap{(1 \le i \le t)}
\end{flalign*} 
Since all members of $\bm{\mu}$ are integers and $e_{I}$ is valued in integers, $\fun{S}_{\mcal{V}[I]}$ can be defined by a primary super $q$-exponential polynomial.
Note that, by \cref{eq:SSA=SumSI:Cn}, this already implies that $\fun{SV}[\:\cdot\:]$ and  $\fun{SSA}[\:\cdot\:]$ can be defined by primary super $q$-exponential polynomials.

Now, let $I$ be a dominant type, namely $n\notin I$. 
Then we have $\mfrak{i}_{\max}=\Set*{t}$, $\ell_{t}=n$, and $\mu_{\max} = \tfrac{n(n+1)}{2}$. 
Therefore, 
\begin{equation*}
	\fun{S}_{\mcal{V}[I]}[r] \sim 
		\left(	C_{I,0} + C_{I,1}(-1)^{r}	\right)
		\cdot q^{\tfrac{n(n+1)}{2} r},
\end{equation*}
where the constants $C_{I,0}$ and $C_{I,1}$ are defined as follows:
\begin{align*}
	C_{I,0} &:= 
		C_{I}\cdot 
		\sum_{\vect{s}\in\F_2^{t}}
			\fun{E}_{I}(\vect{s}),
	&
	C_{I,1} &:= 
		C_{I}\cdot 
		\sum_{\vect{s}\in\F_2^{t}}
			(-1)^{\vect{1}\cdot\vect{s}}
			\fun{E}_{I}(\vect{s}),
\end{align*}
where the constant $C_{I}$ and the function $\fun{E}_{I}\colon\F_2^{t}\to\mbb{Q}[q;-]$ are defined as follows:
\begin{align}
	\label{eq:EIandCI:Cn}%
	C_{I} &:=
		\frac{1}{2}
		\prod_{i=1}^{t-1}
			\left(	q^{	(n-\ell_{i})(n+1-\ell_{i})	}-1	\right)^{-1},
	&
	\fun{E}_{I}(\vect{s}) &:= 
		q^{	e_{I}(\vect{s}) +
				\sum\limits_{i=1}^{t-1}
					\tfrac{1}{2}(n-\ell_{i})(n+1-\ell_{i})s_i	}.
\end{align}
From the definition \cref{eq:DefineEpsilon:Cn} of $e_{I}$, it is clear that 
\begin{equation*}
	\fun{E}_{I}(s_1,\cdots,s_{t-1},0) = 
		\fun{E}_{I}(s_1,\cdots,s_{t-1},1).
\end{equation*}
Therefore, $C_{I,1} = 0$ and we thus have 
\begin{equation}
	\label{eq:AsymptoticOfSI:Cn}
	\fun{S}_{\mcal{V}[I]}[r] \sim 
		C_I\cdot
		\sum_{\vect{s}\in\F_2^{t-1}}
			2\fun{E}_{I}[\vect{s}\sqcup 0]
		\cdot q^{\tfrac{n(n+1)}{2} r},
\end{equation}
where $\vect{s}\sqcup 0$ is the sequence $s_1,\cdots,s_{t-1},0$.

\subsection{Asymptotic growths of \texorpdfstring{$\fun{SSA}[r]$}{SSA} and \texorpdfstring{$\fun{SV}[r]$}{SV}}
We are now going to obtain the asymptotic growth of $\fun{SSA}[r]$.

By \cref{eq:SSA=SumSI:Cn}, \cref{subsec:dominant:Cn}, and \cref{eq:DefineEpsilon:Cn,eq:EIandCI:Cn,eq:AsymptoticOfSI:Cn}, we see that $\fun{SSA}[\:\cdot\:]$ can be defined by a primary super $q$-exponential polynomial so that 
\begin{equation}\label{eq:Asymp:Cn:SSA}
	\fun{SSA}[r] \sim
		C(n) \cdot q^{	\tfrac{n(n+1)}{2} r},
\end{equation}
where the constant $C(n)$ is defined as follows:
\begin{equation}\label{eq:Asymp:Cn:C}
	\index[notation]{C (n)@$C(n)$}%
	C(n)	:= 
	\sum_{n\notin I}
	\left(
		\frac{
			\mscr{P}_{C_n;I}[q]
		}{
			q^{\fun{deg}[\mscr{P}_{C_n;I}]}
		}
		\prod_{i=1}^{t_{I}-1}
			\left(	q^{	(n-\ell_{i}(I))(n+1-\ell_{i}(I))	}-1	\right)^{-1}
		\sum_{\vect{s}\in\F_2^{t_{I}-1}}
			\fun{E}_{C_n;I}[\vect{s}]
	\right),
\end{equation}
where the function $\fun{E}_{C_n;I}\colon\F_2^{t_{I}-1}\to\mbb{Q}[q;-]$ is defined as follows:
\begin{align}\label{eq:Asymp:Cn:E}
	\fun{E}_{C_n;I}[\vect{s}] &:= 
		q^{	\sum\limits_{1 \le i < i' \le t_{I}}
					(\ell_{i}(I)-\ell_{i-1}(I))(\ell_{i'}(I)-\ell_{i'-1}(I))
						\overline{s_i+\cdots+s_{i'-1}}	}\\
	\nonumber&\quad\cdot
		q^{	\sum\limits_{i=1}^{t_{I}-1}
					\tfrac{1}{2}(n-\ell_{i}(I))(n+1-\ell_{i}(I))s_i	}.
\end{align}
As a consequence, we see that $\fun{SV}[\:\cdot\:]$ can be defined by a primary super $q$-exponential polynomial so that
\begin{equation}\label{eq:Asymp:Cn:SV}
	\fun{SV}[r] = 
	\sum_{z=0}^{r}\fun{SSA}[z] \sim 
		\frac{
			q^{\tfrac{n(n+1)}{2}}
		}{
			q^{\tfrac{n(n+1)}{2}} - 1
		}C(n) \cdot q^{	\tfrac{n(n+1)}{2} r}.
\end{equation}

By \cref{eq:Asymp:Cn:SSA,eq:Asymp:Cn:C,eq:Asymp:Cn:E,eq:Asymp:Cn:SV}, 
we have proved \cref{thm:Asymp:Cn}. 
Moreover, by \cref{eq:PoincareCnI}, we have the following explicit formulas for the first factor of $C(n)$:
\begin{align*}
	\mscr{P}_{C_n;I}[q]
	&= 
		\frac{
			[2n]!!(q)
		}{
			\prod\limits_{i=1}^{t}
				[\ell_{i}(I)-\ell_{i-1}(I)]!(q)
		},
	&
	q^{\fun{deg}[\mscr{P}_{C_n;I}]}
	&=
		\frac{
			q^{n^2}
		}{
			\prod\limits_{i=1}^{t}
				q^{\binom{\ell_{i}(I)-\ell_{i-1}(I)}{2}}
		}.
\end{align*}
See \cref{eq:quantum_factorial,eq:PoincareCn} for the definitions of the symbols $[\:\cdot\:]!$ and $[2\:\cdot\:]!!$.

\clearpage
\section{Simplicial volume in buildings of \texorpdfstring{$B_{n}$}{Bn} type}\label{sec:Bn}
In this section, we will prove the $B_{n}$ part of \cref{thm:AsymptoticDominanceOfSV,thm:AsymptoticGrowthOfSV}. 
More precisely, we will prove the following stronger theorem. 
\begin{theorem}\label{thm:Asymp:Bn}
	Let $\mathscr{B}$ be a Bruhat-Tits building of split classical type $B_n$ over a local field $K$ with residue cardinality $q$. 
	Then the simplicial volume $\fun{SV}[\:\cdot\:]$ and the simplicial surface area $\fun{SSA}[\:\cdot\:]$ in it can be defined by super $q$-exponential polynomials whose leading terms are of the form:%
	\footnote{\label{footnote1}The leading terms may give an impression that these $q$-functions can be defined by primary $q$-exponential polynomials when $n=3$. However, we will see this is false in \cref{subsec:Bn:AsymptoticOfSdaggerI}.}
	\begin{align*}
		\fun{SV}[r] &\sim 
			\tilde{C}(n) \cdot q^{	\pi(n) r	},&
		\fun{SSA}[r] &\sim 
			C(n) \cdot q^{	\pi(n) r	},
	\end{align*}
	where $\pi(n)=\tfrac{n^2}{2}$ when $n \ge 4$ and $\pi(3)=5$. 
	The leading coefficients $\tilde{C}(3)$ and $C(3)$ are primary $q$-numbers, not just parity $q$-functions. 
	Moreover, the four $q$-functions $\fun{SV}[2\:\cdot\:]$, $\fun{SV}[2\:\cdot\:+1]$, $\fun{SSA}[2\:\cdot\:]$, and $\fun{SSA}[2\:\cdot\:+1]$ can be defined by primary $q$-exponential polynomials whose leading terms are of the form: 
	\begin{align*}
		\fun{SV}[2r] &\sim 
			\tilde{C}_{0}(n) \cdot q^{	2\pi(n) r	},&
		\fun{SSA}[2r] &\sim 
			C_{0}(n) \cdot q^{	2\pi(n) r	},\\
		\fun{SV}[2r+1] &\sim 
			\tilde{C}_{1}(n) \cdot q^{	2\pi(n) r	},&
		\fun{SSA}[2r+1] &\sim 
			C_{1}(n) \cdot q^{	2\pi(n) r	}.
	\end{align*}
\end{theorem}
We will obtain explicit formulas for the parity functions $\tilde{C}(n)$ and $C(n)$, and the constants $\tilde{C}_{0}(n)$, $C_{0}(n)$, $\tilde{C}_{1}(n)$, and $C_{1}(n)$.

But before proving \cref{thm:Asymp:Bn}, we will first analyze the asymptotic growths of $\fun{SSA}_{\dagger}[r]$ and $\fun{SV}_{\dagger}[r]$, where $\dagger$ denotes ``being special''. We will prove the following.
\begin{theorem}\label{thm:Asymp:BnSp}
	Let $\mathscr{B}$ be a Bruhat-Tits building of split classical type $B_n$ over a local field $K$ with residue cardinality $q$. 
	Then the special simplicial volume $\fun{SV}_{\dagger}[\:\cdot\:]$ and the special simplicial surface area $\fun{SSA}_{\dagger}[\:\cdot\:]$ in it can be defined by super $q$-exponential polynomials whose leading terms are of the form:\cref{footnote1}
	\begin{align*}
		\fun{SV}_{\dagger}[r] &\sim 
			\tilde{C}_{_{\dagger}}(n) \cdot q^{	\pi(n) r	},&
		\fun{SSA}_{\dagger}[r] &\sim 
			C_{_{\dagger}}(n) \cdot q^{	\pi(n) r	},
	\end{align*}
	where $\pi(n)=\tfrac{n^2}{2}$ when $n \ge 4$ and $\pi(3)=5$. 
	The leading coefficients $\tilde{C}_{\dagger}(3)$ and $C_{\dagger}(3)$ are primary $q$-numbers, not just parity $q$-functions.
	Moreover, the four $q$-functions $\fun{SV}_{\dagger}[2\:\cdot\:]$, $\fun{SV}_{\dagger}[2\:\cdot\:+1]$, $\fun{SSA}_{\dagger}[2\:\cdot\:]$, and $\fun{SSA}_{\dagger}[2\:\cdot\:+1]$ can be defined by primary $q$-exponential polynomials whose leading terms are of the form: 
	\begin{align*}
		\fun{SV}_{\dagger}[2r] &\sim 
			\tilde{C}_{{\dagger}0}(n) \cdot q^{	2\pi(n) r	},&
		\fun{SSA}_{\dagger}[2r] &\sim 
			C_{{\dagger}0}(n) \cdot q^{	2\pi(n) r	},\\
		\fun{SV}_{\dagger}[2r+1] &\sim 
			\tilde{C}_{{\dagger}1}(n) \cdot q^{	2\pi(n) r	},&
		\fun{SSA}_{\dagger}[2r+1] &\sim 
			C_{{\dagger}1}(n) \cdot q^{	2\pi(n) r	}.
	\end{align*}
\end{theorem}
We will also give explicit formulas for the parity functions $\tilde{C}_{\dagger}(n)$ and $C_{\dagger}(n)$, and the constants $\tilde{C}_{{\dagger}0}(n)$, $C_{{\dagger}0}(n)$, $\tilde{C}_{{\dagger}1}(n)$, and $C_{{\dagger}1}(n)$. 
The proof of \cref{thm:Asymp:BnSp} turns out to play an essential role in the study of $\fun{SSA}[r]$ and $\fun{SV}[r]$. 

This section is structured as follows. 
In \cref{subsec:Bn:AsymptoticOfSdaggerI}, we will compute the asymptotic growth of $\fun{S}_{\mcal{V}_{\dagger}[I]}[r]$ for each type $I\subset\Delta$. This allows use to find the dominant ones of $\fun{S}_{\mcal{V}_{\dagger}[I]}[r]$, which will be done in \cref{subsec:Bn:DominantSdagger}.
Then in \cref{subsec:Bn:AsymptoticOfSdagger}, we will obtain the asymptotic growths of $\fun{SSA}_{\dagger}[r]$ and $\fun{SV}_{\dagger}[r]$. 
In \cref{subsec:Bn:SdaggerI}, we will analyze $\fun{S}_{\mcal{V}_{\dagger}[I]}[2r]$ and $\fun{S}_{\mcal{V}_{\dagger}[I]}[2r+1]$. Combine \cref{subsec:Bn:AsymptoticOfSdagger,subsec:Bn:SdaggerI}, we finish proving \cref{thm:Asymp:BnSp}. 
After that, in \cref{subsec:Bn:AsymptoticOfSasympXI}, we will estimate the asymptotic growths of $\fun{S}_{\mcal{X}^{0}[I]}[r]$ and $\fun{S}_{\mcal{X}^{1}[I]}[r]$ using the auxiliary functions $\fun{S}^{\asymp}_{\mcal{X}^{0}[I]}$ and $\fun{S}^{\asymp}_{\mcal{X}^{1}[I]}$. 
Note that $\mcal{V}$ is between $\mcal{V}_{\dagger}$ and $\mcal{X}^{0}\cup\mcal{X}^{1}$. 
Therefore, we can combine \cref{subsec:Bn:AsymptoticOfSdaggerI} and \cref{subsec:Bn:AsymptoticOfSasympXI} to estimate the asymptotic growth of each $\fun{S}_{\mcal{V}[I]}[r]$ and find the dominant ones, which will be done in \cref{subsec:Bn:dominant}. 
Once we found the dominant types, we can proceed to compute the asymptotic growth of dominant $\fun{S}_{\mcal{V}[I]}[r]$. 
This will be done in three steps: 
in \cref{subsec:Bn:AsymptoticOfSXI}, we will compute the asymptotic growths of $\fun{S}_{\mcal{X}^{0}[I]}[r]$ and $\fun{S}_{\mcal{X}^{1}[I]}[r]$; in \cref{subsec:Bn:AsymptoticOfSXJI}, we will deduce the asymptotic growth of $\fun{S}_{\mcal{X}_{J}[I]}[r]$ from that of $\fun{S}_{\mcal{V}_{\dagger}[I]}[r]$; finally in \cref{subsec:Bn:AsymptoticOfSVI}, the asymptotic growth of $\fun{S}_{\mcal{V}[I]}[r]$ will be deduced from them. 
Then in \cref{subsec:Bn:AsymptoticOfS}, we will obtain the asymptotic growths of $\fun{SSA}[r]$ and $\fun{SV}[r]$. 
In \cref{subsec:Bn:SI}, we will analyze $\fun{S}_{\mcal{X}_{J}[I]}[2r]$ and $\fun{S}_{\mcal{X}_{J}[I]}[2r+1]$. Combine \cref{subsec:Bn:AsymptoticOfS,subsec:Bn:SI}, we finish proving \cref{thm:Asymp:Bn}. 

Throughout this section, we will heavily use the various index sets $\mcal{V}$, $\mcal{V}_{\dagger}$, $\mcal{X}^{0}$, $\mcal{X}^{1}$, and $\mcal{X}_{J}$. 
We refer to \cref{figure:VerticesOfIBnl1,figure:VerticesOfIBne1} for the structure of them.

\subsection{Asymptotic growth of \texorpdfstring{$\fun{S}_{\mcal{V}_{\dagger}[I]}[r]$}{SVdaggerIr}}
\label{subsec:Bn:AsymptoticOfSdaggerI}
Now, let $I$ be a type and follow \cref{con:type}. 
We are going to compute the asymptotic growth of $\fun{S}_{\mcal{V}_{\dagger}[I]}[r]$. 
We will separate the discussion into two cases: \cref{step:Bn:AsymptoticOfSdaggerI:l1} $\ell_{1}>1$ and \cref{step:Bn:AsymptoticOfSdaggerI:e1} $\ell_{1}=1$.

\begin{step}\label{step:Bn:AsymptoticOfSdaggerI:l1}
	If $\ell_{1}>1$, then by \cref{eq:expressIndexSets_speical:Sphere,eq:HighestRoot_Bn,eq:2rho_Bn}, we have
	\begin{equation*}
		\fun{S}_{\mcal{V}_{\dagger}[I]}[r] = 
			\sum_{\crampedsubstack{
				c_i\in\Z_{>0}\\
				2c_1+\cdots+2c_t = r
			}}q^{	\sum\limits_{i=1}^{t}\ell_{i}(2n-\ell_{i})c_i	}.
	\end{equation*}
	Now, we apply \cref{lem:MultiSum2} to this summation, where the index set $\mfrak{i}$ is $\Set*{1,\cdots,t}$, the partition $\mfrak{i}=\mfrak{i}_{1}\sqcup\mfrak{i}_{2}$ is $\Set*{1,\cdots,t} = \emptyset\sqcup\Set*{1,\cdots,t}$, and the sequence $\bm{\mu}$ is 
	\begin{flalign*}
		&&	\mu_i&=\ell_{i}(2n-\ell_{i}). & \mathllap{(1 \le i \le t)}
	\end{flalign*}
	The knowledge of quadratic function shows that $\mfrak{i}_{\max}=\mfrak{i}_{2\max}=\Set*{t}$ with $\mu_{\max}=\mu_{2\max}=\ell_{t}(2n-\ell_{t})$. 
	Then by \cref{item:lem:MultiSum2:s}, we have 
	\begin{equation}\label{eq:Bn:AsymptoticOfSdaggerI}
		\fun{S}_{\mcal{V}_{\dagger}[I]}[r] \sim  
			\prod_{	i=1	}^{t-1}
				\left(	q^{(\ell_{t}-\ell_{i})(2n-\ell_{t}-\ell_{i})} - 1	\right)^{-1}
			\cdot	\tfrac{1}{2}\left(1+(-1)^{r}\right)
			\cdot q^{	\tfrac{1}{2}\ell_{t}(2n-\ell_{t}) r	}.
	\end{equation}
	In particular, it has order $\tfrac{1}{2}\ell_{t}(2n-\ell_{t})$ and degree $0$. 
\end{step}

\begin{step}\label{step:Bn:AsymptoticOfSdaggerI:e1}
	If $\ell_{1}=1$, then by \cref{eq:expressIndexSets_speical:Sphere,eq:HighestRoot_Bn,eq:2rho_Bn}, we have
	\begin{equation*}
		\fun{S}_{\mcal{V}_{\dagger}[I]}[r] = 
			\sum_{\crampedsubstack{
				c_i\in\Z_{>0}\\
				c_1+2c_2+\cdots+2c_t = r
			}}q^{	\sum\limits_{i=1}^{t}\ell_{i}(2n-\ell_{i})c_i	}.
	\end{equation*}
	Now, we apply \cref{lem:MultiSum2} to this summation, where the index set $\mfrak{i}$ is $\Set*{1,\cdots,t}$, the partition $\mfrak{i}=\mfrak{i}_{1}\sqcup\mfrak{i}_{2}$ is $\Set*{1,\cdots,t} = \Set*{1}\sqcup\Set*{2,\cdots,t}$, and the sequence $\bm{\mu}$ is 
	\begin{flalign*}
		&&	\mu_i&=\ell_{i}(2n-\ell_{i}). & \mathllap{(1 \le i \le t)}
	\end{flalign*}
	The knowledge of quadratic function shows that $\mfrak{i}_{2\max}=\Set*{t}$ and $\mu_{2\max}=\ell_{t}(2n-\ell_{t})$. 

	Depending on $n$ and $\ell_{t}$, there are two possibilities: $2(2n-1) > \ell_{t}(2n-\ell_{t})$ and $2(2n-1) < \ell_{t}(2n-\ell_{t})$. 
	Note that $2(2n-1) = \ell_{t}(2n-\ell_{t})$ is impossible since the left-hand side has remainder $2$ modulo $4$ while the right-hand side is either odd or a multiple of $4$. 

	If $2(2n-1) > \ell_{t}(2n-\ell_{t})$, then by \cref{item:lem:MultiSum2:l}, we have 
	\begin{equation}\label{eq:Bn:AsymptoticOfSdaggerI:l}
		\fun{S}_{\mcal{V}_{\dagger}[I]}[r] \sim  
			\prod_{	i=2	}^{t}
				\left(	q^{2(2n-1)-\ell_{i}(2n-\ell_{i})} - 1	\right)^{-1}
			\cdot q^{	(2n-1) r	}.
	\end{equation}
	In particular, it has order $2n-1$ and degree $0$. 

	If $2(2n-1) < \ell_{t}(2n-\ell_{t})$, then by \cref{item:lem:MultiSum2:s}, we have 
	\begin{align}\label{eq:Bn:AsymptoticOfSdaggerI:s}
		\fun{S}_{\mcal{V}_{\dagger}[I]}[r] &\sim  
			\left(	q^{\ell_{t}(2n-\ell_{t})-2(2n-1)} - 1	\right)^{-1}
			\prod_{	i=2	}^{t-1}
				\left(	q^{(\ell_{t}-\ell_{i})(2n-\ell_{t}-\ell_{i})} - 1	\right)^{-1} \\
		&\nonumber\qquad
			\cdot\tfrac{1}{2}
			\left(	
				\left(	
					1+q^{\tfrac{1}{2}\ell_{t}(2n-\ell_{t})-(2n-1)}	
				\right) 
				+ 
				\left(	
					1-q^{\tfrac{1}{2}\ell_{t}(2n-\ell_{t})-(2n-1)}	
				\right)(-1)^{r}	
			\right) \\
		&\nonumber\qquad
			\cdot q^{	\tfrac{1}{2}\ell_{t}(2n-\ell_{t}) r	}.
	\end{align}
	In particular, it has order $\tfrac{1}{2}\ell_{t}(2n-\ell_{t})$ and degree $0$. 
\end{step}

\subsection{Dominant types for \texorpdfstring{$\fun{S}_{\mcal{V}_{\dagger}[I]}[r]$}{SVdaggerIr}}
\label{subsec:Bn:DominantSdagger}
Now, we are able to figure out for which type $I$, $\fun{S}_{\mcal{V}_{\dagger}[I]}[r]$ is dominant. 

When $n=3$, we have $\ell_{t}(I)\le 3$ for all $I$. 
Therefore, $2(2n-1) > \ell_{t}(I)(2n-\ell_{t}(I))$. 
Hence, $\fun{S}_{\mcal{V}_{\dagger}[I]}[r]$ is dominant exactly when $1\notin I$. 
Note that, such a type $I$ must be one of the following: $\Set*{2,3}$, $\Set*{2}$, $\Set*{3}$, and $\emptyset$. 
Using \cref{eq:Bn:AsymptoticOfSdaggerI:l}, we can deduce the asymptotic growth of dominant $\fun{S}_{\mcal{V}_{\dagger}[I]}[r]$ as follows. 
\begin{align}
	\label{eq:B3:AsymptoticOfSdaggerI:23}
	\fun{S}_{\mcal{V}_{\dagger}[\Set*{2,3}]}[r] &= 
		q^{	(2n-1) r	} = q^{	5r	}, \\
	\label{eq:B3:AsymptoticOfSdaggerI:2}
	\fun{S}_{\mcal{V}_{\dagger}[\Set*{2}]}[r] &\sim  
		\left(	q^{2(2n-1)-\ell_{2}(2n-\ell_{2})} - 1	\right)^{-1}
		\cdot q^{	(2n-1) r	} \\
	\nonumber&= 
		\frac{1}{\left(	q - 1	\right)}
		\cdot q^{	5 r	}, \\
	\label{eq:B3:AsymptoticOfSdaggerI:3}
	\fun{S}_{\mcal{V}_{\dagger}[\Set*{3}]}[r] &\sim  
		\left(	q^{2(2n-1)-\ell_{2}(2n-\ell_{2})} - 1	\right)^{-1}
		\cdot q^{	(2n-1) r	}\\
	\nonumber&=  
		\frac{1}{\left(	q^2 - 1	\right)}
		\cdot q^{	5 r	}, \\
	\label{eq:B3:AsymptoticOfSdaggerI:emptyset}
	\fun{S}_{\mcal{V}_{\dagger}[\emptyset]}[r] &\sim  
		\prod_{	i=2	}^{3}
			\left(	q^{2(2n-1)-\ell_{i}(2n-\ell_{i})} - 1	\right)^{-1}
		\cdot q^{	(2n-1) r	} \\
	\nonumber&= 
		\frac{1}{\left(	q - 1	\right)\left(	q^2 - 1	\right)}
		\cdot q^{	5 r	}.
\end{align}

If $n\ge 4$, then we have $2(2n-1) < \ell_{t}(I)(2n-\ell_{t}(I))$ when $\ell_{t}(I)\ge 4$. 
On the other hand, since $\ell_{t}(I)\le n$, we have
\begin{equation*}
	\mu_{2\max} = \ell_{t}(I)(2n-\ell_{t}(I)) \le n^2.
\end{equation*}
The equality achieves when $\ell_{t}(I)=n$. 
Hence, $\fun{S}_{\mcal{V}_{\dagger}[I]}[r]$ is dominant exactly when $n\notin I$. 
In that case, its asymptotic growth is given by \cref{eq:Bn:AsymptoticOfSdaggerI:s}.

\subsection{Asymptotic growths of \texorpdfstring{$\fun{SSA}_{\dagger}[r]$}{SSAdagger} and \texorpdfstring{$\fun{SV}_{\dagger}[r]$}{SVdagger}}
\label{subsec:Bn:AsymptoticOfSdagger}
We are now able to obtain the asymptotic growth of $\fun{SSA}_{\dagger}[r]$. 
By \cref{eq:SimplicialSurfaceAreaFormulaVar}, we have 
\begin{equation}\label{eq:SSA=SumSI:BnSp}
	\fun{SSA}_{\dagger}[r] = 
		\sum_{I\subset\Delta}
		\frac{
			\mscr{P}_{B_n;I}[q]
		}{
			q^{\fun{deg}[\mscr{P}_{B_n;I}]}
		}\fun{S}_{\mcal{V}_{\dagger}[I]}[r]
		\sim 
		\sum_{I\text{ is dominant}}
		\frac{
			\mscr{P}_{B_n;I}[q]
		}{
			q^{\fun{deg}[\mscr{P}_{B_n;I}]}
		}\fun{S}_{\mcal{V}_{\dagger}[I]}[r].
\end{equation}\label{eq:}
Then by the discussion in \cref{subsec:Bn:DominantSdagger}, we see that 
\begin{equation}\label{eq:Bn:Asymp:SSAdagger}
	\fun{SSA}_{\dagger}[r] \sim
	\begin{dcases*}
		C_{\dagger}(3) \cdot q^{	5 r} 
			& if $n=3$, \\
		C_{\dagger}(n) \cdot q^{	\tfrac{n^2}{2} r} 
			& if $n\ge 4$.
	\end{dcases*}
\end{equation}

When $n=3$, by \cref{eq:B3:AsymptoticOfSdaggerI:23,eq:B3:AsymptoticOfSdaggerI:2,eq:B3:AsymptoticOfSdaggerI:3,eq:B3:AsymptoticOfSdaggerI:emptyset}, the constant $C_{\dagger}(3)$ is defined as follows: 
\begin{align*}
	C_{\dagger}(3) &:= 
		\frac{
			\mscr{P}_{B_3;\Set*{2,3}}[q]
		}{
			q^{\fun{deg}[\mscr{P}_{B_3;\Set*{2,3}}]}
		} + 
		\frac{
			\mscr{P}_{B_3;\Set*{2}}[q]
		}{
			\left(	q - 1	\right)q^{\fun{deg}[\mscr{P}_{B_3;\Set*{2}}]}
		} + 
		\frac{
			\mscr{P}_{B_3;\Set*{3}}[q]
		}{
			\left(	q^2 - 1	\right)q^{\fun{deg}[\mscr{P}_{B_3;\Set*{3}}]}
		} \\
		\nonumber&\qquad + 
		\frac{
			\mscr{P}_{B_3;\emptyset}[q]
		}{
			\left(	q - 1	\right)\left(	q^2 - 1	\right)
			q^{\fun{deg}[\mscr{P}_{B_3;\emptyset}]}
		}. \\
\end{align*}
Moreover, by \cref{eq:PoincareCnI}, we have
\begin{align}\label{eq:B3:Asymp:Cdagger}
	\index[notation]{C dagger (3)@$C_{\dagger}(3)$}%
	C_{\dagger}(3) &= 
		\frac{
			\left(q^6-1\right)
		}{
			\left(q-1\right)q^{5}
		} + 
		\frac{
			\left(q^6-1\right)\left(q^4-1\right)
		}{
			\left(q^2-1\right)\left(q-1\right)^{2}q^{7}
		} + 
		\frac{
			\left(q^6-1\right)\left(q^4-1\right)
		}{
			\left(q-1\right)^{2}\left(	q^2 - 1	\right)q^{8}
		} \\
		\nonumber&\qquad + 
		\frac{
			\left(q^6-1\right)\left(q^4-1\right)\left(q^2-1\right)
		}{
			\left(q-1\right)^{4}\left(	q^2 - 1	\right)q^{9}
		} \\
		\nonumber&=
		\frac{
			\left(q^2+q+1\right) 
			\left(q^2-q+1\right)
			(q+1)  
		}{
			(q-1)^{2}q^{9}
		}
		\left(q^6-q^5+q^4+q^3+q^2+1\right).
\end{align}
As a consequence, we have
\begin{equation}\label{eq:B3:Asymp:SVdagger}
	\fun{SV}_{\dagger}[r] = 
	\sum_{z=0}^{r}\fun{SSA}_{\dagger}[z] \sim 
		\frac{
			q^{5}
		}{
			q^{5} - 1
		}C_{\dagger}(3) \cdot q^{	5 r}.
\end{equation}

When $n\ge 4$, by \cref{eq:Bn:AsymptoticOfSdaggerI,eq:Bn:AsymptoticOfSdaggerI:s}, $C_{\dagger}(n)$ is a parity $q$-function defined as follows:
\index[notation]{C dagger (n)@$C_{\dagger}(n)$}%
\begin{align}
	\label{eq:Bn:Asymp:Cdagger:even}
	C_{\dagger}(n)(\text{even})	&:= 
	\sum_{1,n\notin I}
		\frac{
			\mscr{P}_{B_n;I}[q]
		}{
			q^{\fun{deg}[\mscr{P}_{B_n;I}]}
		}
		\prod_{	i=2	}^{t-1}
			\left(	q^{(n-\ell_{i}(I))^2} - 1	\right)^{-1}
		\cdot
		\frac{1}{	q^{n^2-2(2n-1)} - 1	}\\
	\nonumber&\qquad+
	\sum_{1\in I, n\notin I}
		\frac{
			\mscr{P}_{B_n;I}[q]
		}{
			q^{\fun{deg}[\mscr{P}_{B_n;I}]}
		}
		\prod_{	i=1	}^{t-1}
			\left(	q^{(n-\ell_{i}(I))^2} - 1	\right)^{-1},\\
	\label{eq:Bn:Asymp:Cdagger:odd}
	C_{\dagger}(n)(\text{odd})	&:= 
	\sum_{1,n\notin I}
		\frac{
			\mscr{P}_{B_n;I}[q]
		}{
			q^{\fun{deg}[\mscr{P}_{B_n;I}]}
		}
		\prod_{	i=2	}^{t-1}
			\left(	q^{(n-\ell_{i}(I))^2} - 1	\right)^{-1}
		\cdot
		\frac{
			q^{\tfrac{n^2}{2}-(2n-1)}
		}{
			q^{n^2-2(2n-1)} - 1
		}.
\end{align}
As a consequence, we have
\begin{equation}\label{eq:Bn:Asymp:SVdagger}
	\fun{SV}_{\dagger}[r] = 
	\sum_{z=0}^{r}\fun{SSA}_{\dagger}[z] \sim 
		\tilde{C}_{\dagger}(n)q^{	\tfrac{n^2}{2} r	},
\end{equation}
where the parity $q$-function $\tilde{C}_{\dagger}(n)$ is defined as follows: 
\begin{align}
	\index[notation]{C dagger (n) tilde@$\tilde{C}_{\dagger}(n)$}%
	\label{eq:Bn:Asymp:tCdagger:even}
		\tilde{C}_{\dagger}(n)(\text{even}) &:= 
		\sum_{1,n\notin I}
			\frac{
				\mscr{P}_{B_n;I}[q]
			}{
				q^{\fun{deg}[\mscr{P}_{B_n;I}]}
			}
			\prod_{	i=2	}^{t-1}
				\left(	q^{(n-\ell_{i}(I))^2} - 1	\right)^{-1}
			\cdot
			\frac{
				q^{n^2-(2n-1)}+q^{n^2}
			}{
				\left(	q^{n^2} - 1	\right)
				\left(	q^{n^2-2(2n-1)} - 1	\right)
			} \\
	\nonumber&\qquad+
		\sum_{1\in I, n\notin I}
			\frac{
				\mscr{P}_{B_n;I}[q]
			}{
				q^{\fun{deg}[\mscr{P}_{B_n;I}]}
			}
			\prod_{	i=1	}^{t-1}
				\left(	q^{(n-\ell_{i}(I))^2} - 1	\right)^{-1}
			\cdot
			\frac{	q^{n^2}	}{	q^{n^2}-1	}, \\	
	\label{eq:Bn:Asymp:tCdagger:odd}
		\tilde{C}_{\dagger}(n)(\text{odd})	&:= 
		\sum_{1,n\notin I}
			\frac{
				\mscr{P}_{B_n;I}[q]
			}{
				q^{\fun{deg}[\mscr{P}_{B_n;I}]}
			}
			\prod_{	i=2	}^{t-1}
				\left(	q^{(n-\ell_{i}(I))^2} - 1	\right)^{-1}
			\cdot
			\frac{
				q^{\tfrac{3n^2}{2}-(2n-1)}+q^{\tfrac{n^2}{2}}
			}{
				\left(	q^{n^2} - 1	\right)
				\left(	q^{n^2-2(2n-1)} - 1	\right)
			} \\
	\nonumber&\qquad+
		\sum_{1\in I, n\notin I}
			\frac{
				\mscr{P}_{B_n;I}[q]
			}{
				q^{\fun{deg}[\mscr{P}_{B_n;I}]}
			}
			\prod_{	i=1	}^{t-1}
				\left(	q^{(n-\ell_{i}(I))^2} - 1	\right)^{-1}
			\cdot
			\frac{	q^{\tfrac{n^2}{2}}	}{	q^{n^2}-1	}.
\end{align}

By \cref{eq:Bn:Asymp:SSAdagger,eq:B3:Asymp:Cdagger,eq:B3:Asymp:SVdagger,eq:Bn:Asymp:Cdagger:even,eq:Bn:Asymp:Cdagger:odd,eq:Bn:Asymp:SVdagger,eq:Bn:Asymp:tCdagger:even,eq:Bn:Asymp:tCdagger:odd}, we have proved the asymptotic relations in \cref{thm:Asymp:BnSp}, where 
\begin{align*}
	C_{{\dagger}0}(n) &= 
		C_{\dagger}(n)(\text{even}), &
	C_{{\dagger}1}(n) &= 
		C_{\dagger}(n)(\text{odd})\cdot q^{\pi(n)}, \\
	\tilde{C}_{{\dagger}0}(n) &= 
		\tilde{C}_{\dagger}(n)(\text{even}), &
	\tilde{C}_{{\dagger}1}(n) &= 
		\tilde{C}_{\dagger}(n)(\text{odd})\cdot q^{\pi(n)}.
\end{align*}  
Moreover, by \cref{eq:PoincareCnI}, we have the following explicit formulas:
\begin{align*}
	\mscr{P}_{B_n;I}[q]
	&= 
		\frac{
			[2n]!!(q)
		}{
			\prod\limits_{i=1}^{t}
				[\ell_{i}(I)-\ell_{i-1}(I)]!(q)
		},
	&
	q^{\fun{deg}[\mscr{P}_{B_n;I}]}
	&=
		\frac{
			q^{n^2}
		}{
			\prod\limits_{i=1}^{t}
				q^{\binom{\ell_{i}(I)-\ell_{i-1}(I)}{2}}
		}.
\end{align*}
See \cref{eq:quantum_factorial,eq:PoincareCn} for the definitions of the symbols $[\:\cdot\:]!$ and $[2\:\cdot\:]!!$.

\subsection{Analysis of \texorpdfstring{$\fun{S}_{\mcal{V}_{\dagger}[I]}[2r]$}{SVdaggerI2r} and \texorpdfstring{$\fun{S}_{\mcal{V}_{\dagger}[I]}[2r+1]$}{SVdaggerI2r+1}}
\label{subsec:Bn:SdaggerI}
Now, let $I$ be a type and follow \cref{con:type}. 
We are going to show that $\fun{S}_{\mcal{V}_{\dagger}[I]}[2\:\cdot\:]$ and $\fun{S}_{\mcal{V}_{\dagger}[I]}[2\:\cdot\:+1]$ can be defined by primary $q$-exponential polynomials. 
We will separate the discussion into two cases: \cref{step:Bn:SdaggerI:l1} $\ell_{1}>1$ and \cref{step:Bn:SdaggerI:e1} $\ell_{1}=1$.

\begin{step}\label{step:Bn:SdaggerI:l1}
	If $\ell_{1}>1$, then by \cref{eq:expressIndexSets_speical:Sphere,eq:HighestRoot_Bn,eq:2rho_Bn}, we have
	\begin{align*}
		\fun{S}_{\mcal{V}_{\dagger}[I]}[2r] &= 
			\sum_{\crampedsubstack{
				c_i\in\Z_{>0}\\
				2c_1+\cdots+2c_t = 2r
			}}q^{	\sum\limits_{i=1}^{t}\ell_{i}(2n-\ell_{i})c_i	} = 
			\sum_{\crampedsubstack{
				c_i\in\Z_{>0}\\
				c_1+\cdots+c_t = r
			}}q^{	\sum\limits_{i=1}^{t}\ell_{i}(2n-\ell_{i})c_i	},\\
		\fun{S}_{\mcal{V}_{\dagger}[I]}[2r+1] &= 
			\sum_{\crampedsubstack{
				c_i\in\Z_{>0}\\
				2c_1+\cdots+2c_t = 2r+1
			}}q^{	\sum\limits_{i=1}^{t}\ell_{i}(2n-\ell_{i})c_i	} = 0.
	\end{align*}
	Now, we apply \cref{lem:MultiSum} to $\fun{S}_{\mcal{V}_{\dagger}[I]}[2r]$, where the index set $\mfrak{i}$ is $\Set*{1,\cdots,t}$ and the sequence $\bm{\mu}$ is 
	\begin{flalign*}
		&&	\mu_i&=\ell_{i}(2n-\ell_{i}). & \mathllap{(1 \le i \le t)}
	\end{flalign*}
	Since all members of $\bm{\mu}$ are integers, $\fun{S}_{\mcal{V}_{\dagger}[I]}[2\:\cdot\:]$ can be defined by a primary super $q$-exponential polynomial. 
\end{step}

\begin{step}\label{step:Bn:SdaggerI:e1}
	If $\ell_{1}=1$, then by \cref{eq:expressIndexSets_speical:Sphere,eq:HighestRoot_Bn,eq:2rho_Bn}, we have (noticing the involved change of variables)
	\begin{align*}
		\fun{S}_{\mcal{V}_{\dagger}[I]}[2r] &= 
			\sum_{\crampedsubstack{
				c_i\in\Z_{>0}\\
				c_1+c_2+\cdots+c_t = r
			}}q^{	2(2n-1)c_1 + \sum\limits_{i=2}^{t}\ell_{i}(2n-\ell_{i})c_i	}, \\
		\fun{S}_{\mcal{V}_{\dagger}[I]}[2r+1] &= 
			\sum_{\crampedsubstack{
				c_i\in\Z_{>0}\\
				c_1+c_2+\cdots+c_t = r+1
			}}q^{	(2n-1)(2c_1-1) + \sum\limits_{i=2}^{t}\ell_{i}(2n-\ell_{i})c_i	}.
	\end{align*}
	Now, we apply \cref{lem:MultiSum} to these summations, where the index set $\mfrak{i}$ is $\Set*{1,\cdots,t}$ and the sequence $\bm{\mu}$ is 
	\begin{flalign*}
		&&	\mu_1&=2(2n-1), \\
		&&	\mu_i&=\ell_{i}(2n-\ell_{i}). & \mathllap{(2 \le i \le t)}
	\end{flalign*}
	Since all members of $\bm{\mu}$ are integers, the $q$-functions $\fun{S}_{\mcal{V}_{\dagger}[I]}[2\:\cdot\:]$ and $\fun{S}_{\mcal{V}_{\dagger}[I]}[2\:\cdot\:+1]$ can be defined by primary super $q$-exponential polynomials. 
\end{step}

By \cref{eq:SSA=SumSI:BnSp}, the $q$-functions $\fun{SV}_{\dagger}[2\:\cdot\:]$, $\fun{SV}_{\dagger}[2\:\cdot\:+1]$, $\fun{SSA}_{\dagger}[2\:\cdot\:]$, and $\fun{SSA}_{\dagger}[2\:\cdot\:+1]$ are $\mbb{Q}[q;1]$-combinations of $\fun{S}_{\mcal{V}_{\dagger}[I]}[2\:\cdot\:]$ and $\fun{S}_{\mcal{V}_{\dagger}[I]}[2\:\cdot\:+1]$. We thus finish proving \cref{thm:Asymp:BnSp}.

\subsection{Asymptotic growths of \texorpdfstring{$\fun{S}^{\asymp}_{\mcal{X}^{0}[I]}[r]$}{SasympX0Ir} and \texorpdfstring{$\fun{S}^{\asymp}_{\mcal{X}^{1}[I]}[r]$}{SasympX1Ir}}
\label{subsec:Bn:AsymptoticOfSasympXI}
Now, let $I$ be a type and follow \cref{con:type}. 
We are going to estimate the asymptotic growths of $\fun{S}_{\mcal{X}^{0}[I]}[r]$ and $\fun{S}_{\mcal{X}^{1}[I]}[r]$ up to the leading coefficient.  
We will separate the discussion into two cases: \cref{step:Bn:AsymptoticOfSasympXI:l1} $\ell_{1}>1$ and \cref{step:Bn:AsymptoticOfSasympXI:e1} $\ell_{1}=1$.

\begin{step}\label{step:Bn:AsymptoticOfSasympXI:l1}
	If $\ell_{1} > 1$, then by \cref{eq:2rho_Bn,eq:IndexBsrciBn:pu0}, we have 
	\begin{equation*}
		\fun{S}^{\asymp}_{\mcal{X}^{0}[I]}[r] =
		\sum_{\crampedsubstack{
				c_i\in\Z_{>0}\\
				c_1+\cdots+c_t = r
			}}
			q^{	\sum\limits_{i=1}^{t} 
						\tfrac{1}{2}\ell_{i}(2n-\ell_{i})c_i	}.
	\end{equation*}
	Now, we apply \cref{lem:MultiSum} to above summation, where the index set $\mfrak{i}$ is $\Set*{1,\cdots,t}$ and the sequence $\bm{\mu}$ is 
	\begin{flalign*}
		&&	\mu_i&=\tfrac{1}{2}\ell_{i}(2n-\ell_{i}). & \mathllap{(1 \le i \le t)}
	\end{flalign*}
	The knowledge of quadratic function shows that $\mfrak{i}_{\max}=\Set*{t}$ with $\mu_{\max}=\tfrac{1}{2}\ell_{t}(2n-\ell_{t})$. 
	Then we have 
	\begin{equation*}
		\fun{S}^{\asymp}_{\mcal{X}^{0}[I]}[r] \sim  
			\prod_{	i=1	}^{t-1}
				\left(	q^{\tfrac{1}{2}(\ell_{t}-\ell_{i})(2n-\ell_{t}-\ell_{i})} - 1	\right)^{-1}
			\cdot q^{	\tfrac{1}{2}\ell_{t}(2n-\ell_{t}) r	}.
	\end{equation*}
	Since $\fun{S}_{\mcal{X}^{0}[I]}[r]\asymp\fun{S}^{\asymp}_{\mcal{X}^{0}[I]}[r]$, we see that $\fun{S}_{\mcal{X}^{0}[I]}$ has order $\tfrac{1}{2}\ell_{t}(2n-\ell_{t})$ and degree $0$. 
\end{step}

\begin{step}\label{step:Bn:AsymptoticOfSasympXI:e1}
	If $\ell_{1} = 1$, then by \cref{eq:2rho_Bn,eq:IndexBsrciBn:pu0,eq:IndexBsrciBn:pu1}, we have 
	\begin{align*}
		\fun{S}^{\asymp}_{\mcal{X}^{0}[I]}[r] &=
		\sum_{\crampedsubstack{
				c_i\in\Z_{>0}\\
				c_1+\cdots+c_t = r
			}}
			q^{	(2n-1)c_1 + \sum\limits_{i=2}^{t} 
					\tfrac{1}{2}\ell_{i}(2n-\ell_{i})c_i	},
		&
		\fun{S}^{\asymp}_{\mcal{X}^{1}[I]}[r] &=
			q^{-\tfrac{1}{2}(2n-1)}\fun{S}^{\asymp}_{\mcal{X}^{0}[I]}[r].
	\end{align*}
	Now, we apply \cref{lem:MultiSum} to above summation, where the index set $\mfrak{i}$ is $\Set*{1,\cdots,t}$ and the sequence $\bm{\mu}$ is 
	\begin{flalign*}
		&& 
			\mu_1 &= 2n-1,\\
		&&
			\mu_i &= \tfrac{1}{2}\ell_{i}(2n-\ell_{i}).
			& \mathllap{(1 < i \le t)}
	\end{flalign*}
	Depending on $n$ and $\ell_{t}$, there are two possibilities: $2n-1 > \tfrac{1}{2}\ell_{t}(2n-\ell_{t})$ and $2n-1 < \tfrac{1}{2}\ell_{t}(2n-\ell_{t})$. 
	If $2n-1 > \tfrac{1}{2}\ell_{t}(2n-\ell_{t})$, then we have $\mfrak{i}_{\max}=\Set*{1}$, $\mu_{\max}=2n-1$, and 
	\begin{equation*}
		\fun{S}^{\asymp}_{\mcal{X}^{0}[I]}[r] \sim  
			\prod_{	i=2	}^{t}
				\left(	q^{(2n-1)-\tfrac{1}{2}\ell_{i}(2n-\ell_{i})} - 1	\right)^{-1}
			\cdot q^{	(2n-1) r	}.
	\end{equation*}
	If $2n-1 < \tfrac{1}{2}\ell_{t}(2n-\ell_{t})$, then we have 
	$\mfrak{i}_{\max}=\Set*{t}$, $\mu_{\max}=\tfrac{1}{2}\ell_{t}(2n-\ell_{t})$, and
	\begin{equation*}
		\fun{S}^{\asymp}_{\mcal{X}^{0}[I]}[r] \sim  
			\left(	q^{\tfrac{1}{2}\ell_{t}(2n-\ell_{t})-(2n-1)} - 1	\right)^{-1}
			\prod_{	i=2	}^{t-1}
				\left(	q^{\tfrac{1}{2}(\ell_{t}-\ell_{i})(2n-\ell_{t}-\ell_{i})} - 1	\right)^{-1}
			\cdot q^{	\tfrac{1}{2}\ell_{t}(2n-\ell_{t}) r	}.
	\end{equation*}
	Since $\fun{S}_{\mcal{X}^{0}[I]}[r]\asymp\fun{S}^{\asymp}_{\mcal{X}^{0}[I]}[r]$, $\fun{S}_{\mcal{X}^{1}[I]}[r]\asymp\fun{S}^{\asymp}_{\mcal{X}^{1}[I]}[r]$, and $\fun{S}^{\asymp}_{\mcal{X}^{1}[I]}[r]=q^{-\tfrac{1}{2}(2n-1)}\fun{S}^{\asymp}_{\mcal{X}^{0}[I]}[r]$, we see that $\fun{S}_{(\mcal{X}^{0}\cup\mcal{X}^{1})(I)}$ has order $\max\Set*{2n-1,\tfrac{1}{2}\ell_{t}(2n-\ell_{t})}$ and degree $0$. 
\end{step}

\subsection{Dominant types for \texorpdfstring{$\fun{S}_{\mcal{V}[I]}[r]$}{SVIr}}\label{subsec:Bn:dominant}
We are going to estimate the asymptotic growth of each $\fun{S}_{\mcal{V}[I]}[r]$ and figure out the \emph{dominant} types, namely the types for which $\fun{S}_{\mcal{V}[I]}[r]$ is dominant. 

Let $I$ be a type and follow \cref{con:type}. 
If $\ell_{1}>1$, 
then $\mcal{V}[I]$ is between $\mcal{V}_{\dagger}[I]$ and $\mcal{X}^{0}[I]$ by \cref{figure:VerticesOfIBnl1}. 
Therefore,
\begin{equation*}
	\fun{S}_{\mcal{X}^{0}[I]}[r]\gg
	\fun{S}_{\mcal{V}[I]}[r]\gg
	\fun{S}_{\mcal{V}_{\dagger}[I]}[r].
\end{equation*}
From \cref{subsec:Bn:AsymptoticOfSdaggerI,step:Bn:AsymptoticOfSasympXI:l1}, we see that both $\fun{S}_{\mcal{V}_{\dagger}[I]}$ and $\fun{S}_{\mcal{X}^{0}[I]}$ have order $\tfrac{1}{2}\ell_{t}(2n-\ell_{t})$ and degree $0$. 
Note that  
\begin{equation*}
	\tfrac{1}{2}\ell_{t}(2n-\ell_{t}) \le \tfrac{n^2}{2},
\end{equation*}
where the equality holds exactly when $\ell_{t}=n$. 
We thus see that $I$ is dominant among those satisfying $\ell_{1}(I)>1$ if and only if $n\notin I$. 
In that case, $\fun{S}_{\mcal{V}[I]}[r]$ has order $\tfrac{n^2}{2}$ and degree $0$. 

If $\ell_{1}=1$, 
then $\mcal{V}[I]$ is between $\mcal{V}_{\dagger}[I]$ and $\mcal{X}^{0}[I]\cup\mcal{X}^{1}[I]$ by \cref{figure:VerticesOfIBne1}. 
Therefore,
\begin{equation*}
	\fun{S}_{(\mcal{X}^{0}\cup\mcal{X}^{1})(I)}[r]\gg
	\fun{S}_{\mcal{V}[I]}[r]\gg
	\fun{S}_{\mcal{V}_{\dagger}[I]}[r].
\end{equation*}
Depending on $n$ and $\ell_{t}$, there are two  there are two possibilities: $2n-1 > \tfrac{1}{2}\ell_{t}(2n-\ell_{t})$ and $2n-1 < \tfrac{1}{2}\ell_{t}(2n-\ell_{t})$. 
From \cref{subsec:Bn:AsymptoticOfSdaggerI,step:Bn:AsymptoticOfSasympXI:e1}, we see that $\fun{S}_{\mcal{V}_{\dagger}[I]}$, $\fun{S}_{\mcal{X}^{0}[I]}$, and $\fun{S}_{\mcal{X}^{1}[I]}$ have the same order and degree. 
Hence, $\fun{S}_{\mcal{V}[I]}[r]$ has the same order and degree with them. 
If $n=3$, then we must have $\ell_{t}\le 3$ and hence $2n-1 = 5 > \tfrac{1}{2}\ell_{t}(2n-\ell_{t})$. 
Then $\fun{S}_{\mcal{V}[I]}[r]$ has order $5$ and degree $0$. 
In this case, all types $I$ satisfying $\ell_{1}(I)=1$ are dominant. 
If $n\ge 4$, then we may have $2n-1 < \tfrac{1}{2}\ell_{t}(2n-\ell_{t})$. 
Note that 
\begin{equation*}
	\tfrac{1}{2}\ell_{t}(2n-\ell_{t}) \le \tfrac{n^2}{2},
\end{equation*}
where the equality holds exactly when $\ell_{t}=n$. 
We thus see that $I$ is dominant if and only if $n\notin I$. 
In that case, $\fun{S}_{\mcal{V}[I]}[r]$ has order $\tfrac{n^2}{2}$ and degree $0$. 

To summarize, when $n=3$, a type $I$ is dominant if and only if $1\notin I$; when $n\ge 4$, a type $I$ is dominant if and only if $n\notin I$.

\subsection{Asymptotic growth of dominant \texorpdfstring{$\fun{S}_{\mcal{X}^{0}[I]}[r]$}{SX0Ir} and \texorpdfstring{$\fun{S}_{\mcal{X}^{1}[I]}[r]$}{SX1Ir}}
\label{subsec:Bn:AsymptoticOfSXI}
Now, let $I$ be a type and follow \cref{con:type}. 
We are going to compute the asymptotic growths of $\fun{S}_{\mcal{X}^{0}[I]}[r]$ and $\fun{S}_{\mcal{X}^{1}[I]}[r]$ when $I$ is dominant.

To do this, we pick an arbitrary $x\in\mcal{X}^{0}[I]$ (or $x\in(\mcal{X}^{0}\cup\mcal{X}^{1})(I)$ if $\ell_{1}=1$) and investigate the difference between $2\rho(x)$ and the sum of $\ceil{a(x)}$ for $a(x)>0$. 
To better describe these sums, we follow \cref{con:ell_inverse,con:standard_parity,con:SummationSeq}. 
We will separate the discussion into two cases: \cref{step:Bn:AsymptoticOfSXI:l1} $\ell_{1}>1$ and \cref{step:Bn:AsymptoticOfSXI:e1} $\ell_{1}=1$.

\begin{step}\label{step:Bn:AsymptoticOfSXI:l1}
	We begin with the $\ell_{1}>1$ case. 
	Suppose $x = o+c_1\cdot\tfrac{1}{2}\omega_{\ell_{1}}+\cdots+c_t\cdot\tfrac{1}{2}\omega_{\ell_{t}}\in\mcal{X}^{0}[I]$. By \cref{eq:Roots_Bn}, we have 
	\begin{flalign*}
		&&
		(\chi_{j}-\chi_{j'})(x) &= 
			\tfrac{1}{2}
			\left(
				c_{\ell^{-1}(j)}+\cdots+c_{\ell^{-1}(j')-1}
			\right),
			&\mathllap{(1 \le j < j' \le n)}\\
		&&
		(\chi_{j}+\chi_{j'})(x) &=
				\tfrac{1}{2}
				\left(
					c_{\ell^{-1}(j)}+\cdots+c_{\ell^{-1}(j')-1}
				\right) 
			&\mathllap{(1 \le j < j' \le n)}\\  
		&&
			&\quad +
			c_{\ell^{-1}(j')}+\cdots+c_{t},\\
		&&
		\chi_{j}(x) &= 
			\tfrac{1}{2}
			\left(
				c_{\ell^{-1}(j)}+\cdots+c_{t}
			\right).
			&\mathllap{(1 \le j \le n)}
	\end{flalign*}
	Therefore, we have
	\begin{align*}
		\sum_{a\in\Phi^{+}}\ceil{a(x)} &= 
		\sum_{1 \le j < j' \le n}
		\left(
			\ceil{(\chi_{j}-\chi_{j'})(x)} + 
			\ceil{(\chi_{j}+\chi_{j'})(x)}
		\right) + 
		\sum_{j=1}^{n}\ceil{\chi_{j}(x)} \\
		&= 2\rho(x) + 
		\sum_{1 \le j < j' \le n}
			\overline{	c_{\ell^{-1}(j)}+\cdots+c_{\ell^{-1}(j')-1}	} + 
		\sum_{j=1}^{n}
			\tfrac{1}{2}\overline{	c_{\ell^{-1}(j)}+\cdots+c_{t}	}.
	\end{align*}

	From above analysis, we can define the parity function $e_{\mcal{X}^{0}[I]}$ as follows:
	\begin{align}
		\label{eq:Bn:DefineEpsilonX0:l1}
		e_{\mcal{X}^{0}[I]}(c_1,\cdots,c_{t}) &:=
		\sum_{1 \le i < i' \le t+1}
			(\ell_{i}-\ell_{i-1})(\ell_{i'}-\ell_{i'-1})
			\overline{c_i+\cdots+c_{i'-1}} \\
		\nonumber &\qquad\qquad + 
		\sum_{i=1}^{t}
			\tfrac{1}{2}(\ell_{i}-\ell_{i-1})\overline{	c_{i}+\cdots+c_{t}	}.
	\end{align}
	Then we have 
	\begin{equation*}
		\sum_{a\in\Phi^{+}}\ceil{a(x)} =
			2\rho(x) + e_{\mcal{X}^{0}[I]}(c_1,\cdots,c_{t}).
	\end{equation*}

	Now, we apply \cref{lem:MultiSumSEP} to the following summation.
	\begin{equation*}
		\fun{S}_{\mcal{X}^{0}[I]}[r] = 
			\sum_{\crampedsubstack{
				c_i\in\Z_{>0}\\
				c_1+\cdots+c_t = r
			}}q^{	\sum\limits_{i=1}^{t}\tfrac{1}{2}\ell_{i}(2n-\ell_{i})c_i	+ e_{\mcal{X}^{0}[I]}(c_1,\cdots,c_{t})}.
	\end{equation*}
	Note that the index set $\mfrak{i}$ is $\Set*{1,\cdots,t}$ and the sequence $\bm{\mu}$ is 
	\begin{flalign*}
		&&	\mu_i&=\tfrac{1}{2}\ell_{i}(2n-\ell_{i}). & \mathllap{(1 \le i \le t)}
	\end{flalign*} 
	
	Now, let $I$ be a dominant type, namely $n\notin I$. 
	Then we have $\mfrak{i}_{\max}=\Set*{t}$, $\ell_{t}=n$, and $\mu_{\max} = \tfrac{n^2}{2}$. 
	Therefore, 
	\begin{equation}
		\label{eq:Bn:AsymptoticOfSX0I:l1}
		\fun{S}_{\mcal{X}^{0}[I]}[r] \sim 
			C_{\mcal{X}^{0}[I]}\cdot 
			\left(	
				\left(
					\sum_{\vect{s}\in\F_2^{t}}
						\fun{E}_{\mcal{X}^{0}[I]}(\vect{s})
				\right) + 
				\left(
					\sum_{\vect{s}\in\F_2^{t}}
						(-1)^{\vect{1}\cdot\vect{s}}
						\fun{E}_{\mcal{X}^{0}[I]}(\vect{s})
				\right)(-1)^{r}	
			\right)
			\cdot q^{\tfrac{n^2}{2} r},
	\end{equation}
	where the constant $C_{\mcal{X}^{0}[I]}$ and the function $\fun{E}_{\mcal{X}^{0}[I]}\colon\F_2^{t}\to\mbb{Q}[q;-]$ are defined as follows:
	\begin{align}
		\label{eq:Bn:EX0IandCX0I:l1}%
		C_{\mcal{X}^{0}[I]} &:=
			\frac{1}{2}
			\prod_{i=1}^{t-1}
				\left(	q^{	(n-\ell_{i})^2	}-1	\right)^{-1},	&
		\fun{E}_{\mcal{X}^{0}[I]}(\vect{s}) &:= 
			q^{	e_{\mcal{X}^{0}[I]}(\vect{s}) +
					\sum\limits_{i=1}^{t-1}
						\tfrac{1}{2}(n-\ell_{i})^2 s_i	}.
	\end{align}
\end{step}

\begin{step}\label{step:Bn:AsymptoticOfSXI:e1}
	Now, we turn to $\ell_{1}=1$ case. 
	Let $\square$ be either $0$ or $1$.
	Suppose
	\begin{equation*}
		x = o+(c_1-\tfrac{1}{2}\cdot\square)\cdot\omega_{1}+c_2\cdot\tfrac{1}{2}\omega_{\ell_{2}}+\cdots+c_t\cdot\tfrac{1}{2}\omega_{\ell_{t}}\in\mcal{X}^{\square}[I],
	\end{equation*}
	By \cref{eq:Roots_Bn}, we have
	\begin{flalign*}
		&&
		(\chi_{1}-\chi_{j})(x) &= 
			(c_1-\tfrac{1}{2}\cdot\square) + 
			\tfrac{1}{2}
			\left(
				c_{2}+\cdots+c_{\ell^{-1}(j)-1}
			\right),
			&\mathllap{(1 < j \le n)}\\
		&&
		(\chi_{j}-\chi_{j'})(x) &= 
			\tfrac{1}{2}
			\left(
				c_{\ell^{-1}(j)}+\cdots+c_{\ell^{-1}(j')-1}
			\right),
			&\mathllap{(1 < j < j' \le n)}\\
		&&
		\MoveEqLeft[4]
		(\chi_{1}+\chi_{j})(x)
		&&\mathllap{(1 < j \le n)} \\
		&&&=  
				(c_1-\tfrac{1}{2}\cdot\square) + 
				\tfrac{1}{2}
				\left(
					c_{2}+\cdots+c_{\ell^{-1}(j)-1}
				\right) + 
				c_{\ell^{-1}(j)}+\cdots+c_{t}, \\
		&&
		\MoveEqLeft[4]
		(\chi_{j}+\chi_{j'})(x)
		&&\mathllap{(1 < j < j' \le n)} \\
		&&&=  
				\tfrac{1}{2}
				\left(
					c_{\ell^{-1}(j)}+\cdots+c_{\ell^{-1}(j')-1}
				\right) + 
				c_{\ell^{-1}(j')}+\cdots+c_{t}, \\
		&&
		\chi_{1}(x) &= 
			(c_1-\tfrac{1}{2}\cdot\square) + 
			\tfrac{1}{2}
			\left(
				c_{2}+\cdots+c_{t}
			\right),\\
		&&
		\chi_{j}(x) &= 
			\tfrac{1}{2}
			\left(
				c_{\ell^{-1}(j)}+\cdots+c_{t}
			\right).
			&\mathllap{(1 < j \le n)}
	\end{flalign*}
	Therefore, we have
	\begin{align*}
		\sum_{a\in\Phi^{+}}\ceil{a(x)} 
		&= 2\rho(x) + 
		\sum_{j=2}^{n}
			\overline{	c_{2}+\cdots+c_{\ell^{-1}(j)-1}-\square	} +
			\sum_{2 \le j < j' \le n}
				\overline{	c_{\ell^{-1}(j)}+\cdots+c_{\ell^{-1}(j')-1}	} \\
		&\qquad+ 
		\tfrac{1}{2}\overline{	c_{2}+\cdots+c_{t}-\square	} + 
		\sum_{j=2}^{n}
			\tfrac{1}{2}\overline{	c_{\ell^{-1}(j)}+\cdots+c_{t}	}.
	\end{align*}

	From above analysis, we can define the parity function $e_{\mcal{X}^{\square}[I]}$ ($\square=0,1$) as follows:
	\begin{align}
	\label{eq:Bn:DefineEpsilonX:e1}
		e_{\mcal{X}^{\square}[I]}(c_1,\cdots,c_{t}) &:=
		\sum_{i=1}^{t}
			(\ell_{i+1}-\ell_{i})
			\overline{	c_{2}+\cdots+c_{i}-\square	} \\
		\nonumber &\qquad +  
		\sum_{2 \le i < i' \le t+1}
			(\ell_{i}-\ell_{i-1})(\ell_{i'}-\ell_{i'-1})
			\overline{c_i+\cdots+c_{i'-1}} \\
		\nonumber &\qquad + 
		\tfrac{1}{2}\overline{c_{2}+\cdots+c_{t}-\square} + 
		\sum_{i=2}^{t}
			\tfrac{1}{2}(\ell_{i}-\ell_{i-1})
			\overline{	c_{i}+\cdots+c_{t}	}.
	\end{align}
	Then we have 
	\begin{equation*}
		\sum_{a\in\Phi^{+}}\ceil{a(x)} =
			2\rho(x) + e_{\mcal{X}^{\square}[I]}(c_1,\cdots,c_{t}).
	\end{equation*}

	Now, we apply \cref{lem:MultiSumSEP} to the following summation ($\square=0,1$).
	\begin{equation*}
		\fun{S}_{\mcal{X}^{\square}[I]}[r] = 
			\sum_{\crampedsubstack{
				c_i\in\Z_{>0}\\
				c_1+\cdots+c_t = r
			}}q^{	(2n-1)(c_{1}-\tfrac{1}{2}\cdot\square) + 
				\sum\limits_{i=2}^{t}
					\tfrac{1}{2}\ell_{i}(2n-\ell_{i})c_i + 
				e_{\mcal{X}^{\square}[I]}(c_1,\cdots,c_{t})}.
	\end{equation*}
	Note that the index set $\mfrak{i}$ is $\Set*{1,\cdots,t}$ and the sequence $\bm{\mu}$ is 
	\begin{flalign*}
		&& 
			\mu_1 &= 2n-1,\\
		&&
			\mu_i &= \tfrac{1}{2}\ell_{i}(2n-\ell_{i}).
			& \mathllap{(1 < i \le t)}
	\end{flalign*}
	
	Now, let $I$ be a dominant type.
	Depending on $n$, there are two cases: 
	\cref{step:Bn:AsymptoticOfSXI:e1:n=3} $n=3$ and \cref{step:Bn:AsymptoticOfSXI:e1:n>3} $n\ge 4$. 
	
	\begin{substep}\label{step:Bn:AsymptoticOfSXI:e1:n=3}
		When $n=3$, this means $1\notin I$. 
		Then we have $\mfrak{i}_{\max}=\Set*{1}$ and $\mu_{\max} = 5$.  
		Therefore, for $\square=0,1$, we have
		\begin{equation*}
			\fun{S}_{\mcal{X}^{\square}[I]}[r] \sim 
				C_{\mcal{X}^{\square}[I]}\cdot
				\left(	
					\left(
						\sum_{\vect{s}\in\F_2^{t}}
							\fun{E}_{\mcal{X}^{\square}[I]}(\vect{s})
					\right) + 
					\left(
						\sum_{\vect{s}\in\F_2^{t}}
							(-1)^{\vect{1}\cdot\vect{s}}
							\fun{E}_{\mcal{X}^{\square}[I]}(\vect{s})
					\right)(-1)^{r}	
				\right)
				\cdot q^{5 r},
		\end{equation*}
		where the constant $C_{\mcal{X}^{\square}[I]}$ and the function $\fun{E}_{\mcal{X}^{\square}[I]}\colon\F_2^{t}\to\mbb{Q}[q;-]$ are defined as follows:
		\begin{align}
			\label{eq:B3:CXI:e1}
			C_{\mcal{X}^{\square}[I]} &:=
				\frac{1}{2}q^{	-\tfrac{5}{2}\cdot\square	}
				\prod_{i=2}^{t}
					\left(	q^{	\left(10-\ell_{i}(6-\ell_{i})\right)	}-1	\right)^{-1},	\\
			\label{eq:B3:EXI:e1}
			\fun{E}_{\mcal{X}^{\square}[I]}(\vect{s}) &:= 
				q^{	e_{\mcal{X}^{\square}[I]}(\vect{s}) +
						\sum\limits_{i=2}^{t}
							\left(5-\tfrac{1}{2}\ell_{i}(6-\ell_{i})\right) s_i	}.
		\end{align}
		From the definition \cref{eq:Bn:DefineEpsilonX:e1} of $e_{\mcal{X}^{\square}[I]}$, we have 
		\begin{equation*}
			\fun{E}_{\mcal{X}^{\square}[I]}(0,s_2,\cdots,s_t) = 
			\fun{E}_{\mcal{X}^{\square}[I]}(1,s_2,\cdots,s_t).
		\end{equation*}
		Therefore, we have 
		\begin{equation}
			\label{eq:B3:AsymptoticOfSXI:e1}
			\fun{S}_{\mcal{X}^{\square}[I]}[r] \sim 
				C_{\mcal{X}^{\square}[I]}\cdot
				\left(	
					\sum_{\vect{s}\in\F_2^{t}}
						\fun{E}_{\mcal{X}^{\square}[I]}(\vect{s})
				\right)
				\cdot q^{5 r},
		\end{equation}
	\end{substep}

	\begin{substep}\label{step:Bn:AsymptoticOfSXI:e1:n>3}
		When $n\ge 4$, $I$ is dominant means $n\notin I$. 
		Then we have $\mfrak{i}_{\max}=\Set*{t}$, $\ell_{t}=n$, and $\mu_{\max} = \tfrac{n^2}{2}$. 
		Therefore, for $\square=0,1$, we have
		\begin{equation}
			\label{eq:Bn:AsymptoticOfSXI:e1}
			\fun{S}_{\mcal{X}^{\square}[I]}[r] \sim 
				C_{\mcal{X}^{\square}[I]}\cdot
				\left(	
					\left(
						\sum_{\vect{s}\in\F_2^{t}}
							\fun{E}_{\mcal{X}^{\square}[I]}(\vect{s})
					\right) + 
					\left(
						\sum_{\vect{s}\in\F_2^{t}}
							(-1)^{\vect{1}\cdot\vect{s}}
							\fun{E}_{\mcal{X}^{\square}[I]}(\vect{s})
					\right)(-1)^{r}	
				\right)
				\cdot q^{\tfrac{n^2}{2} r},
		\end{equation}
		where the constant $C_{\mcal{X}^{\square}[I]}$ and the function $\fun{E}_{\mcal{X}^{\square}[I]}\colon\F_2^{t}\to\mbb{Q}[q;-]$ are defined as follows:
		\begin{align}
			\label{eq:Bn:CXI:e1}
			C_{\mcal{X}^{\square}[I]} &:=
				\frac{1}{2}q^{	-\tfrac{1}{2}(2n-1)\cdot\square	}
				\left(	q^{	n^2 - 2(2n-1)	}-1	\right)^{-1}
				\prod_{i=2}^{t-1}
					\left(	q^{	(n-\ell_{i})^2	}-1	\right)^{-1},	\\
			\label{eq:Bn:EXI:e1}
			\fun{E}_{\mcal{X}^{\square}[I]}(\vect{s}) &:= 
				q^{	e_{\mcal{X}^{\square}[I]}(\vect{s}) + 
						\left(	\tfrac{n^2}{2} - (2n-1)	\right) s_1 + 
						\sum\limits_{i=2}^{t-1}
							\tfrac{1}{2}(n-\ell_{i})^2 s_i	}.
		\end{align}
	\end{substep}
\end{step}

\subsection{Asymptotic growth of dominant \texorpdfstring{$\fun{S}_{\mcal{X}_{J}[I]}[r]$}{SXJIr}}
\label{subsec:Bn:AsymptoticOfSXJI}
Now, let $I$ be a type and follow \cref{con:type}. 
We are going to analyze $\fun{S}_{\mcal{X}_{J}[I]}[r]$. 

Suppose $x\in\mcal{X}_{J}[I,r]$, where $I\cap J = \emptyset$. 
Since $\mcal{X}_{\emptyset} = \mcal{V}_{\dagger}$, by \cref{lem:XIJfromXI0Bn}, we can write $x$ as 
$x_0-\sum\limits_{j\in J}\tfrac{1}{2}\omega_{j}$, 
where $x_0\in\mcal{V}_{\dagger}[I,r+\abs*{J}-\delta(J)]$. 
Then we have 
\begin{equation*}
	\sum_{a\in\Phi^{+}}\ceil{a(x)} = 
		2\rho(x_0) + 
		\sum_{a\in\Phi^{+}}
			\ceil*{-\sum\limits_{j\in J}a(\tfrac{1}{2}\omega_{j})}.
\end{equation*}
Note that the last summation gives an integral constant. 
Then we have 
\begin{equation}\label{eq:Bn:SXJIfromSVdaggerI}
	\fun{S}_{\mcal{X}_{J}[I]}[r] = 
	q^{	\sum\limits_{a\in\Phi^{+}}
				\ceil*{-\sum\limits_{j\in J}a(\tfrac{1}{2}\omega_{j})}	}
	\fun{S}_{\mcal{V}_{\dagger}[I]}[r+\abs*{J}-\delta(J)]. 
\end{equation}

Now, we assume $I$ is dominant. 
We will separate the discussion into two cases: \cref{step:B3:AsymptoticOfSXJI} $n=3$ and \cref{step:Bn:AsymptoticOfSXJI} $n\ge 4$.

\begin{step}\label{step:B3:AsymptoticOfSXJI}
	When $n=3$, this means $\ell_{1}=1$. 
	Then the following $J$ appears in \cref{figure:VerticesOfIBne1}: $\Set*{1},\Set*{1,2},\Set*{2,3}$. 
	In those cases, by \cref{eq:Roots_Bn}, we have 
	\begin{align*}
		\abs*{\Set*{1}}-\delta(\Set*{1}) &= 0, &
			\sum_{a\in\Phi^{+}}
			\ceil*{-a(\tfrac{1}{2}\omega_{1})} &= 0,\\
		\abs*{\Set*{1,2}}-\delta(\Set*{1,2}) &= 1, &
			\sum_{a\in\Phi^{+}}
			\ceil*{-a(\tfrac{1}{2}\omega_{1}+\tfrac{1}{2}\omega_{2})} &= -4,\\
		\abs*{\Set*{2,3}}-\delta(\Set*{2,3}) &= 2, &
			\sum_{a\in\Phi^{+}}
			\ceil*{-a(\tfrac{1}{2}\omega_{2}+\tfrac{1}{2}\omega_{3})} &= -6.
	\end{align*}
	Then by \cref{eq:Bn:AsymptoticOfSdaggerI:l,eq:Bn:SXJIfromSVdaggerI}, we have
	\begin{align}
		\label{eq:B3:AsymptoticOfSXJI:1}
		\fun{S}_{\mcal{X}_{\Set*{1}}[I]}[r] &\sim  
			\prod_{	i=2	}^{t}
				\left(	q^{10-\ell_{i}(6-\ell_{i})} - 1	\right)^{-1}
			\cdot q^{	5 r	}, \\
		\label{eq:B3:AsymptoticOfSXJI:12}
		\fun{S}_{\mcal{X}_{\Set*{1,2}}[I]}[r] &\sim  
			q\cdot
			\prod_{	i=2	}^{t}
				\left(	q^{10-\ell_{i}(6-\ell_{i})} - 1	\right)^{-1}
			\cdot q^{	5 r	}, \\
		\label{eq:B3:AsymptoticOfSXJI:23}
		\fun{S}_{\mcal{X}_{\Set*{2,3}}[I]}[r] &\sim  
			q^{	4	}\cdot
			\prod_{	i=2	}^{t}
				\left(	q^{10-\ell_{i}(6-\ell_{i})} - 1	\right)^{-1}
			\cdot q^{	5 r	}.
	\end{align}
\end{step}

\begin{step}\label{step:Bn:AsymptoticOfSXJI}
	Now, we assume $n\ge 4$. 
	Then $I$ is dominant means $\ell_{t}=n$. 
	Depending on $\ell_{1}$, there are two cases: 
	\cref{step:Bn:AsymptoticOfSXJI:l1} $\ell_{1}>1$ and \cref{step:Bn:AsymptoticOfSXJI:e1} $\ell_{1}=1$. 

	\begin{substep}\label{step:Bn:AsymptoticOfSXJI:l1}
		If $\ell_{1}>1$, then the following $J$ appears in \cref{figure:VerticesOfIBnl1}: $\Set*{2,3},\cdots,\Set*{n-1,n}$. 
		In those cases, we have $\abs*{J}-\delta(J)=2$ and by \cref{eq:Roots_Bn}, 
		\begin{flalign}\label{eq:Bn:SumOfJ}
			&&
			\sum_{a\in\Phi^{+}}
					\ceil*{-a(\tfrac{1}{2}\omega_{j}+\tfrac{1}{2}\omega_{j+1})}
			&=	-j(2n-1-j). 
			&\mathllap{(1<j<n)}
		\end{flalign}
		Then by \cref{eq:Bn:AsymptoticOfSdaggerI,eq:Bn:SXJIfromSVdaggerI}, we have ($1<j<n$)
		\begin{equation}\label{eq:Bn:AsymptoticOfSXJI:l1}
			\fun{S}_{\mcal{X}_{\Set*{j,j+1}}[I]}[r] \sim  
				q^{	n^2-j(2n-1-j)	}\cdot
				\prod_{	i=1	}^{t-1}
					\left(	q^{(n-\ell_{i})^2} - 1	\right)^{-1}
				\cdot	\tfrac{1}{2}\left(1+(-1)^{r}\right)
				\cdot q^{	\tfrac{n^2}{2} r	}.
		\end{equation}
	\end{substep}

	\begin{substep}\label{step:Bn:AsymptoticOfSXJI:e1}
		If $\ell_{1}=1$, then the following $J$ appears in \cref{figure:VerticesOfIBne1}: $\Set*{1}$, $\Set*{1,2}$, $\cdots$, $\Set*{n-1,n}$. 
		When $J=\Set*{1}$, we have $\abs*{\Set*{1}}-\delta(\Set*{1}) = 0$ and 
		\begin{equation*}
			\sum_{a\in\Phi^{+}}\ceil*{-a(\tfrac{1}{2}\omega_{1})} = 0.
		\end{equation*}
		Then by \cref{eq:Bn:AsymptoticOfSdaggerI:s,eq:Bn:SXJIfromSVdaggerI}, we have
		\begin{align}\label{eq:Bn:AsymptoticOfSXJI:e1:1}
			\fun{S}_{\mcal{X}_{\Set*{1}}[I]}[r] &\sim  
				\left(	q^{n^2-2(2n-1)} - 1	\right)^{-1}
				\prod_{	i=2	}^{t-1}
					\left(	q^{(n-\ell_{i})^2} - 1	\right)^{-1} \\
			&\nonumber\qquad
				\cdot\tfrac{1}{2}
				\left(	
					\left(	
						1+q^{\tfrac{n^2}{2}-(2n-1)}	
					\right) 
					+ 
					\left(	
						1-q^{\tfrac{n^2}{2}-(2n-1)}	
					\right)(-1)^{r}	
				\right) 
				\cdot q^{	\tfrac{n^2}{2} r	}.
		\end{align}
		When $J=\Set*{1,2}$, we have $\abs*{\Set*{1,2}}-\delta(\Set*{1,2}) = 1$ and 
		\begin{equation*}
			\sum_{a\in\Phi^{+}}\ceil*{-a(\tfrac{1}{2}\omega_{1}+\tfrac{1}{2}\omega_{2})} = -(2n-2).
		\end{equation*}
		Then by \cref{eq:Bn:AsymptoticOfSdaggerI:s,eq:Bn:SXJIfromSVdaggerI}, we have
		\begin{align}\label{eq:Bn:AsymptoticOfSXJI:e1:12}
			\fun{S}_{\mcal{X}_{\Set*{1,2}}[I]}[r] &\sim  
				q^{	\tfrac{n^2}{2} -(2n-2) }\cdot
				\left(	q^{n^2-2(2n-1)} - 1	\right)^{-1}
				\prod_{	i=2	}^{t-1}
					\left(	q^{(n-\ell_{i})^2} - 1	\right)^{-1} \\
			&\nonumber\qquad
				\cdot\tfrac{1}{2}
				\left(	
					\left(	
						1+q^{\tfrac{n^2}{2}-(2n-1)}	
					\right) 
					+ 
					\left(	
						1-q^{\tfrac{n^2}{2}-(2n-1)}	
					\right)(-1)^{r}	
				\right) 
				\cdot q^{	\tfrac{n^2}{2} r	}.
		\end{align}
		When $J=\Set*{j,j+1}$ ($1<j<n$), we have $\abs*{J}-\delta(J)=2$ and \cref{eq:Bn:SumOfJ}. 
		Then by \cref{eq:Bn:AsymptoticOfSdaggerI:s,eq:Bn:SXJIfromSVdaggerI}, we have 
		\begin{align}\label{eq:Bn:AsymptoticOfSXJI:e1:J}
			\fun{S}_{\mcal{X}_{\Set*{j,j+1}}[I]}[r] &\sim  
				q^{	n^2-j(2n-1-j) }\cdot
				\left(	q^{n^2-2(2n-1)} - 1	\right)^{-1}
				\prod_{	i=2	}^{t-1}
					\left(	q^{(n-\ell_{i})^2} - 1	\right)^{-1} \\
			&\nonumber\qquad
				\cdot\tfrac{1}{2}
				\left(	
					\left(	
						1+q^{\tfrac{n^2}{2}-(2n-1)}	
					\right) 
					+ 
					\left(	
						1-q^{\tfrac{n^2}{2}-(2n-1)}	
					\right)(-1)^{r}	
				\right) 
				\cdot q^{	\tfrac{n^2}{2} r	}.
		\end{align}
	\end{substep}
\end{step}

\subsection{Asymptotic growth of dominant \texorpdfstring{$\fun{S}_{\mcal{V}[I]}[r]$}{SVIr}}
\label{subsec:Bn:AsymptoticOfSVI}
We are now able to compute the asymptotic growth of $\fun{S}_{\mcal{V}[I]}[r]$ when $I$ is dominant. 
We will separate the discussion into two cases: \cref{step:B3:AsymptoticOfSVI} $n=3$ and \cref{step:Bn:AsymptoticOfSVI} $n\ge 4$.

\begin{step}\label{step:B3:AsymptoticOfSVI}
	When $n=3$, the dominant types are $\Set*{2,3}$, $\Set*{2}$, $\Set*{3}$, and $\emptyset$. 
	By \cref{figure:VerticesOfIBne1}, we have (where zero summations are omitted)
	\begin{align*}
		\fun{S}_{\mcal{V}[\Set*{2,3}]}[r] &= 
			\fun{S}_{\mcal{X}^{0}[\Set*{2,3}]}[r] + 
			\fun{S}_{\mcal{X}^{1}[\Set*{2,3}]}[r] - 
			\fun{S}_{\mcal{X}_{\Set*{1}}[\Set*{2,3}]}[r], \\
		\fun{S}_{\mcal{V}[\Set*{2}]}[r] &= 
			\fun{S}_{\mcal{X}^{0}[\Set*{2}]}[r] + 
			\fun{S}_{\mcal{X}^{1}[\Set*{2}]}[r] - 
			\fun{S}_{\mcal{X}_{\Set*{1}}[\Set*{2}]}[r], \\
		\fun{S}_{\mcal{V}[\Set*{3}]}[r] &= 
			\fun{S}_{\mcal{X}^{0}[\Set*{3}]}[r] + 
			\fun{S}_{\mcal{X}^{1}[\Set*{3}]}[r] - 
			\fun{S}_{\mcal{X}_{\Set*{1}}[\Set*{3}]}[r] - 
			\fun{S}_{\mcal{X}_{\Set*{1,2}}[\Set*{3}]}[r], \\
		\fun{S}_{\mcal{V}[\emptyset]}[r] &= 
			\fun{S}_{\mcal{X}^{0}[\emptyset]}[r] + 
			\fun{S}_{\mcal{X}^{1}[\emptyset]}[r] - 
			\fun{S}_{\mcal{X}_{\Set*{1}}[\emptyset]}[r] - 
			\fun{S}_{\mcal{X}_{\Set*{1,2}}[\emptyset]}[r] - 
			\fun{S}_{\mcal{X}_{\Set*{2,3}}[\emptyset]}[r].
	\end{align*}
	Therefore, by \cref{eq:B3:AsymptoticOfSXI:e1,eq:B3:AsymptoticOfSXJI:1,eq:B3:AsymptoticOfSXJI:12,eq:B3:AsymptoticOfSXJI:23}, we have 
	\begin{align}
		\label{eq:B3:AsymptoticOfSI:23}
		\fun{S}_{\mcal{V}[\Set*{2,3}]}[r] &\sim 
			\left(
				1 + 1 - 1
			\right)q^{5r} = q^{5r}, \\
		\label{eq:B3:AsymptoticOfSI:2}
		\fun{S}_{\mcal{V}[\Set*{2}]}[r] &\sim 
			\frac{
				(q^2+1) + (q+1) - 1
			}{q-1}q^{5r} = 
			\frac{q^2+q+1}{q-1} q^{5r}, \\
		\label{eq:B3:AsymptoticOfSI:3}
		\fun{S}_{\mcal{V}[\Set*{3}]}[r] &\sim 
			\frac{
				(q^4+1) + (q+1) - 1 - q
			}{q^2-1}q^{5r} = 
			\frac{q^4+1}{q^2-1} q^{5r}, \\
		\label{eq:B3:AsymptoticOfSI:emptyset}
		\fun{S}_{\mcal{V}[\emptyset]}[r] &\sim 
			\frac{
				(2q^4+q^2+1) + (q^2+2q+1) - 1 - q - q^4
			}{
				(q-1)(q^2-1)
			}q^{5r} \\
		\nonumber&= 
			\frac{q^4+2q^2+q+1}{(q-1)(q^2-1)} q^{5r}.
	\end{align}
\end{step}

\begin{step}\label{step:Bn:AsymptoticOfSVI}
	Now, we assume $n\ge 4$. Then $I$ is dominant exactly when $n\notin I$. 
	Depending on $\ell_{1}$, there are two cases: 
	\cref{step:Bn:AsymptoticOfSVI:l1} $\ell_{1}>1$ and \cref{step:Bn:AsymptoticOfSVI:e1} $\ell_{1}=1$. 

	\begin{substep}\label{step:Bn:AsymptoticOfSVI:l1}
		If $\ell_{1}>1$, then by \cref{figure:VerticesOfIBnl1}, we have (including the zero summations)
		\begin{equation*}
			\fun{S}_{\mcal{V}[I]}[r] = 
				\fun{S}_{\mcal{X}^{0}[I]}[r] - 
				\sum_{j=2}^{n-1}
					\fun{S}_{\mcal{X}_{\Set*{j,j+1}}[I]}[r].
		\end{equation*}
		Therefore, by \cref{eq:Bn:AsymptoticOfSX0I:l1,eq:Bn:AsymptoticOfSXJI:l1,eq:Bn:EX0IandCX0I:l1}, we have 
		\begin{equation}\label{eq:Bn:AsymptoticOfSI:l1}
			\fun{S}_{\mcal{V}[I]}[r] \sim 
				\frac{1}{2}
				\prod_{i=1}^{t-1}
					\left(	q^{	(n-\ell_{i})^2	}-1	\right)^{-1}
				\cdot
				\left(			
					C_{I,0} + C_{I,1}(-1)^{r}	
				\right)			 
				\cdot  q^{\tfrac{n^2}{2} r},
		\end{equation}
		where the constants $C_{I,0}$ and $C_{I,1}$ are defined as follows:
		\begin{align}
			\label{eq:Bn:AsymptoticOfSI:l1:C0}
			C_{I,0} &:= 
				\sum_{\vect{s}\in\F_2^{t}}
					q^{	e_{\mcal{X}^{0}[I]}(\vect{s}) +
						\sum\limits_{i=1}^{t-1}
							\tfrac{1}{2}(n-\ell_{i})^2 s_i	} - 
				\sum_{\crampedsubstack{
					1<j<n \\ \Set*{j,j+1}\cap I = \emptyset
				}}
					q^{	n^2-j(2n-1-j)	}, \\
			\label{eq:Bn:AsymptoticOfSI:l1:C1}
			C_{I,1} &:= 
				\sum_{\vect{s}\in\F_2^{t}}
					(-1)^{\vect{1}\cdot\vect{s}}
					q^{	e_{\mcal{X}^{0}[I]}(\vect{s}) +
						\sum\limits_{i=1}^{t-1}
							\tfrac{1}{2}(n-\ell_{i})^2 s_i	} - 
				\sum_{\crampedsubstack{
					1<j<n \\ \Set*{j,j+1}\cap I = \emptyset
				}}
					q^{	n^2-j(2n-1-j)	}.
		\end{align}
		Note that the multivariable parity function $e_{\mcal{X}^{0}[I]}$ is defined in \cref{eq:Bn:DefineEpsilonX0:l1}. 
	\end{substep}
	
	\begin{substep}\label{step:Bn:AsymptoticOfSVI:e1}
		If $\ell_{1}=1$, then by \cref{figure:VerticesOfIBne1}, we have (including the zero summations)
		\begin{equation*}
			\fun{S}_{\mcal{V}[I]}[r] = 
				\fun{S}_{\mcal{X}^{0}[I]}[r] + 
				\fun{S}_{\mcal{X}^{0}[I]}[r] - 
				\fun{S}_{\mcal{X}_{\Set*{1}}[I]}[r] - 
				\sum_{j=1}^{n-1}
					\fun{S}_{\mcal{X}_{\Set*{j,j+1}}[I]}[r].
		\end{equation*}
		Therefore, by \cref{eq:Bn:AsymptoticOfSXI:e1,eq:Bn:CXI:e1,eq:Bn:EXI:e1,eq:Bn:AsymptoticOfSXJI:e1:1,eq:Bn:AsymptoticOfSXJI:e1:12,eq:Bn:AsymptoticOfSXJI:e1:J}, we have 
		\begin{equation}\label{eq:Bn:AsymptoticOfSI:e1}
			\fun{S}_{\mcal{V}[I]}[r] \sim 
				\frac{1}{2}
				\left(	q^{	n^2 - 2(2n-1)	}-1	\right)^{-1}
				\prod_{i=2}^{t-1}
					\left(	q^{	(n-\ell_{i})^2	}-1	\right)^{-1}
				\cdot
				\left(			
					C_{I,0} + C_{I,1}(-1)^{r}	
				\right)			 
				\cdot  q^{\tfrac{n^2}{2} r},
		\end{equation}
		where the constants $C_{I,0}$ and $C_{I,1}$ are defined as follows:
		\begin{align}
			\label{eq:Bn:AsymptoticOfSI:e1:C0}
			C_{I,0} &:= 
				\sum_{\square = 0,1}
				\sum_{\vect{s}\in\F_2^{t}}
					q^{	e_{\mcal{X}^{\square}[I]}(\vect{s}) - 
							\tfrac{1}{2}(2n-1)\cdot\square + 
							\left(	\tfrac{n^2}{2} - (2n-1)	\right) s_1 + 
							\sum\limits_{i=2}^{t-1}
								\tfrac{1}{2}(n-\ell_{i})^2 s_i	} \\
			\nonumber&\qquad - 
				\left(
					1 + 
					\delta_{I}(2)q^{	\tfrac{n^2}{2} -(2n-2) } + 
					\sum_{\crampedsubstack{
						1<j<n \\ \Set*{j,j+1}\cap I = \emptyset
					}}
						q^{	n^2-j(2n-1-j)	}
				\right)
				\left(	
					1+q^{\tfrac{n^2}{2}-(2n-1)}	
				\right), \\
			\label{eq:Bn:AsymptoticOfSI:e1:C1}
			C_{I,1} &:= 
				\sum_{\square = 0,1}
				\sum_{\vect{s}\in\F_2^{t}}
					(-1)^{\vect{1}\cdot\vect{s}}
					q^{	e_{\mcal{X}^{\square}[I]}(\vect{s}) - 
							\tfrac{1}{2}(2n-1)\cdot\square + 
							\left(	\tfrac{n^2}{2} - (2n-1)	\right) s_1 + 
							\sum\limits_{i=2}^{t-1}
								\tfrac{1}{2}(n-\ell_{i})^2 s_i	} \\
			\nonumber&\qquad - 
				\left(
					1 + 
					\delta_{I}(2)q^{	\tfrac{n^2}{2} -(2n-2) } + 
					\sum_{\crampedsubstack{
						1<j<n \\ \Set*{j,j+1}\cap I = \emptyset
					}}
						q^{	n^2-j(2n-1-j)	}
				\right)
				\left(	
					1-q^{\tfrac{n^2}{2}-(2n-1)}	
				\right),
		\end{align}
		where $\delta_{I}(i)=0$ if $1\in I$ and $1$ if not.  
		Note that the multivariable parity functions $e_{\mcal{X}^{\square}[I]}$ ($\square=0,1$) are defined in \cref{eq:Bn:DefineEpsilonX:e1}. 
	\end{substep}
\end{step}

\subsection{Asymptotic growths of \texorpdfstring{$\fun{SSA}[r]$}{SSA} and \texorpdfstring{$\fun{SV}[r]$}{SV}}
\label{subsec:Bn:AsymptoticOfS}
We are now able to obtain the asymptotic growth of $\fun{SSA}[r]$. 
By \cref{eq:SimplicialSurfaceAreaFormula}, we have 
\begin{equation}\label{eq:SSA=SumSI:Bn}
	\fun{SSA}[r] = 
	\sum_{I\subset\Delta}
	\frac{
		\mscr{P}_{B_n;I}[q]
	}{
		q^{\fun{deg}[\mscr{P}_{B_n;I}]}
	}\fun{S}_{\mcal{V}[I]}[r]
	\sim 
	\sum_{I\text{ is dominant}}
	\frac{
		\mscr{P}_{B_n;I}[q]
	}{
		q^{\fun{deg}[\mscr{P}_{B_n;I}]}
	}\fun{S}_{\mcal{V}[I]}[r].
\end{equation}
What remains is to plug in the asymptotic growth of dominant $\fun{S}_{\mcal{V}[I]}[r]$. 
We will separate the discussion into  two cases: \cref{step:B3:AsymptoticOfS} $n=3$ and \cref{step:Bn:AsymptoticOfS} $n\ge 4$.

\begin{step}\label{step:B3:AsymptoticOfS}
	When $n=3$, the dominant types are $\Set*{2,3}$, $\Set*{2}$, $\Set*{3}$, and $\emptyset$. 
	By \cref{eq:B3:AsymptoticOfSI:23,eq:B3:AsymptoticOfSI:2,eq:B3:AsymptoticOfSI:3,eq:B3:AsymptoticOfSI:emptyset,eq:PoincareCnI}, we have 
	\begin{equation}\label{eq:B3:Asymp:SSA}
		\fun{SSA}[r] \sim  
			C(3) \cdot q^{	5 r},
	\end{equation}
	where the constant $C(3)$ is defined as follows: 
	\begin{align}\label{eq:B3:Asymp:C}
		C(3) &:= 
		\frac{
			\mscr{P}_{B_3;\Set*{2,3}}[q]
		}{
			q^{\fun{deg}[\mscr{P}_{B_3;\Set*{2,3}}]}
		} + 
		\frac{
			\mscr{P}_{B_3;\Set*{2}}[q]
		}{
			q^{\fun{deg}[\mscr{P}_{B_3;\Set*{2}}]}
		}
		\frac{q^2+q+1}{q-1} + 
		\frac{
			\mscr{P}_{B_3;\Set*{3}}[q]
		}{
			q^{\fun{deg}[\mscr{P}_{B_3;\Set*{3}}]}
		}
		\frac{q^4+1}{q^2-1} \\
		\nonumber&\qquad + 
		\frac{
			\mscr{P}_{B_3;\emptyset}[q]
		}{
			q^{\fun{deg}[\mscr{P}_{B_3;\emptyset}]}
		}
		\frac{q^4+2q^2+q+1}{(q-1)(q^2-1)} \\
	\nonumber &=
		\frac{
			\left(q^6-1\right)
		}{
			\left(q-1\right)q^{5}
		} + 
		\frac{
			\left(q^6-1\right)\left(q^4-1\right)\left(q^2+q+1\right)
		}{
			\left(q^2-1\right)\left(q-1\right)^{2}q^{7}
		} + 
		\frac{
			\left(q^6-1\right)\left(q^4-1\right)\left(q^4+1\right)
		}{
			\left(q-1\right)^{2}\left(	q^2 - 1	\right)q^{8}
		} \\
	\nonumber&\qquad + 
		\frac{
			\left(q^6-1\right)\left(q^4-1\right)\left(q^2-1\right)
			\left(q^4+2q^2+q+1\right)
		}{
			\left(q-1\right)^{4}\left(	q^2 - 1	\right)q^{9}
		} \\
	\nonumber &=
		\frac{
			\left(q^2+q+1\right) 
			\left(q^2-q+1\right) 
			(q+1) 
		}{
			(q-1)^2 q^9
		} \\
	\nonumber&\qquad 	\cdot
		\left(
			q^8 + q^7 + 3 q^6 + q^5 + 5 q^4 + 3 q^3 + 4 q^2 + q + 1
		\right).
	\end{align}
	As a consequence, we have 
	\begin{equation}\label{eq:B3:Asymp:SV}
		\fun{SV}[r] = \sum_{z=0}^{r}\fun{SSA}[z] \sim  
			\frac{q^5}{q^5-1}C(3) \cdot q^{	5 r}.
	\end{equation}
\end{step}

\begin{step}\label{step:Bn:AsymptoticOfS}
	Now, we assume $n\ge 4$. Then $I$ is dominant exactly when $n\notin I$. 
	By \cref{eq:Bn:AsymptoticOfSI:l1,eq:Bn:AsymptoticOfSI:e1}, we have 
	\begin{equation}\label{eq:Bn:Asymp:SSA}
		\fun{SSA}[r] \sim  
			C(n) \cdot q^{	\tfrac{n^2}{2} r},
	\end{equation}
	where the parity $q$-function $C(n)$ is defined as follows: 
	\index[notation]{C (n)@$C(n)$}%
	\begin{align}
		\label{eq:Bn:Asymp:C}
		C(n)(r) &:=
			\sum_{1,n\notin I}
				\frac{
					\mscr{P}_{B_n;I}[q]
				}{
					q^{\fun{deg}[\mscr{P}_{B_n;I}]}
				}
				\prod_{i=2}^{t-1}
					\left(	q^{	(n-\ell_{i}(I))^2	}-1	\right)^{-1}
				\cdot
				\frac{
					\tfrac{1}{2}\left(C_{I,0}+C_{I,1}(-1)^{r}\right)
				}{
					q^{	n^2 - 2(2n-1)	}-1
				}
		\\
	\nonumber&\qquad +
			\sum_{1\in I, n\notin I}
				\frac{
					\mscr{P}_{B_n;I}[q]
				}{
					q^{\fun{deg}[\mscr{P}_{B_n;I}]}
				}
				\prod_{i=1}^{t-1}
					\left(	q^{	(n-\ell_{i}(I))^2	}-1	\right)^{-1}
				\cdot
				\tfrac{1}{2}\left(C_{I,0}+C_{I,1}(-1)^{r}\right).
	\end{align}
	As a consequence, we have 
	\begin{equation}\label{eq:Bn:Asymp:SV}
		\fun{SV}[r] = 
		\sum_{z=0}^{r}\fun{SSA}[z] \sim 
			\tilde{C}(n)q^{	\tfrac{n^2}{2} r	},
	\end{equation}
	where the parity $q$-function $\tilde{C}(n)$ is defined as follows:  
	\index[notation]{C (n) tilde@$C(n)$}%
	\begin{align}
		\label{eq:Bn:Asymp:tC}
		\tilde{C}(n)(r) &:=
			\sum_{1,n\notin I}
				\frac{
					\mscr{P}_{B_n;I}[q]
				}{
					q^{\fun{deg}[\mscr{P}_{B_n;I}]}
				}
				\prod_{i=2}^{t-1}
					\left(	q^{	(n-\ell_{i}(I))^2	}-1	\right)^{-1}
		\\
	\nonumber&\qquad\quad \cdot
				\frac{
					\tfrac{1}{2}\left(
						(1+q^{\tfrac{n^2}{2}})C_{I,0} + 
						(1-q^{\tfrac{n^2}{2}})C_{I,1}(-1)^{r}
					\right)
				}{
					\left(q^{	n^2	}-1\right)\left(q^{	n^2 - 2(2n-1)	}-1\right)
				}
		\\
	\nonumber&\quad +
			\sum_{1\in I, n\notin I}
				\frac{
					\mscr{P}_{B_n;I}[q]
				}{
					q^{\fun{deg}[\mscr{P}_{B_n;I}]}
				}
				\prod_{i=1}^{t-1}
					\left(	q^{	(n-\ell_{i}(I))^2	}-1	\right)^{-1}
		\\
	\nonumber&\qquad\qquad \cdot
				\frac{
					\tfrac{1}{2}\left(
						(1+q^{\tfrac{n^2}{2}})C_{I,0} + 
						(1-q^{\tfrac{n^2}{2}})C_{I,1}(-1)^{r}
					\right)
				}{
					q^{	n^2	}-1
				}.
	\end{align}

	\begin{remark}
		Note that the constants $C_{I,\square}$ ($\square=0,1$) depends on $I$. 
		When $1\in I$ and $n\notin I$, they are defined in \cref{eq:Bn:AsymptoticOfSI:l1:C0,eq:Bn:AsymptoticOfSI:l1:C1}. 
		When $1,n\notin I$, they are defined in \cref{eq:Bn:AsymptoticOfSI:e1:C0,eq:Bn:AsymptoticOfSI:e1:C1}. 
	\end{remark}

	By \cref{eq:B3:Asymp:SSA,eq:B3:Asymp:C,eq:B3:Asymp:SV,eq:Bn:Asymp:SSA,eq:Bn:Asymp:C,eq:Bn:Asymp:SV,eq:Bn:Asymp:tC}
	we have proved the asymptotic relations in \cref{thm:Asymp:Bn}, where 
	\begin{align*}
		C_{0}(n) &= 
			C(n)(\text{even}), &
		C_{1}(n) &= 
			C(n)(\text{odd})\cdot q^{\pi(n)}, \\
		\tilde{C}_{0}(n) &= 
			\tilde{C}(n)(\text{even}), &
		\tilde{C}_{1}(n) &= 
			\tilde{C}(n)(\text{odd})\cdot q^{\pi(n)}.
	\end{align*}  
	One can see they are primary $q$-numbers by either \cref{subsec:Bn:SI} or direct verification using \cref{eq:Bn:Asymp:C,eq:Bn:Asymp:tC}. 
	Moreover, by \cref{eq:PoincareCnI}, we have the following explicit formulas:
	\begin{align}\label{eq:Bn:PoincareFactors}
		\mscr{P}_{B_n;I}[q]
		&= 
			\frac{
				[2n]!!(q)
			}{
				\prod\limits_{i=1}^{t}
					[\ell_{i}(I)-\ell_{i-1}(I)]!(q)
			},
		&
		q^{\fun{deg}[\mscr{P}_{B_n;I}]}
		&=
			\frac{
				q^{n^2}
			}{
				\prod\limits_{i=1}^{t}
					q^{\binom{\ell_{i}(I)-\ell_{i-1}(I)}{2}}
			}.
	\end{align}
	See \cref{eq:quantum_factorial,eq:PoincareCn} for the definitions of the symbols $[\:\cdot\:]!$ and $[2\:\cdot\:]!!$.
\end{step}

\subsection{Analysis of \texorpdfstring{$\fun{S}_{\mcal{X}_{J}[I]}[2r]$}{SXJI2r} and \texorpdfstring{$\fun{S}_{\mcal{X}_{J}[I]}[2r+1]$}{SXJI2r+1}}\label{subsec:Bn:SI}
Now, let $I$ be a general type and follow \cref{con:type}. 
We are going to show that $\fun{S}_{\mcal{X}_{J}[I]}[2r]$ and $\fun{S}_{\mcal{X}_{J}[I]}[2r+1]$ can be defined by primary $q$-exponential polynomials. 

Suppose $I\cap J = \emptyset$. 
By \cref{eq:Bn:SXJIfromSVdaggerI}, we have ($\square=0,1$)
\begin{equation*}
	\fun{S}_{\mcal{X}_{J}[I]}[2r+\square] = 
	q^{	\sum\limits_{a\in\Phi^{+}}
				\ceil*{-\sum\limits_{j\in J}a(\tfrac{1}{2}\omega_{j})}	}
	\fun{S}_{\mcal{V}_{\dagger}[I]}[2r+\square+\abs*{J}-\delta(J)]. 
\end{equation*}
We have seen that the $q$-functions $\fun{S}_{\mcal{V}_{\dagger}[I]}[2\:\cdot\:]$ and $\fun{S}_{\mcal{V}_{\dagger}[I]}[2\:\cdot\:+1]$ can be defined by primary super $q$-exponential polynomials in \cref{subsec:Bn:SdaggerI}. 
The exponent $\sum\limits_{a\in\Phi^{+}}
\ceil*{-\sum\limits_{j\in J}a(\tfrac{1}{2}\omega_{j})}$ is an integer. 
Therefore, $\fun{S}_{\mcal{X}_{J}[I]}[2\:\cdot\:+\square]$ can be defined by a primary super $q$-exponential polynomial. 

Note that the proof of \cref{lem:VerticesBn} implies
\begin{equation*}
	\mcal{V}[I,r] = 
	\bigcup_{
		J\neq \Set*{1},\Set*{1,2},\cdots,\Set*{n-1,n}
	}\mcal{X}_{J}[I,r].
\end{equation*}
Hence, the $q$-function $\fun{S}_{\mcal{V}[I]}[2\:\cdot\:+\square]$ ($\square=0,1$) is clearly a $\mbb{Q}[q;1]$-combination of $\fun{S}_{\mcal{X}_{J}[I]}[2\:\cdot\:+\square]$. 
On the other hand, by \cref{eq:SSA=SumSI:Bn}, the $q$-functions $\fun{SV}[2\:\cdot\:]$, $\fun{SV}[2\:\cdot\:+1]$, $\fun{SSA}[2\:\cdot\:]$, and $\fun{SSA}[2\:\cdot\:+1]$ are $\mbb{Q}[q;1]$-combinations of $\fun{S}_{\mcal{V}[I]}[2\:\cdot\:]$ and $\fun{S}_{\mcal{V}[I]}[2\:\cdot\:+1]$. 
We thus finish proving \cref{thm:Asymp:Bn}.

\clearpage
\section{Simplicial volume in buildings of \texorpdfstring{$D_{n}$}{Dn} type}\label{sec:Dn}
In this section, we will prove the $D_{n}$ part of \cref{thm:AsymptoticDominanceOfSV,thm:AsymptoticGrowthOfSV}. 
More precisely, we will prove the following stronger theorem. 
\begin{theorem}\label{thm:Asymp:Dn}
	Let $\mathscr{B}$ be a Bruhat-Tits building of split classical type $D_n$ over a local field $K$ with residue cardinality $q$. 
	Then the simplicial volume $\fun{SV}[\:\cdot\:]$ and the simplicial surface area $\fun{SSA}[\:\cdot\:]$ in it can be defined by primary super $q$-exponential polynomials whose leading terms are of the form:
	\begin{align*}
		\fun{SV}[r] &\sim 
			\tilde{C}(n) \cdot 
			\binom{r}{\varepsilon(n)} q^{	\pi(n) r	},&
		\fun{SSA}[r] &\sim 
			C(n) \cdot 
			\binom{r}{\varepsilon(n)} q^{	\pi(n) r	},
	\end{align*}
	where $\varepsilon(n)=1$ and $\pi(n)=\tfrac{n(n-1)}{2}$ when $n \ge 5$, while $\varepsilon(4)=2$ and $\pi(4)=6$. 
	The leading coefficients $\tilde{C}(n)$ and $C(n)$ are primary $q$-numbers, not just parity $q$-functions. 
\end{theorem}
We will obtain explicit formulas for the parity functions $\tilde{C}(n)$ and $C(n)$.

But before proving \cref{thm:Asymp:Dn}, we will first analyze the asymptotic growths of $\fun{SSA}_{\dagger}[r]$ and $\fun{SV}_{\dagger}[r]$, where $\dagger$ denotes ``being special''. We will prove the following.
\begin{theorem}\label{thm:Asymp:DnSp}
	Let $\mathscr{B}$ be a Bruhat-Tits building of split classical type $D_n$ over a local field $K$ with residue cardinality $q$. 
	Then the special simplicial volume $\fun{SV}_{\dagger}[\:\cdot\:]$ and the special simplicial surface area $\fun{SSA}_{\dagger}[\:\cdot\:]$ in it can be defined by primary super $q$-exponential polynomials whose leading terms are of the form:
	\begin{align*}
		\fun{SV}_{\dagger}[r] &\sim 
			\tilde{C}_{_{\dagger}}(n) \cdot 
			\binom{r}{\varepsilon(n)} q^{	\pi(n) r	},&
		\fun{SSA}_{\dagger}[r] &\sim 
			C_{_{\dagger}}(n) \cdot 
			\binom{r}{\varepsilon(n)} q^{	\pi(n) r	},
	\end{align*}
	where $\varepsilon(n)=1$ and $\pi(n)=\tfrac{n(n-1)}{2}$ when $n \ge 5$, while $\varepsilon(4)=2$ and $\pi(4)=6$. 
	The leading coefficients $\tilde{C}_{\dagger}(n)$ and $C_{\dagger}(n)$ are primary $q$-numbers, not just parity $q$-functions.
\end{theorem}
We will also give explicit formulas for the constants $\tilde{C}_{\dagger}(n)$ and $C_{\dagger}(n)$. 
The proof of \cref{thm:Asymp:DnSp} will play an essential role in the study of $\fun{SSA}[r]$ and $\fun{SV}[r]$.

This section is structured as follows. 
In \cref{subsec:Dn:AsymptoticOfSdaggerI}, we will compute the asymptotic growth of $\fun{S}_{\mcal{V}_{\dagger}[I]}[r]$ for each type $I\subset\Delta$. This allows use to find the dominant ones of $\fun{S}_{\mcal{V}_{\dagger}[I]}[r]$, which will be done in \cref{subsec:Dn:DominantSdagger}.
Then in \cref{subsec:Dn:AsymptoticOfSdagger}, we will obtain the asymptotic growths of $\fun{SSA}_{\dagger}[r]$ and $\fun{SV}_{\dagger}[r]$.  
After that, in \cref{subsec:Dn:AsymptoticOfSasympXI}, we will estimate the asymptotic growth of $\fun{S}_{\mcal{X}^{\square\heartsuit}[I]}[r]$ ($\square,\heartsuit$ being $0$ or $1$) using the auxiliary function $\fun{S}^{\asymp}_{\mcal{X}^{\square\heartsuit}[I]}$. 
Note that $\mcal{V}$ is between $\mcal{V}_{\dagger}$ and $\mcal{X}^{00}\cup\mcal{X}^{10}\cup\mcal{X}^{01}\cup\mcal{X}^{11}$. 
Therefore, we can combine \cref{subsec:Dn:AsymptoticOfSdaggerI} and \cref{subsec:Dn:AsymptoticOfSasympXI} to estimate the asymptotic growth of each $\fun{S}_{\mcal{V}[I]}[r]$ and find the dominant ones, which will be done in \cref{subsec:Dn:dominant}. 
Once we found the dominant types, we can proceed to compute the asymptotic growth of dominant $\fun{S}_{\mcal{V}[I]}[r]$. 
This will be done in three steps: 
in \cref{subsec:Dn:AsymptoticOfSXI}, we will compute the asymptotic growth of each $\fun{S}_{\mcal{X}^{\square\heartsuit}[I]}[r]$; in \cref{subsec:Dn:AsymptoticOfSXJI}, we will deduce the asymptotic growth of $\fun{S}_{\mcal{X}_{J}[I]}[r]$ from that of $\fun{S}_{\mcal{V}_{\dagger}[I]}[r]$; then in \cref{subsec:Dn:AsymptoticOfSVI}, the asymptotic growth of $\fun{S}_{\mcal{V}[I]}[r]$ will be deduced from them. 
Finally, in \cref{subsec:Dn:AsymptoticOfS}, we will obtain the asymptotic growths of $\fun{SSA}[r]$ and $\fun{SV}[r]$. 

Throughout this section, we will heavily use the various index sets $\mcal{V}$, $\mcal{V}_{\dagger}$, $\mcal{X}^{00}$, $\mcal{X}^{10}$, $\mcal{X}^{01}$, $\mcal{X}^{11}$, and $\mcal{X}_{J}$. 
We refer to \cref{figure:VerticesOfIDnl1ne,figure:VerticesOfIDnl1e,figure:VerticesOfIDne1ne,figure:VerticesOfIDne1e} for the structure of them.

\subsection{Asymptotic growth of \texorpdfstring{$\fun{S}_{\mcal{V}_{\dagger}[I]}[r]$}{SVdaggerIr}}
\label{subsec:Dn:AsymptoticOfSdaggerI}
Now, let $I$ be a type and follow \cref{con:type}. 
We are going to compute the asymptotic growth of $\fun{S}_{\mcal{V}_{\dagger}[I]}[r]$. 
We will separate the discussion into the following six cases: 
\begin{center}
	\begin{tabular}{|l|c|c|c|}
		\hline
			& 
			$\Set*{n-1,n}\subset I$ & 
			$\abs*{\Set*{n-1,n}\cap I} = 1$ & 
			$\Set*{n-1,n}\cap I = \emptyset$
			\\
		\hline
		$1\in I$ & 
			\cref{step:Dn:AsymptoticOfSdaggerI:l1sub} &
			\cref{step:Dn:AsymptoticOfSdaggerI:l1one} & 
			\cref{step:Dn:AsymptoticOfSdaggerI:l1empty} \\
		$1\notin I$ & 
			\cref{step:Dn:AsymptoticOfSdaggerI:e1sub} &
			\cref{step:Dn:AsymptoticOfSdaggerI:e1one} & 
			\cref{step:Dn:AsymptoticOfSdaggerI:e1empty} \\
		\hline
	\end{tabular}
\end{center}

\begin{step}\label{step:Dn:AsymptoticOfSdaggerI:l1sub}
	Suppose $1\in I$ and $\Set*{n-1,n}\subset I$. 
	By \cref{eq:expressIndexSets_speical:Sphere,eq:HighestRoot_Dn,eq:2rho_Dn}, we have
	\begin{equation*}
		\fun{S}_{\mcal{V}_{\dagger}[I]}[r] = 
			\sum_{\crampedsubstack{
				c_i\in\Z_{>0}\\
				2c_1+\cdots+2c_t = r
			}}q^{	\sum\limits_{i=1}^{t}\ell_{i}(2n-1-\ell_{i})c_i	}.
	\end{equation*}
	Now, we apply \cref{lem:MultiSum2} to this summation, where the index set $\mfrak{i}$ is $\Set*{1,\cdots,t}$, the partition $\mfrak{i}=\mfrak{i}_{1}\sqcup\mfrak{i}_{2}$ is $\Set*{1,\cdots,t} = \emptyset\sqcup\Set*{1,\cdots,t}$, and the sequence $\bm{\mu}$ is 
	\begin{flalign*}
		&&	\mu_i&=\ell_{i}(2n-1-\ell_{i}). & \mathllap{(1 \le i \le t)}
	\end{flalign*}
	Since all members of $\bm{\mu}$ are integers, $\fun{S}_{\mcal{V}_{\dagger}[I]}$ can be defined by a primary super $q$-exponential polynomial. 
	The knowledge of quadratic function shows that $\mfrak{i}_{\max}=\mfrak{i}_{2\max}=\Set*{t}$ with $\mu_{\max}=\mu_{2\max}=\ell_{t}(2n-1-\ell_{t})$. 
	Then by \cref{item:lem:MultiSum2:s}, we have 
	\begin{equation*}
		\fun{S}_{\mcal{V}_{\dagger}[I]}[r] \sim  
			\prod_{	i=1	}^{t-1}
				\left(	q^{(\ell_{t}-\ell_{i})(2n-1-\ell_{t}-\ell_{i})} - 1	\right)^{-1}
			\cdot	\tfrac{1}{2}\left(1+(-1)^{r}\right)
			\cdot q^{	\tfrac{1}{2}\ell_{t}(2n-1-\ell_{t}) r	}.
	\end{equation*}
	In particular, it has order $\tfrac{1}{2}\ell_{t}(2n-1-\ell_{t})$ and degree $0$. 
\end{step}

\begin{step}\label{step:Dn:AsymptoticOfSdaggerI:l1one}
	Suppose $1\in I$ and $\Set*{n-1,n}\cap I$ is a singleton. 
	By \cref{eq:expressIndexSets_speical:Sphere,eq:HighestRoot_Dn,eq:2rho_Dn}, we have
	\begin{equation*}
		\fun{S}_{\mcal{V}_{\dagger}[I]}[r] = 
			\sum_{\crampedsubstack{
				c_i\in\Z_{>0}\\
				2c_1+\cdots+2c_{t-1}+c_{t} = r
			}}q^{	\sum\limits_{i=1}^{t-1}\ell_{i}(2n-1-\ell_{i})c_{i} + \tfrac{n(n-1)}{2}c_{t}	}.
	\end{equation*}
	Now, we apply \cref{lem:MultiSum2} to this summation, where the index set $\mfrak{i}$ is $\Set*{1,\cdots,t}$, the partition $\mfrak{i}=\mfrak{i}_{1}\sqcup\mfrak{i}_{2}$ is $\Set*{1,\cdots,t} = \Set*{t}\sqcup\Set*{1,\cdots,t-1}$, and the sequence $\bm{\mu}$ is 
	\begin{flalign*}
		&&	\mu_i& = \ell_{i}(2n-1-\ell_{i}), & \mathllap{(1 \le i \le t-1)} \\
		&&	\mu_{t}& = \tfrac{n(n-1)}{2}. 
	\end{flalign*}
	Since all members of $\bm{\mu}$ are integers, $\fun{S}_{\mcal{V}_{\dagger}[I]}$ can be defined by a primary super $q$-exponential polynomial. 
	The knowledge of quadratic function shows that $\mfrak{i}_{2\max}=\Set*{t-1}$, $\mu_{2\max}=\ell_{t-1}(2n-1-\ell_{t-1})$, and $2\mu_{1\max}>\mu_{2\max}$. 
	Then by \cref{item:lem:MultiSum2:l}, we have 
	\begin{equation*}
		\fun{S}_{\mcal{V}_{\dagger}[I]}[r] \sim  
			\prod_{	i=1	}^{t-1}
				\left(	q^{(n-\ell_{i})(n-1-\ell_{i})} - 1	\right)^{-1}
			\cdot q^{	\tfrac{n(n-1)}{2} r	}.
	\end{equation*}
	In particular, it has order $\tfrac{n(n-1)}{2}$ and degree $0$. 
\end{step}

\begin{step}\label{step:Dn:AsymptoticOfSdaggerI:l1empty}
	Suppose $1\in I$ and $\Set*{n-1,n}\cap I = \emptyset$. 
	By \cref{eq:expressIndexSets_speical:Sphere,eq:HighestRoot_Dn,eq:2rho_Dn}, we have
	\begin{equation*}
		\fun{S}_{\mcal{V}_{\dagger}[I]}[r] = 
			\sum_{\crampedsubstack{
				c_i\in\Z_{>0}\\
				2c_1+\cdots+2c_{t-2}+c_{t-1}+c_{t} = r
			}}q^{	\sum\limits_{i=1}^{t-2}\ell_{i}(2n-1-\ell_{i})c_{i} + \tfrac{n(n-1)}{2}(c_{t-1}+c_{t})	}.
	\end{equation*}
	Now, we apply \cref{lem:MultiSum2} to this summation, where the index set $\mfrak{i}$ is $\Set*{1,\cdots,t}$, the partition $\mfrak{i}=\mfrak{i}_{1}\sqcup\mfrak{i}_{2}$ is $\Set*{1,\cdots,t} = \Set*{t-1,t}\sqcup\Set*{1,\cdots,t-2}$, and the sequence $\bm{\mu}$ is 
	\begin{flalign*}
		&&	\mu_i& = \ell_{i}(2n-1-\ell_{i}), & \mathllap{(1 \le i \le t-2)} \\
		&&	\mu_{t-1}& = \tfrac{n(n-1)}{2}, \\ 
		&&	\mu_{t}& = \tfrac{n(n-1)}{2}.
	\end{flalign*}
	Since all members of $\bm{\mu}$ are integers, $\fun{S}_{\mcal{V}_{\dagger}[I]}$ can be defined by a primary super $q$-exponential polynomial. 
	The knowledge of quadratic function shows that $\mfrak{i}_{2\max}=\Set*{t-2}$, $\mu_{2\max}=\ell_{t-2}(2n-1-\ell_{t-2})$, and $2\mu_{1\max}>\mu_{2\max}$. 
	Then by \cref{item:lem:MultiSum2:l}, we have 
	\begin{equation}\label{eq:Dn:AsymptoticOfSdaggerI:l1}
		\fun{S}_{\mcal{V}_{\dagger}[I]}[r] \sim  
			\prod_{	i=1	}^{t-2}
				\left(	q^{(n-\ell_{i})(n-1-\ell_{i})} - 1	\right)^{-1}
			\cdot rq^{	\tfrac{n(n-1)}{2} r	}.
	\end{equation}
	In particular, it has order $\tfrac{n(n-1)}{2}$ and degree $1$. 
\end{step}

\begin{step}\label{step:Dn:AsymptoticOfSdaggerI:e1sub}
	Suppose $1\notin I$ and $\Set*{n-1,n}\subset I$. 
	By \cref{eq:expressIndexSets_speical:Sphere,eq:HighestRoot_Dn,eq:2rho_Dn}, we have
	\begin{equation*}
		\fun{S}_{\mcal{V}_{\dagger}[I]}[r] = 
			\sum_{\crampedsubstack{
				c_i\in\Z_{>0}\\
				c_1+2c_2+\cdots+2c_t = r
			}}q^{	\sum\limits_{i=1}^{t}\ell_{i}(2n-1-\ell_{i})c_i	}.
	\end{equation*}
	Now, we apply \cref{lem:MultiSum2} to this summation, where the index set $\mfrak{i}$ is $\Set*{1,\cdots,t}$, the partition $\mfrak{i}=\mfrak{i}_{1}\sqcup\mfrak{i}_{2}$ is $\Set*{1,\cdots,t} = \Set*{1}\sqcup\Set*{2,\cdots,t}$, and the sequence $\bm{\mu}$ is 
	\begin{flalign*}
		&&	\mu_i&=\ell_{i}(2n-1-\ell_{i}). & \mathllap{(1 \le i \le t)}
	\end{flalign*}
	Since all members of $\bm{\mu}$ are integers, $\fun{S}_{\mcal{V}_{\dagger}[I]}$ can be defined by a primary super $q$-exponential polynomial. 
	The knowledge of quadratic function shows that $\mfrak{i}_{2\max}=\Set*{t}$ and $\mu_{2\max}=\ell_{t}(2n-1-\ell_{t})$. 
	On the other side $\mu_{1\max}=(2n-2)$.

	Depending on $n$ and $\ell_{t}$, there are three possibilities. 

	If $2\mu_{1\max} > \mu_{2\max}$, then by \cref{item:lem:MultiSum2:l}, we have 
	\begin{equation*}
		\fun{S}_{\mcal{V}_{\dagger}[I]}[r] \sim  
			\prod_{	i=2	}^{t}
				\left(	q^{2(2n-2)-\ell_{i}(2n-1-\ell_{i})} - 1	\right)^{-1}
			\cdot q^{	(2n-2) r	}.
	\end{equation*}
	In particular, it has order $2n-2$ and degree $0$. 

	If $2\mu_{1\max} < \mu_{2\max}$, then by \cref{item:lem:MultiSum2:s}, we have 
	\begin{align*}
		\fun{S}_{\mcal{V}_{\dagger}[I]}[r] &\sim  
			\left(	q^{\ell_{t}(2n-1-\ell_{t})-2(2n-2)} - 1	\right)^{-1}
			\prod_{	i=2	}^{t-1}
				\left(	q^{(\ell_{t}-\ell_{i})(2n-1-\ell_{t}-\ell_{i})} - 1	\right)^{-1} \\
		&\qquad
			\cdot\tfrac{1}{2}
			\left(	
				\left(	
					1+q^{\tfrac{1}{2}\ell_{t}(2n-1-\ell_{t})-(2n-2)}	
				\right) 
				+ 
				\left(	
					1-q^{\tfrac{1}{2}\ell_{t}(2n-1-\ell_{t})-(2n-2)}	
				\right)(-1)^{r}	
			\right) \\
		&\qquad
			\cdot q^{	\tfrac{1}{2}\ell_{t}(2n-1-\ell_{t}) r	}.
	\end{align*}
	In particular, it has order $\tfrac{1}{2}\ell_{t}(2n-1-\ell_{t})$ and degree $0$. 

	If $2\mu_{1\max} = \mu_{2\max}$, then by \cref{item:lem:MultiSum2:e}, we have 
	\begin{equation*}
		\fun{S}_{\mcal{V}_{\dagger}[I]}[r] \sim  
			\tfrac{1}{2}
			\prod_{	i=2	}^{t-1}
				\left(	q^{(\ell_{t}-\ell_{i})(2n-1-\ell_{t}-\ell_{i})} - 1	\right)^{-1} 
			\cdot rq^{	(2n-2) r	}.
	\end{equation*}
	In particular, it has order $2n-2$ and degree $1$. 
\end{step}

\begin{step}\label{step:Dn:AsymptoticOfSdaggerI:e1one}
	Suppose $1\notin I$ and $\Set*{n-1,n}\cap I$ is a singleton. 
	By \cref{eq:expressIndexSets_speical:Sphere,eq:HighestRoot_Dn,eq:2rho_Dn}, we have
	\begin{equation*}
		\fun{S}_{\mcal{V}_{\dagger}[I]}[r] = 
			\sum_{\crampedsubstack{
				c_i\in\Z_{>0}\\
				c_1+2c_2+\cdots+2c_{t-1}+c_{t} = r
			}}q^{	\sum\limits_{i=1}^{t-1}\ell_{i}(2n-1-\ell_{i})c_i + \tfrac{n(n-1)}{2}c_{t}	}.
	\end{equation*}
	Now, we apply \cref{lem:MultiSum2} to this summation, where the index set $\mfrak{i}$ is $\Set*{1,\cdots,t}$, the partition $\mfrak{i}=\mfrak{i}_{1}\sqcup\mfrak{i}_{2}$ is $\Set*{1,\cdots,t} = \Set*{1,t}\sqcup\Set*{2,\cdots,t-1}$, and the sequence $\bm{\mu}$ is 
	\begin{flalign*}
		&&	\mu_i& = \ell_{i}(2n-1-\ell_{i}), & \mathllap{(1 \le i \le t-1)} \\
		&&	\mu_{t}& = \tfrac{n(n-1)}{2}. 
	\end{flalign*}
	Since all members of $\bm{\mu}$ are integers, $\fun{S}_{\mcal{V}_{\dagger}[I]}$ can be defined by a primary super $q$-exponential polynomial. 
	The knowledge of quadratic function shows that $\mfrak{i}_{2\max}=\Set*{t-1}$ with $\mu_{2\max}=\ell_{t-1}(2n-1-\ell_{t-1})$ and that $t\in\mfrak{i}_{1\max}$ with $2\mu_{1\max} = n(n-1) >\mu_{2\max}$. 

	Depending on $n$, there are two possibilities. 
	
	If $n=4$, then $\mu_{1\max}=(2n-2)$ and hence $\mfrak{i}_{1\max}=\Set*{1,t}$. 
	By \cref{item:lem:MultiSum2:l}, we have 
	\begin{equation*}
		\fun{S}_{\mcal{V}_{\dagger}[I]}[r] \sim  
			\prod_{	i=2	}^{t-1}
				\left(	q^{(4-\ell_{i})(3-\ell_{i})} - 1	\right)^{-1}
			\cdot rq^{	6 r	}.
	\end{equation*}
	In particular, it has order $6$ and degree $1$. 

	If $n\ge 5$, then $\mu_{1\max}>(2n-2)$ and hence $\mfrak{i}_{1\max}=\Set*{t}$. 
	By \cref{item:lem:MultiSum2:l}, we have 
	\begin{equation*}
		\fun{S}_{\mcal{V}_{\dagger}[I]}[r] \sim  
			\prod_{	i=2	}^{t-1}
				\left(	q^{(n-\ell_{i})(n-1-\ell_{i})} - 1	\right)^{-1}
			\cdot
			\frac{
				1+q^{\tfrac{n(n-1)}{2}-(2n-2)}	
			}{
				q^{n(n-1)-2(2n-2)} - 1
			} 
			\cdot q^{	\tfrac{n(n-1)}{2} r	}.
	\end{equation*}
	In particular, it has order $\tfrac{n(n-1)}{2}$ and degree $0$. 
\end{step}

\begin{step}\label{step:Dn:AsymptoticOfSdaggerI:e1empty}
	Suppose $1\notin I$ and $\Set*{n-1,n}\cap I = \emptyset$.  
	By \cref{eq:expressIndexSets_speical:Sphere,eq:HighestRoot_Dn,eq:2rho_Dn}, we have
	\begin{equation*}
		\fun{S}_{\mcal{V}_{\dagger}[I]}[r] = 
			\sum_{\crampedsubstack{
				c_i\in\Z_{>0}\\
				c_1+2c_2+\cdots+2c_{t-2}+c_{t-1}+c_{t} = r
			}}q^{	\sum\limits_{i=1}^{t-2}\ell_{i}(2n-1-\ell_{i})c_i + \tfrac{n(n-1)}{2}(c_{t-1}+c_{t})	}.
	\end{equation*}
	Now, we apply \cref{lem:MultiSum2} to this summation, where the index set $\mfrak{i}$ is $\Set*{1,\cdots,t}$, the partition $\mfrak{i}=\mfrak{i}_{1}\sqcup\mfrak{i}_{2}$ is $\Set*{1,\cdots,t} = \Set*{1,t-1,t}\sqcup\Set*{2,\cdots,t-2}$, and the sequence $\bm{\mu}$ is 
	\begin{flalign*}
		&&	\mu_i& = \ell_{i}(2n-1-\ell_{i}), & \mathllap{(1 \le i \le t-2)} \\
		&&	\mu_{t-1}& = \tfrac{n(n-1)}{2}, \\ 
		&&	\mu_{t}& = \tfrac{n(n-1)}{2}.
	\end{flalign*}
	Since all members of $\bm{\mu}$ are integers, $\fun{S}_{\mcal{V}_{\dagger}[I]}$ can be defined by a primary super $q$-exponential polynomial. 
	The knowledge of quadratic function shows that $\mfrak{i}_{2\max}=\Set*{t-2}$ with $\mu_{2\max}=\ell_{t-2}(2n-1-\ell_{t-2})$ and that $\Set*{t-1,t}\subset\mfrak{i}_{1\max}$ with $2\mu_{1\max} = n(n-1) >\mu_{2\max}$. 

	Depending on $n$, there are two possibilities. 
	
	If $n=4$, then $\mu_{1\max}=(2n-2)$ and hence $\mfrak{i}_{1\max}=\Set*{1,t-1,t}$.   
	By \cref{item:lem:MultiSum2:l}, we have 
	\begin{equation}\label{eq:D4:AsymptoticOfSdaggerI:e1}
		\fun{S}_{\mcal{V}_{\dagger}[I]}[r] \sim  
			\prod_{	i=2	}^{t-2}
				\left(	q^{(4-\ell_{i})(3-\ell_{i})} - 1	\right)^{-1}
			\cdot \binom{r}{2}q^{	6 r	}.
	\end{equation}
	In particular, it has order $6$ and degree $2$. 

	If $n\ge 5$, then $\mu_{1\max}>(2n-2)$ and hence $\mfrak{i}_{1\max}=\Set*{t-1,t}$.  
	By \cref{item:lem:MultiSum2:l}, we have 
	\begin{equation}\label{eq:Dn:AsymptoticOfSdaggerI:e1}
		\fun{S}_{\mcal{V}_{\dagger}[I]}[r] \sim  
			\prod_{	i=2	}^{t-2}
				\left(	q^{(n-\ell_{i})(n-1-\ell_{i})} - 1	\right)^{-1}
			\cdot
			\frac{
				1+q^{\tfrac{n(n-1)}{2}-(2n-2)}	
			}{
				q^{n(n-1)-2(2n-2)} - 1
			} 
			\cdot rq^{	\tfrac{n(n-1)}{2} r	}.
	\end{equation}
	In particular, it has order $\tfrac{n(n-1)}{2}$ and degree $1$. 
\end{step}

Note that, in all cases, $\fun{S}_{\mcal{V}_{\dagger}[I]}$ can be defined by a primary super $q$-exponential polynomial. 
Then by \cref{eq:SimplicialVolumeFormulaVar,eq:SimplicialSurfaceAreaFormulaVar}, we see that $\fun{SV}_{\dagger}[\:\cdot\:]$ and $\fun{SSA}_{\dagger}[\:\cdot\:]$ can be defined by primary super $q$-exponential polynomials.

\subsection{Dominant types for \texorpdfstring{$\fun{S}_{\mcal{V}_{\dagger}[I]}[r]$}{SVdaggerIr}}
\label{subsec:Dn:DominantSdagger}
Now, we are able to figure out for which type $I$, $\fun{S}_{\mcal{V}_{\dagger}[I]}[r]$ is dominant. 
First, we summarize the asymptotic results in \cref{subsec:Dn:AsymptoticOfSdaggerI} as follows.
\begin{center}
	\begin{tabular}{|l|c|c|c|}
		\hline
			& 
			$\Set*{n-1,n}\subset I$ & 
			$\abs*{\Set*{n-1,n}\cap I} = 1$ & 
			$\Set*{n-1,n}\cap I = \emptyset$
			\\
		\hline
		$1\in I$ & 
			$\left(\tfrac{1}{2}\ell_{t}(2n-1-\ell_{t}),0\right)$ &
			$\left(\tfrac{n(n-1)}{2},0\right)$ & 
			$\left(\tfrac{n(n-1)}{2},1\right)$ \\
		\hline
		$1\notin I$ & 
			\begin{tabular}{c}
				$\left(2n-2,0\right)$ \\
				$\left(\tfrac{1}{2}\ell_{t}(2n-1-\ell_{t}),0\right)$ \\
				$\left(2n-2,1\right)$
			\end{tabular} &
			\begin{tabular}{c}
				$\left(6,1\right)$ \\
				$\left(\tfrac{n(n-1)}{2},0\right)$
			\end{tabular} & 
			\begin{tabular}{c}
				$\left(6,2\right)$ \\
				$\left(\tfrac{n(n-1)}{2},1\right)$
			\end{tabular} \\
		\hline
	\end{tabular}
\end{center}
In the table, the pair in each cell tells us the possible order and degree of $\fun{S}_{\mcal{V}_{\dagger}[I]}$.

When $n=4$, we have $\ell_{t}(I)\le 4$ for all $I$. 
Therefore, 
\begin{equation*}
	(2n-2) \ge \tfrac{1}{2}\ell_{t}(I)(2n-1-\ell_{t}(I)).
\end{equation*}
Hence, $\fun{S}_{\mcal{V}_{\dagger}[I]}[r]$ is dominant exactly when $1\notin I$ and $\Set*{n-1,n}\cap I = \emptyset$. 
Note that, such a type $I$ must be either $\Set*{2}$ or $\emptyset$. 
By \cref{eq:D4:AsymptoticOfSdaggerI:e1}, the asymptotic growth of dominant $\fun{S}_{\mcal{V}_{\dagger}[I]}[r]$ are as follows: 
\begin{align}
	\label{eq:D4:AsymptoticOfSdaggerI:2}
	\fun{S}_{\mcal{V}_{\dagger}[\Set*{2}]}[r] &\sim  
		\binom{r}{2}q^{	6 r	}\\
	\label{eq:D4:AsymptoticOfSdaggerI:emptyset}
	\fun{S}_{\mcal{V}_{\dagger}[\emptyset]}[r] &\sim 
		\left(	q^{(4-\ell_{2})(3-\ell_{2})} - 1	\right)^{-1}
		\cdot \binom{r}{2}q^{	6 r	} = 
		\frac{1}{q^{2}-1}\cdot \binom{r}{2}q^{	6 r	}
\end{align}

If $n\ge 4$, then we have $(2n-2)$ is no longer the highest order. 
When $\Set*{n-1,n}\subset I$, we have $\ell_{t}(I)<n-1$ and thus \begin{equation*}
	\tfrac{1}{2}\ell_{t}(I)(2n-1-\ell_{t}(I)) < \tfrac{n(n-1)}{2}. 
\end{equation*}
Therefore, $\fun{S}_{\mcal{V}_{\dagger}[I]}[r]$ is dominant exactly when $\Set*{n-1,n}\cap I = \emptyset$. 
In that case, its asymptotic growth is given by \cref{eq:Dn:AsymptoticOfSdaggerI:l1,eq:Dn:AsymptoticOfSdaggerI:e1}.

\subsection{Asymptotic growths of \texorpdfstring{$\fun{SSA}_{\dagger}[r]$}{SSAdagger} and \texorpdfstring{$\fun{SV}_{\dagger}[r]$}{SVdagger}}
\label{subsec:Dn:AsymptoticOfSdagger}
We are now able to obtain the asymptotic growth of $\fun{SSA}_{\dagger}[r]$. 
By \cref{eq:SimplicialSurfaceAreaFormulaVar}, we have 
\begin{equation}\label{eq:SSA=SumSI:DnSp}
	\fun{SSA}_{\dagger}[r] = 
		\sum_{I\subset\Delta}
		\frac{
			\mscr{P}_{D_n;I}[q]
		}{
			q^{\fun{deg}[\mscr{P}_{D_n;I}]}
		}\fun{S}_{\mcal{V}_{\dagger}[I]}[r]
		\sim 
		\sum_{I\text{ is dominant}}
		\frac{
			\mscr{P}_{D_n;I}[q]
		}{
			q^{\fun{deg}[\mscr{P}_{D_n;I}]}
		}\fun{S}_{\mcal{V}_{\dagger}[I]}[r].
\end{equation}
Then by the discussion in \cref{subsec:Dn:DominantSdagger}, we see that 
\begin{equation}\label{eq:Dn:Asymp:SSAdagger}
	\fun{SSA}_{\dagger}[r] \sim
	\begin{dcases*}
		C_{\dagger}(4) \cdot \binom{r}{2}q^{	6 r	} 
			& if $n=4$, \\
		C_{\dagger}(n) \cdot rq^{	\tfrac{n(n-1)}{2} r	} 
			& if $n\ge 5$.
	\end{dcases*}
\end{equation}

When $n=4$, by \cref{eq:D4:AsymptoticOfSdaggerI:2,eq:D4:AsymptoticOfSdaggerI:emptyset}, the constant $C_{\dagger}(4)$ is defined as follows: 
\begin{equation*}
	C_{\dagger}(4) = 
	\frac{
		\mscr{P}_{D_4;\Set*{2}}[q]
	}{
		q^{\fun{deg}[\mscr{P}_{D_4;\Set*{2}}]}
	} + 
	\frac{
		\mscr{P}_{D_4;\emptyset}[q]
	}{
		\left(	q^2 - 1	\right)
		q^{\fun{deg}[\mscr{P}_{D_4;\emptyset}]}
	}. 
\end{equation*}
Moreover, by \cref{eq:PoincareDnI}, we have
\begin{align}\label{eq:D4:Asymp:Cdagger}
	\index[notation]{C dagger (4)@$C_{\dagger}(4)$}%
	C_{\dagger}(4) &= 
		\frac{
			\left(q^6-1\right)\left(q^4-1\right)^2
		}{
			\left(q-1\right)^{3} q^{11}
		} + 
		\frac{ 
			\left(q^6-1\right)\left(q^4-1\right)^2
		}{
			\left(q-1\right)^{4} q^{12}
		} 
		\\
		\nonumber&=
		\frac{
			\left(q^2+q+1\right)
			\left(q^2-q+1\right)^2 
			\left(q^2+1\right)^2 
			(q+1)^3 
		}{(q-1) q^{12}}.
\end{align}
As a consequence, we have
\begin{equation}\label{eq:D4:Asymp:SVdagger}
	\fun{SV}_{\dagger}[r] = 
	\sum_{z=0}^{r}\fun{SSA}_{\dagger}[z] \sim 
		\frac{
			q^{6}
		}{
			q^{6} - 1
		}C_{\dagger}(4) \cdot \binom{r}{2}q^{	6 r}.
\end{equation}

When $n\ge 5$, by \cref{eq:Dn:AsymptoticOfSdaggerI:l1,eq:Dn:AsymptoticOfSdaggerI:e1}, the constant $C_{\dagger}(n)$ is defined as follows:
\index[notation]{C dagger (n)@$C_{\dagger}(n)$}%
\begin{align}
	\label{eq:Dn:Asymp:Cdagger}
	C_{\dagger}(n)	&:= 
	\sum_{1,n-1,n\notin I}
		\frac{
			\mscr{P}_{D_n;I}[q]
		}{
			q^{\fun{deg}[\mscr{P}_{D_n;I}]}
		}
		\prod_{	i=2	}^{t-2}
			\left(	q^{(n-\ell_{i}(I))(n-1-\ell_{i}(I))} - 1	\right)^{-1}
		\cdot
		\frac{
			1+q^{\tfrac{n(n-1)}{2}-(2n-2)}	
		}{
			q^{n(n-1)-2(2n-2)} - 1
		} \\
	\nonumber&\qquad+
	\sum_{1\in I, n-1,n\notin I}
		\frac{
			\mscr{P}_{D_n;I}[q]
		}{
			q^{\fun{deg}[\mscr{P}_{D_n;I}]}
		}
		\prod_{	i=1	}^{t-2}
			\left(	q^{(n-\ell_{i}(I))(n-1-\ell_{i}(I))} - 1	\right)^{-1}.
\end{align}
As a consequence, we have
\begin{equation}\label{eq:Dn:Asymp:SVdagger}
	\fun{SV}_{\dagger}[r] = 
	\sum_{z=0}^{r}\fun{SSA}_{\dagger}[z] \sim 
		\frac{q^{\tfrac{n(n-1)}{2}}}{q^{\tfrac{n(n-1)}{2}}-1}
		C_{\dagger}(n) \cdot rq^{	\tfrac{n(n-1)}{2} r	}.
\end{equation}

By \cref{eq:Dn:Asymp:SSAdagger,eq:D4:Asymp:Cdagger,eq:D4:Asymp:SVdagger,eq:Dn:Asymp:Cdagger,eq:Dn:Asymp:SVdagger}, we have proved \cref{thm:Asymp:DnSp}.  
Moreover, by \cref{eq:PoincareDnIs}, we have the following explicit formulas:
\begin{align*}
	\mscr{P}_{D_n;I}[q]
	&= 
		\frac{
			[2(n-1)]!!(z)\cdot[n](z)
		}{
			\prod\limits_{i=1}^{t-1}
				[\ell_{i}(I)-\ell_{i-1}(I)]!(z)
		},
	&
	q^{\fun{deg}[\mscr{P}_{D_n;I}]}
	&=
		\frac{
			q^{n(n-1)}
		}{
			\prod\limits_{i=1}^{t-1}
				q^{\binom{\ell_{i}(I)-\ell_{i-1}(I)}{2}}
		}.
\end{align*}
See \cref{lem:PoincareOfXn,eq:quantum_factorial,eq:PoincareCn} for the definitions of the symbols $[\:\cdot\:]$, $[\:\cdot\:]!$, and $[2\:\cdot\:]!!$.

\subsection{Asymptotic growths of \texorpdfstring{$\fun{S}^{\asymp}_{\mcal{X}^{\square\heartsuit}[I]}[r]$}{SasympXIr}}
\label{subsec:Dn:AsymptoticOfSasympXI}
Now, let $I$ be a type and follow \cref{con:type}. 
We are going to estimate the asymptotic growth of $\fun{S}_{\mcal{X}^{\square\heartsuit}[I]}[r]$ ($\square,\heartsuit$ being $0$ or $1$) up to the leading coefficient. 
We will separate the discussion into the following six cases: 
\begin{center}
	\begin{tabular}{|l|c|c|c|}
		\hline
			& 
			$\Set*{n-1,n}\subset I$ & 
			$\abs*{\Set*{n-1,n}\cap I} = 1$ & 
			$\Set*{n-1,n}\cap I = \emptyset$
			\\
		\hline
		$1\in I$ & 
			\cref{step:Dn:AsymptoticOfSasympXI:l1sub} &
			\cref{step:Dn:AsymptoticOfSasympXI:l1one} & 
			\cref{step:Dn:AsymptoticOfSasympXI:l1empty} \\
		$1\notin I$ & 
			\cref{step:Dn:AsymptoticOfSasympXI:e1sub} &
			\cref{step:Dn:AsymptoticOfSasympXI:e1one} & 
			\cref{step:Dn:AsymptoticOfSasympXI:e1empty} \\
		\hline
	\end{tabular}
\end{center}

\begin{step}\label{step:Dn:AsymptoticOfSasympXI:l1sub}
	Suppose $1\in I$ and $\Set*{n-1,n}\subset I$. 
	By \cref{figure:VerticesOfIDnl1ne}, we only need to consider $\mcal{X}^{00}[I]$. 
	By \cref{eq:2rho_Dn,eq:IndexBsrciDn:pu00}, we have
	\begin{equation*}
		\fun{S}^{\asymp}_{\mcal{X}^{00}[I]}[r] =
		\sum_{\crampedsubstack{
				c_i\in\Z_{>0}\\
				c_1+\cdots+c_t = r
			}}
			q^{	\sum\limits_{i=1}^{t} 
						\tfrac{1}{2}\ell_{i}(2n-1-\ell_{i})c_i	}.
	\end{equation*}
	Now, we apply \cref{lem:MultiSum} to above summation, where the index set $\mfrak{i}$ is $\Set*{1,\cdots,t}$ and the sequence $\bm{\mu}$ is 
	\begin{flalign*}
		&&	\mu_i&=\tfrac{1}{2}\ell_{i}(2n-1-\ell_{i}). & \mathllap{(1 \le i \le t)}
	\end{flalign*}
	The knowledge of quadratic function shows that $\mfrak{i}_{\max}=\Set*{t}$ with $\mu_{\max}=\tfrac{1}{2}\ell_{t}(2n-1-\ell_{t})$. 
	Then we have 
	\begin{equation*}
		\fun{S}^{\asymp}_{\mcal{X}^{00}[I]}[r] \sim  
			\prod_{	i=1	}^{t-1}
				\left(	q^{\tfrac{1}{2}(\ell_{t}-\ell_{i})(2n-1-\ell_{t}-\ell_{i})} - 1	\right)^{-1}
			\cdot q^{	\tfrac{1}{2}\ell_{t}(2n-1-\ell_{t}) r	}.
	\end{equation*}
	Since $\fun{S}_{\mcal{X}^{00}[I]}[r]\asymp\fun{S}^{\asymp}_{\mcal{X}^{00}[I]}[r]$, it has order $\tfrac{1}{2}\ell_{t}(2n-1-\ell_{t})$ and degree $0$. 
\end{step}

\begin{step}\label{step:Dn:AsymptoticOfSasympXI:l1one}
	Suppose $1\in I$ and $\Set*{n-1,n}\cap I$ is a singleton. 
	By \cref{figure:VerticesOfIDnl1ne}, we only need to consider $\mcal{X}^{00}[I]$. 
	By \cref{eq:2rho_Dn,eq:IndexBsrciDn:pu00}, we have
	\begin{equation*}
		\fun{S}^{\asymp}_{\mcal{X}^{00}[I]}[r] =
		\sum_{\crampedsubstack{
				c_i\in\Z_{>0}\\
				c_1+\cdots+c_t = r
			}}
			q^{	\sum\limits_{i=1}^{t-1} 
						\tfrac{1}{2}\ell_{i}(2n-1-\ell_{i})c_i 
					+ \tfrac{n(n-1)}{2}c_{t}	}.
	\end{equation*}
	Now, we apply \cref{lem:MultiSum} to above summation, where the index set $\mfrak{i}$ is $\Set*{1,\cdots,t}$ and the sequence $\bm{\mu}$ is 
	\begin{flalign*}
		&&	\mu_i&=\tfrac{1}{2}\ell_{i}(2n-1-\ell_{i}), & \mathllap{(1 \le i \le t-1)} \\
		&&	\mu_t&=\tfrac{n(n-1)}{2}.
	\end{flalign*}
	The knowledge of quadratic function shows that $\mfrak{i}_{\max}=\Set*{t}$ with $\mu_{\max}=\tfrac{n(n-1)}{2}$. 
	Then we have 
	\begin{equation*}
		\fun{S}^{\asymp}_{\mcal{X}^{00}[I]}[r] \sim  
			\prod_{	i=1	}^{t-1}
				\left(	q^{\tfrac{1}{2}(n-\ell_{i})(n-1-\ell_{i})} - 1	\right)^{-1}
			\cdot q^{	\tfrac{n(n-1)}{2} r	}.
	\end{equation*}
	Since $\fun{S}_{\mcal{X}^{00}[I]}[r]\asymp\fun{S}^{\asymp}_{\mcal{X}^{00}[I]}[r]$, it has order $\tfrac{n(n-1)}{2}$ and degree $0$. 
\end{step}

\begin{step}\label{step:Dn:AsymptoticOfSasympXI:l1empty}
	Suppose $1\in I$ and $\Set*{n-1,n}\cap I = \emptyset$. 
	By \cref{figure:VerticesOfIDnl1e}, we only need to consider $\mcal{X}^{00}[I]$ and $\mcal{X}^{01}[I]$. 
	By \cref{eq:2rho_Dn,eq:IndexBsrciDn:pu00,eq:IndexBsrciDn:pu01}, we have
	\begin{align*}
		\fun{S}^{\asymp}_{\mcal{X}^{00}[I]}[r] &=
		\sum_{\crampedsubstack{
				c_i\in\Z_{>0}\\
				c_1+\cdots+c_t = r
			}}
			q^{	\sum\limits_{i=1}^{t-2} 
						\tfrac{1}{2}\ell_{i}(2n-1-\ell_{i})c_i 
					+ \tfrac{n(n-1)}{2}(c_{t-1}+c_{t})	}, \\
		\fun{S}^{\asymp}_{\mcal{X}^{01}[I]}[r] &=
			q^{	-\tfrac{n(n-1)}{2} }
			\fun{S}^{\asymp}_{\mcal{X}^{00}[I]}[r]\asymp
			\fun{S}^{\asymp}_{\mcal{X}^{00}[I]}[r].
	\end{align*}
	Now, we apply \cref{lem:MultiSum} to the first summation, where the index set $\mfrak{i}$ is $\Set*{1,\cdots,t}$ and the sequence $\bm{\mu}$ is 
	\begin{flalign*}
		&&	\mu_i &=\tfrac{1}{2}\ell_{i}(2n-1-\ell_{i}), & \mathllap{(1 \le i \le t-2)} \\
		&&	\mu_{t-1} &=\tfrac{n(n-1)}{2},\\
		&&	\mu_t &=\tfrac{n(n-1)}{2}.
	\end{flalign*}
	The knowledge of quadratic function shows that $\mfrak{i}_{\max}=\Set*{t-1,t}$ with $\mu_{\max}=\tfrac{n(n-1)}{2}$. 
	Then we have 
	\begin{equation*}
		\fun{S}^{\asymp}_{\mcal{X}^{00}[I]}[r] \sim  
			\prod_{	i=1	}^{t-2}
				\left(	q^{\tfrac{1}{2}(n-\ell_{i})(n-1-\ell_{i})} - 1	\right)^{-1}
			\cdot rq^{	\tfrac{n(n-1)}{2} r	}.
	\end{equation*}
	In particular, it has order $\tfrac{n(n-1)}{2}$ and degree $1$. 
	We then know that $\fun{S}^{\asymp}_{\mcal{X}^{01}[I]}[r]$ also has the same order and degree. 
	Since $\fun{S}_{\mcal{X}^{00}[I]}[r]\asymp\fun{S}^{\asymp}_{\mcal{X}^{00}[I]}[r]$ and 
	$\fun{S}_{\mcal{X}^{01}[I]}[r]\asymp\fun{S}^{\asymp}_{\mcal{X}^{01}[I]}[r]$, 
	we see that $\fun{S}_{(\mcal{X}^{00}\cup\mcal{X}^{01})(I)}[r]$ has order $\tfrac{n(n-1)}{2}$ and degree $1$. 
\end{step}

\begin{step}\label{step:Dn:AsymptoticOfSasympXI:e1sub}
	Suppose $1\notin I$ and $\Set*{n-1,n}\subset I$. 
	By \cref{figure:VerticesOfIDne1ne}, we only need to consider $\mcal{X}^{00}[I]$ and $\mcal{X}^{10}[I]$. 
	By \cref{eq:2rho_Dn,eq:IndexBsrciDn:pu00,eq:IndexBsrciDn:pu10}, we have
	\begin{align*}
		\fun{S}^{\asymp}_{\mcal{X}^{00}[I]}[r] &=
		\sum_{\crampedsubstack{
				c_i\in\Z_{>0}\\
				c_1+\cdots+c_t = r
			}}
			q^{	(2n-2)c_1 +
					\sum\limits_{i=2}^{t} 
						\tfrac{1}{2}\ell_{i}(2n-1-\ell_{i})c_i	}, \\
		\fun{S}^{\asymp}_{\mcal{X}^{10}[I]}[r] &=
			q^{	-\tfrac{1}{2}(2n-2)	}
			\fun{S}^{\asymp}_{\mcal{X}^{00}[I]}[r]\asymp
			\fun{S}^{\asymp}_{\mcal{X}^{00}[I]}[r].
	\end{align*}
	Now, we apply \cref{lem:MultiSum} to the first summation, where the index set $\mfrak{i}$ is $\Set*{1,\cdots,t}$ and the sequence $\bm{\mu}$ is 
	\begin{flalign*}
		&&	\mu_1&=2n-2,\\
		&&	\mu_i&=\tfrac{1}{2}\ell_{i}(2n-1-\ell_{i}). & \mathllap{(2 \le i \le t)}
	\end{flalign*}
	The knowledge of quadratic function shows that $\mfrak{i}_{\max}\subset\Set*{1,t}$ with 
	\begin{equation*}
		\mu_{\max} = 
			\max\Set*{2n-2,\tfrac{1}{2}\ell_{t}(2n-1-\ell_{t})}. 
	\end{equation*}

	Depending on $n$ and $\ell_{t}$, there are three possibilities. 

	If $2n-2 > \tfrac{1}{2}\ell_{t}(2n-1-\ell_{t})$, then we have $\mfrak{i}_{\max}=\Set*{1}$, $\mu_{\max}=2n-2$, and
	\begin{equation*}
		\fun{S}^{\asymp}_{\mcal{X}^{00}[I]}[r] \sim
			\prod_{	i=2	}^{t}
				\left(	
					q^{ (2n-2)-\tfrac{1}{2}\ell_{i}(2n-1-\ell_{i}) } - 1	
				\right)^{-1}
			\cdot q^{	(2n-2) r	}.
	\end{equation*} 
	Then we can deduce that $\fun{S}_{(\mcal{X}^{00}\cup\mcal{X}^{10})(I)}[r]$ has order $2n-2$ and degree $0$. 

	If $2n-2 < \tfrac{1}{2}\ell_{t}(2n-1-\ell_{t})$, then we have $\mfrak{i}_{\max}=\Set*{t}$, $\mu_{\max}=\tfrac{1}{2}\ell_{t}(2n-1-\ell_{t})$, and
	\begin{equation*}
		\fun{S}^{\asymp}_{\mcal{X}^{00}[I]}[r] \sim
			\left(
				q^{\tfrac{1}{2}\ell_{t}(2n-1-\ell_{t})-(2n-2)}-1
			\right)^{-1}
			\prod_{	i=2	}^{t-1}
				\left(	q^{\tfrac{1}{2}(\ell_{t}-\ell_{i})(2n-1-\ell_{t}-\ell_{i})} - 1	\right)^{-1}
			\cdot q^{	\tfrac{1}{2}\ell_{t}(2n-1-\ell_{t}) r	}.
	\end{equation*} 
	Then we can deduce that $\fun{S}_{(\mcal{X}^{00}\cup\mcal{X}^{10})(I)}[r]$ has order $\tfrac{1}{2}\ell_{t}(2n-1-\ell_{t})$ and degree $0$. 

	If $2n-2 = \tfrac{1}{2}\ell_{t}(2n-1-\ell_{t})$, then we have $\mfrak{i}_{\max}=\Set*{1,t}$ and
	\begin{equation*}
		\fun{S}^{\asymp}_{\mcal{X}^{00}[I]}[r] \sim
			\prod_{	i=2	}^{t-1}
				\left(	q^{\tfrac{1}{2}(\ell_{t}-\ell_{i})(2n-1-\ell_{t}-\ell_{i})} - 1	\right)^{-1}
			\cdot rq^{	(2n-2) r	}.
	\end{equation*} 
	Then we can deduce that $\fun{S}_{(\mcal{X}^{00}\cup\mcal{X}^{10})(I)}[r]$ has order $2n-2$ and degree $1$. 
\end{step}

\begin{step}\label{step:Dn:AsymptoticOfSasympXI:e1one}
	Suppose $1\notin I$ and $\Set*{n-1,n}\cap I$ is a singleton.  
	By \cref{figure:VerticesOfIDne1ne}, we only need to consider $\mcal{X}^{00}[I]$ and $\mcal{X}^{10}[I]$. 
	By \cref{eq:2rho_Dn,eq:IndexBsrciDn:pu00,eq:IndexBsrciDn:pu10}, we have
	\begin{align*}
		\fun{S}^{\asymp}_{\mcal{X}^{00}[I]}[r] &=
		\sum_{\crampedsubstack{
				c_i\in\Z_{>0}\\
				c_1+\cdots+c_t = r
			}}
			q^{	(2n-2)c_1 +
					\sum\limits_{i=2}^{t-1} 
						\tfrac{1}{2}\ell_{i}(2n-1-\ell_{i})c_i
					+ \tfrac{n(n-1)}{2}c_{t}	}, \\
		\fun{S}^{\asymp}_{\mcal{X}^{10}[I]}[r] &=
			q^{	-\tfrac{1}{2}(2n-2)	}
			\fun{S}^{\asymp}_{\mcal{X}^{00}[I]}[r]\asymp
			\fun{S}^{\asymp}_{\mcal{X}^{00}[I]}[r].
	\end{align*}
	Now, we apply \cref{lem:MultiSum} to the first summation, where the index set $\mfrak{i}$ is $\Set*{1,\cdots,t}$ and the sequence $\bm{\mu}$ is 
	\begin{flalign*}
		&&	\mu_1&=2n-2,\\
		&&	\mu_i&=\tfrac{1}{2}\ell_{i}(2n-1-\ell_{i}), & \mathllap{(2 \le i \le t-1)}\\
		&&	\mu_t&=\tfrac{n(n-1)}{2}.
	\end{flalign*}
	The knowledge of quadratic function shows that $t\in\mfrak{i}_{\max}\subset\Set*{1,t}$ with $\mu_{\max} = \tfrac{n(n-1)}{2}$. 

	Depending on $n$, there are two possibilities. 

	If $n=4$, then we have $\mfrak{i}_{\max}=\Set*{1,t}$ and
	\begin{equation*}
		\fun{S}^{\asymp}_{\mcal{X}^{00}[I]}[r] \sim
			\prod_{	i=2	}^{t-1}
				\left(	
					q^{ \tfrac{1}{2}(4-\ell_{i})(3-\ell_{i}) } - 1	
				\right)^{-1}
			\cdot rq^{	6 r	}.
	\end{equation*} 
	Then we can deduce that $\fun{S}_{(\mcal{X}^{00}\cup\mcal{X}^{10})(I)}[r]$ has order $6$ and degree $1$. 

	If $n\ge 5$, then we have $\mfrak{i}_{\max}=\Set*{t}$ and
	\begin{equation*}
		\fun{S}^{\asymp}_{\mcal{X}^{00}[I]}[r] \sim
			\left(
				q^{\tfrac{n(n-1)}{2}-(2n-2)}-1
			\right)^{-1}
			\prod_{	i=2	}^{t-1}
				\left(	
					q^{ \tfrac{1}{2}(n-\ell_{i})(n-1-\ell_{i}) } - 1	
				\right)^{-1}
			\cdot q^{	\tfrac{n(n-1)}{2} r	}.
	\end{equation*} 
	Then we can deduce that $\fun{S}_{(\mcal{X}^{00}\cup\mcal{X}^{10})(I)}[r]$ has order $\tfrac{n(n-1)}{2}$ and degree $0$. 
\end{step}

\begin{step}\label{step:Dn:AsymptoticOfSasympXI:e1empty}
	Suppose $1\notin I$ and $\Set*{n-1,n}\cap I = \emptyset$.  
	By \cref{figure:VerticesOfIDne1e}, we have to consider all the sets $\mcal{X}^{00}[I]$, $\mcal{X}^{01}[I]$, $\mcal{X}^{10}[I]$, and $\mcal{X}^{11}[I]$. 
	By \cref{eq:2rho_Dn,eq:IndexBsrciDn:pu00,eq:IndexBsrciDn:pu01,eq:IndexBsrciDn:pu10,eq:IndexBsrciDn:pu11}, we have
	\begin{align*}
		\fun{S}^{\asymp}_{\mcal{X}^{00}[I]}[r] &=
		\sum_{\crampedsubstack{
				c_i\in\Z_{>0}\\
				c_1+\cdots+c_t = r
			}}
			q^{	(2n-2)c_1 +
					\sum\limits_{i=2}^{t-2} 
						\tfrac{1}{2}\ell_{i}(2n-1-\ell_{i})c_i
					+ \tfrac{n(n-1)}{2}(c_{t-1}+c_{t})	}, \\
		\fun{S}^{\asymp}_{\mcal{X}^{01}[I]}[r] &=
			q^{	-\tfrac{n(n-1)}{2}	}
			\fun{S}^{\asymp}_{\mcal{X}^{00}[I]}[r]\asymp
			\fun{S}^{\asymp}_{\mcal{X}^{00}[I]}[r], \\
		\fun{S}^{\asymp}_{\mcal{X}^{10}[I]}[r] &=
			q^{	-\tfrac{1}{2}(2n-2)	}
			\fun{S}^{\asymp}_{\mcal{X}^{00}[I]}[r]\asymp
			\fun{S}^{\asymp}_{\mcal{X}^{00}[I]}[r], \\
		\fun{S}^{\asymp}_{\mcal{X}^{11}[I]}[r] &=
			q^{	-\tfrac{1}{2}(2n-2)-\tfrac{n(n-1)}{2}	}
			\fun{S}^{\asymp}_{\mcal{X}^{00}[I]}[r]\asymp
			\fun{S}^{\asymp}_{\mcal{X}^{00}[I]}[r].
	\end{align*}
	Now, we apply \cref{lem:MultiSum} to the first summation, where the index set $\mfrak{i}$ is $\Set*{1,\cdots,t}$ and the sequence $\bm{\mu}$ is 
	\begin{flalign*}
		&&	\mu_1 &=2n-2,\\
		&&	\mu_i &=\tfrac{1}{2}\ell_{i}(2n-1-\ell_{i}), & \mathllap{(2 \le i \le t-2)}\\
		&&	\mu_{t-1} &=\tfrac{n(n-1)}{2},\\
		&&	\mu_t &=\tfrac{n(n-1)}{2}.
	\end{flalign*}
	The knowledge of quadratic function shows that $\Set*{t-1,t}\subset\mfrak{i}_{\max}\subset\Set*{1,t-1,t}$ with $\mu_{\max} = \tfrac{n(n-1)}{2}$. 

	Depending on $n$, there are two possibilities. 

	If $n=4$, then we have $\mfrak{i}_{\max}=\Set*{1,t-1,t}$ and
	\begin{equation*}
		\fun{S}^{\asymp}_{\mcal{X}^{00}[I]}[r] \sim
			\prod_{	i=2	}^{t-2}
				\left(	
					q^{ \tfrac{1}{2}(4-\ell_{i})(3-\ell_{i}) } - 1	
				\right)^{-1}
			\cdot \binom{r}{2}q^{	6 r	}.
	\end{equation*} 
	Then we can deduce that $\fun{S}_{(\mcal{X}^{00}\cup\mcal{X}^{01}\cup\mcal{X}^{10}\cup\mcal{X}^{11})(I)}[r]$ has order $6$ and degree $2$. 

	If $n\ge 5$, then we have $\mfrak{i}_{\max}=\Set*{t-1,t}$ and
	\begin{equation*}
		\fun{S}^{\asymp}_{\mcal{X}^{00}[I]}[r] \sim
			\left(
				q^{\tfrac{n(n-1)}{2}-(2n-2)}-1
			\right)^{-1}
			\prod_{	i=2	}^{t-2}
				\left(	
					q^{ \tfrac{1}{2}(n-\ell_{i})(n-1-\ell_{i}) } - 1	
				\right)^{-1}
			\cdot rq^{	\tfrac{n(n-1)}{2} r	}.
	\end{equation*} 
	Then we can deduce that $\fun{S}_{(\mcal{X}^{00}\cup\mcal{X}^{01}\cup\mcal{X}^{10}\cup\mcal{X}^{11})(I)}[r]$ has order $\tfrac{n(n-1)}{2}$ and degree $1$. 
\end{step}

\subsection{Dominant types for \texorpdfstring{$\fun{S}_{\mcal{V}[I]}[r]$}{SVIr}}\label{subsec:Dn:dominant}
We are going to estimate the asymptotic growth of each $\fun{S}_{\mcal{V}[I]}[r]$ and figure out the \emph{dominant} types, namely the types for which $\fun{S}_{\mcal{V}[I]}[r]$ is dominant. 

Let $I$ be a type and follow \cref{con:type}. 
Depending on $I$, the set $\mcal{V}[I]$ is contained in various sets $\mcal{X}^{\cup}[I]$, where 
\begin{equation*}
	\mcal{X}^{\cup}[I] = 
	\begin{dcases*}
		\mcal{X}^{00}[I] & 
			if $1\in I$ and $\Set*{n-1,n}\cap I \neq \emptyset$, \\
		\mcal{X}^{00}[I]\cup\mcal{X}^{01}[I] & 
			if $1\in I$ and $\Set*{n-1,n}\cap I = \emptyset$, \\
		\mcal{X}^{00}[I]\cup\mcal{X}^{10}[I] & 
			if $1\notin I$ and $\Set*{n-1,n}\cap I \neq \emptyset$, \\
		\mcal{X}^{00}[I]\cup\mcal{X}^{01}[I]\cup\mcal{X}^{10}[I]\cup\mcal{X}^{11}[I] & 
			if $1\notin I$ and $\Set*{n-1,n}\cap I = \emptyset$.
	\end{dcases*}
\end{equation*}
Refer to \cref{figure:VerticesOfIDnl1ne,figure:VerticesOfIDnl1e,figure:VerticesOfIDne1ne,figure:VerticesOfIDne1e}. 
Then we have
\begin{equation*}
	\fun{S}_{\mcal{X}^{\cup}[I]}[r]\gg
	\fun{S}_{\mcal{V}[I]}[r]\gg
	\fun{S}_{\mcal{V}_{\dagger}[I]}[r].
\end{equation*}

We summarize \cref{subsec:Dn:AsymptoticOfSasympXI} as follows.
\begin{center}
	\begin{tabular}{|l|c|c|c|}
		\hline
			& 
			$\Set*{n-1,n}\subset I$ & 
			$\abs*{\Set*{n-1,n}\cap I} = 1$ & 
			$\Set*{n-1,n}\cap I = \emptyset$
			\\
		\hline
		$1\in I$ & 
			$\left(\tfrac{1}{2}\ell_{t}(2n-1-\ell_{t}),0\right)$ &
			$\left(\tfrac{n(n-1)}{2},0\right)$ & 
			$\left(\tfrac{n(n-1)}{2},1\right)$ \\
		\hline
		$1\notin I$ & 
			\begin{tabular}{c}
				$\left(2n-2,0\right)$ \\
				$\left(\tfrac{1}{2}\ell_{t}(2n-1-\ell_{t}),0\right)$ \\
				$\left(2n-2,1\right)$
			\end{tabular} &
			\begin{tabular}{c}
				$\left(6,1\right)$ \\
				$\left(\tfrac{n(n-1)}{2},0\right)$
			\end{tabular} & 
			\begin{tabular}{c}
				$\left(6,2\right)$ \\
				$\left(\tfrac{n(n-1)}{2},1\right)$
			\end{tabular} \\
		\hline
	\end{tabular}
\end{center}
In the table, the pair in each cell tells us the possible order and degree of $\fun{S}_{\mcal{X}^{\cup}[I]}[r]$.

Comparing this table with the discussion in \cref{subsec:Dn:DominantSdagger}, we see the followings.
\begin{enumerate}
	\item When $n=4$, a type $I$ is dominant if and only if $\Set*{1,n-1,n}\cap I = \emptyset$. In that case, $\fun{S}_{\mcal{V}[I]}$ has order $6$ and degree $2$. 
	\item When $n\ge 5$, a type $I$ is dominant if and only if $\Set*{n-1,n}\cap I = \emptyset$. In that case, $\fun{S}_{\mcal{V}[I]}$ has order $\tfrac{n(n-1)}{2}$ and degree $1$. 
\end{enumerate}

\subsection{Asymptotic growth of dominant \texorpdfstring{$\fun{S}_{\mcal{X}^{\square\heartsuit}[I]}[r]$}{SXIr}}\label{subsec:Dn:AsymptoticOfSXI}
Now, let $I$ be a type and follow \cref{con:type}. 
We are going to compute the asymptotic growths of $\fun{S}_{\mcal{X}^{0}[I]}[r]$ and $\fun{S}_{\mcal{X}^{1}[I]}[r]$ when $I$ is dominant. 
To do this, we pick an arbitrary $x\in\mcal{X}^{\square\heartsuit}[I]$ and investigate the difference between $2\rho(x)$ and the sum of $\ceil{a(x)}$ for $a(x)>0$. 
To better describe these sums, we follow \cref{con:ell_inverse,con:standard_parity,con:SummationSeq}. 

By \cref{subsec:Dn:dominant}, a necessary condition for $I$ being dominant is $\Set*{n-1,n}\cap I = \emptyset$. 
We will assume that $I$ satisfies this condition. 
Depending on $1\in I$ or not, we will separate the discussion into two cases: \cref{step:Dn:AsymptoticOfSXI:l1} and \cref{step:Dn:AsymptoticOfSXI:e1}.

\begin{step}\label{step:Dn:AsymptoticOfSXI:l1}
	Suppose $1\in I$. 
	By \cref{figure:VerticesOfIDnl1e}, we only need to consider $\mcal{X}^{00}[I]$ and $\mcal{X}^{01}[I]$. 
	Let $\heartsuit$ be either $0$ or $1$. Suppose 
	\begin{align*}
		x &= 
			o + 
			c_1\cdot\tfrac{1}{2}\omega_{\ell_{1}} + \cdots + 
			c_{t-2}\cdot\tfrac{1}{2}\omega_{\ell_{t-2}} \\
			&\qquad + 
			(c_{t-1}-\tfrac{1}{2}\cdot\heartsuit)\cdot\omega_{n-1} + 
			(c_{t}-\tfrac{1}{2}\cdot\heartsuit)\cdot\omega_{n}
			\in\mcal{X}^{0\heartsuit}[I].
	\end{align*}
	By \cref{eq:Roots_Dn}, we have 
	\begin{flalign*}
		&&
		\MoveEqLeft
		(\chi_{j}-\chi_{j'})(x)	
		&	\mathllap{(1 \le j < j' \le n-1)}\\
		&&&=
			\tfrac{1}{2}
			\left(
				c_{\ell^{-1}(j)}+\cdots+c_{\ell^{-1}(j')-1}
			\right),	\\
		&&
		\MoveEqLeft
		(\chi_{j}+\chi_{j'})(x)	
		&	\mathllap{(1 \le j < j' \le n-1)}\\
		&&&=
				\tfrac{1}{2}
				\left(
					c_{\ell^{-1}(j)}+\cdots+c_{\ell^{-1}(j')-1}
				\right) 
				+ c_{\ell^{-1}(j')}+\cdots+c_{t}-\heartsuit,\\
		&&
		\MoveEqLeft
		(\chi_{j}-\chi_{n})(x)	
		&	\mathllap{(1 \le j \le n-1)}\\
		&&&=
			\tfrac{1}{2}
			\left(
				c_{\ell^{-1}(j)}+\cdots+c_{t-2}
			\right)+c_{t-1}-\tfrac{1}{2}\cdot\heartsuit,	\\
		&&
		\MoveEqLeft
		(\chi_{j}+\chi_{n})(x)		
		&	\mathllap{(1 \le j \le n-1)}\\
		&&&=
			\tfrac{1}{2}
			\left(
				c_{\ell^{-1}(j)}+\cdots+c_{t-2}
			\right)+c_{t}-\tfrac{1}{2}\cdot\heartsuit.
	\end{flalign*}
	Therefore, we have
	\begin{align*}
		\sum_{a\in\Phi^{+}}\ceil{a(x)} &= 
			\sum_{1 \le j < j' \le n}
			\left(
				\ceil{(\chi_{j}-\chi_{j'})(x)} + 
				\ceil{(\chi_{j}+\chi_{j'})(x)}
			\right) \\
			&= 2\rho(x) + 
			\sum_{1 \le j < j' \le n-1}
				\overline{	c_{\ell^{-1}(j)}+\cdots+c_{\ell^{-1}(j')-1}	} 
			+ 
			\sum_{j=1}^{n-1}
				\overline{	c_{\ell^{-1}(j)}+\cdots+c_{t-2}-\heartsuit	}.
	\end{align*}

	From above analysis, we can define the parity functions $e_{\mcal{X}^{0\heartsuit}[I]}$ ($\heartsuit=0,1$) as follows:
	\begin{align}
		\label{eq:Dn:DefineEpsilonX:l1}
		e_{\mcal{X}^{0\heartsuit}[I]}(c_1,\cdots,c_{t}) &:=
			\sum_{1 \le i < i' \le t-1}
				(\ell_{i}-\ell_{i-1})(\ell_{i'}-\ell_{i'-1})
				\overline{c_i+\cdots+c_{i'-1}}\\
		&\nonumber\qquad + 
			\sum_{i=1}^{t-1}
				(\ell_{i}-\ell_{i-1})
				\overline{	c_{i}+\cdots+c_{t-2}-\heartsuit	}.
	\end{align}
	Then we have 
	\begin{equation*}
		\sum_{a\in\Phi^{+}}\ceil{a(x)} =
			2\rho(x) + e_{\mcal{X}^{0\heartsuit}[I]}(c_1,\cdots,c_{t}).
	\end{equation*}

	Now, we apply \cref{lem:MultiSumSEP} to the following summation ($\heartsuit=0,1$).
	\begin{equation*}
		\fun{S}_{\mcal{X}^{0\heartsuit}[I]}[r] = 
			\sum_{\crampedsubstack{
				c_i\in\Z_{>0}\\
				c_1+\cdots+c_t = r+\heartsuit
			}}q^{	
					\sum\limits_{i=1}^{t-2}
						\tfrac{1}{2}\ell_{i}(2n-1-\ell_{i})c_i	
					+ \tfrac{n(n-1)}{2}(c_{t-1}+c_{t}-\heartsuit)
					+ e_{\mcal{X}^{0\heartsuit}[I]}(c_1,\cdots,c_{t})
				}.
	\end{equation*}
	Note that the index set $\mfrak{i}$ is $\Set*{1,\cdots,t}$ and the sequence $\bm{\mu}$ is 
	\begin{flalign*}
		&&	\mu_i &=\tfrac{1}{2}\ell_{i}(2n-1-\ell_{i}), & \mathllap{(1 \le i \le t-2)} \\
		&&	\mu_{t-1} &=\tfrac{n(n-1)}{2},\\
		&&	\mu_t &=\tfrac{n(n-1)}{2}.
	\end{flalign*}
	Since all members of $\bm{\mu}$ are integers, $\fun{S}_{\mcal{X}^{0\heartsuit}[I]}[r]$ can be defined by a primary super $q$-exponential polynomial.  	
	The knowledge of quadratic function shows that $\mfrak{i}_{\max}=\Set*{t-1,t}$ with $\mu_{\max}=\tfrac{n(n-1)}{2}$. 
	Therefore, for $\heartsuit=0,1$, we have
	\begin{equation*}
		\fun{S}_{\mcal{X}^{0\heartsuit}[I]}[r] \sim 
			C_{\mcal{X}^{0\heartsuit}[I]}\cdot
			\left(	
				\left(
					\sum_{\vect{s}\in\F_2^{t}}
						\fun{E}_{\mcal{X}^{0\heartsuit}[I]}(\vect{s})
				\right) + 
				\left(
					\sum_{\vect{s}\in\F_2^{t}}
						(-1)^{\vect{1}\cdot\vect{s}}
						\fun{E}_{\mcal{X}^{0\heartsuit}[I]}(\vect{s})
				\right)(-1)^{r+\heartsuit}	
			\right)
			\cdot rq^{	\tfrac{n(n-1)}{2} r},
	\end{equation*}
	where the constant $C_{\mcal{X}^{0\heartsuit}[I]}$ and the function $\fun{E}_{\mcal{X}^{0\heartsuit}[I]}\colon\F_2^{t}\to\mbb{Q}[q;-]$ are defined as follows:
	\begin{align}
		\label{eq:Dn:CXI:l1}
		C_{\mcal{X}^{0\heartsuit}[I]} &:=
			\frac{1}{4}
			\prod_{i=1}^{t-2}
				\left(	q^{	(n-\ell_{i})(n-1-\ell_{i})	}-1	\right)^{-1},	\\
		\label{eq:Dn:EXI:l1}
		\fun{E}_{\mcal{X}^{0\heartsuit}[I]}(\vect{s}) &:= 
			q^{	e_{\mcal{X}^{0\heartsuit}[I]}(\vect{s}) + 
					\sum\limits_{i=1}^{t-2}
						\tfrac{1}{2}(n-\ell_{i})(n-1-\ell_{i}) s_i	}.
	\end{align}
	From the definition \cref{eq:Dn:DefineEpsilonX:l1} of $e_{\mcal{X}^{0\heartsuit}[I]}$, we see that $\fun{E}_{\mcal{X}^{0\heartsuit}[I]}(s_1,\cdots,s_t)$ does not depend on $s_{t-1}$ and $s_{t}$. 
	Therefore, we have 
	\begin{equation}
		\label{eq:Dn:AsymptoticOfSXI:l1}
		\fun{S}_{\mcal{X}^{0\heartsuit}[I]}[r] \sim 
			C_{\mcal{X}^{0\heartsuit}[I]}\cdot
			\left(
				\sum_{\vect{s}\in\F_2^{t}}
					\fun{E}_{\mcal{X}^{0\heartsuit}[I]}(\vect{s})
			\right) 
			\cdot rq^{	\tfrac{n(n-1)}{2} r}.
	\end{equation}
\end{step}

\begin{step}\label{step:Dn:AsymptoticOfSXI:e1}
	Suppose $1\in I$. 
	By \cref{figure:VerticesOfIDne1e}, we have to consider all the sets $\mcal{X}^{00}[I]$, $\mcal{X}^{01}[I]$, $\mcal{X}^{10}[I]$, and $\mcal{X}^{11}[I]$. 
	Let $\square,\heartsuit$ be either $0$ or $1$. Suppose 
	\begin{align*}
		x &= 
			o + 
			(c_{1}-\tfrac{1}{2}\cdot\square)\cdot\omega_{1} + 
			c_2\cdot\tfrac{1}{2}\omega_{\ell_{2}} + \cdots + 
			c_{t-2}\cdot\tfrac{1}{2}\omega_{\ell_{t-2}} \\
			&\qquad + 
			(c_{t-1}-\tfrac{1}{2}\cdot\heartsuit)\cdot\omega_{n-1} + 
			(c_{t}-\tfrac{1}{2}\cdot\heartsuit)\cdot\omega_{n}
			\in\mcal{X}^{\square\heartsuit}[I].
	\end{align*}
	By \cref{eq:Roots_Dn}, we have 
	\begin{flalign*}
		&&
		\MoveEqLeft
		(\chi_{1}-\chi_{n})(x)	\\
		&&&=
			(c_{1}-\tfrac{1}{2}\cdot\square) + 
			\tfrac{1}{2}
			\left(
				c_{2}+\cdots+c_{t-2}
			\right)+c_{t-1}-\tfrac{1}{2}\cdot\heartsuit,	\\
		&&
		\MoveEqLeft
		(\chi_{1}+\chi_{n})(x)	\\
		&&&=
			(c_{1}-\tfrac{1}{2}\cdot\square) + 
			\tfrac{1}{2}
			\left(
				c_{2}+\cdots+c_{t-2}
			\right)+c_{t}-\tfrac{1}{2}\cdot\heartsuit,	\\
		&&
		\MoveEqLeft
		(\chi_{1}-\chi_{j})(x)	
		&&	\mathllap{(1 < j \le n-1)}\\
		&&&=
			(c_{1}-\tfrac{1}{2}\cdot\square) + 
			\tfrac{1}{2}
			\left(
				c_{2}+\cdots+c_{\ell^{-1}(j)-1}
			\right),	\\
		&&
		\MoveEqLeft
		(\chi_{1}+\chi_{j})(x)	
		&&	\mathllap{(1 < j \le n-1)}\\
		&&&=
			(c_{1}-\tfrac{1}{2}\cdot\square) + 
			\tfrac{1}{2}
			\left(
				c_{2}+\cdots+c_{\ell^{-1}(j)-1}
			\right) + c_{\ell^{-1}(j)}+\cdots+c_{t} - \heartsuit,\\
		&&
		\MoveEqLeft
		(\chi_{j}-\chi_{n})(x)	
		&&	\mathllap{(1 < j \le n-1)}\\
		&&&=
			\tfrac{1}{2}
			\left(
				c_{\ell^{-1}(j)}+\cdots+c_{t-2}
			\right)+c_{t-1}-\tfrac{1}{2}\cdot\heartsuit,	\\
		&&
		\MoveEqLeft
		(\chi_{j}+\chi_{n})(x)	
		&&	\mathllap{(1 < j \le n-1)}\\
		&&&=
			\tfrac{1}{2}
			\left(
				c_{\ell^{-1}(j)}+\cdots+c_{t-2}
			\right)+c_{t}-\tfrac{1}{2}\cdot\heartsuit,	\\
		&&
		\MoveEqLeft
		(\chi_{j}-\chi_{j'})(x)	
		&&	\mathllap{(1 < j < j' \le n-1)}\\
		&&&=
			\tfrac{1}{2}
			\left(
				c_{\ell^{-1}(j)}+\cdots+c_{\ell^{-1}(j')-1}
			\right),	\\
		&&
		\MoveEqLeft
		(\chi_{j}+\chi_{j'})(x)	
		&&	\mathllap{(1 < j < j' \le n-1)}\\
		&&&=
				\tfrac{1}{2}
				\left(
					c_{\ell^{-1}(j)}+\cdots+c_{\ell^{-1}(j')-1}
				\right) + c_{\ell^{-1}(j')}+\cdots+c_{t} - \heartsuit.
	\end{flalign*}
	Therefore, we have
	\begin{align*}
		\sum_{a\in\Phi^{+}}\ceil{a(x)} &=
			\sum_{1 \le j < j' \le n}
			\left(
				\ceil{(\chi_{j}-\chi_{j'})(x)} + 
				\ceil{(\chi_{j}+\chi_{j'})(x)}
			\right) \\
			&= 2\rho(x) + 
			\overline{	c_{2}+\cdots+c_{t-2}-\square-\heartsuit	} + 
			\sum_{j=2}^{n-1}
				\overline{	c_{2}+\cdots+c_{\ell^{-1}(j)-1}-\square	} \\
			&\qquad + 
			\sum_{j=2}^{n-1}
				\overline{	c_{\ell^{-1}(j)}+\cdots+c_{t-2}-\heartsuit	} + 
			\sum_{1 < j < j' \le n-1}
				\overline{	c_{\ell^{-1}(j)}+\cdots+c_{\ell^{-1}(j')-1}	}.
	\end{align*}

	From above analysis, we can define the parity functions $e_{\mcal{X}^{\square\heartsuit}[I]}$ as follows:
	\begin{align}
		\label{eq:Dn:DefineEpsilonX:e1}
		e_{\mcal{X}^{\square\heartsuit}[I]}(c_1,\cdots,c_{t}) &:=
			\overline{	c_{2}+\cdots+c_{t-2}-\square-\heartsuit	} \\
		&\nonumber\qquad +  
			\sum_{i=2}^{t-1}
				(\ell_{i}-\ell_{i-1})
				\overline{	c_{2}+\cdots+c_{i-1}-\square	}\\
		&\nonumber\qquad + 
			\sum_{i=2}^{t-1}
				(\ell_{i}-\ell_{i-1})
				\overline{	c_{i}+\cdots+c_{t-2}-\heartsuit	}\\
		&\nonumber\qquad + 
			\sum_{2 \le i < i' \le t-1}
				(\ell_{i}-\ell_{i-1})(\ell_{i'}-\ell_{i'-1})
				\overline{c_i+\cdots+c_{i'-1}}.
	\end{align}
	Then we have 
	\begin{equation*}
		\sum_{a\in\Phi^{+}}\ceil{a(x)} =
			2\rho(x) + e_{\mcal{X}^{\square\heartsuit}[I]}(c_1,\cdots,c_{t}).
	\end{equation*}

	Now, we apply \cref{lem:MultiSumSEP} to the following summation ($\square,\heartsuit=0,1$). 
	\begin{equation*}
		\fun{S}_{\mcal{X}^{\square\heartsuit}[I]}[r] = 
			\sum_{\crampedsubstack{
				c_i\in\Z_{>0}\\
				c_1+\cdots+c_t = r+\heartsuit
			}}q^{	
					(2n-2)(c_{1}-\tfrac{1}{2}\cdot\square)
					\sum\limits_{i=2}^{t-2}
						\tfrac{1}{2}\ell_{i}(2n-1-\ell_{i})c_i	
					+ \tfrac{n(n-1)}{2}(c_{t-1}+c_{t}-\heartsuit)
					+ e_{\mcal{X}^{\square\heartsuit}[I]}(c_1,\cdots,c_{t})
				}.
	\end{equation*}
	Note that the index set $\mfrak{i}$ is $\Set*{1,\cdots,t}$ and the sequence $\bm{\mu}$ is 
	\begin{flalign*}
		&&	\mu_{1} &=2n-2,\\
		&&	\mu_i &=\tfrac{1}{2}\ell_{i}(2n-1-\ell_{i}), & \mathllap{(2 \le i \le t-2)} \\
		&&	\mu_{t-1} &=\tfrac{n(n-1)}{2},\\
		&&	\mu_t &=\tfrac{n(n-1)}{2}.
	\end{flalign*}
	Since all members of $\bm{\mu}$ are integers, $\fun{S}_{\mcal{X}^{\square\heartsuit}[I]}[r]$ can be defined by a primary super $q$-exponential polynomial.  	
	The knowledge of quadratic function shows that $\Set*{t-1,t}\subset\mfrak{i}_{\max}\subset\Set*{1,t-1,t}$ with $\mu_{\max}=\tfrac{n(n-1)}{2}$. 

	Depending on $n$, there are two possibilities. 

	If $n=4$, then we have $\mfrak{i}_{\max}=\Set*{1,t-1,t}$ and ($\square,\heartsuit=0,1$)
	\begin{equation*}
		\fun{S}_{\mcal{X}^{\square\heartsuit}[I]}[r] \sim 
			C_{\mcal{X}^{\square\heartsuit}[I]}\cdot
			\left(	
				\left(
					\sum_{\vect{s}\in\F_2^{t}}
						\fun{E}_{\mcal{X}^{\square\heartsuit}[I]}(\vect{s})
				\right) + 
				\left(
					\sum_{\vect{s}\in\F_2^{t}}
						(-1)^{\vect{1}\cdot\vect{s}}
						\fun{E}_{\mcal{X}^{\square\heartsuit}[I]}(\vect{s})
				\right)(-1)^{r+\heartsuit}	
			\right)
			\cdot \binom{r}{2}q^{	6 r	},
	\end{equation*}
	where the constant $C_{\mcal{X}^{\square\heartsuit}[I]}$ and the function $\fun{E}_{\mcal{X}^{\square\heartsuit}[I]}\colon\F_2^{t}\to\mbb{Q}[q;-]$ are defined as follows:
	\begin{align}
		\label{eq:D4:CXI:e1}
		C_{\mcal{X}^{\square\heartsuit}[I]} &:=
			\frac{1}{8}q^{	-3\cdot\square	}
			\prod_{i=2}^{t-2}
				\left(	q^{	(4-\ell_{i})(3-\ell_{i})	}-1	\right)^{-1},	\\
		\label{eq:D4:EXI:e1}
		\fun{E}_{\mcal{X}^{\square\heartsuit}[I]}(\vect{s}) &:= 
			q^{	e_{\mcal{X}^{\square\heartsuit}[I]}(\vect{s}) + 
					\sum\limits_{i=2}^{t-2}
						\tfrac{1}{2}(4-\ell_{i})(3-\ell_{i}) s_i	}.
	\end{align}
	From the definition \cref{eq:Dn:DefineEpsilonX:e1} of $e_{\mcal{X}^{\square\heartsuit}[I]}$, we see that $\fun{E}_{\mcal{X}^{\square\heartsuit}[I]}(s_1,\cdots,s_t)$ does not depend on $s_1$, $s_{t-1}$, and $s_{t}$. 
	Therefore, we have 
	\begin{equation}
		\label{eq:D4:AsymptoticOfSXI:e1}
		\fun{S}_{\mcal{X}^{\square\heartsuit}[I]}[r] \sim 
			C_{\mcal{X}^{\square\heartsuit}[I]}\cdot
			\left(
				\sum_{\vect{s}\in\F_2^{t}}
					\fun{E}_{\mcal{X}^{\square\heartsuit}[I]}(\vect{s})
			\right) 
			\cdot \binom{r}{2}q^{	6 r	}.
	\end{equation}

	If $n\ge 5$, then we have $\mfrak{i}_{\max}=\Set*{t-1,t}$ and ($\square,\heartsuit=0,1$)
	\begin{equation*}		
		\fun{S}_{\mcal{X}^{\square\heartsuit}[I]}[r] \sim 
			C_{\mcal{X}^{\square\heartsuit}[I]}\cdot
			\left(	
				\left(
					\sum_{\vect{s}\in\F_2^{t}}
						\fun{E}_{\mcal{X}^{\square\heartsuit}[I]}(\vect{s})
				\right) + 
				\left(
					\sum_{\vect{s}\in\F_2^{t}}
						(-1)^{\vect{1}\cdot\vect{s}}
						\fun{E}_{\mcal{X}^{\square\heartsuit}[I]}(\vect{s})
				\right)(-1)^{r+\heartsuit}	
			\right)
			\cdot rq^{	\tfrac{n(n-1)}{2} r},
	\end{equation*}
	where the constant $C_{\mcal{X}^{\square\heartsuit}[I]}$ and the function $\fun{E}_{\mcal{X}^{\square\heartsuit}[I]}\colon\F_2^{t}\to\mbb{Q}[q;-]$ are defined as follows:
	\begin{align}
		\label{eq:Dn:CXI:e1}
		C_{\mcal{X}^{\square\heartsuit}[I]} &:=
			\frac{1}{4}
			\frac{
				q^{	-\tfrac{1}{2}(2n-2)\cdot\square	}
			}{
				q^{	n(n-1) - 2(2n-2)	}-1
			}
			\prod_{i=2}^{t-2}
				\left(	q^{	(n-\ell_{i})(n-1-\ell_{i})	}-1	\right)^{-1},	\\
		\label{eq:Dn:EXI:e1}
		\fun{E}_{\mcal{X}^{\square\heartsuit}[I]}(\vect{s}) &:= 
			q^{	e_{\mcal{X}^{\square\heartsuit}[I]}(\vect{s}) + 
					\left(	\tfrac{n(n-1)}{2} - (2n-2)	\right) s_1 + 
					\sum\limits_{i=2}^{t-2}
						\tfrac{1}{2}(n-\ell_{i})(n-1-\ell_{i}) s_i	}.
	\end{align}
	From the definition \cref{eq:Dn:DefineEpsilonX:e1} of $e_{\mcal{X}^{\square\heartsuit}[I]}$, we see that $\fun{E}_{\mcal{X}^{\square\heartsuit}[I]}(s_1,\cdots,s_t)$ does not depend on $s_{t-1}$ and $s_{t}$. 
	Therefore, we have 
	\begin{equation}
		\label{eq:Dn:AsymptoticOfSXI:e1}
		\fun{S}_{\mcal{X}^{\square\heartsuit}[I]}[r] \sim 
			C_{\mcal{X}^{\square\heartsuit}[I]}\cdot
			\left(
				\sum_{\vect{s}\in\F_2^{t}}
					\fun{E}_{\mcal{X}^{\square\heartsuit}[I]}(\vect{s})
			\right) 
			\cdot rq^{	\tfrac{n(n-1)}{2} r}.
	\end{equation}
\end{step}

\subsection{Asymptotic growth of dominant \texorpdfstring{$\fun{S}_{\mcal{X}_{J}[I]}[r]$}{SXJIr}}\label{subsec:Dn:AsymptoticOfSXJI}
Now, let $I$ be a type and follow \cref{con:type}. 
We are going to analyze $\fun{S}_{\mcal{X}_{J}[I]}[r]$. 

Suppose $x\in\mcal{X}_{J}[I,r]$, where $I\cap J = \emptyset$. 
Since $\mcal{X}_{\emptyset} = \mcal{V}_{\dagger}$, by \cref{lem:XIJfromXI0Dn}, we can write $x$ as 
$x_0-\sum\limits_{j\in J}\tfrac{1}{2}\omega_{j}$, 
where $x_0\in\mcal{V}_{\dagger}[I,r+\abs*{J}-\delta(J)]$. 
Then we have 
\begin{equation*}
	\sum_{a\in\Phi^{+}}\ceil{a(x)} = 
		2\rho(x_0) + 
		\sum_{a\in\Phi^{+}}
			\ceil*{-\sum\limits_{j\in J}a(\tfrac{1}{2}\omega_{j})}.
\end{equation*}
Note that the last summation gives an integral constant. 
Then we have 
\begin{equation}\label{eq:Dn:SXJIfromSVdaggerI}
	\fun{S}_{\mcal{X}_{J}[I]}[r] = 
	q^{	\sum\limits_{a\in\Phi^{+}}
				\ceil*{-\sum\limits_{j\in J}a(\tfrac{1}{2}\omega_{j})}	}
	\fun{S}_{\mcal{V}_{\dagger}[I]}[r+\abs*{J}-\delta(J)]. 
\end{equation}
In particular, each $\fun{S}_{\mcal{X}_{J}[I]}$ can be defined by a primary super $q$-exponential polynomial. 
Since $\mcal{V}[I,r]$ is a disjointed union of various $\mcal{X}_{J}[I,r]$, we see that $\fun{S}_{\mcal{V}[I]}$ can be defined by a primary super $q$-exponential polynomial. 
Then by \cref{eq:SimplicialVolumeFormula,eq:SimplicialSurfaceAreaFormula}, we see that $\fun{SV}[\:\cdot\:]$ and $\fun{SSA}[\:\cdot\:]$ can be defined by primary super $q$-exponential polynomials.

Now we assume that $I$ is dominant. 
We will separate the discussion into two cases: \cref{step:D4:AsymptoticOfSXJI} $n=4$ and \cref{step:Dn:AsymptoticOfSXJI} $n\ge 5$.

\begin{step}\label{step:D4:AsymptoticOfSXJI}
	Suppose $n=4$. 
	Then we have $\Set*{1,3,4}\cap I = \emptyset$. 
	The following $J$ appears in \cref{figure:VerticesOfIDne1e}: 
	$\Set*{1}$, $\Set*{1,2}$, $\Set*{3,4}$, $\Set*{2,3,4}$. 
	In those cases, by \cref{eq:Roots_Dn}, we have 
	\begin{align*}
		\abs*{\Set*{1}}-\delta(\Set*{1}) &= 0,
		&
		\sum_{a\in\Phi^{+}}
			\ceil*{-a(\tfrac{1}{2}\omega_{1})} 
			&= 0, \\
		\abs*{\Set*{1,2}}-\delta(\Set*{1,2}) &= 1,
		&
		\sum_{a\in\Phi^{+}}
			\ceil*{-a(\tfrac{1}{2}\omega_{1}+\tfrac{1}{2}\omega_{2})} 
			&= -5, \\
		\abs*{\Set*{3,4}}-\delta(\Set*{3,4}) &= 1,
		&
		\sum_{a\in\Phi^{+}}
			\ceil*{-a(\tfrac{1}{2}\omega_{3}+\tfrac{1}{2}\omega_{4})} 
			&= -3, \\
		\abs*{\Set*{2,3,4}}-\delta(\Set*{2,3,4}) &= 2,
		&
		\sum_{a\in\Phi^{+}}
			\ceil*{-a(\tfrac{1}{2}\omega_{2}+\tfrac{1}{2}\omega_{3}+\tfrac{1}{2}\omega_{4})} 
			&= -8.
	\end{align*}
	Then by \cref{eq:D4:AsymptoticOfSdaggerI:e1,eq:Dn:SXJIfromSVdaggerI}, we have 
	\begin{align}
		\label{eq:D4:AsymptoticOfSXJI:1}
		\fun{S}_{\mcal{X}_{\Set*{1}}[I]}[r] &=
		\prod_{	i=2	}^{t-2}
			\left(	q^{(4-\ell_{i})(3-\ell_{i})} - 1	\right)^{-1}
		\cdot \binom{r}{2}q^{	6 r	}, \\
		\label{eq:D4:AsymptoticOfSXJI:12}
		\fun{S}_{\mcal{X}_{\Set*{1,2}}[I]}[r] &=
		q^{	6-5	}\cdot
		\prod_{	i=2	}^{t-2}
			\left(	q^{(4-\ell_{i})(3-\ell_{i})} - 1	\right)^{-1}
		\cdot \binom{r}{2}q^{	6 r	}, \\
		\label{eq:D4:AsymptoticOfSXJI:34}
		\fun{S}_{\mcal{X}_{\Set*{3,4}}[I]}[r] &=
		q^{	6-3	}\cdot
		\prod_{	i=2	}^{t-2}
			\left(	q^{(4-\ell_{i})(3-\ell_{i})} - 1	\right)^{-1}
		\cdot \binom{r}{2}q^{	6 r	}, \\
		\label{eq:D4:AsymptoticOfSXJI:234}
		\fun{S}_{\mcal{X}_{\Set*{2,3,4}}[I]}[r] &=
		q^{	12-8	}\cdot
		\prod_{	i=2	}^{t-2}
			\left(	q^{(4-\ell_{i})(3-\ell_{i})} - 1	\right)^{-1}
		\cdot \binom{r}{2}q^{	6 r	}.
	\end{align}
\end{step}
\begin{step}\label{step:Dn:AsymptoticOfSXJI}
	Suppose $n \ge 5$. Then we have $\Set*{n-1,n}\cap I = \emptyset$. 
	Depending on $\ell_{1}$, there are two cases: \cref{step:Dn:AsymptoticOfSXJI:l1} $\ell_{1}>1$ and \cref{step:Dn:AsymptoticOfSXJI:e1} $\ell_{1}=1$. 
	
	\begin{substep}\label{step:Dn:AsymptoticOfSXJI:l1}
		If $\ell_{1}>1$, then we consider \cref{figure:VerticesOfIDnl1e} and the following $J$: $\Set*{2,3}$, $\cdots$, $\Set*{n-3,n-2}$, $\Set*{n-1,n}$, and $\Set*{n-2,n-1,n}$. 
		When $J=\Set*{j,j+1}$, where $2\le j\le n-3$, we have $\abs*{J}-\delta(J) = 2$ and by \cref{eq:Roots_Dn}, 
		\begin{equation*}
			\sum_{a\in\Phi^{+}}
				\ceil*{-a(\tfrac{1}{2}\omega_{j}+\tfrac{1}{2}\omega_{j+1})} = 
				-j(2n-j).
		\end{equation*}
		Then by \cref{eq:Dn:AsymptoticOfSdaggerI:l1,eq:Dn:SXJIfromSVdaggerI}, we have 
		\begin{equation}
			\label{eq:Dn:AsymptoticOfSXJI:l1}
			\fun{S}_{\mcal{X}_{\Set*{j,j+1}}[I]}[r] = 
				q^{	n(n-1) -j(2n-j)	}
				\prod_{	i=1	}^{t-2}
					\left(	q^{(n-\ell_{i})(n-1-\ell_{i})} - 1	\right)^{-1}
				\cdot rq^{	\tfrac{n(n-1)}{2} r	}.
		\end{equation}
		When $J=\Set*{n-1,n}$, we have $\abs*{J}-\delta(J) = 1$ and by \cref{eq:Roots_Dn}, 
		\begin{equation*}
			\sum_{a\in\Phi^{+}}
				\ceil*{-a(\tfrac{1}{2}\omega_{n-1}+\tfrac{1}{2}\omega_{n})} = 
				-\tfrac{(n-1)(n-2)}{2}.
		\end{equation*}
		Then by \cref{eq:Dn:AsymptoticOfSdaggerI:l1,eq:Dn:SXJIfromSVdaggerI}, we have 
		\begin{equation}
			\label{eq:Dn:AsymptoticOfSXJI:l1:nn}
			\fun{S}_{\mcal{X}_{\Set*{n-1,n}}[I]}[r] = 
				q^{	\tfrac{n(n-1)}{2} -\tfrac{(n-1)(n-2)}{2}	}
				\prod_{	i=1	}^{t-2}
					\left(	q^{(n-\ell_{i})(n-1-\ell_{i})} - 1	\right)^{-1}
				\cdot rq^{	\tfrac{n(n-1)}{2} r	}.
		\end{equation}
		When $J=\Set*{n-2,n-1,n}$, we have $\abs*{J}-\delta(J) = 2$ and by \cref{eq:Roots_Dn}, 
		\begin{equation*}
			\sum_{a\in\Phi^{+}}
				\ceil*{-a(\tfrac{1}{2}\omega_{n-2}+\tfrac{1}{2}\omega_{n-1}+\tfrac{1}{2}\omega_{n})} = 
				-(n-1)^2.
		\end{equation*}
		Then by \cref{eq:Dn:AsymptoticOfSdaggerI:l1,eq:Dn:SXJIfromSVdaggerI}, we have 
		\begin{equation}
			\label{eq:Dn:AsymptoticOfSXJI:l1:nnn}
			\fun{S}_{\mcal{X}_{\Set*{n-2,n-1,n}}[I]}[r] = 
				q^{	n(n-1) -(n-1)^2	}
				\prod_{	i=1	}^{t-2}
					\left(	q^{(n-\ell_{i})(n-1-\ell_{i})} - 1	\right)^{-1}
				\cdot rq^{	\tfrac{n(n-1)}{2} r	}.
		\end{equation}
	\end{substep}

	\begin{substep}\label{step:Dn:AsymptoticOfSXJI:e1}
		If $\ell_{1}=1$, then we consider \cref{figure:VerticesOfIDne1e} and the following $J$: $\Set*{1}$, $\Set*{1,2}$, $\Set*{2,3}$, $\cdots$, $\Set*{n-3,n-2}$, $\Set*{n-1,n}$, and $\Set*{n-2,n-1,n}$.  
		When $J=\Set*{1}$, we have $\abs*{J}-\delta(J) = 0$ and by \cref{eq:Roots_Dn}, 
		\begin{equation*}
			\sum_{a\in\Phi^{+}}
				\ceil*{-a(\tfrac{1}{2}\omega_{1})} = 0.
		\end{equation*}
		Then by \cref{eq:Dn:AsymptoticOfSdaggerI:e1,eq:Dn:SXJIfromSVdaggerI}, we have 
		\begin{equation}
			\label{eq:Dn:AsymptoticOfSXJI:e1:1}
			\fun{S}_{\mcal{X}_{\Set*{1}}[I]}[r] = 
				\prod_{	i=2	}^{t-2}
					\left(	q^{(n-\ell_{i})(n-1-\ell_{i})} - 1	\right)^{-1}
				\cdot
				\frac{
					1+q^{\tfrac{n(n-1)}{2}-(2n-2)}	
				}{
					q^{n(n-1)-2(2n-2)} - 1
				} 
				\cdot rq^{	\tfrac{n(n-1)}{2} r	}.
		\end{equation}
		When $J=\Set*{1,2}$, we have $\abs*{J}-\delta(J) = 1$ and by \cref{eq:Roots_Dn}, 
		\begin{equation*}
			\sum_{a\in\Phi^{+}}
				\ceil*{-a(\tfrac{1}{2}\omega_{1}+\tfrac{1}{2}\omega_{2})} = -(2n-3).
		\end{equation*}
		Then by \cref{eq:Dn:AsymptoticOfSdaggerI:e1,eq:Dn:SXJIfromSVdaggerI}, we have 
		\begin{align}
			\label{eq:Dn:AsymptoticOfSXJI:e1:12}
			\fun{S}_{\mcal{X}_{\Set*{1}}[I]}[r] &= 
				q^{	\tfrac{n(n-1)}{2}-(2n-3)	}
				\prod_{	i=2	}^{t-2}
					\left(	q^{(n-\ell_{i})(n-1-\ell_{i})} - 1	\right)^{-1}
				\\
				\nonumber&\qquad
				\cdot
				\frac{
					1+q^{\tfrac{n(n-1)}{2}-(2n-2)}	
				}{
					q^{n(n-1)-2(2n-2)} - 1
				} 
				\cdot rq^{	\tfrac{n(n-1)}{2} r	}.
		\end{align}
		When $J=\Set*{j,j+1}$, where $2\le j\le n-3$, we have $\abs*{J}-\delta(J) = 2$ and by \cref{eq:Roots_Dn}, 
		\begin{equation*}
			\sum_{a\in\Phi^{+}}
				\ceil*{-a(\tfrac{1}{2}\omega_{j}+\tfrac{1}{2}\omega_{j+1})} = 
				-j(2n-j).
		\end{equation*}
		Then by \cref{eq:Dn:AsymptoticOfSdaggerI:e1,eq:Dn:SXJIfromSVdaggerI}, we have 
		\begin{align}
			\label{eq:Dn:AsymptoticOfSXJI:e1:j}
			\fun{S}_{\mcal{X}_{\Set*{j,j+1}}[I]}[r] &= 
				q^{	n(n-1) -j(2n-j)	}
				\prod_{	i=2	}^{t-2}
					\left(	q^{(n-\ell_{i})(n-1-\ell_{i})} - 1	\right)^{-1}
				\\
				\nonumber&\qquad
				\cdot
				\frac{
					1+q^{\tfrac{n(n-1)}{2}-(2n-2)}	
				}{
					q^{n(n-1)-2(2n-2)} - 1
				} 
				\cdot rq^{	\tfrac{n(n-1)}{2} r	}.
		\end{align}
		When $J=\Set*{n-1,n}$, we have $\abs*{J}-\delta(J) = 1$ and by \cref{eq:Roots_Dn}, 
		\begin{equation*}
			\sum_{a\in\Phi^{+}}
				\ceil*{-a(\tfrac{1}{2}\omega_{n-1}+\tfrac{1}{2}\omega_{n})} = 
				-\tfrac{(n-1)(n-2)}{2}.
		\end{equation*}
		Then by \cref{eq:Dn:AsymptoticOfSdaggerI:e1,eq:Dn:SXJIfromSVdaggerI}, we have 
		\begin{align}
			\label{eq:Dn:AsymptoticOfSXJI:e1:nn}
			\fun{S}_{\mcal{X}_{\Set*{n-1,n}}[I]}[r] &= 
				q^{	\tfrac{n(n-1)}{2} -\tfrac{(n-1)(n-2)}{2}	}
				\prod_{	i=2	}^{t-2}
					\left(	q^{(n-\ell_{i})(n-1-\ell_{i})} - 1	\right)^{-1}
				\\
				\nonumber&\qquad
				\cdot
				\frac{
					1+q^{\tfrac{n(n-1)}{2}-(2n-2)}	
				}{
					q^{n(n-1)-2(2n-2)} - 1
				} 
				\cdot rq^{	\tfrac{n(n-1)}{2} r	}.
		\end{align}
		When $J=\Set*{n-2,n-1,n}$, we have $\abs*{J}-\delta(J) = 2$ and by \cref{eq:Roots_Dn}, 
		\begin{equation*}
			\sum_{a\in\Phi^{+}}
				\ceil*{-a(\tfrac{1}{2}\omega_{n-2}+\tfrac{1}{2}\omega_{n-1}+\tfrac{1}{2}\omega_{n})} = 
				-(n-1)^2.
		\end{equation*}
		Then by \cref{eq:Dn:AsymptoticOfSdaggerI:e1,eq:Dn:SXJIfromSVdaggerI}, we have 
		\begin{align}
			\label{eq:Dn:AsymptoticOfSXJI:e1:nnn}
			\fun{S}_{\mcal{X}_{\Set*{n-2,n-1,n}}[I]}[r] &= 
				q^{	n(n-1) -(n-1)^2	}
				\prod_{	i=2	}^{t-2}
					\left(	q^{(n-\ell_{i})(n-1-\ell_{i})} - 1	\right)^{-1}
				\\
				\nonumber&\qquad
				\cdot
				\frac{
					1+q^{\tfrac{n(n-1)}{2}-(2n-2)}	
				}{
					q^{n(n-1)-2(2n-2)} - 1
				} 
				\cdot rq^{	\tfrac{n(n-1)}{2} r	}.
		\end{align}
	\end{substep}
\end{step}

\subsection{Asymptotic growth of dominant \texorpdfstring{$\fun{S}_{\mcal{V}[I]}[r]$}{SVIr}}\label{subsec:Dn:AsymptoticOfSVI}
We are now able to compute the asymptotic growth of $\fun{S}_{\mcal{V}[I]}[r]$ when $I$ is dominant. 
We will separate the discussion into two cases: \cref{step:D4:AsymptoticOfSVI} $n=4$ and \cref{step:Dn:AsymptoticOfSVI} $n\ge 5$.

\begin{step}\label{step:D4:AsymptoticOfSVI}
	Suppose $n=4$. 
	Then the dominant types are $\Set*{2}$ and $\emptyset$. 
	By \cref{figure:VerticesOfIDne1e}, we have 
	\begin{align*}
		\fun{S}_{\mcal{V}[\Set*{2}]}[r] &=
			\fun{S}_{\mcal{X}^{00}[\Set*{2}]}[r] + 
			\fun{S}_{\mcal{X}^{01}[\Set*{2}]}[r] + 
			\fun{S}_{\mcal{X}^{10}[\Set*{2}]}[r] + 
			\fun{S}_{\mcal{X}^{11}[\Set*{2}]}[r] \\
			&\qquad - 
			\fun{S}_{\mcal{X}_{\Set*{1}}[\Set*{2}]}[r] -
			\fun{S}_{\mcal{X}_{\Set*{3,4}}[\Set*{2}]}[r], \\ 
		\fun{S}_{\mcal{V}[\emptyset]}[r] &=
			\fun{S}_{\mcal{X}^{00}[\emptyset]}[r] + 
			\fun{S}_{\mcal{X}^{01}[\emptyset]}[r] + 
			\fun{S}_{\mcal{X}^{10}[\emptyset]}[r] + 
			\fun{S}_{\mcal{X}^{11}[\emptyset]}[r] \\
			&\qquad - 
			\fun{S}_{\mcal{X}_{\Set*{1}}[\emptyset]}[r] -
			\fun{S}_{\mcal{X}_{\Set*{1,2}}[\emptyset]}[r] -
			\fun{S}_{\mcal{X}_{\Set*{3,4}}[\emptyset]}[r] -
			\fun{S}_{\mcal{X}_{\Set*{2,3,4}}[\emptyset]}[r].
	\end{align*}
	Therefore, by \cref{eq:Dn:DefineEpsilonX:e1,eq:D4:CXI:e1,eq:D4:EXI:e1,eq:D4:AsymptoticOfSXI:e1,eq:D4:AsymptoticOfSXJI:1,eq:D4:AsymptoticOfSXJI:12,eq:D4:AsymptoticOfSXJI:34,eq:D4:AsymptoticOfSXJI:234}, we have
	\begin{align}
		\label{eq:D4:AsymptoticOfSVI:2}
		\fun{S}_{\mcal{V}[\Set*{2}]}[r] &\sim 
			\left(1+q^3+1+q-1-q^3\right)
			\cdot \binom{r}{2}q^{	6 r	} = 
			\left(q+1\right)\cdot\binom{r}{2}q^{	6 r	}, \\
		\label{eq:D4:AsymptoticOfSVI}
		\fun{S}_{\mcal{V}[\emptyset]}[r] &\sim 
			\left((1+q^5)+(q^3+q^4)+(1+q)+(1+q^2)\right. \\
			\nonumber&\qquad
			\left.-1-q-q^3-q^4\right)\left(q^2-1\right)^{-1}
			\cdot \binom{r}{2}q^{	6 r	} \\
			\nonumber&= 
			\frac{
				q^5+q^2+q+1
			}{
				q^2-1
			}
			\cdot\binom{r}{2}q^{	6 r	}.
	\end{align}
\end{step}

\begin{step}\label{step:Dn:AsymptoticOfSVI} 
	Now, we assume $n\ge 5$. 
	Then $I$ is dominant exactly when $\Set*{n-1,n}\cap I = \emptyset$. 
	Depending on $\ell_{1}$, there are two cases: \cref{step:Dn:AsymptoticOfSVI:l1} $\ell_{1}>1$ and \cref{step:Dn:AsymptoticOfSVI:e1} $\ell_{1}=1$.  

	\begin{substep}\label{step:Dn:AsymptoticOfSVI:l1}
		When $\ell_{1}>1$, by \cref{figure:VerticesOfIDnl1e}, we have (including the zero summations)
		\begin{align*}
			\fun{S}_{\mcal{V}[I]}[r] &= 
			\fun{S}_{\mcal{X}^{00}[I]}[r] + 
			\fun{S}_{\mcal{X}^{01}[I]}[r] + 
			\fun{S}_{\mcal{X}^{10}[I]}[r] + 
			\fun{S}_{\mcal{X}^{11}[I]}[r] \\
			&\qquad -
			\sum_{j=2}^{n-3}
			\fun{S}_{\mcal{X}_{\Set*{j,j+1}}[I]}[r] -
			\fun{S}_{\mcal{X}_{\Set*{n-1,n}}[I]}[r] -
			\fun{S}_{\mcal{X}_{\Set*{n-2,n-1,n}}[I]}[r].
		\end{align*} 
		Therefore, by \cref{eq:Dn:CXI:l1,eq:Dn:EXI:l1,eq:Dn:AsymptoticOfSXI:l1,eq:Dn:AsymptoticOfSXJI:l1,eq:Dn:AsymptoticOfSXJI:l1:nn,eq:Dn:AsymptoticOfSXJI:l1:nnn}, we have 
		\begin{align}\label{eq:Dn:AsymptoticOfSVI:l1}
			\fun{S}_{\mcal{V}[I]}[r] &\sim 
			\prod_{i=1}^{t-2}
				\left(	q^{	(n-\ell_{i})(n-1-\ell_{i})	}-1	\right)^{-1}
			\cdot
			C_{I}
			\cdot rq^{	\tfrac{n(n-1)}{2} r},
		\end{align}
		where the constant $C_{I}$ is defined as follows:
		\begin{align}\label{eq:Dn:AsymptoticOfSVI:C:l1}
			C_{I} &= 
			\sum_{\heartsuit=0,1}
			\sum_{s_1,\cdots,s_{t-2}\in\F_2}
				q^{	e_{\mcal{X}^{0\heartsuit}[I]}(s_1,\cdots,s_{t-2},0,0) + 
						\sum\limits_{i=1}^{t-2}
							\tfrac{1}{2}(n-\ell_{i})(n-1-\ell_{i}) s_i	}	\\
			\nonumber&\qquad 	- 
			\sum_{\crampedsubstack{
				2\le j\le n-3 \\
				\Set*{j,j+1}\cap I = \emptyset
			}}q^{ (n-j)^2-n }
			- 
			(1+\delta_{I}(n-2))q^{ n-1 },
		\end{align}
		where $\delta_{I}(i)=0$ if $i\in I$ and $1$ is not. 
		Note that the definition of the multivariable parity function $e_{\mcal{X}^{0\heartsuit}[I]}$ is in \cref{eq:Dn:DefineEpsilonX:l1}.
	\end{substep}

	\begin{substep}\label{step:Dn:AsymptoticOfSVI:e1}
		When $\ell_{1}=1$, by \cref{figure:VerticesOfIDne1e}, we have (including the zero summations)
		\begin{align*}
			\fun{S}_{\mcal{V}[I]}[r] &= 
			\fun{S}_{\mcal{X}^{00}[I]}[r] + 
			\fun{S}_{\mcal{X}^{01}[I]}[r] + 
			\fun{S}_{\mcal{X}^{10}[I]}[r] + 
			\fun{S}_{\mcal{X}^{11}[I]}[r] \\
			&\qquad -
			\fun{S}_{\mcal{X}_{\Set*{1}}[I]}[r] -
			\fun{S}_{\mcal{X}_{\Set*{1,2}}[I]}[r] -
			\sum_{j=2}^{n-3}
			\fun{S}_{\mcal{X}_{\Set*{j,j+1}}[I]}[r] \\
			&\qquad -
			\fun{S}_{\mcal{X}_{\Set*{n-1,n}}[I]}[r] -
			\fun{S}_{\mcal{X}_{\Set*{n-2,n-1,n}}[I]}[r].
		\end{align*} 
		Therefore, by \cref{eq:Dn:CXI:e1,eq:Dn:EXI:e1,eq:Dn:AsymptoticOfSXI:e1,eq:Dn:AsymptoticOfSXJI:e1:1,eq:Dn:AsymptoticOfSXJI:e1:12,eq:Dn:AsymptoticOfSXJI:e1:j,eq:Dn:AsymptoticOfSXJI:e1:nn,eq:Dn:AsymptoticOfSXJI:e1:nnn}, we have 
		\begin{align}\label{eq:Dn:AsymptoticOfSVI:e1}
			\fun{S}_{\mcal{V}[I]}[r] &\sim 
			\left(
				q^{	(n-4)(n-1)	}-1
			\right)^{-1}
			\prod_{i=2}^{t-2}
				\left(	q^{	(n-\ell_{i})(n-1-\ell_{i})	}-1	\right)^{-1}
			\cdot
			C_{I}
			\cdot rq^{	\tfrac{n(n-1)}{2} r},
		\end{align}
		where the constant $C_{I}$ is defined as follows:
		\begin{align}\label{eq:Dn:AsymptoticOfSVI:C:e1}
			C_{I} &= 
			\sum_{\square,\heartsuit=0,1}
			\sum_{s_1,\cdots,s_{t-2}\in\F_2}
				q^{	e_{\mcal{X}^{\square\heartsuit}[I]}(\vect{s}) - 
						(n-1)\cdot\square + 
						\tfrac{(n-4)(n-1)}{2} s_1 + 
						\sum\limits_{i=2}^{t-2}
							\tfrac{1}{2}(n-\ell_{i})(n-1-\ell_{i}) s_i	}\\
			\nonumber&\qquad 	- 
			\Bigg(
				1+
				\delta_{I}(2)q^{	\tfrac{n(n-1)}{2}-(2n-3)	} + 
				\sum_{\crampedsubstack{
					2\le j\le n-3 \\
					\Set*{j,j+1}\cap I = \emptyset
				}}q^{ (n-j)^2-n } \\
			\nonumber&\qquad 
				+ 
				(1+\delta_{I}(n-2))q^{ n-1 }
			\Bigg)
			\cdot
			\left(1+q^{\tfrac{n(n-1)}{2}-(2n-2)}\right),
		\end{align}
		Note that the definition of the multivariable parity function $e_{\mcal{X}^{\square\heartsuit}[I]}$ is in \cref{eq:Dn:DefineEpsilonX:e1}.
	\end{substep}
\end{step}

\subsection{Asymptotic growths of \texorpdfstring{$\fun{SSA}[r]$}{SSA} and \texorpdfstring{$\fun{SV}[r]$}{SV}}\label{subsec:Dn:AsymptoticOfS}
We are now able to obtain the asymptotic growth of $\fun{SSA}[r]$. 
By \cref{eq:SimplicialSurfaceAreaFormula}, we have 
\begin{equation}\label{eq:SSA=SumSI:Dn}
	\fun{SSA}[r] = 
	\sum_{I\subset\Delta}
	\frac{
		\mscr{P}_{D_n;I}[q]
	}{
		q^{\fun{deg}[\mscr{P}_{D_n;I}]}
	}\fun{S}_{\mcal{V}[I]}[r]
	\sim 
	\sum_{I\text{ is dominant}}
	\frac{
		\mscr{P}_{D_n;I}[q]
	}{
		q^{\fun{deg}[\mscr{P}_{D_n;I}]}
	}\fun{S}_{\mcal{V}[I]}[r].
\end{equation}
What remains is to plug in the asymptotic growth of dominant $\fun{S}_{\mcal{V}[I]}[r]$. 
We will separate the discussion into  two cases: \cref{step:D4:AsymptoticOfS} $n=4$ and \cref{step:Dn:AsymptoticOfS} $n\ge 5$.

\begin{step}\label{step:D4:AsymptoticOfS}
	When $n=4$, the only dominant types are $\Set*{2}$ and $\emptyset$. 
	Then by \cref{eq:D4:AsymptoticOfSVI:2,eq:D4:AsymptoticOfSVI,eq:PoincareDnIs}, we have 
	\begin{equation}\label{eq:D4:Asymp:SSA}
		\fun{SSA}[r] \sim  
			C(4) \cdot\binom{r}{2}q^{	6 r	},
	\end{equation}
	where the constant $C(4)$ is defined as follows: 
	\begin{align}\label{eq:D4:Asymp:C}
		C(4) &:= 
		\frac{
			\mscr{P}_{D_4;\Set*{2}}[q]
		}{
			q^{\fun{deg}[\mscr{P}_{D_4;\Set*{2}}]}
		}
		(q+1) + 
		\frac{
			\mscr{P}_{D_4;\emptyset}[q]
		}{
			q^{\fun{deg}[\mscr{P}_{D_4;\emptyset}]}
		}		
		\frac{
			q^5+q^2+q+1
		}{
			q^2-1
		} \\
	\nonumber &=
		\frac{
			\left(q^6-1\right)\left(q^4-1\right)^2
			\left(q+1\right)
		}{
			\left(q-1\right)^{3} q^{11}
		} + 
		\frac{ 
			\left(q^6-1\right)\left(q^4-1\right)^2
			\left(q^5+q^2+q+1\right)
		}{
			\left(q-1\right)^{4} q^{12}
		} \\
	\nonumber &=
		\frac{
			\left(q^2+q+1\right)
			\left(q^2-q+1\right)^2
			\left(q^2+1\right)^3 
			(q+1)^4 
		}{
			(q-1) q^{12}
		}.
	\end{align}
	As a consequence, we have 
	\begin{equation}\label{eq:D4:Asymp:SV}
		\fun{SV}[r] = \sum_{z=0}^{r}\fun{SSA}[z] \sim  
			\frac{q^6}{q^6-1}C(4) \cdot \binom{r}{2}q^{	6 r	}.
	\end{equation}
\end{step}

\begin{step}\label{step:Dn:AsymptoticOfS}
	Now, we assume $n\ge 5$. Then $I$ is dominant exactly when $\Set*{n-1,n}\cap I = \emptyset$. 
	By \cref{eq:Dn:AsymptoticOfSVI:l1,eq:Dn:AsymptoticOfSVI:e1}, we have 
	\begin{equation}\label{eq:Dn:Asymp:SSA}
		\fun{SSA}[r] \sim  
			C(n) \cdot rq^{	\tfrac{n(n-1)}{2} r},
	\end{equation}
	where the constant $C(n)$ is defined as follows: 
	\index[notation]{C (n)@$C(n)$}%
	\begin{align}\label{eq:Dn:Asymp:C}
		C(n)(r) &:=
			\sum_{1,n-1,n\notin I}
				\frac{
					\mscr{P}_{D_n;I}[q]
				}{
					q^{\fun{deg}[\mscr{P}_{D_n;I}]}
				}
				\prod_{i=2}^{t-2}
					\left(	q^{	(n-\ell_{i}(I))(n-1-\ell_{i}(I))	}-1	\right)^{-1}
				\cdot
				\frac{
					C_{I}
				}{
					q^{	(n-4)(n-1)	}-1
				}
		\\
	\nonumber&\qquad +
			\sum_{1\in I, n-1,n\notin I}
				\frac{
					\mscr{P}_{D_n;I}[q]
				}{
					q^{\fun{deg}[\mscr{P}_{D_n;I}]}
				}
				\prod_{i=1}^{t-2}
					\left(	q^{	(n-\ell_{i}(I))(n-1-\ell_{i}(I))	}-1	\right)^{-1}
				\cdot
				C_{I}.
	\end{align}
	As a consequence, we have 
	\begin{equation}\label{eq:Dn:Asymp:SV}
		\fun{SV}[r] = 
		\sum_{z=0}^{r}\fun{SSA}[z] \sim 
			\frac{
				q^{\tfrac{n(n-1)}{2}}
			}{
				q^{\tfrac{n(n-1)}{2}} - 1
			}C(n)
			\cdot
			q^{	\tfrac{n(n-1)}{2} r	}.
	\end{equation}

	\begin{remark}
		Note that the constant $C_{I}$ depends on $I$. 
		When $1\in I$ and $n-1,n\notin I$, it is defined in \cref{eq:Dn:AsymptoticOfSVI:C:l1}. 
		When $1,n-1,n\notin I$, it is defined in \cref{eq:Dn:AsymptoticOfSVI:C:e1}. 
	\end{remark}

	By \cref{eq:D4:Asymp:SSA,eq:D4:Asymp:C,eq:D4:Asymp:SV,eq:Dn:Asymp:SSA,eq:Dn:Asymp:C,eq:Dn:Asymp:SV}
	we have proved \cref{thm:Asymp:Dn}. 
	Moreover, by \cref{eq:PoincareDnIs}, we have the following explicit formulas:
	\begin{align}\label{eq:Dn:PoincareFactors}
		\mscr{P}_{D_n;I}[q]
		&= 
			\frac{
				[2(n-1)]!!(z)\cdot[n](z)
			}{
				\prod\limits_{i=1}^{t-1}
					[\ell_{i}(I)-\ell_{i-1}(I)]!(z)
			},
		&
		q^{\fun{deg}[\mscr{P}_{D_n;I}]}
		&=
			\frac{
				q^{n(n-1)}
			}{
				\prod\limits_{i=1}^{t-1}
					q^{\binom{\ell_{i}(I)-\ell_{i-1}(I)}{2}}
			}.
	\end{align}
	See \cref{lem:PoincareOfXn,eq:quantum_factorial,eq:PoincareCn} for the definitions of the symbols $[\:\cdot\:]$, $[\:\cdot\:]!$, and $[2\:\cdot\:]!!$.
\end{step}


\clearpage
\appendix
\section{Poincar\'{e} polynomials of irreducible root systems}\label{sec:Poincares}
In this appendix, we will work out a closed formula for the Poincar\'{e} polynomial $\mscr{P}_{X_{n};I}$ of each irreducible reduced root system $\Phi$ of type $X_n$ and each type $I$. 

\subsection{Poincar\'{e} polynomials of \texorpdfstring{$A_{n}$}{An}}
First, it is clear that $\mscr{P}_{A_{0}}[z]=\mscr{P}_{\emptyset}[z]=1$. 
We then assume that $n\ge 1$. 

Let $\Phi$ be a root system of type $A_{n}$. 
Then the Dynkin diagram with the label of simple roots in $\Phi$ is the following one: 
\begin{equation*}
	\dynkin[
				labels={1,2,n-1,n},
				label macro/.code={a_{\drlap{#1}}},
			] A{}
\end{equation*}
By \cite{Bourbaki}*{chap.VI, \S 4, no.7}, 
the degrees of its Weyl group are $d_{i}=i+1$. 
Therefore, by \cref{lem:PoincareOfXn}, we have 
\begin{equation}
	\label{eq:PoincareAn}
	\mscr{P}_{A_{n}}[z] = 
		\prod_{i=1}^n[i+1](z).
\end{equation}
In particular, $\fun{deg}[\mscr{P}_{A_{n}}]=\binom{n+1}{2}$. 
Note that $[1](z)$ is the constant $1$. 
Hence, $\mscr{P}_{A_{n}}[z]$ equals to the following 
\emph{$z$-factorial} polynomial: 
\index{z-factorial@$z$-factorial}%
\index[notation]{n factorial@$[n]"!(z)$}%
\begin{equation}\label{eq:quantum_factorial}
	[n+1]!(z) := \prod_{i=1}^{n+1}[i](z).
\end{equation}
We also need the following \emph{$z$-multinomial} polynomial:
\index{z-multinomial@$z$-multinomial}%
\index[notation]{n multinomial@$\qbinom{n}{n_0,\cdots,n_k}(z)$}%
\begin{equation}\label{eq:quantum_multinomial}
	\qbinom{n}{n_0,\cdots,n_k}(z) := 
		\frac{
			[n]!(z)
		}{
			[n_0]!(z)\cdots[n_k]!(z)
		},
\end{equation}
where $n_0+\cdots+n_k=n$ is a partition of $n$ into natural numbers.
	
Let $I$ be a type and follow \cref{con:type}. 
Then the Dynkin diagram of the subsystem $\Phi_{I}$ with labels is the following one:
\begin{equation*}
	\dynkin[
				labels={1,\ell_{1}-1,\ell_{1},\ell_{1}+1,\ell_{i}-1,\ell_{i},\ell_{i}+1,n},
				label macro/.code={a_{\drlap{#1}}},
			] A{*.*X*.*X*.*}
\end{equation*}
Hence, $\Phi_{I}$ is of type 
$A_{\ell_{1}-\ell_{0}-1} \times\cdots\times A_{\ell_{t+1}-\ell_{t}-1}$, where $\ell_{t+1}$ is defined to be $n+1$. 
\index[notation]{l t+1@$\ell_{t+1}=n+1$}%
Note that $(\ell_{1}-\ell_{0}) + \cdots + (\ell_{t+1}-\ell_{t}) = n+1$.
Then we have
\begin{equation}
	\label{eq:PoincareAnI}
	\mscr{P}_{A_{n},I}[z] =
		\frac{
			\mscr{P}_{A_{n}}[z]
		}{
			\prod\limits_{i=1}^{t+1}\mscr{P}_{A_{\ell_{i}-\ell_{i-1}-1}}[z]
		} =  
		\qbinom{n+1}{\ell_{1}-\ell_{0},\cdots,\ell_{t+1}-\ell_{t}}(z).
\end{equation}
In particular, $\fun{deg}[\mscr{P}_{A_{n},I}]=\binom{n+1}{2}-\sum\limits_{i=1}^{t+1}\binom{\ell_{i}-\ell_{i-1}}{2}$.

\subsection{Poincar\'{e} polynomials of \texorpdfstring{$B_{n}$}{Bn} and \texorpdfstring{$C_{n}$}{Cn}}
First note that the two types of root systems share the same Weyl group, hence the same Poincar\'{e} polynomial. 
It suffices to only consider one of them. 
We will consider $C_{n}$. 
	
Let $\Phi$ be a root system of type $C_{n}$. 
Then the Dynkin diagram with the label of simple roots in $\Phi$ is the following one: 
\begin{equation*}
	\dynkin[
				labels={1,2,n-2,n-1,n},
				label macro/.code={a_{\drlap{#1}}},
			] C{}
\end{equation*}
When $n=0$ or $1$, we can see that $C_{n}=A_{n}$. 
We then assume that $n\ge 2$. 
By \cite{Bourbaki}*{chap.VI, \S 4, no.5 and no.6}, 
the degrees of its Weyl group are $d_{i}=2i$. 
Therefore, by \cref{lem:PoincareOfXn}, we have 
\begin{equation}
	\label{eq:PoincareCn}
	\mscr{P}_{C_{n}}[z] = \prod_{i=1}^n[2i](z).
\end{equation}
\index[notation]{2n factorial factorial@$[2n]"!"!(z)$}%
\index[notation]{0 factorial factorial@$[0]"!"!$}%
In particular, $\fun{deg}[\mscr{P}_{C_{n}}] = n^2$.
We use $[2n]!!(z)$ to denote the right-hand side and use the convention that $[0]!! = 1$.
	
Let $I$ be a type and follow \cref{con:type}. 
Then the Dynkin diagram of the subsystem $\Phi_{I}$ with labels is one of the following three: (focusing on position of $\ell_{t}$)
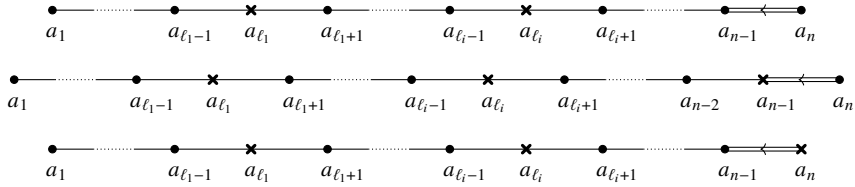
\begin{gather*}
	\dynkin[
				labels={1,\ell_{1}-1,\ell_{1},\ell_{1}+1,\ell_{i}-1,\ell_{i},\ell_{i}+1,n-1,n},
				label macro/.code={a_{\drlap{#1}}},
			] C{*.*X*.*X*.**} \\
	\dynkin[
				labels={1,\ell_{1}-1,\ell_{1},\ell_{1}+1,\ell_{i}-1,\ell_{i},\ell_{i}+1,n-2,n-1,n},
				label macro/.code={a_{\drlap{#1}}},
			] C{*.*X*.*X*.*X*} \\
	\dynkin[
				labels={1,\ell_{1}-1,\ell_{1},\ell_{1}+1,\ell_{i}-1,\ell_{i},\ell_{i}+1,n-1,n},
				label macro/.code={a_{\drlap{#1}}},
			] C{*.*X*.*X*.*X}
\end{gather*}
In either case, $\Phi_{I}$ is of type 
$A_{\ell_{1}-\ell_{0}-1} \times\cdots\times A_{\ell_{t}-\ell_{t-1}-1} \times C_{n-\ell_{t}}$ 
(notice that $C_{0}=A_{0}$ and $C_{1}=A_{1}$).
Then we have
\begin{align}
	\label{eq:PoincareCnI}
	\mscr{P}_{C_{n},I}[z] &=
		\frac{
			\mscr{P}_{C_{n}}[z]
		}{
			\prod\limits_{i=1}^{t}
				\mscr{P}_{A_{\ell_{i}-\ell_{i-1}-1}}[z]^{-1}
			\cdot\mscr{P}_{C_{n-\ell_{t}}}[z]
		} \\
	\nonumber &= 
		\frac{
			[2n]!!(z)
		}{
			\prod\limits_{i=1}^{t}
				[\ell_{i}-\ell_{i-1}]!(z)
			\cdot[2(n-\ell_{t})]!!(z)
		}.
\end{align}
In particular, 
$\fun{deg}[\mscr{P}_{C_{n},I}] = n^2 - \sum\limits_{i=1}^{t}\binom{\ell_{i}-\ell_{i-1}}{2} - (n-\ell_{t})^2$.

\subsection{Poincar\'{e} polynomials of \texorpdfstring{$D_{n}$}{Dn}}
\label{subsec:PoincareDn}
Let $\Phi$ be a root system of type $D_{n}$. 
Then the Dynkin diagram with the label of simple roots in $\Phi$ is the following one: 
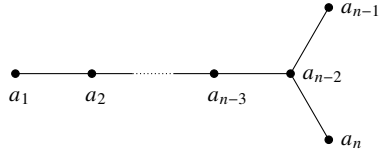
\begin{equation*}
	\dynkin[
				labels={1,2,n-3,n-2,n-1,n},
				label macro/.code={a_{\drlap{#1}}},
				label directions={,,,right,,}
			] D{}
\end{equation*}
When $n=0$, $1$, or $3$, we can see that $D_{n}=A_{n}$. 
When $n=2$, we have $D_{2} = A_{1}\times A_{1}$.  
We then assume that $n\ge 4$. 
\cite{Bourbaki}*{chap.VI, \S 4, no.8}, 
the degrees of $W$ are $d_{i}=2i$ for $i<n$ and $d_{n}=n$. 
Therefore, by \cref{lem:PoincareOfXn}, we have 
\begin{equation}
	\label{eq:PoincareDn}
	\mscr{P}_{D_{n}}[z] = 
		\prod_{i=1}^{n-1}[2i](z)\cdot[n](z) = 
		[2(n-1)]!!(z)\cdot[n](z). 
\end{equation}
In particular, $\fun{deg}[\mscr{P}_{D_{n}}]=n(n-1)$. 

Let $I$ be a type and follow \cref{con:type}. 
When $\ell_{t} < n-2$, the Dynkin diagram of the subsystem $\Phi_{I}$ with labels is one of the followings: (focusing on position of $\ell_{t}$)
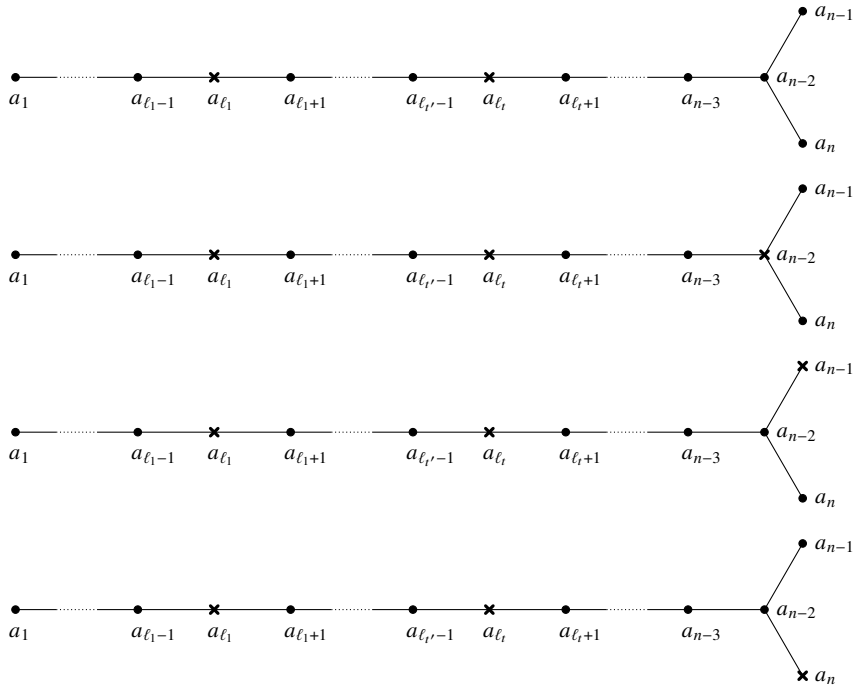
\begin{gather*}
	\dynkin[
		labels={1,\ell_{1}-1,\ell_{1},\ell_{1}+1,\ell_{t'}-1,\ell_{t},\ell_{t}+1,n-3,n-2,n-1,n},
		label macro/.code={a_{\drlap{#1}}},
		label directions={,,,,,,,,right,,}
	] D{*.*X*.*X*.****} \\			
	\dynkin[
		labels={1,\ell_{1}-1,\ell_{1},\ell_{1}+1,\ell_{t'}-1,\ell_{t},\ell_{t}+1,n-3,n-2,n-1,n},
		label macro/.code={a_{\drlap{#1}}},
		label directions={,,,,,,,,right,,}
	] D{*.*X*.*X*.*X**} \\		
	\dynkin[
		labels={1,\ell_{1}-1,\ell_{1},\ell_{1}+1,\ell_{t'}-1,\ell_{t},\ell_{t}+1,n-3,n-2,n-1,n},
		label macro/.code={a_{\drlap{#1}}},
		label directions={,,,,,,,,right,,}
	] D{*.*X*.*X*.**X*} \\		
	\dynkin[
		labels={1,\ell_{1}-1,\ell_{1},\ell_{1}+1,\ell_{t'}-1,\ell_{t},\ell_{t}+1,n-3,n-2,n-1,n},
		label macro/.code={a_{\drlap{#1}}},
		label directions={,,,,,,,,right,,}
	] D{*.*X*.*X*.***X}
\end{gather*}

In the first two cases, $\Phi_{I}$ is of the type 
$A_{\ell_{1}-\ell_{0}-1} \times \cdots \times A_{\ell_{t}-\ell_{t-1}-1} \times D_{n-\ell_{t}}$ (Noticing that is $D_{n-\ell_{t}}$ if $n-\ell_{t} < 4$).  
Therefore, we have 
\begin{align}
	\label{eq:PoincareDnI}
	\mscr{P}_{D_{n},I}[z] &=
		\frac{
			\mscr{P}_{D_{n}}[z]
		}{
			\prod\limits_{i=1}^{t}
				\mscr{P}_{A_{\ell_{i}-\ell_{i-1}-1}}[z]^{-1}
			\cdot\mscr{P}_{D_{n-\ell_{t}}}[z]
		} \\
	\nonumber &= 
		\frac{
			[2(n-1)]!!(z)\cdot[n](z)
		}{
			\prod\limits_{i=1}^{t}
				[\ell_{i}-\ell_{i-1}]!(z)
			\cdot[2(n-\ell_{t}-1)]!!(z)
			\cdot[n-\ell_{t}](z)
		}.
\end{align}
In particular, 
$\fun{deg}[\mscr{P}_{D_{n},I}] = n(n-1) - \sum\limits_{i=1}^{t}\binom{\ell_{i}-\ell_{i-1}}{2} - (n-\ell_{t})(n-\ell_{t}-1)$.

In the last two cases, $\Phi_{I}$ is of the type 
$A_{\ell_{1}-\ell_{0}-1} \times \cdots \times A_{\ell_{t-1}-\ell_{t-2}-1} \times A_{n-\ell_{t-1}-1}$.
Therefore, we have 
\begin{align}
	\label{eq:PoincareDnIs}
	\mscr{P}_{D_{n},I}[z] &=
		\frac{
			\mscr{P}_{D_{n}}[z]
		}{
			\prod\limits_{i=1}^{t-1}
				\mscr{P}_{A_{\ell_{i}-\ell_{i-1}-1}}[z]^{-1}
			\cdot\mscr{P}_{A_{n-\ell_{t-1}-1}}[z]
		} \\
	\nonumber &= 
		\frac{
			[2(n-1)]!!(z)\cdot[n](z)
		}{
			\prod\limits_{i=1}^{t-1}
				[\ell_{i}-\ell_{i-1}]!(z)
			\cdot[n-\ell_{t-1}](z)
		}.
\end{align}
In particular, 
$\fun{deg}[\mscr{P}_{D_{n},I}] = n(n-1) - \sum\limits_{i=1}^{t-1}\binom{\ell_{i}-\ell_{i-1}}{2} - \binom{n-\ell_{t-1}}{2}$.

\section*{Acknowledgements}
\addcontentsline{toc}{section}{Acknowledgements}
The author thanks his PhD thesis advisor Prof. Junecue Suh for his advices and useful discussions. 
The author also thanks Prof. Martin Weissman for his expert answers to several technical questions on Bruhat-Tits buildings.


\phantomsection



\printindex
\printindex[notation]

\end{document}